\documentclass[10pt]{article}
\usepackage{graphics}
\usepackage[pdftex]{hyperref}
\usepackage{amssymb,amsmath,amsfonts}
\usepackage{mathtools}

\DeclarePairedDelimiter\floor{\lfloor}{\rfloor}
\usepackage{datetime}
\usepackage{color}
\usepackage{ulem}
\usepackage{lipsum}
\usepackage{amsfonts}
\usepackage{enumerate}
\usepackage{cite}
\usepackage{tikz}
\usetikzlibrary{matrix}
\usepackage[pdftex]{hyperref}
\RequirePackage{amscd}
\RequirePackage{epic}
\RequirePackage[all]{xy}
\RequirePackage{url}
\RequirePackage[shortlabels]{enumitem}
\usepackage{empheq}
\usepackage{epsfig}
\usepackage{lineno} % per numerare le righe 
\usepackage{perpage}
\usepackage{upgreek}
\usepackage{bbm}

\MakePerPage{linenumbers}
\usepackage[english]{babel}

\usepackage[utf8]{inputenc}
\usepackage{cite}
\RequirePackage{amscd}
\RequirePackage{epic}
\RequirePackage{eepic}
\RequirePackage[all]{xy}
\RequirePackage{url}
\RequirePackage[shortlabels]{enumitem}
\usepackage{empheq}
\usepackage{nomencl}
\usepackage{epsfig}
\usepackage{graphicx}
\usepackage{stix}

\def\beq{\begin{equation}}
\def\eeq{\end{equation}}

\newcommand{\eps}{\varepsilon}

\newcommand{\gradl}{	\left.\frac{\partial P}{\partial x_\ell}\right|_{\jet(t)}}

\newcommand{\sommar}{\sum_{\substack{\mu\in\N^m\\  2\le |\mu|\le r}}}

\newcommand{\comment}[1]{}

\newcommand{\abs}[1]{{\left|#1\right|}}
\newcommand{\norm}[1]{{\left |#1\right |}}
\newcommand{\T}{\mathbb{T}}

\newcommand{\R}{\mathbb{R}}
\newcommand{\C}{\mathbb{C}}

\newcommand{\sgn}{\operatorname{sgn}}

\newcommand{\Omegone}{\boldsymbol{\Omega}_n^{r,\ts}}
\newcommand{\OmegoneI}{\boldsymbol{\Omega}_n^{r,\ts}}
\newcommand{\OmegoneIsprimo}{\boldsymbol{\Omega}_n^{r,\ts'}}

\newcommand{\co}[1]{\textit{#1}}%corsivo
\newcommand{\id}{\operatorname{id}}

\usepackage{amsthm}

\newcommand{\arco}{\Theta_m}
\newcommand{\arcox}{\Theta_m^1}
\newcommand{\arcoi}{\Theta_m^i}
\newcommand{\centina}{\vartheta_m(s)}
\newcommand{\centinax}{\vartheta^1_m(s)}

\newcommand{\centinai}{\vartheta^i_m(s)}
\newcommand{\centinadue}{\vartheta^1_{m}(s,2)}
\newcommand{\centinauno}{\vartheta^1_m(1)}

\newcommand{\centinaduei}{\vartheta^i_{m}(s,2)}
\newcommand{\arcoh}{\widehat \Theta_m}
\newcommand{\arcoxh}{\widehat \Theta_m^1}
\newcommand{\arcoih}{\widehat \Theta_m^i}
\newcommand{\centinah}{\widehat \vartheta_m(s)}
\newcommand{\centinaxh}{\widehat \vartheta^1_m(s)}

\newcommand{\centinaih}{\widehat \vartheta^i_m(s)}
\newcommand{\centinadueh}{ \vartheta^1_{m}(s,2)}
\newcommand{\centinaunoh}{\widehat \vartheta^1_m(1)}

\newcommand{\centinadueih}{\widehat \vartheta^i_{m}(s,2)}
\newcommand{\centinaunoih}{\widehat \vartheta^i_m(1)}
\newcommand{\Hp}{\mathbb{H}^1(P)}
\newcommand{\Hpa}{\mathbb{H}^1(\tP_\ta)}

\newcommand{\Gp}{\mathbb{G}^1(P)}

\newcommand{\Hpi}{\mathbb{H}^i(P)}

\newcommand{\sommalphater}{\sum_{\substack{\mu\in\cE_m(\ell,\alpha)\\  \mu\in\mathcal{M}_m(\alpha)\\\mu_\ell\neq 0 }}}

\newcommand{\sommalphateri}{\sum_{\substack{\mu\in\cE_m(i,\alpha)\\  \mu\in\mathcal{M}_m(\alpha)\\\mu_i\neq 0 }}}

\newcommand{\sommalphateruno}{\sum_{\substack{\mu\in\cE_m(1,\alpha)\\  \mu\in\mathcal{M}_m(\alpha)\\\mu_1\neq 0 }}}

\newcommand{\sommalphaquater}{\sum_{i=2}^{m}\sum_{\beta=1}^{\alpha-1}}

\newcommand{\mtj}{\widetilde{\mu}_j}

\newcommand{\twg}{\mathcal T(h,\Gamma^m)}
\newcommand{\twgpo}{\mathcal T(\tT_0(h,r,n),\Gamma^m)}

\newcommand{\twgq}{\mathcal T(Q,\Gamma^m)}
\newcommand{\twgqxi}{\mathcal T_\xi(Q,\Gamma^m)}

\RequirePackage{url}
\newtheorem{prop}{Proposition}[section]
\newtheorem{thm}{Theorem}
\newtheorem*{thm*}{Theorem}
\newtheorem*{cor*}{Corollary}

\newtheorem{cor}{Corollary}
\newtheorem{lemma}{Lemma}
\theoremstyle{remark}
\newtheorem{ex}{Example}
\newtheorem{rmk}{Remark}[section]
\theoremstyle{definition}
\newtheorem{defn}{Definition}

\newcommand\myoverset[2]{\overset{\textstyle #1\mathstrut}{#2}}

\numberwithin{equation}{section}
\numberwithin{thm}{section}
\numberwithin{defn}{section}
\numberwithin{prop}{section}
\numberwithin{cor}{section}
\numberwithin{lemma}{section}
\numberwithin{rmk}{section}
\newcommand{\alg}{{\bm{\alpha}}}

\newcommand{\scA}{{\mathscr{A}}}
\newcommand{\scK}{{\mathscr{K}}}

\newcommand{\sL}{{\mathsf{\Lambda}}}

\newcommand{\N}{{\mathbb N}}

\renewcommand{\S}{{\mathbb S}}
\newcommand{\cA}{{\mathcal A}}

\newcommand{\cC}{{\mathcal C}}
\newcommand{\cD}{{\mathcal D}}
\newcommand{\cE}{{\mathcal E}}

\newcommand{\cG}{{\mathcal G}}
\newcommand{\cH}{{\mathcal H}}
\newcommand{\cI}{{\mathcal I}}
\newcommand{\cJ}{{\mathcal J}}

\newcommand{\cL}{{\mathcal L}}
\newcommand{\cM}{{\mathcal M}}

\newcommand{\cP}{{\mathcal P}}

\newcommand{\cR}{{\mathcal R}}
\newcommand{\cS}{{\mathcal S}}
\newcommand{\cT}{{\mathcal T}}
\newcommand{\cU}{{\mathcal U}}
\newcommand{\cV}{{\mathcal V}}
\newcommand{\cW}{{\mathcal W}}

\newcommand{\cZ}{{\mathcal Z}}

\newcommand{\jet}{{\cJ_{s,\gamma}}}

\newcommand{\jetm}{{\cJ_{s_m,\gamma}}}
\newcommand{\jeta}{{\mathtt J_{s,\gamma,\ta}}}

\newcommand{\jetam}{{\mathtt J_{s_m,\gamma,\ta}}}

\newcommand{\jetunoa}{{\tJ_{1,\gamma,\ta}}}
\newcommand{\jeth}{{\cJ_{s,\gamma}}}

\newcommand{\jetah}{{\mathtt J_{s,\gamma,\ta}}}
\newcommand{\jetamh}{{\widehat\mathtt J_{s_m,\gamma,\ta}}}

\usepackage{float}
\usepackage{aliascnt}
\newaliascnt{eqfloat}{equation}
\newfloat{eqfloat}{h}{eqflts}
\floatname{eqfloat}{Table}

\newcommand*{\ORGeqfloat}{}
\let\ORGeqfloat\eqfloat
\def\eqfloat{%
	\let\ORIGINALcaption\caption
	\def\caption{%
		\addtocounter{equation}{-1}%
		\ORIGINALcaption
	}%
	\ORGeqfloat
}

\newcommand{\fB}{{\mathfrak{B}}}

\newcommand{\fD}{{\mathfrak{D}}}

\newcommand{\fU}{{\mathfrak{U}}}

\newcommand{\ta}{{\mathtt{a}}}
\newcommand{\tb}{{\mathtt{b}}}

\newcommand{\td}{{\mathtt{d}}}

\newcommand{\tf}{{\mathtt{f}}}

\newcommand{\tti}{{\mathtt{t}}}
\newcommand{\tu}{{\mathtt{u}}}

\usepackage[LGR,T1]{fontenc} % to enable Greek encoding

\DeclareSymbolFont{ttgreek}{LGR}{cmtt}{m}{n}
\DeclareMathSymbol{\tgamma}{\mathord}{ttgreek}{`g}
\DeclareMathSymbol{\tmu}{\mathord}{ttgreek}{`m}

\newcommand{\ti}{{\mathtt{i}}}
\newcommand{\tj}{{\mathtt{j}}}
\newcommand{\tJ}{{\mathtt{J}}}

\newcommand{\tp}{{\mathtt{p}}}
\newcommand{\tir}{{\mathtt{r}}}

\newcommand{\ty}{{\mathtt{y}}}

\usepackage{scalerel,stackengine, amssymb}

\newcommand{\tA}{{\mathtt{A}}}
\newcommand{\tB}{{\mathtt{B}}}
\newcommand{\tC}{{\mathtt{C}}}
\newcommand{\tD}{{\mathtt{D}}}

\newcommand{\tF}{{\mathtt{F}}}
\newcommand{\tG}{{\mathtt{G}}}
\newcommand{\tH}{{\mathtt{H}}}
\newcommand{\tK}{{\mathtt{K}}}
\newcommand{\tI}{{\mathtt{I}}}
\newcommand{\tL}{{\mathtt{L}}}
\newcommand{\tM}{{\mathtt{M}}}

\newcommand{\tP}{{\mathtt{P}}}
\newcommand{\tQ}{{\mathtt{Q}}}
\newcommand{\tR}{{\mathtt{R}}}
\newcommand{\tS}{{\mathtt{S}}}
\newcommand{\ts}{{\mathtt{s}}}
\newcommand{\tT}{{\mathtt{T}}}
\newcommand{\tU}{{\mathtt{U}}}

\newcommand{\tY}{{\mathtt{Y}}}

\usepackage{bm}

\newcommand{\grad}{\nabla}

\newcommand\norma[1]{\left\lVert#1\right\rVert}
\newcommand{\sommalpha}{\sum_{\substack{\mu\in\N^m\\  2\le |\mu|\le \alpha+1}}}

\newcommand{\sommalphabis}{\underset{\substack{\nu_\ell(i,\beta)\in \cE_m(\ell,\alpha)\\ \nu_\ell(i,\beta)\neq 0}}{\sum_{i=1}^{m}\sum_{\beta=1}^{\alpha}}}

\usepackage{tipa}
\usepackage{upgreek}

\newcommand{\meas}{{\rm meas}}

\newcommand{\diag}{{\rm diag}}

\newcommand{\sA}{{\mathsf{A}}}
\newcommand{\sS}{{\mathsf{S}}}
\newcommand{\nucleo}{{\mathscr{N}}}
\newcommand{\scS}{{\mathscr{S}}}
\newcommand{\scV}{{\mathscr{V}}}

\newcommand{\nnorm}[1]{{\left\vert\kern-0.25ex\left\vert\kern-0.25ex\left\vert #1 
		\right\vert\kern-0.25ex\right\vert\kern-0.25ex\right\vert}}
\usepackage{framed,enumitem} 

\newcommand{\Pol}{\mathcal{P}(r,n)}
\newcommand{\Pollo}{\mathcal{P}^\star(r,n)}
\newcommand{\Polpo}{\mathcal{P}^\star(r',n)}
\newcommand{\Polm}{\mathcal{P}(r,m)}

\title{\bf 
	
	Semi-algebraic Geometry and generic Hamiltonian stability\\ \ \\
	}

\begin{document}

\author{ 
%Nota di Luca Biasco e Luigi Chierchia \\{\bf Scientific chapter:} {\sl Mathematical analysis.} \\
 S. Barbieri\\ 
 \\ \footnotesize Universitat de Barcelona
\\ \footnotesize Gran Via de les Corts Catalanes, 585 - 08007 Barcelona, Spain
\\
{\footnotesize barbieri@ub.edu}
\\ 
}

\maketitle



\tableofcontents

\section*{Abstract}

The steepness property is a local geometric transversality condition on the gradient of a $C^2$-function which proves fundamental in order to ensure the stability of sufficiently-regular nearly-integrable Hamiltonian systems over long timespans. Steep functions were originally introduced by Nekhoroshev, who also proved their genericity. Namely, given a pair of positive integers $r,n$, with $r$ high enough, and a point $I_0\in \R^n$, the Taylor polynomials of those $C^{2r-1}$ functions which are not steep around $I_0$ are contained in a semi-algebraic set of positive codimension in the space of polynomials of $n$ variables and degree bounded by $r$. The demonstration of this result was originally published in 1973 and has been hardly studied ever since, probably due to the fact that it involves no arguments of dynamical systems: it makes use of quantitative reasonings of real-algebraic geometry and complex analysis.   The aim of the present work is two-fold. In the first part, the original proof of the genericity of steepness is rewritten by making use of modern tools of real-algebraic geometry: this allows to clarify the original reasonings, that were obscure or sketchy in many parts. In particular, Yomdin's Lemma on the analytic reparametrization of semi-algebraic sets, together with non trivial estimates on the codimension of certain algebraic varieties, turn out to be the fundamental ingredients to prove the genericity of steepness. The second part of this work is entirely new and is devoted to the formulation of explicit algebraic criteria to check steepness of any given sufficiently regular function, which constitutes a very important result for applications, as the original definition of steepness is not constructive. These criteria involve both the derivatives of the studied function up to any given order and external real parameters that, generically, belong to compact sets.

\section{Introduction}
\subsection{Hamiltonian formalism and nearly-integrable systems}

Hamiltonian formalism is the natural setting appearing in the study of many physical systems. Namely, for any given positive integer $n$, we consider a symplectic manifold $\cM$ of dimension $2n$, endowed with a skew-symmetric two-form $\omega$, and with a function $H\in C^2(\cM)$ classically called Hamiltonian. A Hamiltonian system on $\cM$ is the dynamical system governed by the vector field $X_H$ verifying
\begin{equation}\label{sistema_dinamico_hamiltoniano}
i_{X_H}\omega:=\omega(X_H,\cdot)=dH\ .
\end{equation}
In the simplest case, we consider the motion of a point on a $n$-dimensional Riemannian manifold $\cR$ endowed with the euclidean metric - called the configuration manifold - under Newton's second law
$$
\ddot{q}=-\nabla U(q)\ ,
$$ 
where $U$ is a smooth potential function, and $q$ is a system of local coordinates for $\mathcal{R}$.
This system can be conjugated by duality due to Legendre's transformation and reads 
\begin{equation}\label{eq_hamilton}
\dot{p}=-\partial_q H(p,q)\quad ; \quad \dot{q}=\partial_p H(p,q)
\end{equation}
where $H(p,q)$ is a real smooth function on the cotangent bundle $T^\star \mathcal{R}$,  and $p$ is the local coordinate conjugated to $q$. In this example, if one takes $\cM\equiv T^\star \cR$, and if one chooses $(p,q)$ to be Darboux's coordinates associated to the two-form $\omega(p,q)\equiv \sum_{j=1}^n dp_i\wedge dq_i$, then system \eqref{eq_hamilton} is locally equivalent to \eqref{sistema_dinamico_hamiltoniano}. 

Among Hamiltonian system, an important rôle is played by those which are integrable by quadrature. Due to the classical Liouville-Arnol'd Theorem, under general topological and algebraic assumptions, an integrable system depending on $2n$ variables ($n$ degrees of freedom) can be conjugated to a Hamiltonian system on the cotangent bundle of the $n$-dimensional torus $\mathbb{T}^n$, whose equations of motion take the form 
$$
\dot{I}=-\partial_\vartheta h(I)=0\quad , \qquad \dot{\vartheta}=\partial_I h(I) \quad,
$$
where $(I,\vartheta)\in \mathbb{R}^n\times \mathbb{T}^n $ are called action-angle coordinates. Therefore, the phase space of an integrable system is foliated by invariant tori carrying the linear motions of the angular variables (called quasi-periodic motions).

Integrable systems are exceptional\footnote{Three examples of integrable systems are the classical Kepler's problem, the harmonic oscillator, and Lagrange's top.}, but many important physical problems can be described by Hamiltonian systems which are close to integrable. 
Namely, the dynamics of a nearly-integrable Hamiltonian system is described by a Hamiltonian function whose form in action-angle coordinates  reads
$$
H(I,\vartheta):=h(I)+\varepsilon f(I,\vartheta)\quad , \qquad  (I,\vartheta)\in \mathbb{R}^n\times \mathbb{T}^n
$$
where $\varepsilon$ is a small parameter that tunes the size of the perturbation $\eps f$ w.r.t. the integrable part $h$.

The structure of the phase space of this kind of systems can be inferred with the help of classical Kolmogorov-Arnol'd-Moser (KAM) theory. Namely, under the generic non-degeneracy condition that $\grad h$ is a local diffeomorphism, a Cantor-like set of positive Lebesgue measure of invariant tori carrying quasi-periodic motions for the integrable flow persists under a suitably small perturbation (see e.g. ref. \cite{Arnold_2010}, \cite{Chierchia_2008}). As this Cantor-like set is nowhere dense, it is extremely difficult to determine numerically whether a given solution is quasi-periodic or not.

%\eject

Moreover, for a  Hamiltonian system depending on $n$ degrees of freedom (hence a
$2n$-dimensional system), the invariant tori provided by classical KAM theory are $n$-dimensional. Hence, if $n=2$,
any pair of invariant tori disconnects the three-dimensional energy level, so that the solutions of the perturbed system are global and bounded over {\it infinite} times. However, an arbitrary large drift of the orbits is possible in case $n\geq 3$. Actually, in ref. \cite{Arnold_1964} Arnol'd proposed an example of a nearly-integrable Hamiltonian system where an arbitrary large instability of the action variables occurs for an arbitrarily small perturbation. This phenomenon is known under the name of {\it Arnol'd's diffusion} (see ref. \cite{Kaloshin_Zhang_2022} and references therein for the most recent developments in this field). Thus, results of stability for quasi integrable Hamiltonian systems which
are valid for an open set of initial condition can only be proved over {\it finite} times.

\subsection{Long time stability of nearly-integrable systems}

In the 1970s, Nekhoroshev\footnote{See \cite{Nekhoroshev_1977}, or \cite{Guzzo_Chierchia_Benettin_2016} for a more modern presentation} proved that if we consider a real-analytic, integrable Hamiltonian whose gradient satisfies a suitable, quantitative transversality condition known as {\it steepness}, then, for any sufficiently small perturbation the solutions of the perturbed system are stable and have a very long time of existence\footnote{The time of stability depends of the regularity of the considered system and is exponential (polynomial) in the inverse of the size of the perturbation if the total Hamiltonian belongs to the Gevrey (Hölder) class. See \cite{Marco_Sauzin_2003}, \cite{Bounemoura_2010}, \cite{Barbieri_Marco_Massetti_2022}.}.

The original definition of steepness given by Nekhoroshev is quite involved and will be discussed at length in the sequel. In order to grasp an idea of what this property means, it is worth mentioning that a real-analytic function is steep if and only if it has no isolated critical points and if any of its restrictions to any affine proper subspace admits only isolated critical points (see \cite{Ilyashenko_1986} and \cite{Niederman_2006}). This is especially satisfied in the important case of an integrable Hamiltonian which is strictly convex in the action variables, since its critical points are non-degenerated.

Actually, the vast majority of the works on Nekhoroshev's theory concerns small perturbations of convex integrable Hamiltonian but Nekhoroshev also proved in \cite{Nekhoroshev_1973} that - unlike convexity - the steepness condition is generic, both in measure and topological sense. Namely, given a pair of positive integers $r,n$, with $r$ high enough, and a point $I_0\in \R^n$, the Taylor polynomials of those $C^{2r-1}$ functions which are not steep around $I_0$ are contained in a semi-algebraic set of positive codimension in the space of polynomials of $n$ variables and degree bounded by $r$. The proof and the refinement of this property constitute the first part of the present work. However, before presenting this and the other main results, we would like to highlight that - even though convex systems are common \footnote{See e.g.  \cite{Niederman_1996}, \cite{Barbieri_Niederman_2020} in the study of the three-body problem, \cite{Bambusi_Fuse_2017}, \cite{Bambusi_Fuse_Sansottera_2018} in the context of central force motions, and \cite{Bambusi_Giorgilli_1993}, \cite{Bambusi_1999}, \cite{Poschel_1999}, \cite{Faou_Grebert_2013} in the framework of infinite-dimensional Hamiltonian systems.} - non-convex integrable Hamiltonians occur in the investigation of important problems of mechanics.

\medskip
Namely, we consider a symplectic manifold $({\mathcal M},\Omega)$ of dimension $2n$, $n\in \N$, where $\Omega$ is an everywhere non-degenerate closed $2$-form, a smooth symplectic vector field $X$ on ${\mathcal M}$ (meaning that the one-form ${i}_X\Omega$ is closed) and an equilibrium point $p^* \in {\mathcal M}$, that is $X(p^*)=0$.

We are interested in studying whether $p^*$ is stable or not.

Since we are in a conservative case, a first observation is that, if $p^*$ is stable, then the spectrum of the linearized system around $p^*$ is $\{\pm i\alpha_1 ,\hdots ,\pm i\alpha_n\}$ where $\alpha_1 ,\hdots ,\alpha_n$ are reals, so that $p^*$ is an elliptic equilibrium position.

The problem being local, we can assure without any loss of generality (this is specified in \cite{  Bounemoura_Fayad_Niederman_2020}) that we are working in $(M,\Omega)=(\R^{2n},\Omega_0)$ where $\Omega_0:=dx\wedge dy$ is the canonical symplectic structure of $\R^{2n}$ for the conjugated Darboux variables $(x,y)\in\mathbb{R}^n\times\mathbb{R}^n$. Moreover, under generic assumptions (see again \cite{  Bounemoura_Fayad_Niederman_2020}), we can assume that the considered system derives from a Hamiltonian of the form:
\begin{equation} \label{Hamintro} H(x,y)= \sum_{j=1}^n \alpha_j(x_j^2+y_j^2)/2 + O_3(x,y). \end{equation}
Our standing assumption from now on is that $H$ is real-analytic. Such a system, under a suitable rescaling, can be considered as nearly-integrable.

In this setting, there are two cases for which one knows that stability holds true for the  considered equilibrium. 

The first case is when the quadratic part $H_2$ is sign-definite, or, equivalently, when the components of the vector $\alpha \in \R^{n}$ have the same sign. Indeed, the Hamiltonian function has then a strict minimum (or maximum) at the origin, and as this function is constant along the flow (it is in particular a Lyapounov function) one can construct, using standard arguments, a basis of neighborhoods of the origin which are invariant, and the latter property is obviously equivalent to stability.

The second case is when $n=2$ and when the so called Arnol'd's iso-energetic non-degeneracy condition is satisfied. Then, KAM stability occurs in every energy level passing sufficiently close to the origin,  implying Lyapounov stability, due to the fact that the two-dimensional tori disconnect each three-dimensional energy level (see for instance \cite{Arn61} and \cite{Mos62}). It is easy to see that the Arnold iso-energetic non-degeneracy condition is generic in measure and topology as a function of the coefficients of the $O_4(x,y)$ part of the Taylor expansion of $H$ around the origin.

In the other cases, a large unstability due to Arnold diffusion can occur (see \cite{Fayad_2023}), but it has been proved in \cite{Bounemoura_Fayad_Niederman_2020} that, generically, any solution starting sufficiently close to the equilibrium point remains close to it for an interval of time which is double-exponentially large ($\exp\circ\exp$) with respect to the inverse of the distance to the equilibrium point.  The latter result is obtained by making use of Nekhoroshev's theory and relies crucially on the genericity of steep functions, since one needs to build a suitable steep integrable approximation of the complete system.

The same issue arises in order to apply Nekhoroshev's theory to concrete examples. Especially, in Celestial Mechanics, we have important problems where an elliptic equilibrium arises with a quadratic term in (\ref{Hamintro}) which is not sign definite: this is the case for the Lagrange's equilibrium points L4, L5 in the restricted three body problem (see \cite{Benettin_Fasso_Guzzo_1998}) and in the averaged ("secular") planetary three body problem (this is due to the Herman's resonance, see \cite{Fejoz_2004} and \cite{Malige_2012}). The latter system is a crucial approximation to apply Hamiltonian perturbation theory (hence KAM or Nekhoroshev theory) in Celestial Mechanics. Moreover, we cannot always build an integrable approximation of this kind of systems which is convex in action variables, hence we have to consider steep non convex Hamiltonians in order to infer stability results with the help of Nekhoroshev's theory. For the study of the Lagrange's equilibrium points, it is possible in most cases to recover steepness by considering higher order approximation (see \cite{Benettin_Fasso_Guzzo_1998}), actually this corresponds to general considerations on functions with three variables which will be specified in the sequel. For the secular planetary three-body problem, the associated Hamiltonian is not convex w.r.t. the actions (see \cite{Pinzari_2013}) and much more variables are involved than for the Lagrange's points, hence we really need new criteria to ensure that a given function is steep or not in this case.  Up to now, generic explicit conditions for steepness were known only for functions of three (the  conditions given by Nekhoroshev in \cite{Nekhoroshev_1977}), four (see \cite{Schirinzi_Guzzo_2013}) or five variables (see \cite{Barbieri_2022}). The second part of this work is devoted to proving explicit conditions for steepness which are generic for functions of an arbitrary number of variables.

\medskip

It can also be specified that if the steepness condition is dropped, large instablities may occur over times of order $1/\varepsilon$, which is the shortest possible time of drift when considering  perturbations of magnitude $O(\varepsilon)$ (see \cite{Niederman_2006} and \cite{BK_2014}). 

In the context of KAM theory, Herman (see \cite{Herman_1992}) has shown that the lack of steepness of the integrable Hamiltonian allows to build perturbations for which one can find a $G_\delta -$dense set of
initial conditions leading to orbits whose action components are unbounded while the integrable Hamiltonian can also be Kolmogorov non-degenerate, hence most of the orbits lie on invariant tori and we have simultanously, existence of large zones of stability and instability.

Steepness also arises in the framework of Arnol'd's diffusion (see \cite{Berti_Biasco_Bolle_2003}) for the optimality of the time of diffusion. Finally, in ref. \cite{Bambusi_Langella_2022} it is shown that Nekhoroshev's classical proof of stability for perturbations of steep integrable Hamiltonian systems is also relevant in the study of PDE's, considered as infinite-dimensional Hamiltonian systems.  

\subsection{Genericity and explicit criteria for steepness}

\medskip

Now, we specify Nekhoroshev's effective result of stability (see refs. \cite{Nekhoroshev_1977}, \cite{Nekhoroshev_1979}), which is valid for an open set of initial conditions  provided that the total Hamiltonian is regular enough and that its integrable part satisfies the following transversality property on its gradient:

\begin{defn}[Steepness]\label{def steep} Fix $\delta>0$, $R>0$.
	A $C^2$ function $h: B^n(0, R+2\delta)\to\R$ is \co{steep} in $B^n(0,R)$ with \co{steepness indices} $\boldsymbol{\upalpha}_1,\ldots,\boldsymbol{\upalpha}_{n-1}\ge 1$ and \co{steepness coefficients} $C_1,\ldots, C_{n-1}, \delta$ if:
	\begin{enumerate}
		\item $\inf_{I\in B^n(0,R)} ||\nabla h(I)|| > 0$;
		\vskip1mm
		\item for any $I \in B^n(0,R)$, for any integer $1\le m < n$, and for any $m$-dimensional subspace $\Gamma^m$ orthogonal to $\nabla h(I)$ and endowed with the induced euclidean metric, one has:
		\begin{equation}\label{cuore_steepness}
		\max_{0\le \eta \le \xi} \,  \min_{u\in\Gamma^m,\,||u||_2=\eta} ||\,\pi_{\Gamma^m}\,\nabla h(I+u)\,|| > C_m \xi^{\boldsymbol{\upalpha}_m},\quad \forall \xi \in (0,\delta],
		\end{equation}
		where $\pi_{\Gamma^m}$ stands for the orthogonal projection on $\Gamma^m$.
	\end{enumerate}
	
\end{defn}

\begin{rmk}\label{viadotto}
	Since in definition \ref{def steep} the subspace $\Gamma^m\subset \R^n$ is endowed with the induced euclidean metric, for all $u\in\Gamma^m$ one has $||\pi_{\Gamma^m}\,\nabla h(I+u)||=||\nabla (h|_{I+\Gamma^m})(I+u)||$, where $h|_{I+\Gamma^m}$ indicates the restriction of $h$ to the affine subspace $I+\Gamma^m$. 
\end{rmk}

As it is showed in \cite{Niederman_2006}, in the analytic case a function is steep if and only if, on any affine hyperplane $I+\Gamma_m$, there exists no curve $\gamma$ with one endpoint in $I$ such that the restriction $\nabla (h|_{I+\Gamma^m})$ vanishes identically on $\gamma$. From a heuristic point of view, for any value $m\in\{1,...,n-1\}$ the gradient $\nabla h$ must "bend" towards $I+\Gamma^m$ when "travelling" along the curve $\gamma\in I+\Gamma^m$, so that critical points for the restriction of $h$ to $I+\Gamma^m$ must not accumulate.

Once the concept of steep function is introduced, Nekhoroshev's effective result of stability reads
\begin{thm}[Nekhoroshev, 1977]
	Consider a nearly-integrable Hamiltonian system governed by
	$$
	H(I,\vartheta):=h(I)+\varepsilon f(I,\vartheta)\quad , \qquad H\in C^\omega(B^n(0,r)\times \T^n)\ ,
	$$
	where $B^n(0,r)$ is the open ball of radius $r$ in $\mathbb{R}^n$, and $h$ is assumed to be steep. Then there exist positive constants $a,b,\varepsilon_0, C_1,C_2,C_3$ such that, for any $\varepsilon\in [0,\varepsilon_0)$ and for any initial condition not too close to the boundary, one has
	$
	|I(t)-I(0)|\leq C_2\varepsilon^b
	$
	for any time $t$ satisfying
	$
	|t|\leq C_1\exp\left(C_3/\varepsilon^{a}\right)\ .
	$
\end{thm}

\begin{rmk}
	The presence of a bound of the kind $
	|I(t)-I(0)|\leq C_2\varepsilon^b
	$ on the variation of the action variables is a consequence of the steepness property, whereas the time of stability depends on the regularity of the function $H$ at hand. In the original formulation by Nekhoroshev, $H$ is considered to be real-analytic, which yielded an exponentially-long time in the inverse of the size of the pertubation (see also \cite{Guzzo_Chierchia_Benettin_2016}). Exponentially-long times of stability hold also in case $H$ is Gevrey (see \cite{Marco_Sauzin_2003}), whereas only polynomially-long times of stability can be ensured for $C^\infty$  and Hölder  functions (see refs. \cite{Bambusi_Langella_2020}, \cite{Barbieri_Marco_Massetti_2022}).
	
\end{rmk}

As it has already been anticipated in the previous paragraph, the steepness property is generic - both in measure and in topological sense - in the space of Taylor polynomials of sufficiently high order of smooth functions. Namely, let $r,n\geq 2$ be two positive integers. We indicate by $\Pol\subset  \R[x_1,\dots,x_n]$ the subspace of real polynomials in $n$ variables having degree bounded by $r$. For any point $I_0\in \R^n$, and any function $f$ of class $C^r$ near $I_0$, we call {\it r-jet of $f$ at $I_0$} the Taylor polynomial of $f$ up to order $r$ calculated at $I_0$. Moreover, we say that a subset $\cA\subset \R^n$ is {\it semi-algebraic} if it is the finite union of subsets determined by a finite number of polynomial equalities or inequalities (see Definition \ref{union}). Nekhoroshev proved in \cite{Nekhoroshev_1973}-\cite{Nekhoroshev_1979} that

\begin{thm}[Nekhoroshev, 1973-1979]\label{Teorema_uno}
	There exists a "bad" semi-algebraic subset $\Omega(r,n)$ of $\Pol$ such that any function $h$ of class $C^{2r-1}$ around a non-critical point $I_0\in \R^n$, and whose $r$-jet at $I_0$ lies outside of $\Omega(r,n)$, is steep in a neighborhood of $I_0$ with uniform indices.
	Moreover, the codimension of $\Omega(r,n)$ in $\Pol$ becomes positive for $r\gtrsim[ n^2/4]$. 
\end{thm}

Although Nekhoroshev's Theory has been a classic subject of study in the dynamical systems community for more than forty years, the proof of Theorem \ref{Teorema_uno} has remained poorly understood.  This is possibly due to the fact that such a demonstration does not involve any arguments of dynamical systems, but combines quantitative reasonings of real-algebraic geometry and complex analysis. Moreover, real-algebraic geometry was at a more rudimentary level than nowadays at the time that Nekhoroshev's was writing; for this reason, important properties  of real-algebraic geometry are discovered\footnote{For example, it is remarkable that, up the author's knowledge, a fundamental Bernstein's inequality for algebraic functions is proved for the first-time in Nekhoroshev's work (see \cite{Barbieri_Niederman_2022}). Such a result seems to have passed unnoticed, until it has been widely rediscovered and generalized in the late nineties by Roytwarf and Yomdin in \cite{Roytwarf_Yomdin_1998}, and subsequently developed by several authors.} in \cite{Nekhoroshev_1973} at the same time that they are used to prove Theorem \ref{Teorema_uno}, which makes that work obscure in many parts. In addition, the proofs of some lemmas in that work are sketchy or presented in an old-fashioned way. For these reasons, the first part of this work is devoted to proving and refining Theorem \ref{Teorema_uno} by making use of modern results of real-algebraic geometry. As we will discuss in detail in the sequel, Yomdin's Lemma about the analytic reparametrization of semi-algebraic sets (see \cite{Yomdin_2008}) turns out to be the fundamental ingredient of real-algebraic geometry which is used in the proof of the genericity of steepness.

Moreover, since the definition of steepness is not constructive, it is difficult to directly establish whether a given function is steep or not. Up to the author's knowledge, there are only three articles on this topic (see \cite{Schirinzi_Guzzo_2013}, \cite{Chierchia_Faraggiana_Guzzo_2019}, \cite{Barbieri_2022}) which concern only polynomials of low degree depending on a small number of variables. 
Actually, by developing the arguments used by Nekhoroshev to prove Theorem \ref{Teorema_uno}, it is possible to deduce explicit sufficient algebraic conditions for steepness involving the derivatives  up to an arbitrary order of functions depending of an arbitrary number of variables. This proves fundamental for applications of Nekhoroshev's theory to physical models. The second part of this work is dedicated to this topic. Namely, we will prove refined versions of the Theorems below.

\begin{thm}\label{Teorema_due}
	The semi-algebraic set $\Omega(r,n)$ in Theorem \ref{Teorema_uno} verifies
	\begin{equation}\label{omegaproiettato}
	\Omega(r,n)=\text{closure}\left(\bigcup_{m=1}^{n-1}\text{Proj}_{\cP(r,n)} Z(r,m,n)\right) 
	\end{equation}
	where, for any $m\in \{1,\dots,n-1\}$, $Z(r,m,n)$ is a semi-algebraic set of $\Pol\times \R^K\times\R^n\times \tU(m-1,n)$, $K=K(r,m)$ is a suitable positive integer, and $\tU(m-1,n)$ is the compact $m-1$-dimensional Stiefel manifold in $\R^n$ (see section \ref{definitions} for its definition).
	
	Moreover, for any $m\in\{1,\dots,n-1\}$, the form of $Z(r,m,n)$ can be explicitly computed. 
\end{thm}

\begin{rmk}\label{scriterio}
	Theorem \ref{Teorema_due} is a first example of an explicit criterion for steepness. Infact, as it is known, there exist general algorithms of real-algebraic geometry that allow to compute the explicit form of the projection and of the closure of any given semi-algebraic set (see e.g. \cite{Basu_Pollack_Roy_2006}). Therefore, at least in principle, it is possible to compute the r.h.s. of \eqref{omegaproiettato} - hence the form of $\Omega(r,n)$ - as the expression of $Z(r,m,n)$ is known due to Theorem \ref{Teorema_due}. However, the complexity of the classic algorithms grows double-exponentially in the number of variables, so that they are of little use in practice (see \cite{Heintz_Roy_Solerno_1990}).
\end{rmk}

\begin{rmk}\label{scriteriato}
	
	Alternatively, one could use Theorem \ref{Teorema_due} in order to check if a function $h$ of class $C^{2r-1}$ around a point $I_0$ is steep in the following way. Indicating by $\tT_{I_0}(h,r,n)$ the $r$-jet of $h$ at $I_0$, by \eqref{omegaproiettato} one could check whether there exists $\tau>0$ such that, for any $m\in\{1,\dots,n-1\}$, for any choice of parameters $\beta\in \R^K\times \R^n\times \tU(m-1,n)$, and for any polynomial $S\in \Pol$ which is $\tau$-close to $\tT_{I_0}(h,r,n)$, the pair $(S,\beta)$ lies outside of $Z(r,m,n)$. This would guarantee that $\tT_{I_0}(h,r,n)$ lies outside of $\text{ closure}\left(\bigcup_{m=1}^{n-1}\text{Proj}_{\cP(r,n)} Z(r,m,n)\right)$, so that  \eqref{omegaproiettato} and Theorem \ref{Teorema_uno} would ensure steepness. This is indeed one possibility, and we will make it more explicit in the next section (see Theorem B). However, this criterion involves checking an explicit condition for a non-compact set of parameters (the first components of the vectors $\beta$ above lie in $\R^K\times \R^n$, whereas the remaining ones belong to the compact Stiefel manifold $\tU(m-1,n)$). Nevertheless, as we show below, on "most subspaces" steepness can be checked by making use of criteria involving only parameters belonging to a compact set. 
	
\end{rmk}

Namely, let $h$ be a function of class $C^{2r-1}$ around the origin, verifying $\nabla h(0)\neq 0$. Then,

\begin{thm}\label{Teorema_tre}
	It is possible to find explicit algebraic criteria involving the derivatives of $h$ up to order $r$ that ensure that $h$ is steep in a neighborhood of the origin on the one-dimensional subspaces, with uniform index $\alg_1$ and uniform coefficients $C_1,\delta$. 
	
	Moreover, for any $m\in\{2,\dots,n-1\}$, one has the following properties.
	\begin{enumerate}
		\item 
		$h$ is steep at the origin with index $\alg_m=1$ on the $m$-dimensional subspaces orthogonal to $\nabla h(0)\neq 0$ on which the restriction of the hessian $D^2 h(0)$ is non-degenerate.	
		\item On the $m$-dimensional subspaces of $\nabla h(0)^\perp$ on which the restriction of $D^2 h(0)$ has exactly one null eigenvalue, it is possible to construct explicit algebraic criteria for steepness that involve the $r$-jet of $h$ at the origin and a finite number of real parameters belonging to a compact subset. These criteria can be constructed starting from the explicit form of subset $Z(r,m,n)$ in Theorem \ref{Teorema_due} by the means of algorithms involving only linear operations\footnote{Hence, much simpler algorithms than the general algorithms of real-algebraic geometry.}. 
	\end{enumerate}

\end{thm}

Therefore, explicit criteria for steepness at the origin involving only the $r$-jet of $h$ exist in case $m=1$. In case $m\in\{2,\dots,n-1\}$, instead, with the exception of the $m$-dimensional subspaces of $\nabla h^\perp(0)$ on which the restriction of $D^2h(0)$ has two or more null eigenvalues, steepness can be checked by using a criterion which is simpler than those stated in Remarks \ref{scriterio}-\ref{scriteriato}.

 Moreover, for any value of $m\in\{1,\dots,n-1\}$, the Hessian at the origin of a generic function $h$ is non-degenerate on most subspaces of the $m$-dimensional Grassmannian $\tG(m,n)$, as the following result shows. 

\begin{thm}\label{Teorema_quattro}
	Consider an integer $m\in\{2,\dots,n-1\}$. For any bilinear, symmetric, non-degenerate form $\mathsf{B}: \R^n\times \R^n\longrightarrow \R$, the $m$-dimensional subspaces on which the restriction of $\mathsf{B}$ is degenerate are contained in a submanifold of codimension one in the Grassmannian $\tG(m,n)$. 
\end{thm}

\subsection{Informal presentation of the proofs}

Roughly speaking, the proof of Theorem \ref{Teorema_uno} is split into two parts:

\begin{itemize}
	\item General considerations of semi-algebraic geometry allow to prove that the complicated condition \eqref{cuore_steepness} arising in the definition of steepness is an open property in the space of polynomials $\Pol$. Namely, if \eqref{cuore_steepness} holds for a given polynomial $Q\in \Pol$, then it holds also in a neighborhood of $Q$ with uniform indices $\alg_1,\dots,\alg_m$ and uniform coefficients $C_1,\dots,C_m,\delta$.
	
	\item Also, condition \eqref{cuore_steepness} is verified for given values of $\alg_1,\dots,\alg_{n-1}$ and for some $C_1,\dots,C_{n-1},\delta$ if the Taylor polynomial at $I_0$ of the studied function $h$ lies outside of the closure of a set of polynomials whose coefficients satisfy a certain number of degeneracy conditions. A detailed analysis of these conditions shows that, for sufficiently high $\alg_1,\dots,\alg_{n-1}$, they only admit a non-generic set of solutions. 
\end{itemize}

In the present work, we have results on the two sides of the proof.

\medskip 

\subsubsection{Reparametrization of semi-algebraic sets and Bernstein's inequality}

We revisit Nekhoroshev's reasonings of semi-algebraic geometry under the light of more recent results in the field.

Due to \eqref{cuore_steepness}, a central point to check steepness of a given function $h$ at a point $I\in \R^n$ consists in ensuring a minimal growth of the projection of its gradient on any affine subspace orthogonal to $\grad h(I)\neq 0$. For a fixed affine subspace $I+\Gamma$ equipped with local coordinates and with the induced euclidean metric, by Remark \ref{viadotto} the projection of $\grad h(I)$ on $\Gamma$ corresponds to the gradient of the restriction $h|_{I+\Gamma}$ expressed in the local coordinates. Hence, one is led to study the locus of minima of $||\nabla h|_{I+\Gamma}||$. By the above considerations, without entering into too many technicalities, a crucial step in Nekhoroshev's proof of the genericity of steepness consists in considering, for any fixed polynomial $P\in \R[X_1,\ldots ,X_m]$, the semi-algebraic set - called thalweg nowadays (see \cite{Bolte_Daniilidis_Ley_Mazet_2010}) - determined by :
\begin{equation}\label{Thalweg}
{\mathbb{R}}^m\supset {\mathcal{T}}_P :=\big{\{ }u\in\R^m|\ \vert\vert\nabla P(u)\vert\vert\leq\vert\vert\nabla P(v)\vert\vert\ \forall v\in\R^m \text{  s.t. } \vert\vert u\vert\vert =\vert\vert v\vert\vert\big{\} }\ .
\end{equation}

Nekhoroshev shows that $\cT_P$ contains the image of a semi-algebraic curve\footnote{I.e. a curve having semi-algebraic graph, see also Definition \ref{funzione_semi_algebraica}.} $\gamma$ which admits a holomorphic extension with the exception of a finite set of singular complex points whose cardinality depends only on the degree of $P$ and on the number of variables. In particular, one can ensure the existence of a uniform real interval of analyticity and of a uniform complex analyticity width for the curve $\gamma$, independently on the choice of the polynomial $P\in \Polm$. More specifically, the graph of $\gamma$ can be parametrized by analytic-algebraic\footnote{I.e. analytic maps whose graph solves a non-zero polynomial of two variables.} maps, and the existence of a Bernstein's-like inequality controlling uniformly the growth of this kind of functions in the complex plane ensures uniform upper bounds on the derivatives of these charts (see \cite{Roytwarf_Yomdin_1998}, \cite{Yomdin_2008}, \cite{Yomdin_2011}, \cite{Barbieri_Niederman_2022} and references therein for a modern presentation).   

Actually, this result about the thalweg in \cite{Nekhoroshev_1973} is a particular case of a general theorem due to Yomdin \cite{Yomdin_2008} about analytic reparametrizations of semi-algebraic sets (the finitely-differentiable case was firstly stated by Yomdin and Gromov in refs.\cite{Yomdin_1987},\cite{Gromov_1987} and then proved by Burguet in \cite{Burguet_2008}). Generally speaking, the reparametrization of a semi-algebraic set $A$ is a subdivision of $A$ into semi-algebraic pieces $A_j$ each of which is the image of a semi-algebraic function\footnote{That is, a function whose graph is a semi-algebraic set, see also Definition \ref{funzione_semi_algebraica}.} of the unit cube.  The uniform control on the parametrization of the curve $\gamma$ is unavoidable in \cite{Nekhoroshev_1973}, since it ensures that - for a smooth function - steepness is an open property. 

Moreover, it is proved that the coefficients of the Taylor expansions of non-steep functions satisfy suitable algebraic conditions (one has a "finite-jet" determinacy of steepness). In this way, the study of the genericity of steepness is reduced to a finite-dimensional setting which involves polynomials of bounded order and this is another crucial step in order to prove the genericity. 

\medskip

It is worth adding some remarks about the fact that Nekhoroshev proves a kind of Bernstein's inequality for algebraic functions (see \cite{Nekhoroshev_1973}, Lemma 5.1, p.446). Namely, Nekhoroshev proves that an algebraic function which is real-analytic over a real interval $I$ admits a bound on its growth on the complex plane which only depends on its maximum over $I$ and on a constant depending on the degree of the polynomial solving its graph and on the size of its complex domain of holomorphy. This result is proved by exploiting the properties of algebraic curves of complex polynomials in two variables, and by making an intensive use of complex analysis. The original statements are difficult to disentangle from the context of the genericity of steepness and the proofs are very sketchy. This is different from the strategy used by Roytwarf and Yomdin (see \cite{Roytwarf_Yomdin_1998},\cite{Yomdin_2011} and references therein) which relies on arguments of analytic geometry. Since we have not been able to find any reference that shows Nekhoroshev's proof of Bernstein's inequality in detail except for the original paper \cite{Nekhoroshev_1973}, we have clarified and extended Nekhoroshev's reasonings in \cite{Barbieri_Niederman_2022}, and this allowed to obtain a simpler proof on the existence of a Bernstein's inequality for algebraic functions. 
\medskip

It is also worth mentioning that, in connection with arithmetic, the steepness condition is introduced to prevent the abundance of rational vectors on certain sets and it can be noticed that deep applications of the controlled analytic reparametrization of semi-algebraic sets yield bounds on the number of integer points in semi-algebraic sets (see \cite{Binyamini_Novikov_2019} and \cite{Cluckers_Pila_Wilkie_2020}). In the future, these ideas may help to spread light on the connection between the stability of nearly-integrable Hamiltonian systems and the arithmetic properties of semi-algebraic sets. 

We also mention that, in the study of PDEs, the Yomdin-Gromov's algebraic lemma was used by Bourgain, Goldstein, and Schlag \cite{Bourgain_Goldstein_Schlag_2002} to bound the number of integer points in a two-dimensional semi-algebraic set.

\subsubsection{Degeneracy condition}

We describe heuristically the second part of the proof of Theorem \ref{Teorema_uno}. To make things simple, we restrict this informal discussion to the case of a real-analytic function $h$ around the origin. By formula \eqref{cuore_steepness}, if $h$ is non-steep at the origin, then for some $m\in \{1,\dots,n-1\}$ there exists a $m$-dimensional subspace $\Gamma^m$ and a curve $\gamma\subset\Gamma^m$ starting at the origin on which the projection $(\pi_{\Gamma^m} \grad h)|_\gamma$ is identically null. Assuming that $\gamma$ is sufficiently regular, this means that $(\pi_{\Gamma^m} \grad h)|_\gamma$ has a zero of infinite order at the origin. It can also be shown (see Theorem \ref{arco_minimale}) that the the curve $\gamma$ on which such a condition is satisfied must possess a precise form. By these arguments, one can write down explicitly the equations imposing to the derivatives of the function $(\pi_{\Gamma^m} \grad h)|_\gamma$ to be all identically null. Moreover, by performing complicated computations it is possible to check that these equations are all linearly independent. Then, estimates on the codimension of a projected set show that, if one bounds the order of the derivatives that are being considered in the equations by a positive integer $r\ge 2$, the Taylor polynomials of non-steep functions belong to a semi-algebraic set of positive codimension in the space of polynomials $\Pol$. It is this kind of computations - which are expressed explicitly for the first time in this work - that allow to prove Theorem \ref{Teorema_due}.

Moreover, by construction, the equations that we are considering depend also on the Taylor coefficients of the curve $\gamma$ and on the vectors spanning the considered subspace $\Gamma^m$. This explains the presence of the space of parameters $\R^K$ and of the Stiefel manifold in the statement of Theorem \ref{Teorema_due}. On the one hand, by suitably exploiting the form of the equations, one can prove Theorem \ref{Teorema_tre}. On the other hand, Theorem \ref{Teorema_quattro} is independent and its proof relies on the construction of a suitable system of coordinates for the Grassmannian. 

\subsubsection{Structure of the work}

This work is organized as follows: section \ref{definitions} sets the main notations and definitions, whereas the main results (refined versions of Theorems \ref{Teorema_uno}-\ref{Teorema_due}-\ref{Teorema_tre}-\ref{Teorema_quattro}) are stated in section \ref{results}. Section \ref{thalweg_sec} construction of the thalweg and the reparametrization of semi-algebraic sets, whereas section \ref{rs_vanishing_polynomials} is devoted to the study of the degeneracy condition described in the paragraph  above. Section \ref{genericity} puts together the results of sections \ref{thalweg_sec} and \ref{rs_vanishing_polynomials} in order to prove the genericity of steepness. Finally, sections \ref{Prova Th B}-\ref{Theorem C prova}-\ref{Prova C1-C4} contain the proof of the explicit criteria for steepness. 

%\begin{thm}\label{Teorema_tre}
%	When in Theorem \ref{Teorema_uno} one has $r\in\{2,3,4\}$, the set $\Omega(r,n)$ is closed and its form can be explicitly computed starting from the form of the sets $Z(r,m,n)$ in Theorem \ref{Teorema_due} by the means of algorithms involving only linear operations. 

%	When $r\ge 5$, one has the disjoint union
%	\begin{equation}
%	\Omega(r,n)=\Omega_1(r,n)\bigsqcup \Omega_2(r,n)
%	\end{equation}
%	where $\Omega_1(r,n)$ is closed and its form can be explicitly computed starting from the form of the sets $Z(r,m,n)$ in Theorem \ref{Teorema_due} by the means of algorithms involving only linear operations. 

%	Moreover, in $\Pol$ one has
%	$$
%	\text{codim } \Omega_2(r,n)>\text{codim } \Omega_1(r,n)=\text{codim } \Omega(r,n)\ .
%	$$
%	\end{thm}

\section{Main notations and definitions}\label{definitions}
\subsection*{Norms}
For any $n\in \N^\star$, and for any $\nu\in \N^\star\cup\{\infty\}$, we denote by $||\cdot||_\nu$ the standard $\ell^\nu$-norm in $\R^n$.

For any integer $q\geq 1$, and for any open subset $D$ of $\R^n$,
the symbol $C^{q}(D)$ indicates the set of $q$-times continuously differentiable maps 
$f:D\to \R$. Moreover, we indicate by $C_b^q(D)$ the subset of $C^q(D)$ containing those functions $f$ satisfying
\begin{equation}\label{norma texas}
\norma{f}_{C^q(D)} := \sup_{\substack{\alpha\in \N^n\\ \abs{\alpha} \le q}}\, \sup_{x\in D} \abs{\partial^\alpha f(x)}<+\infty\ .
\end{equation}
In particular, $\big(C_b^q(D),\norma{\cdot  }_{C^q(D)}\big)$ is a Banach space with multiplicative norm\footnote{That is, satisfying an inequality
	of the form $\norm{fg}\leq K\norm{f}\,\norm{g}$ for a suitable constant $K$.}. 
\subsection*{Sets}
In the sequel, we will make use of the following notations:
\begin{itemize}
	\item For any $d>0$ and for any $f\in C_b^q(D)$, the symbol $\boldsymbol{\fB}^q(f,d,D)$ indicates the infinite-dimensional open ball of radius $d$ centered at $f$ for the norm \eqref{norma texas};
	\item $\cD_\rho(z_0)$ indicates the open complex disc of radius $\rho>0$ centered at $z_0\in\C$;
	\item $B^n(I,R)$ indicates the real ball of radius $R>0$ centered at $I\in\R^n$.
	\item For any connected set $\cA\subset \C$, we denote the complex polydisk of width $r$ around $\cA$ by
	$$
	(\cA)_r:=\left\{z\in \C\,\text{ s.t. }\, ||z-\cA||_2:=\inf_{z'\in\cA}||z-z'||<r\right\}\ .
	$$
\end{itemize}

\subsection*{Notations of real-algebraic geometry}

For any pair $(r,n)$ of positive integers, and for any function $h$ of class $C^r$ in a neighborhood of some point $I_0\in\R^n$, we denote by
\begin{itemize}
	\item $\Pol\subset \R[X_1,\dots,X_n]$ the subspace of polynomials over the real field in $n$ real variables with zero constant term and whose degree is bounded by $r$;
	\item $\Pollo\subset \Pol$ is the subset of those polynomials $Q$ that verify $\grad Q(0)\neq 0$;
	\item $\tT_{I_0}(h,r,n)\in\Pol$ the Taylor polynomial at order $r$ of the function $h(I)-h(I_0)$ centered at $I_0$.
\end{itemize}

\medskip 

Now, let $k,n$ be positive integers, with $k\le n$. 

\begin{itemize}

	\item We indicate by $\tG(k,n)$ the $k$-dimensional Grassmanniann in $\R^n$, i.e. the compact manifold of $k$-dimensional linear subspaces in $\R^n$. 
	
	\item We also denote by $\tU(k,n)$ the compact $k$-dimensional Stiefel manifold in $\R^n$, that is the manifold of ordered orthonormal $k$-tuples of vectors in $\R^n$.

\end{itemize}

For any set $\cA\subset\R^n\times\R^m$, we indicate by $\Pi_{x}\cA$ its projection onto the first $n$ coordinates, that is the set
$$
\Pi_{x}\cA:=\{x\in\R^n:\exists y\in\R^m | (x,y)\in \cA\}\ .
$$
Finally, let $m,n$ be positive integers satisfying $m\le n$, and let $\{u_1\dots,u_m\}$ be a set of linearly independent vectors in $\R^n$. For each $i\in\{1,\dots,n\}$ and $j\in\{1,\dots,m\}$, we denote the $i$-th component of the vector $u_j$ by $u_j^i$. For any multi-index $\mu=(\mu_1,\dots,\mu_m)\in \N^m$, we set $|\mu|:=||\mu||_1$. Given $I_0\in \R^n$ and a function $h$ of class $C^{|\mu|}$ in a neighborhood of $I_0$, we also denote by
\begin{align}\label{forma}
\begin{split}
&h_{I_0}^{|\mu|}\big[\stackrel{ \mu_1 }{\overbrace{u_1}},\dots,\stackrel{ \mu_m }{\overbrace{u_m}}\big]:=\\
&\sum_{\substack{i_{11},\,\dots\,,i_{1\mu_1}=1\\\dots\\i_{m1},\,\dots\,,i_{m\mu_m}=1}}^n \frac{\partial^{|\mu|} h(I_0)}{\partial I_{i_{11}}\dots\partial I_{i_{1\mu_1}}\dots\partial I_{i_{m1}}\dots\partial I_{i_{m\mu_m}}}\,u_1^{i_{11}}\dots u_1^{i_{1\mu_1}}\dots u_m^{i_{m1}}\dots u_m^{i_{m\mu_m}}
\end{split} 
\end{align}
the $\mu$-th order multilinear form associated to the $\mu$-th coefficient of the Taylor expansion around $I_0$ of the restriction of $h$ to $\text{Span }(u_1,\dots,u_m)$.

\section{Main results}\label{results}

\subsection{Genericity of steepness}

As we discussed in Theorem \ref{Teorema_uno} in the Introduction, the steepness property is generic, both in measure and topological sense, in the space of jets of functions of sufficiently high regularity. In this paragraph, we will give a more quantitative version of this result. Namely, the statement below is a refined version of Nekhoroshev's Theorem on the genericity of steepness, which can be found in refs. \cite{Nekhoroshev_1973}-\cite{Nekhoroshev_1979}. 
\begin{thm*}[\bf A]\label{Gen_steepness}
	Let $r,n\geq 2$ be two integers, and let $\ts:=(s_1,\dots,s_{n-1})\in \N^{n-1}$ be a vector satisfying $1\le s_m\le r-1$ for all $m=1,\dots,n-1$. 
	
	There exists a closed semi-algebraic subset $\OmegoneI$ of $ \Pol$ such that, for any $I_0\in \R^n$, for any real number $\varrho>0$, for any open, bounded domain $\mathscr{D}\subset C_b^{2r-1}(\overline B^n(I_0,\varrho))$, 
	and for any function $h$ verifying  \begin{enumerate}
		\item $h\in \mathscr{D}$\ ,
		\item $\nabla h(I_0)\neq 0$\ ,
		\item $
		\left|\left|\tT_{I_0}(h,r,n)-\OmegoneI\right|\right|_\infty:=\displaystyle \inf_{ Q\in \OmegoneI}||\tT_{I_0}(h,r,n)- Q||_\infty>\tau>0 \ ,
		$ 
	\end{enumerate}
	one can introduce positive constants $\overline \epsilon=\overline \epsilon(r,\ts,\tau,n)$, $R=R(r,\ts,\tau,n,\varrho)$, $C_m=C_m(r,s_m,\tau,n)$ - where $m=1,\dots,n-1$ - and $\delta=\delta(r,\ts,\tau,n,\varrho,\mathscr{D})$ so that any function $$
	f\in \mathscr{D}\quad , \qquad f\in \fB^{2r-2}(h,\eps,\overline B^n(I_0,\varrho))\qquad \eps\in[0,\overline \epsilon]
	$$
	is steep in $B^n(I_0,R)$, with steepness coefficients $C_m,\delta$ and with steepness indices bounded by 
	\begin{equation}\label{bound}
	\overline \alg_m(s):=
	\begin{cases}
	s_m\quad , \qquad &\text{ if $m=1$ }\\
	2s_m-1\quad, \qquad &\text{ if  $m\ge 2$ }\ .\\
	\end{cases}
	\end{equation}
	Moreover, in $\Pol$ one has
	\begin{equation}\label{codimension}
	\text{  codim } \Omegone \ge \max\left\{0,\min_{m\in\{1,\dots,n-1\}}\{s_m-m(n-m-1)\}\right\}\ .
	\end{equation}
\end{thm*}
With the setting of Theorem \eqref{Gen_steepness}, we give the following
\begin{defn}\label{ordine_steepness}
	A function $f\in C_b^{2r-1}(\overline B^n(I_0,\varrho))$ satisfying the hypotheses of Theorem A is said to be steep at order $r$ at the point $I_0$ for the vector $\ts$. The functions $g\in \mathscr{D}$ verifying
	$$
	\grad g(I_0)\neq 0\quad , \qquad 	\tT_{I_0}(g,r,n) \in \bigcup_{\substack{\ts\in \N^{n-1}\\1\le s_m\le r-1\\ \forall m\in\{1,\dots,n-1\}}}\Pol \backslash \OmegoneI 
	$$ 
	are said to be steep at order $r$ at the point $I_0$. 
\end{defn}

With respect to the original result by Nekhoroshev, a few aspects are refined or clarified in Theorem \ref{Gen_steepness}
\begin{enumerate}
	\item The dependence of the steepness coefficients $C_m$, $m\in\{1,\dots,n-1\}$, and $\delta$ on the distance $\tau$ to the bad set $\OmegoneI$, is made explicit. In particular, as it will be shown in section \ref{genericity}, for fixed $n,r,\ts$, when $\tau\rightarrow 0$, then both $\delta\rightarrow 0$ and $C_m\rightarrow 0$ for all $m=1,\dots,n-1$, whereas the bounds $\alg_m$ on the steepness indices are left unchanged. Hence, when $\tau\rightarrow 0$, steepness may "break down" due to the steepness coefficients tending to zero (whereas the steepness indices of those functions whose $r$-jet lies outside of $\OmegoneI$ stay uniformly bounded).
	
	It is important to stress that the above reasonings do not necessarily imply that a function $g\in C_b^{2r-1}(\overline B^n(I_0,\varrho))$ whose $r$-jet satisfies $\tT_{I_0}(g,r,n)\in \OmegoneI$ - for some vector $\ts\in \N^{n-1}$ as the one in Theorem A - is non-steep. For example, if for two vectors $\ts,\ts'\in \N^{n-1}$, $\ts'\neq \ts$, having the same properties of the one in Theorem A, one has  $||\tT_{I_0}(g,r,n)- \OmegoneIsprimo||_\infty>0$,  $||\tT_{I_0}(g,r,n)- \OmegoneI||_\infty=0$, then $g$ is steep at order $r$ at $I_0$ for the vector $\ts'$ but not for the vector $\ts$. Hence, it admits different bounds on the steepness indices than the functions whose jets lie outside of $\OmegoneI$. Therefore, the only result that one can infer from the relation $\tT_{I_0}(g,r,n)\in \OmegoneI$ is that, in case $g$ were steep, its steepness indices would admit a different upper bound than the one in \eqref{bound}. 
	
	By Theorem A, definition \ref{ordine_steepness}, and by the above discussion,  the set 
	\begin{equation}
	\bigcap_{\substack{\ts\in \N^{n-1}\\1\le s_m\le r-1\\ \forall m\in\{1,\dots,n-1\}}}\OmegoneI \subset \Pol 
	\end{equation}
	contains the $r$-jets of all $C_b^{2r-1}$ functions around the non-critical point $I_0$ which are non-steep at order $r$ at $I_0$. In the same way, when $r\rightarrow+\infty$, the Taylor coefficients of all non-steep analytic functions at the non-critical point $I_0$ belong to the set
	\begin{equation}\label{infinito}
	\omega_n(I_0):=	\bigcap_{r=2}^\infty \bigcap_{\substack{\ts\in \N^{n-1}\\1\le s_m\le r-1\\ \forall m\in\{1,\dots,n-1\}}}\OmegoneI \quad \subset\quad  \bigcup_{r=2}^\infty \Pol  \ .
	\end{equation}  
	
	Relation \eqref{infinito} is the explicit version of what Nekhoroshev meant when he wrote {\it"Hamiltonians which fail to be steep at a noncritical point are infinitely
		singular: they satisfy an infinite number of independent conditions on the Taylor coefficients"} (see \cite{Nekhoroshev_1973}, p. 426).
	
	Actually, the strategy of proof given in the present work follows Nekhoroshev's reasonings by showing - for any given $Q\in \Pol$- the existence of an arc, whose image is contained in the thalweg $\cT_Q$ (see Definition \ref{Thalweg}), admitting a fitted parametrization whose derivatives are controlled by constants depending only on the number of variables $n$ and on the degree $r$, but not on $Q$. This is a particular occurrence of the fact that - with the exception of small neighborhoods around singularities - a semi-algebraic set can be reparametrized by a collection of images of holomorphic functions having a domain of analyticity and an upper bound on their derivatives which depend only on the number of variables, and on the number and on the degrees of the polynomials involved in the construction (see \cite{Yomdin_2008} and \cite{Yomdin_2015} for a two-dimensional semi-algebraic set, and \cite{Binyamini_Novikov_2019}, \cite{Cluckers_Pila_Wilkie_2020} for higher dimensional sets with more general properties than semi-algebraicness). The considered analytic reparametrization is a partial extension of a theorem (called algebraic lemma) due to Yomdin \cite{Yomdin_1987} and Gromov \cite{Gromov_1987} which ensures that, for any semi-algebraic set, there exists a collection of $C^k -$mappings which parametrize entirely the considered set. Actually, in our proof (see Theorem \ref{arco_minimale}), we ask for the mappings to be holomorphic, and we must remove small neighborhoods around their singularities; also, the use of $C^k -$reparametrizations would not allow to obtain a characterization of non-steep functions as in \ref{infinito}, where a control of all the derivative up to infinity is needed.

	\item The vector $\ts$ does not appear in the original statement. Indeed, Nekhoroshev implicitly sets
	\begin{equation}\label{scelta_Nekho}
	s_m=\overline s_m:=
	\begin{cases}
	\max\left\{1,r-1 +m(n-m-1)-\displaystyle \frac{n(n-2)}{4}\right\}\quad , \qquad \text{  for } n \text{  even }\\
	\max\left\{1,r-1 +m(n-m-1)-\displaystyle \frac{(n-1)^2}{4}\right\}\quad , \qquad  \text{  for } n \text{  odd }\ .
	\end{cases}
	\end{equation} 
	From a heuristic point of view, in ref. \cite{Nekhoroshev_1979} this choice was probably conceived in the following way: in estimate \eqref{codimension}, one may want to get rid of the quantity $m(n-m-1)$, which attains the maximal value $n(n-2)/4$ for $m=n/2$ when $n$ is even, and $(n-1)^2/4$ for $m=(n-1)/2$ when $n$ is odd. Hence, \eqref{scelta_Nekho} is the best choice which allows to get rid of the term $-m(n-m-1)$ in \eqref{codimension} and which still guarantees the essential condition $1\le s_m\le r-1$.  
	
	Choice \eqref{scelta_Nekho}, in our case, yields (see \eqref{bound}-\eqref{codimension})
	{\small
		\begin{equation}\label{bound2}
		\overline \alg_m(\overline s_m):=
		\begin{cases}
		\max\left\{1,r-1 +m(n-m-1)-\displaystyle \frac{n(n-2)}{4}\right\}\ ,\ &\text{  for $n$  even, $m=1$}\\
		\max\left\{1,r-1 +m(n-m-1)-\displaystyle \frac{(n-1)^2}{4}\right\}\ ,\   &\text{  for $n$  odd, $m=1$ }\\
		\max\left\{1,2r-3 +2m(n-m-1)-\displaystyle\frac{ n(n-2)}{2}\right\}\ ,\  &\text{  for $n$  even, $m\ge 2$}\\
		\max\left\{1,2r-3 +2m(n-m-1)-\displaystyle\frac{ (n-1)^2}{2}\right\}\ ,\   &\text{  for $n$  odd, $m\ge 2$ }\ ,\\
		\end{cases}
		\end{equation}}
	and 
	\begin{equation}\label{codimension2}
	\text{  codim }\OmegoneI\ge \begin{cases}
	\max\left\{0,r-1-\displaystyle\frac{n(n-2)}{4}\right\},\ \text{if $n$ is even}\\
	\max\left\{0,r-1-\displaystyle\frac{(n-1)^2}{4}\right\},\ \text{if $n$ is odd}\ .\\
	\end{cases}
	\end{equation}
	For $m=1$, the bound \eqref{bound2} on the steepness indices is half of the one found in \cite{Nekhoroshev_1979}. For $m\ge 2$, the estimates in \cite{Nekhoroshev_1979} coincide with \eqref{bound2}. Estimate \eqref{codimension2} on the codimension coincides with the one in \cite{Nekhoroshev_1979}. 
	\item Theorem \ref{Gen_steepness} holds true even for functions in the class $C_b^{r+1}$. In that case, one considers $1\le s_m\le \floor{r/2}$ for any $m=1,\dots,n-1$, which yields worse estimates both on the indices of steepness and on the codimension. This is the case which was considered in \cite{Nekhoroshev_1973}, whereas the regularity $C_b^{2r-1}$ was introduced in \cite{Nekhoroshev_1979}.  
\end{enumerate}
\subsection{Explicit algebraic criteria for steepness}

The kind of genericity stated in Theorem A implies the classic notions of genericity in topological and in measure sense. However, Theorem A alone is not sufficient when dealing with applications of Nekhoroshev's Theory to physical models. Infact, in order to infer long-time stability of a sufficiently regular integrable model Hamiltonian $h$ under any sufficiently small and regular perturbation, one needs to have a criterion to check at which points of its domain the given function $h$ is steep. As we shall show in the sequel, establishing a criterion of this kind is a non-trivial developement of the proof of the genericity of steepness. Namely, in the rest of this section we will present explicit algebraic criteria for steepness which involve the Taylor coefficients at any order of the studied function.

In order to give rigorous statements, we firstly need to introduce some notations.   

\subsubsection{Some additional notations}

For any pair of positive integers $n\ge 2$ and $k\in\{2,\dots,n\}$, we introduce the notation
\begin{equation}
\scV^1(k,n):=\{(v,u_2,\dots,u_k)\in \R^n\times \tU(k-1,n)|\text{ rk}(v,u_2,\dots,u_k)=k  \}\ .
\end{equation}

For fixed integers $m\ge 2$ and $\alpha\ge 0$, for any $\beta\in\{0,...,\alpha\}$, and for any $i\in\{1,...,m\}$, we also define the family of multi-indices

\begin{align}\label{mulan}
\N^m\ni\nu(i,\beta):=
\begin{cases}
&(\beta+1,0,...,0)\ ,\ \text{ for } i=1\\
&(\beta,0,...,0,1,0,...,0)\ ,\ \text{ for } i=2,...,m\\
\end{cases}
\end{align}
where the "$1$" fills the  $i$-th slot for $i=2,...,m$. When $\alpha\ge 1$, we denote the multi-indices $\mu\in\N^m$ of length $2\le|\mu|:=||\mu||_1\le\alpha+1$ not belonging to the family \eqref{mulan} with
\begin{equation}\label{sid}
\mathcal{M}_m(\alpha):=\{\mu\in\N^m\ ,\ \ 2\le|\mu|\le \alpha+1\}\backslash \bigcup_{\substack{i=1,...,m\\\beta=1,...,\alpha}}\{\nu(i,\beta)\}\ .
\end{equation}

Moreover, for given values of $\alpha\in \N$, $\mu=(\mu_1,\dots,\mu_m)\in \N^m$ and $\ell\in\{1,...,m\}$, we indicate by $\widetilde \mu(\ell)=(\widetilde \mu_1(\ell),\dots, \widetilde \mu_m(\ell))$ the multi-index for which $\widetilde\mu_i(\ell):=\mu_i-\delta_{i\ell}$, $i=1,\dots,m$, and we introduce the sets
\begin{align}\label{giammai}
\begin{split}
\mathcal{G}_m(\widetilde{\mu}(\ell),\alpha):=\biggl\{&k:=(k_{22},\dots,k_{2\alpha},\dots, k_{m2},\dots,k_{m\alpha})\in\N^{(m-1)\times(\alpha-1)}:\\
&\sum_{i=2}^{\alpha}k_{ji}=\mtj(\ell)\ ,\ \ \widetilde{\mu}_1(\ell)+\sum_{j=2}^{m}\sum_{i=2}^{\alpha}i\ k_{ji}=\alpha\biggl\} \ ,
\end{split}
\end{align}  
\begin{equation}\label{nonvuoto}
\cE_m(\ell,\alpha):=\{\mu\in \N^m\ |\  \mathcal{G}_m(\widetilde{\mu}(\ell),\alpha)\neq \varnothing\}\ .
\end{equation}

For any $k\in \mathcal{G}_m(\widetilde{\mu}(\ell),\alpha)$, we write $k!:=k_{22}!\dots k_{2\alpha}!\dots k_{m2}!\dots k_{m\alpha}!$ and, similarly, for any $\mu \in \N^m$, we use the notation $\mu!:=\mu_1!\dots\mu_m!\,$.

Finally, we consider a quadruplet of positive integers $r\ge 2,n\ge 3$, $1\le s\le r-1$, $2\le m \le n-1$, a point $I_0\in \R^n$, and a function $h$ of order $C^{r}$ around $I_0$. 

With this setting, for any given $\alpha\in \{1,\dots,s\}$, and $\ell\in\{1,\dots,m\}$,  we introduce the functions 
$$
\cH_{m,\ell,\alpha}^{h,I_0}: \scV^1(m,n)\times \R^{(m-1) s}\rightarrow\R\ \ \ 
$$ 
associating to any element 
$
(v,u_2,\dots,u_m)\in \scV^1(m,n)
$
and to any vector 
$
a(m,s)=(a_{21},\dots,a_{2s},\dots,a_{m1},\dots,a_{ms})\in \R^{(m-1)\times s}
$
the following quantities
 
\begin{align}\label{polpette}
\begin{split}
&\cH_{m,1,1}^{h,I_0}(v,u_2,\dots,u_m,a(m,s)):=h^2_{I_0}\left[v,v\right]\quad , \qquad \alpha=1\ ,\ \ \ell=1\\   &\cH_{m,\ell,1}^{h,I_0}(v,u_2,\dots,u_m,a(m,s)):=h^2_{I_0}\left[v,u_\ell\right]\quad ,\qquad \alpha=1\ ,\ \  \ell\in\{2,\dots,m\}\\
&\cH_{m,1,2}^{h,I_0}(v,u_2,\dots,u_m,a(m,s)):=h^3_{I_0}[v,v,v]\quad , \qquad  \text{if } s\ge 2\ , \ \text{ for }\alpha=2\ ,\  \text{and }\ \ell=1
\\
&------------------------------------\\
& \text{  if $s\ge 3$, for $\alpha \in\{3,\dots,s\}$, and $\ell=1$ }\\
\\
& \cH_{m,1,\alpha}^{h,I_0}(v,u_2,\dots,u_m,a(m,s)):=\\
& \frac{1}{\alpha! }h_{I_0}^{\alpha+1}\big[v,\dots,v\big]+\sum_{\beta=1}^{\alpha-1}\frac{1}{(\beta-1)!} h_{I_0}^{\beta+1}\big[\stackrel{ \beta  }{\overbrace{v}},\sum_{i=2}^m a_{i(\alpha-(\beta-1))}\,u_i\big] \\
&	+\sum_{\substack{\mu\in\cE_m(1,\alpha)\\  \mu\in\mathcal{M}_m(\alpha)\\\mu_1\neq 0 }}\, \sum_{k\in \cG_m(\widetilde \mu(1),\alpha)}\frac{h^{|\mu|}_{I_0}\big[\stackrel{ \mu_1-1  }{\overbrace{v}},\stackrel{ k_{22}  }{\overbrace{a_{22}u_2}},\dots, \stackrel{ k_{2\alpha}  }{\overbrace{a_{2\alpha}u_2}},\dots,\stackrel{ k_{m2}  }{\overbrace{a_{m2}u_m}},\dots, \stackrel{ k_{m\alpha}  }{\overbrace{a_{m\alpha}u_m}},v\big]}{(\mu_1-1)!\,k!}	\\
\end{split}
\end{align}

\begin{align}\label{polpettone}
\begin{split}
& \text{  if $s\ge 2$, for $\alpha \in\{2,\dots,s\}$, and $\ell\in\{2,\dots,m\}$ }\\
&	\cH_{m,\ell,\alpha}^{h,I_0}(v,u_2,\dots,u_m,a(m,s)):= \frac{1}{\alpha! }h^{\alpha+1}_{I_0}\big[\stackrel{ \alpha  }{\overbrace{v}},u_\ell\big]\\
& + \sum_{\substack{\mu\in\cE_m(\ell,\alpha)\\  \mu\in\mathcal{M}_m(\alpha)\\\mu_\ell\neq 0 }}\,\sum_{k\in \cG_m(\widetilde \mu(\ell),\alpha)} \frac{h^{|\mu|}_{I_0}\big[\stackrel{ \mu_1  }{\overbrace{v}},\stackrel{ k_{22}  }{\overbrace{a_{22}u_2}},\dots, \stackrel{ k_{2\alpha}  }{\overbrace{a_{2\alpha}u_2}},...,\stackrel{ k_{m2}  }{\overbrace{a_{m2}u_m}},..., \stackrel{ k_{m\alpha}  }{\overbrace{a_{m\alpha}u_m}},u_\ell\big]}{\mu_1!\,k!}.
\end{split}
\end{align}

\subsubsection{Theorem B and related corollaries}

With the setting above, we can state the first explicit criterion for steepness. Its Corollary B2 is a refined version of Theorem \ref{Teorema_due}.

\begin{thm*}[{\bf B}]\label{condizioni_esplicite_1}
	Let $r,n\ge 2$ be two integers, and let $\ts:=(s_1,\dots,s_{n-1})\in \N^{n-1}$ be a vector verifying $1\le s_i\le r-1$ for all $i=1,\dots,n-1$. Consider a point $I_0\in \R^n$, a real number $\varrho>0$, and a function $h$ of class $C^{2r-1}_b(B^n(I_0,\varrho))$ verifying $\nabla h(I_0)\neq 0$. 
	\\
	i) If the system
	\begin{equation}\label{s_1+1 jet}
	\begin{cases}
	w\in \S^n\\
	h_{I_0}^1[w]=h_{I_0}^2[w,w]=\dots =h_{I_0}^{s_1+1}[w,\dots,w]=0
	\end{cases} 
	\end{equation}
	has no solution\footnote{In this case, $h$ is said to be $s_1+1$-jet non-degenerate at the origin.}, then $h$ is steep  around the point $I_0$ on the affine subspaces of dimension one, with steepness index bounded by $s_1$.
	
	\medskip 
	
	In the sequel, we set $N=N(r,n):=\dim \Pol$.
	
	\medskip 
	
	ii) If, for some $m\in\{2,\dots,n-1\}$, there exists $R_m>0$ such that for any polynomial $S\in B^N(\tT_{I_0}(h,r,n),R_m)$ the system
	\begin{align}\label{non_tocca}
	\begin{split}
	\begin{cases}
	(u_1,\dots,u_m)\in \tU(m,n)\\[0.5em]
	a(m,s_m):=(a_{21},\dots,a_{2s_m},\dots,a_{m1},\dots,a_{ms_m})\in \R^{(m-1)\times s_m}\\[0.5em]
	v=u_1+\sum_{j=2}^m a_{j1}\,u_j\\[0.5em]
	S^1_{I_0}[v]=S^1_{I_0}[u_2]=\dots= S^1_{I_0}[u_m]=0\\[0.5em]
	\sum_{\ell=1}^m\sum_{\alpha=1}^{s_m}\left|\cH_{m,\ell,\alpha}^{S,I_0}(v,u_2,\dots,u_m,a(m,s_m))\right|=0
	\end{cases}
	\end{split}
	\end{align}
	has no solution, then $h$ is steep  in a neighborhood of $I_0$ on the affine subspaces of dimension $m$, with steepness index bounded by
	$
	\alg_m\le 2s_m-1\ .
	$
\end{thm*}

Though the quantities which are involved in Theorem B are quite cumbersome, the idea behind the result is not difficult to grasp: condition \eqref{non_tocca} amounts to asking that there exists a neighborhood of polynomials around the $r$-jet
of $h(I)-h(I_0)$ which lies outside of the closed semi-algebraic set $\Omegone$ defined in Theorem A. This will be made clearer in Corollary B2.

As it will be discussed in the technical sections of the present work (see e.g. section \ref{rs_vanishing_polynomials}), for any $m\in\{2,\dots,n-1\}$, the real parameters $a_{21},\dots,a_{2s_m},\dots,a_{m1},\dots,a_{ms_m}$ appearing in \eqref{polpette}-\eqref{polpettone} and in the statement of Theorem B represent the Taylor coefficients of analytic curves of the type
\begin{equation}\label{sirio}
\gamma(t):=
\begin{cases}
x_1(t):=t\quad , \qquad x_j(t):=\sum_{i=1}^{+\infty}a_{ji}t^i\qquad  j=2,\dots,m
\end{cases}
\end{equation}
which, for any $m$-dimensional affine subspace $I_0+\Gamma^m$, contain the locus of minima of the projection $||\pi_{\Gamma^m}\grad \tT_{I_0}(h,r,n)||$ appearing in \eqref{cuore_steepness}. For any given $\Gamma^m$ and for any regular function $h$ around $I_0$, the existence of a minimal curve of the form \eqref{sirio} is ensured by Theorem \ref{arco_minimale} in the sequel. 

Theorem B comes together with important corollaries. 

The following one is well-known: its statement can be found in \cite{Nekhoroshev_1977}, whereas its proof can be found in \cite{Chierchia_Faraggiana_Guzzo_2019}. As we shall see, in our context it is a simple consequence of Theorem B. 
\begin{cor*}[{\bf B1}]
	Consider an integer $n\ge2$, a point $I_0\in\R^n$, and a function $h$ of class $C^{5}$ around $I_0$, satisfying $\nabla h(I_0)\neq 0$. If the system
	\begin{equation}\label{3-jet}
	\begin{cases}
	w\in \S^n\\
	h^1_{I_0}[w]=h^2_{I_0}[w,w]=h^3_{I_0}[w,w,w]=0 \qquad \text{  ($3$-jet non degeneracy) }
	\end{cases}
	\end{equation}
	has no solution, then $h$ is steep in a neighborhood of $I_0$, and its indices satisfy
	$$
	\alpha_1=2\quad , \qquad \max_{m=2,\dots,n-1}\{\alpha_m\}\le 3\ .
	$$ 
\end{cor*}
\begin{rmk}
	Actually, as a more careful analysis of $3$-jet non-degenerate functions shows, the result is true for $C^4$ functions and one can take $ \max_{m=1,\dots,n-1}\{\alpha_m\}\le 2
	$ (see \cite{Chierchia_Faraggiana_Guzzo_2019}). 
\end{rmk}

Thanks to Theorem B, moreover, we have a more explicit characterisation of the sets $\OmegoneI$ appearing in the statement of Theorem A. Namely, as it was the case in the hypotheses of Theorem B, we consider two integers $r,n\ge 2$, and a vector $\ts:=(s_1,\dots,s_{n-1})\in \N^{n-1}$,  with $1\le s_i\le r-1$ for all $i=1,\dots,n-1$. Also, we take a point $I_0\in \R^n$. This time, differently to what we did in Theorem B, we do not consider a fixed function.

\begin{cor*}[{\bf B2}]\label{Toyota_Corolla}
	For $n\ge 2$, and $m=1$, we indicate by $
	\cZ^{r,s_1,1}_n
	$ the algebraic set of $\Pollo\times \S^n$ determined by 
	\begin{equation}\label{so un cavolo}
	\begin{cases}
	w\in \S^n\quad , \qquad Q\in\Pollo\\
	Q_{I_0}^1[w]=Q_{I_0}^2[w,w]=\dots =Q_{I_0}^{s_1+1}[w,\dots,w]=0\ .
	\end{cases}
	\end{equation}
	For $n\ge 3$, and for any given $m\in\{2,\dots,n-1\}$, we denote by $\cZ^{r,s_m,m}_n$ the algebraic set of $\Pollo\times \R^{(m-1)s_m}\times \scV^1(m,n)$ determined by 
	\begin{align}\label{machenneso}
	\begin{split}
	\begin{cases}
	(u_1,\dots,u_m)\in \tU(m,n)\quad , \qquad Q\in \Pollo\\
	a(m,s_m):=(a_{21},\dots,a_{2s_m},\dots,a_{m1},\dots,a_{ms_m})\in \R^{(m-1)\times s_m}\\
	v=u_1+\sum_{j=2}^m a_{j1}\,u_j\\
\sum_{\ell=1}^m\sum_{\alpha=1}^{s_m}\left|\cH_{m,\ell,\alpha}^{Q,I_0}(v,u_2,\dots,u_m,a(m,s_m))\right|=0\ .
	\end{cases}
	\end{split}
	\end{align}
	
	With this setting, one has
	\begin{equation}\label{tutti_insieme_appassionatamente}
	\bigcup_{m=1}^{n-1}\text{  closure }\left( \Pi_{\Pollo}\, \cZ^{r,s_m,m}_n\right)=\OmegoneI\ ,
	\end{equation}
	where $\OmegoneI$ is the semi-algebraic set introduced in Theorem A.
\end{cor*}
\begin{rmk}\label{commentario}
	Since the sets $\cZ^{r,s_m,m}_n$, $m=1,\dots,n-1$ in Corollary \ref{Toyota_Corolla} are algebraic, the Theorem of Tarski and Seidenberg (see Th. \ref{Tarski_Seidenberg}) - together with expression \eqref{tutti_insieme_appassionatamente} and Proposition \ref{closure-interior-boundary} - ensures that $\OmegoneI$ is a semi-algebraic set of $\Pol$, as we already knew by Theorem A. Moreover, it is worth to notice that - at least in principle - one could find the explicit expression for $\OmegoneI$. Infact, the Theorem of Tarski and Seidenberg is somewhat "constructive", in the sense that there exist algorithms that allow to find the explicit expression for the projection and the closure of any semi-algebraic set (see e.g. \cite{Basu_Pollack_Roy_2006}). However, these general algorithms are not very useful in applications, as their complexity grows double-exponentially with the number of the involved variables (see \cite{Heintz_Roy_Solerno_1990}). As we shall show in Theorems C1-C2 below, in "most cases" (in a sense that will be clarified in Theorem C3) the sets $\cZ^{r,s_m,m}_n$, with $m=2,\dots,n-1$ can be projected onto $\Pollo\times  \scV^1(m,n)$ with the help of a simple algorithm involving only linear operations. Moreover, such a projection yields a closed semi-algebraic set of $\Pollo\times  \scV^1(m,n)$. This implies a further criterion to check steepness of a given function. 
\end{rmk}

Finally, using Theorem B we can state a sufficient condition for non-steepness at a given point, namely
\begin{cor*}[{\bf B3}]
	Consider a point $I_0\in \R^n$, and a function $h$ in the real-analytic class around $I_0$ verifying $\nabla h(I_0)\neq 0$. 
	
	If at least one of the two following conditions is satisfied, then $h$ is non-steep at $I_0$.
	
	\begin{enumerate}
		
	\item  There exists $w\in \S^n$ such that
	$
	h_{I_0}^r[w]=0\ ,\ \  \forall r\in \N\ . 
	$
	
	\item  For some $m\in\{2,\dots,n-1\}$, there exist
	\begin{enumerate}
		
		\item  $m-1$ real sequences $\{a_{ji}\}_{i\in \N}$, $j=2,\dots,m$ and a number $\tti>0$ such that the expansions $\sum_{i=1}^{+\infty}a_{ji}t^i$ admit a radius of convergence greater than $\tti$ for all $m$ ;
		\item $m$ linearly-independent vectors $v,u_2,\dots,u_m\in \scV^1(m,n)$ ;
		
	\end{enumerate}
	such that for all integer $r\ge 2$ the following system is satisfied:
	\begin{align}\label{sountuboio}
	\begin{split}
	\begin{cases}
	(u_1,\dots,u_m)\in \tU(m,n)\\
	a(m,r-1):=(a_{21},\dots,a_{2(r-1)},\dots,a_{m1},\dots,a_{m(r-1)})\in \R^{(m-1)\times (r-1)}\\
	v=u_1+\sum_{j=2}^m a_{j1}\,u_j\\
	h^1_{I_0}[v]=h^1_{I_0}[u_2]=\dots= h^1_{I_0}[u_m]=0\\
	\sum_{\ell=1}^m\sum_{\alpha=1}^{r-1}\left|\cH_{m,\ell,\alpha}^{h,I_0}(v,u_2,\dots,u_m,a(m,s_m))\right|=0\ .
	\end{cases}
	\end{split}
	\end{align}
	
\end{enumerate}

\end{cor*}
\begin{rmk}
	Since we consider any $r\ge 2$, we have an infinite system. 
\end{rmk}
%{\small
%	\begin{align}\label{infinity}
%	\begin{split}
%	&	h^1_{I_0}[v]=h^1_{I_0}[u_2]=\dots= h^1_{I_0}[u_m]=0\\
%	&	h^2_{I_0}[v,v]=h^2_{I_0}[v,u_2]=\dots=h^2_{I_0}[v,u_m]=0\\
%	&	h^3_{I_0}[v,v,v]=0\\
%	\\
%	& \forall\  \alpha\in \N, \alpha\ge 3,\\
%	&\frac{1}{\alpha!} h_{I_0}^{\alpha+1}\big[v,\dots,v\big]+\sum_{\beta=1}^{\alpha-1}\frac{1}{(\beta-1)!} h_{I_0}^{\beta+1}\big[\stackrel{ \beta  }{\overbrace{v}},\sum_{i=2}^m a_{i(\alpha-(\beta-1))}\,u_i\big] \\
%	&	+\sum_{\substack{\mu\in\N^m\\  \mu\in\mathcal{M}(\alpha)\\\mu_1\neq 0 }}\sum_{k\in \cG(\widetilde \mu(1),\alpha)}\frac{h^{|\mu|}_{I_0}\big[\stackrel{ \mu_1-1  }{\overbrace{v}},\stackrel{ k_{22}  }{\overbrace{a_{22}u_2}},\dots, \stackrel{ k_{2\alpha}  }{\overbrace{a_{2\alpha}u_2}},\dots,\stackrel{ k_{m2}  }{\overbrace{a_{m2}u_m}},\dots, \stackrel{ k_{m\alpha}  }{\overbrace{a_{m\alpha}u_m}},v\big]}{(\mu_1-1)!\, k!}
%	=0\\
%	\\
%	& \forall\  \alpha'\in \N, \alpha'\ge 2,\\
%	& \frac{1}{\alpha'!} h^{\alpha'+1}_{I_0}\big[\stackrel{ \alpha'  }{\overbrace{v}},u_\ell\big]\\
%	&+ \sum_{\substack{\mu\in\N^m\\  \mu\in\mathcal{M}(\alpha')\\\mu_\ell\neq 0 }} \frac{1}{\mu_1!}\sum_{k\in \cG(\widetilde \mu(\ell),\alpha')} \frac{h^{|\mu|}_{I_0}\big[\stackrel{ \mu_1  }{\overbrace{v}},\stackrel{ k_{22}  }{\overbrace{a_{22}u_2}},\dots, \stackrel{ k_{2\alpha'}  }{\overbrace{a_{2\alpha'}u_2}},\dots,\stackrel{ k_{m2}  }{\overbrace{a_{m2}u_m}},\dots, \stackrel{ k_{m\alpha'}  }{\overbrace{a_{m\alpha'}u_m}},u_\ell\big]}{k!}=0\ .
%	\\
%	\end{split}
%	\end{align} }

\subsubsection{Theorems C1-C2-C3}

As we have showed above, Theorem B constitutes an explicit criterion for steepness which, however, for any given value of $n\ge 3$, $m\in\{2,\dots,n-1\}$ and $s_m\in\{1,\dots,r-1\}$ depends on the additional parameters $a_{21},\dots,a_{2s_m},\dots,a_{m1},\dots,a_{ms_m}\in \R^{(m-1)s_m}$ and on the vectors $v,u_2,\dots,u_m  \in \scV^1(m,n)$. As we have already pointed out in Remark \ref{commentario}, it is possible in principle to reduce these quantities from system \eqref{non_tocca}, by the means of classical algorithms of semi-algebraic geometry (see \cite{Basu_Pollack_Roy_2006}). However, in general the complexity of the latter grows double exponentially in the number of variables (see \cite{Heintz_Roy_Solerno_1990}) making them of little use in practice. 

However, since the quantities in \eqref{polpette}-\eqref{polpettone} are explicit, one may attempt to exploit their specific form in order to find an algorithm which is simpler than the classic ones and that eliminates at least the parameters $a_{22},\dots,a_{2s_m},\dots,a_{m2},\dots,a_{ms_m}$ from system \eqref{non_tocca}. In this way, one would have an explicit criterion for steepness involving only the multilinear forms of the tested function $h$ up to a given order, the parameters $a_{21},\dots,a_{m1}$, and the vectors $v,u_2,\dots,u_m$. Moreover, as we shall show in the sequel, without any loss of generality the numbers $a_{21},\dots,a_{m1}$ and the vector $v$ can be assumed to belong to a compact subset, whereas the vectors $u_2,\dots,u_m$ belong to $\tU(n,m-1)$, which is compact by definition. Hence, having an explicit criterion for steepness involving only the coefficients $a_{21},\dots,a_{m1}$ and the vectors $v,u_2,\dots,u_m$ as additional free parameters would be a qualititative improvement w.r.t. Theorem B in view of possible applications, as one would only have to consider parameters belonging to a compactum. Moreover, the presence of the vectors $v,u_2,\dots,u_m$ allows to keep track of the subspaces one is working on; namely, it is possible to isolate the subspaces where the studied function is non-steep. 

As we prove in sections \ref{Theorem C prova}-\ref{Prova C1-C4}, for a generic regular test function $h$ and for any $m\in\{2,\dots,n\}$,  on most of the $m$-dimensional subspaces of the Grassmannian $\tG(m,n)$ one is able to apply an explicit criterion to check steepness that does not involve the parameters $a_{22},\dots,a_{2s_m},\dots,a_{m2},\dots,a_{ms_m}$: this is the content of Theorems C1-C2-C3.

In order to state these results, we start by considering an integer $n\ge 3$, and a function $h$ of class $C^{2}$ around the origin, verifying $\grad  h(0) \neq  0$. Now, for any $m\in\{2,\dots,n-1\}$ we need to consider some subsets of the Grassmannian manifold $\tG(m,n)$. 

\begin{defn}\label{D}
	For any pair of integers $m \in \{2, ..., n - 1\}$ and  $j\in\{0,\dots,m\}$,  we indicate by $\sL_j(h,m,n)$ (resp. $\sL_{\ge j}(h,m,n)$) the subset of $\tG(m,n)$ containing those $m$-dimensional subspaces $\Gamma^m$ verifying
	\begin{enumerate}
		\item $	\grad h(0)\perp \Gamma^m$;
		\item 	the Hessian matrix of the restriction of $h$ to $\Gamma^m$, calculated at the origin, has exactly $j$ null eigenvalues (resp. at least $j$ null eigenvalues).
	\end{enumerate}
\end{defn}

With this definition, for any fixed $m\in\{2,\dots,n-1\}$, we have the partition
\begin{equation}
\{\Gamma^m\in \tG(m,n)| \Gamma^m\perp \grad h(0)\}=\sL_0(h,m,n)\bigsqcup \sL_1(h,m,n)\bigsqcup \sL_{\ge 2}(h,m,n)\ .
\end{equation}
We are now ready to state Theorems C1-C2-C3. 

Consider two positive integers $r,n \ge 2$, a vector $\ts:=(s_1,\dots,s_{n-1})\in \N^{n-1}$,  with $1\le s_i\le r-1$ for all $i=1,\dots,n-1$, and a function $h$ of class $C_b^{2r-1}$ around the origin, satisfying $\grad  h(0) \neq  0$. Then, for any given $m \in \{2, ..., n - 1\}$, one has the following results (which, considered together, are refined versions of Theorems \ref{Teorema_tre}-\ref{Teorema_quattro} in the introduction):

\begin{thm*}[\bf C1]\label{C1}
	$h$ is steep at the origin, with index $\alpha_m=1$, on the $m$-dimensional subspaces belonging to $\sL_0(h,m,n)$.	
\end{thm*}

\begin{thm*}[\bf C2]\label{C3}
	
	If $s_m\ge 2$  there exist two semi-algebraic sets
	$$
	\scA_1(r,s_m,n,m),\scA_2(r,s_m,n,m) \subset \Pollo \times \R^{m-1}\times \scV^1(m,n)
	$$
	enjoying the following properties:
	
	\begin{enumerate}
		\item The form of $\scA_1(r,s_m,n,m)$ can be explicitly computed starting from the expression of set $\cZ^{r,s_m,m}_n$ in \eqref{machenneso} by the means of an algorithm involving only linear operations.
		\item If system
		\begin{align}
		\begin{split}
		&\begin{cases}
		%(a_{21},\dots,a_{m1})\in \overline B^{m-1}(\scK)\\
		(u_1,\dots,u_m)\in \tU(m,n)\quad , \qquad  % \qquad v:=u_1+\sum_{i=2}^m a_{i1} u_i
		\text{ Span }(u_1,u_2,\dots,u_m)\in \sL_1(h,m,n)\\
		(\tT_{0}(h,r,n),0, u_1,u_2,\dots,u_m)\in \scA_1(r,s_m,n,m)
		\end{cases}
		\end{split}
		\end{align}
		has no solution, then $h$ is steep at the origin with index $\alpha_m\le 2s_m-1$  on any subspace $\Gamma^m\in \sL_1(h,m,n)$. 
		\item There exists a positive constant $\scK=\scK(r,n,m)$ such that if
		\begin{align}\label{brutto}
		\begin{split} 
		&\begin{cases}
		(a_{21},\dots,a_{m1})\in \overline B^{m-1}(\scK)\\
		(u_1,\dots,u_m)\in \tU(m,n)\quad , \qquad v:=u_1+\sum_{i=2}^m a_{i1} u_i\\
		\text{ Span }(v,u_2,\dots,u_m)\in \sL_{\ge 2}(h,m,n)\\
		(\tT_{0}(h,r,n),a_{21},\dots,a_{m1}, u_1,u_2,\dots,u_m)\in  \scA_2(r,s_m,n,m)
		\end{cases}
		\end{split}
		\end{align}
		has no solution, then $h$ is steep at the origin with index $\alpha_m\le 2s_m-1$ on any subspace $\Gamma^m\in \sL_{\ge 2}(h,m,n)$.
	\end{enumerate}
	
\end{thm*}

\begin{rmk}
	We observe that the statement above gives no information about the explicit expression of subset $\scA_2(r,s_m,n,m)$. As it will be shown in sections \ref{Theorem C prova}-\ref{Prova C1-C4}, the linear algorithm used to deduce the form of $\scA_1(r,s_m,n,m)$ starting from set $\cZ^{r,s_m,m}_n$ in \eqref{machenneso} fails in case $v,u_2,\dots,u_m$ span a subspace belonging to $\sL_{\ge 2}(h,m,n)$. Therefore, in order to find the explicit expression for $\scA_2(r,s_m,n,m)$ one is obliged to apply the classical, much slower algorithms of real-algebraic geometry to the set $\cZ^{r,s_m,m}_n$ determined by system \eqref{machenneso}. Thus, checking steepness on the subspaces of $\sL_{\ge 2}(h,m,n)$ is more complicated than on those belonging to $\sL_1(h,m,n)$, as one is obliged either to apply slow algorithms to find the explicit expression of system \eqref{brutto}, or to use the statement of Theorem B, which nevertheless depends on the non-compact real coefficients $a_{22},\dots,a_{2s_m},\dots,a_{m2},\dots,a_{ms_m}$. 
	
	However, for a generic function $h$ the subsets $\sL_1(h,m,n)$ and $\sL_{\ge 2}(h,m,n)$ are "small" inside the Grassmannian $\tG(m,n)$. Namely, in Theorem C3 below we prove that for any $m\in\{2,\dots,n-1\}$ and for any  bilinear symmetric non-degenerate form $\mathsf{B}: \R^n\times \R^n\longrightarrow\R$, the subspaces of dimension $m$ on which the restriction of $\mathsf{B}$ has one or two null eigenvalues are rare in $\tG(m,n)$, both in measure and in topological sense.

\end{rmk}

\begin{thm*}[\bf C3]\label{C4}
	Let $\mathsf{B}:\R^n\times \R^n \longrightarrow \R$ be a bilinear, symmetric, nondegenerate form, and let 
	$m\in\{2,\dots,n-1\}$ be a positive integer. 
	
	For $j\in\{1,2\}$, denote by $\tG_{\ge j}(\mathsf{B},m,n)\subset \tG(m,n)$ the subset of linear $m$-dimensional subspaces on which the restriction of $\mathsf{B}$ has at least $j$ null eigenvalues. 
	
	Then
	
	\begin{enumerate}
		\item $\tG_{\ge 1}(\mathsf{B},m,n)$ is contained in a submanifold of codimension one in $\tG(m,n)$; \item  $\tG_{\ge 2}(\mathsf{B},m,n)$ is obtained by the intersection of  $\tG_{\ge 1}(\mathsf{B},m,n)$ with another subset contained in a submanifold of codimension one in $\tG(m,n)$.

	\end{enumerate}
	
\end{thm*}

Finally, we state the following conjecture, which will hopefully be proved in a future work.

{\bf Conjecture}: for a generic bilinear form $\mathsf{B}$, the subset $\tG_{\ge 2}(\mathsf{B},m,n)$ appearing in Theorem C$3$ is contained in a submanifold of codimension two in $\tG(m,n)$. 

\subsubsection{Some observations on Theorems C$1$-C$2$}\label{si rompe la steepness}

	We stress that the criteria for steepness stated in Theorems C$1$-C$2$ are only valid for a given function on specific subspaces. They do not guarantee steepness on the same subspaces for an open neighborhood of functions around the studied one. Moreover, Theorems C$1$-C$2$ yield a uniform bound on the indices of steepness on the subspaces where their hypotheses are satisfied, but do not give any information on the coefficients of steepness, which may tend to zero as a subspace where the conditions are not verified is approached. This is why these theorems are somewhat "punctual", whereas Theorem A is "local" (it ensures steepness for an open set of functions, in a neighborhood of a considered point, with uniform estimates on the indices and coefficients). 
	
	To see an example showing why Theorems C$1$-C$2$ cannot be extended, in general, to an open set of functions, and why they do not provide uniform steepness coefficients, consider 
	$$
	h(I_1,I_2,I_3)=\frac{(I_1+I_2)^2}{2}+I_3
	$$ 
	and study on which subspaces of the origin it verifies the steepness condition. Even though Theorems C$1$-C$2$ deal with steepness on $m$-dimensional subspaces, with $m\ge 2$, for the sake of simplicity we will only consider one-dimensional subspaces in the sequel. Infact, the phenomena that we will point out already suggest what kind of "problems" may arise in higher dimensions when trying to extend the results of Theorems C$1$-C$2$ either to an open set of functions, or when trying to get uniform estimates on the steepness coefficients.
	
	The Hessian matrix of $h$ at the origin reads
	\begin{equation}
		h^2_{(0,0)}=
		\left(
		\begin{matrix}
		1 & 1  & 0\\	
		1 & 1  & 0\\
		0 & 0  & 0
		\end{matrix}
		\right) 
	\end{equation}
	and the vectors $(u_1,u_2,u_3)\in \R^3$ which are both isotropic for $h^2_{(0,0)}$ and orthogonal to $\grad h(0)$ must verify
	$$
	\begin{cases}
	(u_1,u_2,u_3)\cdot \grad h(0)=u_3=0\\
	h^2_{(0,0)}[u,u]=(u_1,u_2,u_3)		\left(
	\begin{matrix}
	1 & 1  & 0\\	
	1 & 1  & 0\\
	0 & 0  & 0
	\end{matrix}
	\right) 
		\left(
	\begin{matrix}
	u_1\\
	u_2\\
	u_3
	\end{matrix}
	\right) = (u_1+u_2)^2=0 
	
	\end{cases}
	$$	
	that is, they belong to the "bad" line
	$$
	\Gamma^1_{bad}:=\text{Span} \{(u_1,u_2,u_3)\ |\ u_3=0\ ,\ \ u_1+u_2=0\}
	$$
	 obtained by the intersection of the planes $u_3=0$ and $u_1+u_2=0$. Not only the Hessian is completely degenerate on $\Gamma^1_{bad}$: on a such a subspace the function $h$ is not even steep, as the gradient of the restriction 
	$
	\grad (h|_{\Gamma^1_{bad}})=I_1+I_2=0
	$
	is identically null, so that condition \eqref{cuore_steepness} cannot be verified. However, for any parameter $\eps\neq  0$, the Hessian is non-degenerate on the line
	$$
	\Gamma^{1,\eps}_{good}:=\text{Span}\{(1+\eps,-1,0)\}
	$$
as
$$
h^2_{(0,0)}[(1+\eps,-1,0),(1+\eps,-1,0)]=\eps^2> 0\ .
$$
This amounts to saying that $h$ is convex at the origin along $\Gamma^{1,\eps}_{good}$, hence that is it steep at the origin on the subspace $\Gamma^{1,\eps}_{good}$ with index equal to one. However, when $\eps$ is chosen to be arbitrarily small, $\Gamma^{1,\eps}_{good}$ approaches $\Gamma^{1}_{bad}  $ and the steepness property breaks down as the coefficients of steepness $C_1,\delta$ tend to zero. Hence, in the higher dimensional case, if the conditions of Theorems C$1$-C$2$ are only matched on certain good subspaces, it is reasonable to expect that one cannot have any uniform estimate on the steepness coefficients, as these may go to zero when a bad subspace is approached. 

Moreover, we introduce the one-parameter family of functions
$$
H_\lambda(I_1,I_2,I_3)=\frac{(I_1+(1+\lambda)I_2)^2}{2}+I_3\quad , \qquad \lambda\in \R
$$
verifying $H_0\equiv h$. For any given $\lambda_0\in \R$, it is plain to check that the function $H_{\lambda_0}$ is non-steep on the subspace $\Gamma^{1,\lambda_0}_{good}$, as the gradient of the restriction $\grad( H_{\lambda_0}|_{\Gamma^{1,\lambda_0}_{good}})$ is identically zero. However, as we have showed previously, the function $H_0$ is steep on $\Gamma^{1,\lambda}_{good}$ for any $\lambda\neq 0$. Therefore, we have showed that for any given $\lambda\neq 0$, there exists a subspace $\Gamma^{1,\lambda}_{good}$ on which $H_0$ is steep and a $\lambda$-close function $H_\lambda$ for the $C^2$-norm which is non-steep on $\Gamma^{1,\lambda}_{good}$. Hence, in Theorems C$1$-C$2$ it is not reasonable to expect that when $h$ is steep in a punctured neighborhood of the Grassmannian around a given subspace, then the same holds true for all functions in a neighborhood of $h$.

\section{The Thalweg and its properties}\label{thalweg_sec}

It is clear from Definition \ref{def steep} that studying the steepness property at the origin of a given function $h\in C^2(B^n(0,2\delta),\R)$ verifying $\nabla h(0)\neq 0$,  amounts to studying the projection of its gradient on any $m$-dimensional subspace $\Gamma^m$ perpendicular to $ \nabla h(0)$, with $m\in\{1,...,n-1\}$. More precisely, given $\delta>0$, for any fixed $\eta\in(0,\delta]$ we are interested in the quantity $$
\mu_h(\Gamma^m,\eta):=\min_{u\in\Gamma^m,\,||u||_2=\eta} ||\pi_{\Gamma^m}\,\nabla h(u)||_2 \ .
$$ 
Since, for any given $\Gamma^m$ orthogonal to $\grad h(0)$ and for any $\eta\in(0,\delta]$, the value $\mu_h(\Gamma^m,\eta)$ is attained at some point of the $m$-dimensional sphere 
$$
\cS_\eta^m:=\{u\in\Gamma^m\ |\ ||u||_2=\eta\}\ ,
$$ it makes sense to give the following 
\begin{defn}\label{thalweg}
	We call {\it Thalweg} of $h$ on $\Gamma^m$ the set
	$$
	\twg:=\{I^\star\in\Gamma^m : ||\pi_{\Gamma^m}\,\nabla h(I^\star)||_2=\mu_h(\Gamma^m,\eta) \text{\quad for \quad } 0\le \eta:=||I^\star||_2\le \delta \}\ .
	$$
\end{defn}

\medskip

In the sequel, we will be interested in studying the thalweg $\twgpo$ of the Taylor polynomial $\tT_0(h,r,n)$. Namely, the goal of this section is to prove the following
\begin{thm}{(Nekhoroshev, \cite{Nekhoroshev_1973})}\label{arco_minimale}
	For any pair of integers $r,n\ge 2$, and for any real $\delta>0$, consider a function $h\in C^r(B^n(0,2\delta))$ verifying $\nabla h(0)\neq 0$. Then, for any given number $m\in\{1,...,n-1\}$, for any $m$-dimensional subspace $\Gamma^m$ orthogonal to $\nabla h(0)$ there exists a semi-algebraic curve 
	$
	\gamma$ with values in $\twgpo
	$ 
	such that $\gamma(0)=0$ and 
	\begin{enumerate}
		\item For any fixed $\eta\ge 0$, the intersection $\text{ Im}(\gamma)\cap\mathcal{S}_\eta^m $ is a singleton;
		\item There exists a positive integer $\td=\td(r,n,m)$ that bounds the diagram (see Def. \ref{diagram}) of $\text{graph}(\gamma)$;
		\item There exists $\tK=\tK(r,n,m)>1$ such that, for any $\lambda>0$, the curve $\gamma$ is real-analytic on some closed interval $\mathtt{I}_\lambda\subset [-\lambda,\lambda]$ of length $\lambda/\tK$, with complex analyticity width $\lambda/\tK$;  
		\item Over $\mathtt{I}_\lambda$, $\gamma$ is an $i$-arc, i.e. it can be parametrized by
		$$
		\gamma(t):=\begin{cases}
		x_i(t)=t \qquad &\text{  for some }i\in\{1,...,m\}\\
		x_j(t)=f_j(t)\qquad &\text{  for all } j\in\{1,...,m\},\ j\neq i
		\end{cases}
		\qquad t\in\mathtt I_\lambda
		$$ 
		where the $f_j(t)$ are Nash (i.e. analytic-algebraic) functions;
		\item $\gamma$ satisfies a Bernstein's inequality on its Taylor coefficients over the interval $\tI_\lambda$. Namely,
		%\begin{align}\label{Bernstein}
		%\begin{split}
		%&\max_{z\in (\mathtt{I}_\lambda)_{\lambda/\tK}}|f_j(z)|\le \tM\,\max_{t\in\mathtt{I}_\lambda}|f_j(t)| \quad 
		%\text{ for all } j\in\{1,...,m\},\  j\neq i\ .
		%\end{split}
		%\end{align} 
		indicating by
		$$
		f_j(t)=\sum_{\beta=0}^{+\infty} a_{j\beta}( u)\,t^\beta\quad , \qquad j\in\{1,...,m\}\ ,\ \ j\neq i\ ,
		$$
		the Taylor expansion of $f_j$ at some point $u\in\tI_\lambda $, there exists a positive constants $K_2(r,n,m)$, and $\tM=\tM(r,n,m,\beta):=\beta!\times (\tK(r,n,m))^\beta\times  K_2(r,n,m)$ for which the following uniform estimate holds:
		\begin{equation}\label{Bernie}
		\max_{u\in\tI_\lambda}|a_{j\beta}(u)|\le \frac{\tM}{\lambda^{\beta-1}}\ .
		\end{equation}
	\end{enumerate}
	
\end{thm}
\begin{rmk}
	The Theorem above corresponds to reasonings holding in the polynomial setting. Moreover, the constants $\td,\tK, \tM$ depend only on the degree of the considered polynomial. Consequently, this Theorem holds uniformly for any $r$-jet of any function $h\in C^r(B^n(0,2\delta),\R)$.
\end{rmk}
\begin{rmk}
	Nekhoroshev calls $\gamma$ "minimal arc with uniform characteristics" (see \cite{Nekhoroshev_1973}, section 4). In that work, the statement of Theorem \ref{arco_minimale} is not given in the form above but is rather split in dispersed parts. Moreover, many of the modern tools of real-algebraic geometry were lacking at that time, so that the redaction of his work appears quite obscure in some parts. These two elements makes difficult for the reader to reconstruct simply Theorem \ref{arco_minimale} from Nekhoroshev's original paper. 
\end{rmk}
\begin{rmk}
	The Bernstein's inequality at point 5 of Theorem \ref{arco_minimale} is essential in order to have stable\footnote{In the sense given in Th. \ref{Gen_steepness}, that is valid for an open set of functions.} lower estimates for the steepness coefficients of $h$. For more details about this result, which is interesting in itself and has applications in various fields of mathematics, see refs. \cite{Roytwarf_Yomdin_1998} and \cite{Barbieri_Niederman_2022}. 
\end{rmk}
Some intermediate Lemmas are needed before demonstrating Theorem \ref{arco_minimale}.
\begin{lemma}\label{thalweg_semialgebrico}
	Consider any triplet of integers $r,n\ge 2$ and $m\in\{1,...,n-1\}$. There exists $\mathsf{d}=\mathsf{d}(r,n,m)\ge 0$ such that, for any $Q\in\Pol$ verifying $\grad Q(0)\neq 0$, and for any subspace $\Gamma^m$ perpendicular to $ \nabla Q(0)$, the thalweg $\twgq$ is a semi-algebraic set satisfying  $\diag(\twgq)\le \mathsf{d}$ (see Def. \ref{diagram}). Moreover, for any fixed $\eta_0>0$, the intersection of $\twgq$ with the sphere $\cS^m_{\eta_0}\subset \Gamma^m$ is compact.   
\end{lemma}
\begin{proof}
	$\Gamma^m$ is obviously isomorphic to $\R^m$ with the metric induced by the euclidean ambient space $\R^n$, and thus admits a global system of orthonormal coordinates $x=(x_1,...,x_m)$. We denote by $P(x)\in\Polm$ the restriction of $Q(I)$ to $\Gamma^m\simeq\R^m$. Since we endow $\Gamma^m$ with the induced euclidean metric, studying the norm of the projection of $\nabla_I Q(I)$ on $\Gamma^m$ amounts to studying the induced norm of $\nabla_x P(x)$ on $\Gamma^m\simeq\R^m$.
	Now, consider the semi-algebraic set 
	\begin{equation}\label{Ecors}
	\cE:=\{(x,y,\eta)\in\R^{2m}\times \R: ||x||_2^2=||y||_2^2=\eta^2, \eta>0, ||\nabla P(x)||_2>||\nabla P(y)||_2\}
	\end{equation}
	By the Theorem of Tarski and Seidenberg \ref{Tarski_Seidenberg} and Proposition \ref{complementary}, we have that the set
	$
	\R^m\backslash\pi_{x}\cE:=\{x\in\R^m\times\R:\forall (y,\eta)\in\R^m\times\R, (x,y,\eta)\not\in\cE\}
	$
	is semi-algebraic. We claim that it coincides with $\twgq$. Infact, by the definition of $\cE$, it is clear that for any given $\eta>0$ one has 
	\begin{equation}\label{minimal}
	x\in\R^m\backslash\pi_{x}\cE\ \Longleftrightarrow \ ||\nabla P(x)||_2\le  ||\nabla P(y)||_2 \text{ for all }y\in \R^m \text{ s.t. } {||y||_2}^2={||x||_2}^2\ ,
	\end{equation}
	so that $x\in\R^m\backslash\pi_{x}\cE$ is the locus of minima on any given sphere for $||\nabla P||_2$ (that is for $||\pi_{\Gamma^m}\nabla Q||_2$), that is it coincides with the Thalweg $\twgq$. Moreover, since $\deg P\le r$, the diagram of $\cE$ is uniformly bounded w.r.t. any $P\in\Polm$ and, again by the Theorem of Tarski and Seidenberg, the same is true for $\pi_x\cE$ and for $\R^m\backslash\pi_x\cE\equiv \twgq$.    
	
	It remains to prove that $\twgq\cap \cS^m_{\eta_0}$ is compact. By construction, $\twgq\cap \cS^m_{\eta_0}$ is the locus of minima of $||\grad P||_2$ on $\cS^m_{\eta_0}$.  Since the restriction of the function $||\grad P||_2$ to $\cS^m_{\eta_0}$ is continuous for the topology induced by $\Gamma^m$, the inverse image of its minimal value on $\cS^m_{\eta_0}$ is closed. Since $\cS^m_{\eta_0}$ is compact, the thesis follows.	
\end{proof}
The next Lemma shows how an analytic curve with uniform characteristics can be extracted from the Thalweg.

\begin{lemma}\label{Jess}
	Fix a triplet of integers $r,n\ge 2$ and $m\in\{1,..,n-1\}$. There exist positive constants $D(r,n,m)\in \N$, and $K_i=K_i(r,n,m)\in \R$, $i=1,2$, such that, for any $\xi>0$, for any polynomial $Q\in \Pol$ satisfying $\nabla Q(0)\neq 0$, and for any $m$-dimensional subspace $\Gamma^m$ orthogonal to $ \nabla Q(0)$, there exists a semi-algebraic curve $\phi=(\phi_1(\eta),...,\phi_m(\eta)):[0,\xi]\longrightarrow\twgq$ having the following properties
	\begin{enumerate}
		\item  For any fixed $\eta\in[0,\xi]$, the intersection $\text{ Im}(\phi)\cap\mathcal{S}_\eta^m $ is a singleton;
		\item The diagram of $\text{graph}(\phi)$ (see Definition \ref{diagram}) is bounded by $D$;
		\item There exists a closed interval $\cI_{\xi}\subset[0,\xi]$ of length $\xi/K_1$ over which $\phi$ is real-analytic, with complex analyticity width $\xi/K_1$;
		\item On the closed complex polydisk $(\cI_\xi)_{\xi/K_1}$ of width $\xi/K_1$ around $\cI_\xi$, one has the uniform Bernstein's inequality
		$$ \max_{\eta\in(\cI_\xi)_{\xi/K_1}}|\phi_j(\eta)|\le K_2\, \xi\qquad  \text{ for any $j\in\{1,...,m\}$.}
		$$ 
		
	\end{enumerate}
\end{lemma}
\begin{proof}
	
	As in Lemma \ref{thalweg_semialgebrico}, we consider the isomorphism $\Gamma^m\simeq \R^m$, and we endow $\Gamma^m$ with a global system of orthonormal coordinates $x=(x_1,...,x_m)$ for the scalar product induced by the ambient space $\R^n$. We proceed by steps. At the first step, we build a semi-algebraic function $\phi_1$ associating to a sphere $\cS^m_\eta$ of given radius  $\eta> 0$ the minimal value attained by the coordinate $x_1$ on $\cS^m_\eta$. At Step 2 we apply Yomdin's reparametrization Lemma (see appendix \ref{riparametrizzazione}) to one of the algebraic components of $\phi_1$ and we get a function with the suitable properties. Finally, at Step 3, we repeat the same construction for the other coordinates. 
	
	{\it Step 1.}
	For any $\xi>0$, by Lemma \ref{thalweg_semialgebrico} the set 
	$$
	\twgqxi:=\twgq\cap \{x\in\R^m: ||x||_2^2\le \xi^2\}
	$$ 
	is  semi-algebraic and its diagram is bounded by a positive constant $\mathsf{d}=\mathsf{d}(r,n,m)$. By the Theorem of Tarski and Seidenberg \ref{Tarski_Seidenberg}, the continuous function 
	$$
	f_1:=\twgqxi\longrightarrow\R\quad ,\qquad x\longmapsto x_1
	$$
	is semi-algebraic and its diagram is bounded by a quantity depending only on $r,n,m$. Infact, 
	$$
	\text{graph}(f_1):=\Pi_{\R^m\times \R}\{(u,v)\in \R^m\times \R^m \ |\ u\in \twgqxi, u=v\}\ .
	$$ 
	Moreover, Lemma \ref{thalweg_semialgebrico} ensures that for any $0<\eta_0<\xi$ the set $\twgqxi\cap \cS^m_{\eta_0}$ is compact, so that $f_1$ has a minimum on it. On the other hand, the function $g_1:\twgqxi\longrightarrow \R$, $x\longmapsto ||x||_2$ is also semi-algebraic and its diagram is bounded by a quantity depending only on $r,n,m$, because
	$$
	\text{graph}(g_1):=\{(x,y)\in \R^{m}\times \R\ |\ x\in \twgqxi\ ,\ \  ||x||_2^2-y^2=0\}\ .
	$$
	Then, by applying Proposition \ref{vietnamiti}, we have that the function
	\begin{equation}\label{psiuno}
	\phi_1: [0,\xi] \longrightarrow \R\qquad \eta\longmapsto \inf_{x\in g_1^{-1}(\eta)}f_1(x)=\min_{x\in \twgqxi\cap\cS^m_\eta}\{x_1\}
	\end{equation}
	is semi-algebraic and we indicate by $d_1=d_1(r,n,m)$ its diagram.
	
	{\it Step 2.} Corollary \ref{pollo} ensures the existence of a number $N_1=N_1(d_1)$ and of an open interval $\cI^1\subset [0,\xi]$ of length $\xi/N_1$ over which the restriction $\phi_1|_{\cI^1}$ is algebraic. By Proposition \ref{p-valency}, $\phi_1|_{\cI^1}$ is $d_1$-valent and has no more than $d_1$ zeros on its domain so that there exists an interval $\cJ^1\subset\cI^1$ of length $\xi/(N_1(d_1)\times(d_1+1))$ over which the restriction $\phi_1|_{\cJ^1}$ has definite sign. Without loss of generality we can assume $ \phi_1(\eta)\ge 0$ for all $\eta\in \cJ^1$ (one considers $-\phi_1$ otherwise). 
	
	We denote by $\xi_1$ and $\xi_2$ the extremal points of the interval $\cJ^1$ and we rescale the domain by setting 
	\begin{equation}\label{phiuno}
	\varphi_1:[-1,1]\longrightarrow \R\quad ,\qquad \varphi_1(u):=\phi_1\left(\frac{u+1}{2}\xi_2-\frac{u-1}{2}\xi_1\right)\ .
	\end{equation}
	We also define the function
	\begin{equation}\label{phiunotilde}
	\widetilde \varphi_1(u):=\displaystyle \frac{\varphi_1(u)}{\xi}\quad ,\qquad 0\le \widetilde \varphi_1(u)\le 1\ ,
	\end{equation}
	which satisfies the hypotheses of Theorem \ref{yomdin}. With the notations of Theorem \ref{yomdin}, we choose the value $\delta:=\displaystyle \frac{1}{8\tY_1(d_1)}$ so that, once at most $\tY_1$ neighborhoods of length $2\delta$ around the singularities of $\widetilde \varphi_1$ are eliminated from $[-1,1]$, the remaining set has a measure which is no less than
	$
	2-\displaystyle\sum_{i=1}^{\tY_1}2\times \displaystyle 1/(8\tY_1)=7/4 \ .
	$
	Moreover, the number of the partition intervals is bounded by the uniform quantity $\tY_2\log_2\left(8\tY_1\right)$, so that there exists an interval $\Delta_1\subset [-1,1]$ verifying 
	$$
	|\Delta_1| =\frac{7}{4\,\tY_2\log_2\left(8\tY_1\right)}
	$$
	on which $\widetilde \varphi_1$ is real-analytic with uniform analyticity width $\frac32|\Delta_1|$. Infact, by Proposition \ref{larghezza_analiticita}, the complex singularities of $\widetilde \varphi_1$ are at distance no less than $3|\Delta_1|$ from the center $c_1$ of $\Delta_1$. By Theorem \ref{yomdin}, $\Delta_1$ can be affinely reparametrized by a function $\psi_1:[-1,1]\longrightarrow \Delta_1$, which maps the closed complex disc $\overline \cD_3(0)$ into the closed complex disc $\overline \cD_{\varrho}(c_1)$ of radius $\varrho:=\frac32 |\Delta_1|$. Hence,  we can write
	\begin{align}
	\begin{split}
	\max_{u\in \cD_{\varrho}(c_1)}|\widetilde \varphi_1(u)|&\le 	\max_{u\in \cD_{\varrho}(c_1)}|\widetilde \varphi_1(u)-\widetilde\varphi_1(0)|+|\widetilde\varphi_1(0)|\\
	&=\max_{z\in \cD_{3 }(0)}|\widetilde \varphi_1\circ\psi_1(z)-\widetilde\varphi_1\circ\psi_1(0)|+|\widetilde\varphi_1(0)|\\
	\end{split} 
	\end{align}
	so that, by Definition \ref{reparametrization} and Theorem \ref{yomdin} and by the fact that $|\widetilde\varphi_1(u)|\le 1$ for any $u\in [-1,1]$, we obtain
	\begin{align}\label{uniform}
	\max_{u\in \cD_{\varrho}(c_1)}|\widetilde \varphi_1(u)|&\le 	2\ .
	\end{align}
	Scaling back to the original variables, by \eqref{phiuno} the interval $\Delta_1$ is mapped into an interval $\Delta_1^\xi$ of length $|\Delta_1^\xi|=|\Delta_1|\frac{\xi_2-\xi_1}{2}=|\Delta_1|\frac{\xi}{2N_1(d_1)\times(d_1+1)}$ and center $\widehat c_1$ and, in the same way, the radius rescales as $\varrho\rightsquigarrow \varrho\frac{\xi}{2N_1(d_1)\times(d_1+1)}$. Therefore, taking into account \eqref{phiunotilde} and \eqref{uniform}, there exists a uniform constant $\tM_1=\tM_1(d_1)$ such that the following Bernstein's inequality is satisfied
	\begin{equation}\label{Bern}
	\max_{z\in \cD_{\xi/\tM_1}( \widehat c_1)}| \phi_1(z)|\le 2\xi	\ .
	\end{equation}
	%	Finally, estimate \eqref{Bern} together with the standard Cauchy inequality yields
	%	\begin{equation}\label{Bern_der}
	%	\max_{z\in \cD_{\xi/(2\tM_1)}(\mathtt c_1)}| \phi_1'(z)|\le 4 \tM_1\ .	
	%	\end{equation}

	{\it Step 3. } Since $\Delta_1^\xi$ is compact, $\phi_1(\Delta_1^\xi)$ is also compact and the inverse image $\mathtt U^1_{\xi}(Q,\Gamma^m):=g_1^{-1}(\phi_1^{-1}(\phi_1( \Delta_1^\xi)))$ is closed. Moreover, since the diagrams of $\phi_1$ and $g_1$ depend only on $r,n,m$, then by Propositions \ref{immagine}-\ref{composta}-\ref{inversa} the diagram of $\mathtt U^1_{\xi}(Q,\Gamma^m)$ also depends only on $r,n,m$. Hence, for any fixed $\eta\in\Delta_1^\xi$ we have that the set $\cS^m_{\eta}\cap \mathtt U^1_{\xi}(Q,\Gamma^m)$, which contains the points of the Thalweg  that have minimal coordinate $x_1$ on the sphere of radius $\eta$, is compact and semialgebraic with a bound on its diagram depending only on $r,n,m$. Hence, the coordinate $x_2$ admits a minimum on this set and we can repeat the same argument of Step 2 on the function
	\begin{equation}\label{psidue}
	\phi_2: \Delta_1^\xi \longrightarrow \R\qquad \eta\longmapsto \min_{x\in g_2^{-1}(\eta)}f_2(x)=\min_{x\in \cS^m_{\eta}\cap \mathtt U^1_{\xi}(Q,\Gamma^m)}\{x_2\}
	\end{equation}
	where we have set 
	$
	f_2:\cS^m_{\eta_0}\cap \mathtt U^1_{\xi}(Q,\Gamma^m)\longrightarrow\R\ ,\ \  x\longmapsto x_2 \quad 
	$ and 
	$
	g_2:\cS^m_{\eta_0}\cap \mathtt U^1_{\xi}(Q,\Gamma^m)\longrightarrow \R,  x\longmapsto ||x||_2\ .
	$
	The curve $\phi:=(\phi_1,\phi_2,...,\phi_m)$ is constructed by iterating this procedure $m$ times.
	
	Points 1, 2, and 3 of the thesis follows easily from this construction. %Point 2 follows by induction since, by the arguments in Step 2 of the proof, for any given interval of given length, each of the functions $\phi_j$, $j=1,...,m$, is real-analytic on a subinterval whose length is uniformly (for all polynomials of degree $r$) bounded from below and the complex analyticity width admits the same lower bound. 
	Point 4 is a consequence of estimate \eqref{Bern} applied to the complex polydisk of uniform width $\xi/K_1$ around the common uniform real interval of analyticity $\cI_\xi$ of the functions $\phi_1,...,\phi_m$.

\end{proof}
We are now ready to state the proof of Theorem \ref{arco_minimale}. 

\begin{proof}{(Theorem  \ref{arco_minimale})}
	We assume the setting of Lemma \ref{Jess} with $Q$ equal to the Taylor expansion $\tT_0(h,r,n)$, and we proceed by steps. At Step 1, we show that there exists a component $\phi_i$, $i\in\{1,\dots,m\}$, of the curve $\phi$ introduced in Lemma \ref{Jess} whose first derivative admits a lower bound on a domain of uniform length. Then, at the second step, we use this fact to apply a quantitative inverse function Theorem and we reparametrize $\phi$ by the $i$-th coordinate. Steps $3$ and $4$ contain, respectively, the proofs of points 1-4 and of point 5 in the statement. 
	
	{\it Step 1.} We cut the uniform interval of analyticity $\cI_\xi$ into three equal intervals and we denote by $\widehat \cI_\xi$ the central one, whose length is $|\widehat\cI_\xi|=|\cI_\xi|/3$. We indicate by $\widehat \xi_1$ and $\widehat \xi_2$ the extreme points of $\widehat\cI_\xi$. Since for any given $\eta\in \widehat\cI_\xi$ by Lemma \ref{Jess} we have $\eta^2=\phi_1^2(\eta)+...+\phi_m^2(\eta)$, there must be some component $\phi_i$ of the curve, with $i\in\{1,...,m\}$, verifying
	\begin{equation}\label{tricchete}
	|\phi_i(\widehat\xi_2)-\phi_i(\widehat \xi_1)|\ge\frac{|\widehat\cI_\xi|}{m}= \frac{\xi}{3mK_1}\ .
	\end{equation}
	At the same time, for some point $\widehat \xi_3\in \widehat\cI_\xi$ we have
	\begin{equation}\label{tracchete}
	|\phi_i(\widehat\xi_2)-\phi_i(\widehat \xi_1)|= |\phi_i'(\widehat\xi_3)||\widehat\cI_\xi|=|\phi_i'(\widehat\xi_3)|\frac{\xi}{3K_1}\ .
	\end{equation}
	On the one hand, relations \eqref{tricchete} and \eqref{tracchete} together imply
	\begin{equation}\label{enne}
	|\phi_i'(\widehat\xi_3)|\ge \frac{1}{m}\ .
	\end{equation}
	On the other hand, for any $\eta\in [\widehat \xi_3-|\widehat \cI_\xi|,\widehat \xi_3+|\widehat \cI_\xi|]\subset\cI_\xi$ one has the estimate
	\begin{equation}
	|\phi_i'(\eta)-\phi_i'(\widehat\xi_3)|\le\max_{\cI_\xi}|\phi''_i| |\eta-\widehat \xi_3|
	\end{equation}
	which, thanks to the classic Cauchy estimate and to the Bernstein inequality of Lemma \ref{Jess}, implies
	\begin{equation}\label{luca}
	|\phi_i'(\eta)-\phi_i'(\widehat\xi_3)|\le\frac{2K_1^2}{\xi^2}\max_{(\cI_\xi)_{\xi/K_1}}|\phi_i| |\eta-\widehat \xi_3|\le \frac{2K_1^2K_2}{\xi} |\eta-\widehat \xi_3|\ .
	\end{equation}
	Hence, for any $\eta$ in the interval $J_\xi:=\left[\widehat \xi_3-\displaystyle\frac{\xi}{4mK_1^2K_2},\widehat \xi_3+\displaystyle\frac{\xi}{4mK_1^2K_2}\right]\subset \cI_\xi$ we have by \eqref{enne} and \eqref{luca} that 
	\begin{equation}\label{stima dal basso}
	|\phi_i'(\eta)|\ge \left||\phi_i'(\widehat\xi_3)|-|\phi_i'(\eta)-\phi_i'(\widehat\xi_3)|\right|\ge \frac{1}{m}-\frac{2K_1^2K_2}{\xi}\frac{\xi}{4mK_1^2K_2}=\frac{1}{2m}\ .
	\end{equation}
	
	{\it Step 2.} By Lemma \ref{Jess} and by the construction at Step 1 we can apply the quantitative local inversion Theorem \ref{loc-inv} for $\phi_i$ at any point $\eta\in J_\xi\subset \cI_\xi$. By making use of the notations in Theorem \ref{loc-inv}, we can set the uniform parameters
	\begin{equation}
	R:= \frac{\xi}{K_1} 
	\quad ,\qquad |\phi'_i(\eta)|\ge \frac{1}{2m}\quad ,\qquad  \max_{(J_\xi)_{\xi/K_1}}|\phi_j''|\le 2\frac{K_1^2K_2}{\xi}\ . 
	\end{equation}
	Hence, $\phi_i$ is invertible in the complex closed polydisk $(J_\xi)_{R'/16}$ around the real interval $J_\xi$, where 
	$$
	R':=\frac12\times \min\left\{R,\displaystyle \frac{\min_{J_\xi}|\phi_i'|}{ \max_{(J_\xi)_{\xi/(2K_1)}}|\phi_i''|},\right\}=\frac{\xi}{8mK_1^2K_2}\ .
	$$
	Since, by construction, $\phi_i$ is real-analytic in $J_\xi$, the continuity of the derivative ensures that $\phi_i(J_\xi)$ is an interval of $\R$ and, by the definition of the function $\phi_i$ in \eqref{psiuno}, one has $\phi_i(J_\xi)\subset [-\xi,\xi]$. The inverse function is analytic in the complex polydisc of uniform width 
	\begin{equation}\label{larghezza}
	R'':=\min_{J_\xi}|\phi_i'|\frac{R'}{8}\ge \frac{R'}{16m}=\frac{\xi}{128\,m^2K_1^2K_2}
	\end{equation}
	around $\phi_i(J_\xi)$. Moreover, using \eqref{stima dal basso}, one has that  
	\begin{equation}\label{uniform_interval}
	|\phi_i(J_\xi)|\ge\min_{J_\xi}|\phi_i'|\times |J_\xi|\ge \frac{1}{2m}\times \frac{\xi}{2mK_1^2K_2}\ .
	\end{equation}
	{\it Step 3.} 	Point 1 of Theorem \ref{arco_minimale} follows by Point 1 of Lemma \ref{Jess} and by the local inversion Theorem applied at Step 2. Points 2, 3, and 4 of Theorem \ref{arco_minimale} are also immediate consequences of the local inversion Theorem at Step 2. 
	
	Namely, by keeping in mind the notations at Point 4 of Theorem \ref{arco_minimale}, the curve $\gamma:=\phi\circ\phi_i^{-1}$ can be defined as
	\begin{equation}\label{j_arco}
	\gamma(t):=\begin{cases}
	x_i=t\\
	x_j(t)=f_j(x_i):=\phi_j(\phi_i^{-1}(x_i)) \qquad \text{  for all }j\in\{1,...,m\}\,,\quad j\neq i\ .
	\end{cases}
	\end{equation}
	
	The existence of an interval of analyticity with uniform length and complex width for $\gamma$ is a consequence of \eqref{larghezza} and \eqref{uniform_interval} and the constant $\tK$ in the statement can be taken equal to
	\begin{equation}\label{tkappa}
	\tK:= 128\,m^2K_1^2K_2\ .
	\end{equation}
	Indeed, for any $0<\lambda\le\xi$, $\tI_\lambda$ can be chosen to be any interval of length $\lambda/\tK$ contained in the interval $\phi_i(J_\lambda)\subset [-\lambda, \lambda]$ (see \eqref{uniform_interval}). For later convenience, we also observe that the above discussion implies that 
	\begin{equation}\label{inclusione}
	(\cI_\lambda)_{\lambda/K_1}\supset (J_\lambda)_{\lambda/K_1}\supset \phi_i^{-1}(\tI_\lambda)_{\lambda/\tK} \quad ,\qquad \forall \lambda\in(0, \xi]\ .
	\end{equation}

	The fact that the diagram of $\text{graph}(\gamma)$ depends only on $r,n,m$ is an immediate consequence of \eqref{j_arco}, together with point 2 of Lemma \ref{Jess} and with Propositions \ref{composta}-\ref{inversa}. 
	
	{\it Step 4.} It remains to prove the Bernstein's inequality at Point 5 of the statement. 
	By Lemma \ref{Jess}, for any $\lambda>0$ we have
	\begin{equation}\label{cheneso}
	\max_{\eta\in(\cI_\lambda)_{\lambda/K_1}}|\phi_j(\eta)|\le K_2\,\lambda
	\end{equation}
	for some uniform constant $K_2=K_2(r,n,m)$ and for any $j\in\{1,...,m\}$. %Estimate \eqref{cheneso}, together with the classic Cauchy inequality, implies 
	%\begin{equation}
	%	\max_{\eta\in(\cI_\xi)_{\xi/(2K_1)}}|\phi_j^{(k)}(\eta)|\le\frac{k!\,2^k\,K_1^k }{\xi^k}\times  \max_{\eta\in(\cI_\xi)_{\xi/K_1}}|\phi_j(\eta)|\le 2K_1\,K_2 \, k!\left(2\frac{K_1 }{\xi}\right)^{k-1}\ .
	%	\end{equation}
	By construction in \eqref{j_arco},
	$
	f_j(x_1)=\phi_j\circ\phi_i^{-1}(x_i)
	$, and for any $\beta\in\N\cup\{0\}$ the classic Cauchy estimate implies
	\begin{equation}\label{cauchy}
	\max_{t\in\tI_\lambda}|f_j^{(\beta)}(t)|\le \beta!\,\tK^\beta\frac{\max_{z\in(\tI_\lambda)_{\lambda/\tK}}|f_j(z)|}{\lambda^\beta}=\beta!\, \tK^\beta\frac{\max_{z\in(\tI_\lambda)_{\lambda/\tK}}|\phi_j\circ\phi_i^{-1}(z)|}{\lambda^\beta} \ .
	\end{equation}
	For any $0<\lambda\le \xi$, by \eqref{inclusione} one has $\phi_i^{-1}((\tI_\lambda)_{\lambda/\tK})\subset (\cI_\lambda)_{\lambda/K_1}$. Taking this into account, \eqref{cauchy} and \eqref{cheneso} yield
	\begin{equation}
	\max_{t\in\tI_\lambda}|f_j^{(\beta)}(t)|\le \beta!\,\tK^\beta\frac{\max_{\eta\in(\cI_\lambda)_{\lambda/K_1}}|\phi_j(\eta)|}{\lambda^\beta}\le \beta!\,  \tK^\beta \frac{K_2\lambda}{\lambda^\beta}=\beta!\,\tK^\beta \frac{K_2}{\lambda^{\beta-1}} \ .
	\end{equation}
	The thesis at Point 5 in the statement follows by setting $\tM=\beta!\times K_2\times \tK^\beta$.

\end{proof}

\section{$s$-vanishing polynomials}\label{rs_vanishing_polynomials}
We take into account the results and the notations of the previous section, in particular Theorem \ref{arco_minimale}.

\subsection{Heuristics and Definitions}

The goal of the first part of this paragraph is to provide the reader with a heuristic justification for introducing the special class of $s$-vanishing polynomials in the study of the genericity of steepness. A rigorous description of the rôle played by these polynomials will be given in the next paragraphs and sections. 

%In Lemma \ref{arco_minimale}, we have seen that the Thalweg $\twgpo$ on any subspace $\Gamma^m\subset \nabla h(0)^\perp$ of the $r$-jet at the origin $(h,r,n)$ of any function $h\in C^r(B^n(0,2\delta))$ contains the image of a curve $\gamma $ which is real-analytic in a uniform interval with uniform analyticity width. Moreover, on this interval $\gamma$ can be parametrized by one of the coordinates and satisfies a uniform Bernstein inequality. The adjective "uniform" in this context refers to the fact that the involved quantities depend only on $r,n,m$. 

For any fixed integer $n\ge 2$, we consider the euclidean space $\R^n$ and we endow any of its linear subspaces with the induced metric. For any pair of positive integers $1\le m\le n-1$ and $r\ge 2$, for any given function $h$ of class $C^r$ near the origin verifying $\nabla h(0)\neq 0$, and for any $m$-dimensional subspace $\Gamma^m$ orthogonal to $ \nabla h(0)$, by  Def. \ref{thalweg}, the set $\twgpo$ is the locus of minima of
$||\pi_{\Gamma^m}\nabla \tT_0(h,r,n)||_2$ on the spheres $\cS^m_\eta(0)\subset \Gamma^m$, with $\eta>0$. 

In Theorem \ref{arco_minimale} we have proved the existence of a minimal semi-algebraic arc $\gamma$ (see \eqref{j_arco}) of diagram $\td(r,n,m)$ parametrized by one coordinate and whose image is contained in the thalweg $\twgpo$. Due to Proposition \ref{pollide} - $\gamma$ is piecewise algebraic, with a maximal number of algebraic components depending only on its diagram $\td(r,n,m)$. With the exception of a finite set of complex points, any algebraic function admits a local holomorphic extension, and the number of its singularities is bounded by a quantity depending only on its diagram (see appendix \ref{riparametrizzazione}, or \cite{Barbieri_Niederman_2022} for more details). Therefore, $\gamma(t)$ is real-analytic with the exception of a finite number of points whose cardinality is bounded uniformly by a quantity depending solely on $\td(r,n,m)$. In particular, for any $\lambda>0$, this ensures the existence of an interval $\tI_\lambda\subset [-\lambda, \lambda]$ of uniform length $\lambda/\tK(r,m,n)$, where $\tK=\tK(r,m,n)$ is a suitable constant, over which $\gamma(t)$ is real-analytic with complex analyticity width $\lambda/\tK$.  

By the above reasonings, for sufficiently small $\lambda>0$ the interval $(-3\lambda,3\lambda)$ contains no singularities of $\gamma(t)$. In particular, $\gamma(t)$ is real analytic in $\tL_\lambda:=(\lambda,2\lambda)$, with complex analyticity width $\lambda$, and the same holds also for ${||\left.\pi_{\Gamma^m}\nabla \tT_0(h,r,n)\right|_{\gamma(t)}||_2}^2$ in that interval. Hence, if the function ${||\left.\pi_{\Gamma^m}\nabla \tT_0(h,r,n)\right|_{\gamma(t)}||_2}^2$ has a zero of infinite order at some point $t^\star\in\tL_\lambda$, then it is identically null in $\tL_\lambda$.
Then, by  Definition \ref{def steep}, and by the minimality of $\gamma$, this implies that the polynomial $\tT_0(h,r,n)$ cannot satisfy the steepness property at the origin on the subspace $\Gamma^m$. 

We claim that a kind of converse result - involving $\tI_\lambda$ instead of $\tL_\lambda$ - is also true: if $\tT_0(h,r,n)$ is non-steep at the origin on the subspace $\Gamma^m$, then ${||\left.\pi_{\Gamma^m}\nabla \tT_0(h,r,n)\right|_{\gamma(t)}||_2}^2$ must have a zero of infinite order in $\tI_\lambda$. This observation is fundamental in order to prove Theorem A. Actually, the necessity of a zero of infinite order has been proved for a real-analytic function in \cite{Niederman_2006} via the curve-selection Lemma; however, in the polynomial setting considered here, we have a much stronger quantitative result.  

Motivated by this heuristic argument, we are interested in studying the properties of those real polynomials of $m\ge 1$ variables whose gradient has a zero of sufficiently high order on some curve $\gamma$ parametrized by one coordinate. In a first moment, we do not consider the fact that these polynomials are the restrictions to a $m$-dimensional subspace $\Gamma^m$ of polynomials defined in $\R^n$, with $n>m$. This will be taken into account in section \ref{genericity}.  Therefore, we give the following definitions: 
\begin{defn}\label{arco}
	We indicate by $\arco$ the set of curves $\gamma(t)$ with values in $\R^m$ such that $\gamma(t)$ is real-analytic around the origin, and $\gamma(0)=\dot\gamma(0)=0$. 
	\begin{rmk}\label{1-coord}
		By the inverse function theorem, for any element $\gamma\in \Theta_m$ there exists some neighborhood $U_\gamma$ and some  $k\in\{1,...,m\}$ such that for all $t\in U_\gamma$ one has the parametrization
		\begin{equation}\label{parametrizzazione a una coordinata}
			\gamma(t):=
			\begin{cases}
			x_k(t)=t\\
			x_j(t)=f_j(t) \quad \forall  j\in\{1,...,m\}\ ,\ \ j\neq k
			\ .
			\end{cases} 
		\end{equation}

		From now on, when considering an element of $\Theta_m$, we will always assume implicitly that it is parametrized as in \eqref{parametrizzazione a una coordinata}. 
	\end{rmk}
	For fixed $i\in\{1,\dots, m\}$, we denote by $\arcoi$ the subset of  curves in $\arco$ that can be parametrized by the $i$-th coordinate. Clearly, one has the decomposition
	$$
	\arco=\bigcup_{i=1}^m\arcoi\ .
	$$
\end{defn}
\begin{rmk}
	We are asking the arc $\gamma$ to be analytic at the origin, but the minimal arc obtained in Theorem \ref{arco_minimale} did not necessarily have this property (the origin was not included, in general, in the uniform interval of analyticity $\tI_\lambda\subset [-\lambda,\lambda]$). As it has already been discussed in the introduction, this is an issue that comes from the use of analytic reparametrizations of semi-algebraic sets. We will deal with this apparent difficulty in section \ref{genericity}.  
\end{rmk}
\begin{rmk}
	For the moment, we do not make any assumption on the sizes of the neighborhoods of analyticity $U_\gamma$ of the arcs in $\arco$. Hence, the results of this section do not require any uniform lower bound on $|U_\gamma|$ as in Theorem \ref{arco_minimale}. Nevertheless, the existence of a uniform lower bound will prove to be necessary in order to to demonstrate the results of section \ref{genericity}.   
\end{rmk}
\begin{defn}
	For any pair of integers $s\ge 1$ and $i\in\{1,\dots,m\}$, we indicate by
	$
	\centina
	$ (resp. $\centinai$)
	the subset of $(\cP^\star(s,1))^m=\cP^\star(s,1)\times \dots \times \cP^\star(s,1)$ containing the truncations at order $s$ of the Taylor expansions at the origin of all curves in $\arco$ (resp. in $\arcoi$). The elements of $\centina$ will henceforth be referred to as $s$-truncations. 
\end{defn}

Clearly, one  has the following decomposition:
\begin{equation}\label{ovvia}
\centina = \bigcup_{i=1}^m\centinai\ .
\end{equation}

\begin{rmk}\label{dimensione}
	Clearly, $\centina$ (resp. $\centinai$) is isomorphic to the set of $s$-jets of curves in $\arco$ (resp. $\arcoi$). Moreover, for each $i\in\{1,\dots,m\}$, the set $\centinai $ is isomorphic to $\R^{(m-1)\times s}$, since for any curve $\gamma\in\arcoi$ its $s$-truncation $\cJ_{s,\gamma} \in\centinai$ is determined by the first $s$ Taylor coefficients at the origin of the functions $f_j$, with $j\in\{1,...,m\}$, $j\neq i$.
\end{rmk}

\begin{defn}\label{s-vanishing}
	Fix three integers $r\ge 2$, $m\ge 1$, and $1\le s\le r-1$. A polynomial $P\in\Polm$ is said to be {\itshape $s$-vanishing} if there exists an arc $\gamma\in\arco$ such that on its $s$-truncation $\jet\in\centina$
	the gradient of $P$ has a zero of order $s$ at the origin, namely
	\begin{equation}\label{singolare}
	\frac{d^\alpha}{dt^\alpha}\left(\left.\frac{\partial P}{\partial x_\ell}\right|_{\jet(t)}\right)_{t=0}=0\ ,  \forall\ \ell\in\{1,...,m\}\ ,\ \forall \ \alpha\in\{0,...,s\}\ .
	\end{equation}
	The set of $s$-vanishing polynomials in $\Polm$ is denoted by $\sigma(r,s,m)$.
\end{defn}

%\begin{rmk}
%	Differently from the arcs in $\arco$, the minimal arc in Lemma \ref{arco_minimale} was not automatically parametrized by the first coordinate on its uniform interval of analyticity. However, as we shall see in the next sections, $s$-vanishing polynomials in $\R^m$ are to be intended as restrictions to $m$-dimensional subspaces $\Gamma^m$ of polynomials defined in $\R^n$, with $n>m$. Since the choice of an orthonormal basis spanning $\Gamma^m$ is free, up to changing the order of the vectors in the basis, one can always reduce to the case where $x_1(t)=t$. 
%\end{rmk}
In paragraph \ref{algebraic}, we shall investigate the properties of the set $\sigma(r,s,m)\subset \Polm$ of $s$-vanishing polynomials: in particular, we shall prove that 
\begin{enumerate}
	\item for any given values of $r\ge 2$, $m\ge 1$, $1\le s\le r-1$, it is the semi-algebraic projection onto $\Polm$ of an algebraic set $Z(r,s,m)$ of $ \Polm\times \centina$ whose ideal can be explicitly computed ;
	\item 
	it has positive codimension.
\end{enumerate}
Secondly, in paragraph \ref{geometric}, we shall show that any polynomial $P$ belonging to the complementary of the closure of $\sigma(r,s,m)$ in $\Polm$ satisfies a "stable" lower estimate on its gradient. As we shall see, "stable" means that the estimate holds uniformly true for any polynomial belonging to a neighborhood of $P$.

Finally, in section \ref{genericity}, we shall prove that a polynomial $Q\in \Pol$ satisfying $\grad Q(0)\neq 0$ is steep around the origin iff there exists $1\le s\le r-1$ such that, for all $m\in\{1,...,n-1\}$, the restriction of $Q$ to any $m$-dimensional linear subspace $\Gamma^m$ perpendicular to $ \nabla Q(0)$ is contained in the complementary of $\text{closure}(\sigma(r,s,m))$ in $\Polm$.
\subsection{Algebraic properties}\label{algebraic}
We assume the notations of the previous paragraph, and we consider a triplet of integers $r\ge 2$, $m\ge 1$, $1\le s\le r-1$. We work in the euclidean space $\R^m$ equipped with  coordinates $(x_1,...,x_m)$, and we consider a polynomial $P=P(x)\in\Polm$ verifying the $s$-vanishing condition on the $s$-truncation $\jet\in \centina $ of some curve $\gamma\in\arco$.
Unless explicitly specified, we will henceforth work in the case in which $\gamma$ is parametrized by the first coordinate, as the generalization to other cases is immediate. Hence, from now on we set $\gamma(t):=(t,x_2(t),\dots,x_m(t))\in\arcox$.

\subsubsection{Case $m=1$}
We observe that, for $m=1$, we have the following simple result:
\begin{lemma}\label{m=1}
	For $m=1$, a polynomial $P(x)=\sum_{\substack{\mu\in \N\\1\le\mu\le r}}p_\mu x^\mu$ of one real variable belongs to the set $\sigma(r,s,1)\subset \cP(r,1)$ if and only if
	\begin{equation}\label{che palle}
			p_\mu=0\qquad \forall \mu\in \N \text{ such that } 1\le\mu\le s+1\ . 
	\end{equation}
	
	Moreover, $\sigma(r,s,1)$ is closed and its codimension in $\cP(r,1)$ is equal to $s+1$. 
\end{lemma}
\begin{proof}
	For $m=1$, the set $\Theta_1$ is the singleton containing the line $\gamma(t):=x(t)=t$. 
	
	By Definition \ref{s-vanishing}, it is clear that a polynomial verifying condition \eqref{che palle} in the statement satisfies also the $s$-vanishing condition. Conversely, again by Definition \ref{s-vanishing}, it is plain to check that the $s$-vanishing condition for $m=1$ imposes that the coefficients of the studied polynomial must be null up to order $s+1$. The closure of $\sigma(r,s,1)$ is due to continuity, whereas $\text{codim }\sigma(r,s,1)=s+1$ as such a set is determined by $s+1$ independent equations in $\cP(r,1)$.  
\end{proof}

\subsubsection{Notations (case $n\ge 3$, $2\le m\le n-1$)}\label{notations}

Up to the end of this paragraph, we will restrict to the case $n\ge 3$, $2\le m\le n-1$. The goal is to introduce useful notations in order to study the properties of $s$-vanishing polynomials in $\Polm$. 

Using standard notations, we set $\mu:=(\mu_1,...,\mu_m)\in\N^m\ ,\ \ |\mu|:=||\mu||_1$, and for any $P\in\Polm$ we write
\begin{equation}\label{polinomio alla romana}
P(x):=\sum_{\substack{\mu\in\N^m\\  1\le |\mu|\le r}} p_{(\mu_1,...,\mu_m)} x_1^{\mu_1}...\,x_m^{\mu_m}=:\sum_{\substack{\mu\in\N^m\\  1\le |\mu|\le r}} p_\mu x^\mu\ ,
\end{equation}
where we have taken into account the fact that $P$ has no constant term by the definition of $\Polm$.

 We also consider a curve $\gamma\in \arcox$ and - for $j=2,...,m$ - we develop its components $f_j$ at the origin, and we write
\begin{equation}\label{fj}
x_j(t)=f_j(t)=:\sum_{i=1}^{+\infty}a_{ji}t^i\ ,
\end{equation}
where we have taken into account the fact that $\gamma(0)=0$ by Definition \ref{arco}. Thus, the $s$-truncation $\jet\in\centinax$ of the curve $\gamma$ is identified by the $(m-1)s$ real coefficients $(a_{21},...,a_{2s},...,a_{m1},...,a_{ms})$ of the truncated expansion, namely
\begin{equation}\label{tronco}
\jet=\jet(t)=\left(t, \sum_{i=1}^{s}a_{2i}t^i\,,\dots,\sum_{i=1}^{s}a_{mi}t^i\right)\ . 
\end{equation}

In the rest of this paragraph, we will try to find an explicit expression for the $s$-vanishing condition in terms of the coefficients of $P$ and $\jet$.
We first observe that the $s$-vanishing condition \eqref{singolare} for $\alpha=0$ implies 
\begin{equation}\label{linear_zero}
p_\mu=0 \quad \text{  for all } \mu\in\N^m \text{  such that } |\mu|=1\ .
\end{equation}
Thus, without any loss of generality, in \eqref{polinomio alla romana} we can only consider the multi-indices $\mu\in\N^m$ that satisfy $2\le |\mu|\le r$. Moreover, for $\ell=1,...,m$, the $\ell$-th component of the gradient of $P$ reads
\begin{equation}\label{gradiente}
\frac{\partial P}{\partial x_\ell}:=\sommar \mu_\ell\ p_{\mu} x^{\widetilde{\mu}(\ell)}\ ,\ \ \widetilde{\mu}_j(\ell):=\mu_j-\delta_{j\ell}\ ,\ \ j=1,...,m\ ,\ \ |\mu|=|\widetilde{\mu}(\ell)|+1\ ,
\end{equation}
where $\delta_{j\ell}$ is the Kronecker symbol. At this point, we indicate by 
\begin{equation}\label{Phiuno}
\Phi^1:\Polm\times\centinax\longrightarrow\R^M\times \R^{(m-1)s}\ ,\ \ M:=\dim\Polm
\end{equation}
the trivial chart associating
$ (P,\jet)\longmapsto (p_\mu,a_{21},...,a_{2s},...,a_{m1},...,a_{ms})
$
and we define the functions
$
q_{\ell\alpha}^1:\R^M\times\R^{(m-1)s}\longrightarrow \R \ ,\ \  \ell\in \{1,...,m\}\ ,\ \  \alpha\in\{0,...,s\}
$
in the following way:

\begin{align}\label{qla}
\begin{split}
&	\text{  For }\alpha=0 \qquad q_{\ell 0}^1\circ\Phi^1(P,\jet):=\left(\gradl\right)_{t=0}=\ p_{(0,...,0,1,0,...,0)} \\
&	\text{  For }\alpha\in\{1,...,s\}\qquad  q_{\ell\alpha}^1\circ\Phi^1(P,\jet):=\frac{d^\alpha}{dt^\alpha}\left(\gradl\right)_{t=0}\\
&	=\frac{d^\alpha}{dt^\alpha}\left[\sommar \mu_\ell\  p_{\mu}t^{\widetilde{\mu}_1(\ell)}\left(\sum_{k=1}^{s}a_{2k}t^k\right)^{\widetilde{\mu}_2(\ell)}...\left(\sum_{j=1}^{s}a_{mj}t^j\right)^{\widetilde{\mu}_m(\ell)}\right]_{t=0}
\end{split}
\end{align}
where the "$1$" fills the $\ell$-th slot in the multi-index at the rightest member of the first line and where the last line is obtained by injecting \eqref{tronco} into expression \eqref{gradiente}. 

\begin{rmk}\label{i-esima}
	In a similar way, when $\gamma$ is parametrized by the $i$-th coordinate, with $i\neq 1$, one can denote by $
	\Phi^i:\Polm\times\centinai\longrightarrow\R^M\times \R^{(m-1)s}
	$
	the chart associating
	$ (P,\jet)\mapsto (p_\mu,a_{11},...,a_{1s},..., a_{(i-1)1},...,a_{(i-1)s},a_{(i+1)1},...,a_{(i+1)s},...,a_{m1},...,a_{ms})
	$, and introduce the maps $q_{\ell\alpha}^i\circ\Phi^i(P,\jet)$, $\ell\in\{1,\dots,m\}$, $\alpha\in\{0,\dots,s\}$ exchanging the rôle of the first coordinate with that of the $i$-th coordinate in \eqref{qla}. 
	
\end{rmk} 

\begin{rmk}\label{nota}
	Comparing expressions \eqref{qla} with Definition \ref{s-vanishing} and expression \eqref{linear_zero}, and taking Remark \ref{i-esima} into account, it is easy to see that any given polynomial $P\in\Polm$ satisfies the $s$-vanishing condition on some truncation $\jet\in \centina$ if and only if 
	\begin{equation}\label{sintetica}
	q_{\ell\alpha}^i\circ\Phi^i(P,\jet)=0\ ,\ \ \text{ for some } i \in\{2,\dots, m\} \ ,\ \ \text{for all } \ell\in\{1,...,m\}\ ,\ \ \alpha\in\{0,...,s\}\ .
	\end{equation}
\end{rmk}

By the above discussion, we see that the set of $s$-vanishing polynomials in $\Polm$ is given by
\begin{equation}\label{ping}
\sigma(r,s,m)=\bigcup_{i=1}^m\sigma^i(r,s,m)\ ,
\end{equation} 
where we have introduced the sets
\begin{equation}\label{pong}
\sigma^i(r,s,m):=\Pi_{\Polm}Z^i(r,s,m)
\end{equation}
and
\begin{align}\label{zrsm}
\begin{split}
Z^i(r,s,m)
:= \{&(P,\jet)\in\Polm\times \centinai\}|\\& q_{\ell \alpha}^i\circ\Phi^i(P,\jet)=0
\text{  for all } \ell\in\{1,...,m\}, \alpha\in\{0,...,s\}\}\ .
\end{split}
\end{align}
We also set
\begin{equation}\label{pang}
Z(r,s,m):=\bigcup_{i=1}^m Z^i(r,s,m)\ .
\end{equation}
It turns out that the ideal of $Z(r,s,m)$ can be explicitly computed for any given value of the integers $r\ge 2$, $1\le s\le r-1$, and $m\ge 2$ \footnote{The case $m=1$ is easier, see Lemma \ref{m=1}.}, i.e., one can find explicit expressions for the quantities $q_{\ell \alpha}^i\circ\Phi^i(P,\jet)$, for any value of $\ell\in\{1,...,m\}$,  $\alpha\in\{0,...,s\}$, and $i\in\{1,...,m\}$.

Before stating this result, for any given value of $i\in\{1,\dots,m\}$ we will introduce new global charts for $\Polm\times \centinai$ which - though unessential for the validity of our results - yield nicer expressions for the equations $q_{\ell\alpha}^i\circ\Phi^i(P,\jet)=0$ than the standard chart $\Phi^i$. 
As it will be shown in the next paragraph, the variables $a_{j1}, j\neq i,$ associated to the linear terms of the $s$-truncation $\jet$ can be incorporated in the coordinates of the polynomial $P$. This simplifies the calculations and yields more readable formulas. 

\subsubsection{A useful chart for $\Polm\times \centina$ (case $n\ge 3$, $2\le m\le n-1$)}\label{appropriata}

Here too, we restrict to the case $n\ge 3$, $2\le m\le n-1$.

Once again, we only consider the case in which $\gamma$ is parametrized by the coordinate $i=1$, the other cases being trivial generalizations. Some of the quantities introduced in the sequel should be labeled with an index $1$, as their definition depends in an obvious way from the choice of the parametrizing coordinate. However, in order not to burden notations, we drop, when possible, the reference to the fact that we are considering the case $i=1$.

In order to define a new chart for $\Polm\times \centinax$, we start by observing that, if we denote by $A_1,...,A_m$ the canonical basis associated to the coordinates $x_1,...,x_m$ in $\R^m$, for any fixed vector $\tb:=(b_{21},...,b_{m1})\in\R^{m-1}$, we can define the new parametric basis
\begin{equation}\label{chgt_base}
v_\tb:=A_1+b_{21}A_2+b_{31}A_3+...+b_{m1}A_m\ ,\ \ u_2:=A_2\ ,\ \ ...\ ,\ \ u_m:=A_m
\end{equation}
associated to the parametric change of variables 
\begin{equation}\label{chi}
\cL_\tb:\R^m\longrightarrow\R^m\ ,\ \  (x_1,...,x_m)\longmapsto \ty(\tb):=(\ty_1,\ty_2(\tb),...,\ty_{m}(\tb)), 
\end{equation}
where
\begin{equation}\label{coordinate}
\ty_1:=x_1\quad ,\qquad  \ty_2=\ty_2(\tb):=x_2-b_{21}\,x_1\quad \dots \quad \ty_m=\ty_m(\tb):=x_m-b_{m1}\,x_1\ .
\end{equation}
Obviously, for any fixed $\tb\in\R^{m-1}$, the change of coordinates \eqref{coordinate} in $\R^m$ induces a change of coordinates also in $\Polm$. Infact, the pull-back of the polynomial $P(x)$ is indicated by
\begin{equation}\label{pol_trasf}
\tP_\tb(\ty(\tb)):=P\circ\cL_\tb^{-1}(\ty(\tb))=\sum_{\substack{\mu\in\N^m\\  2\le |\mu|\le r}}p_{\mu} (x(\ty(\tb)))^\mu=:\sum_{\substack{\mu\in\N^m\\  2\le |\mu|\le r}} \tp_\mu(p_\mu,\tb) \ty^\mu(\tb)\ ,  
\end{equation}
where the new coefficients $\tp_\mu=\tp_\mu(p_\mu,b_{21},....,b_{m1})$ are polynomial functions of coefficients $p_\mu$ and on the parameters $\tb\in\R^{m-1}$. For any given $\tb\in\R^{m-1}$, there is a $1-1$ correspondence between the quantities $p_\mu$ and $\tp_\mu$, as they represent the coordinates of the same polynomial written in different bases. Moreover, by \eqref{fj}, in the new variables \eqref{coordinate} the components of the push-forward $\cL_\tb\circ \jet=\cL_\tb\circ \jet\in\centinax$ of any $s$-truncation $\jet\in\centinax$ read 
\begin{equation}\label{tfj}
\ty_1(t)=t\quad , \qquad \ty_j(t):=(a_{j1}-b_{j1})t+\sum_{i=2}^{s}a_{ji}t^i\quad,\qquad  j\in\{2,...,m\}\ .
\end{equation}

By looking at expression \eqref{tfj}, we see that, for any $s\ge 2$ and for any given $s$-truncation $\jet\in\centinax$,  it is possible to find a parametric change of coordinates $\cL_\tb$ in $\R^m$ such that the image $\cL_\tb\circ\jet\in \centinax$ has no linear terms except for the parametrizing component: it suffices to choose $\tb=\ta:=(a_{21},...,a_{m1})$, with $(a_{21},...,a_{m1})$ the coefficients of the linear terms of $\jet$. Taking \eqref{pol_trasf} into account, this also entails that, for any given truncation $\jet\in\centinax$,  there exists an associated set of coordinates $\tp_\mu=\tp_\mu(p_\mu,\ta)$, with $\ta:=(a_{21},...,a_{m1})$, in the space of polynomials $\Polm$. Hence, we can parametrize the coordinates of any polynomial $P$ through the linear coefficients of any truncation $\jet$. 

Namely, we firstly observe that $\centinauno$ is isomorphic to
\begin{equation}
\centinauno\simeq\{\ta:=(a_{21},\dots,a_{m1})\in \R^{m-1}\}\ .
\end{equation}
Also, indicating by $\centinadue$ the subset of $s$-truncations having null linear terms, we have that
\begin{equation}
\centinadue\simeq \{(a_{22},\dots,a_{2s},\dots,a_{m2},\dots,a_{ms})\in \R^{(m-1)(s-1)}\}\ .
\end{equation}
Clearly, one has 
\begin{equation}\label{ordinedue}
\centinax=\centinauno\times\centinadue\ .
\end{equation}
At this point, with the notation in \eqref{pol_trasf}, we define the invertible transformation
\begin{equation}\label{effeuno}
	\mathscr{F}^1: \Polm\times \R^{m-1}\longrightarrow\R^M\times\R^{m-1}\quad , \qquad M:=\dim \Polm 
\end{equation}
associating 
\begin{equation}\label{lunetta}
\left(P=\sum_{\substack{\mu\in \N^m\\1\le |\mu |\le r}}p_\mu x^\mu,\ta\right)\longmapsto \left(\tp_{\mu'}\left((p_\mu)_{\substack{\mu\in \N^m\\1\le |\mu |\le r}},\ta\right)_{\substack{\mu'\in \N^m\\1\le |\mu' |\le r}},\ta\right)
\end{equation}
and we indicate its image by
\begin{equation}\label{W}
W^1(r,m):=\mathscr{F}^1(\Polm\times \R^{m-1})\ .
\end{equation}
In other words, $W^1(r,m)$ is constructed by attaching to any point  $\ta\in \R^{m-1}\simeq\centinauno$ the fiber of all polynomials in $\Polm$ expressed in the variables \eqref{coordinate} associated to the value $\ta$. 

%Moreover, setting $M:=\dim  \Polm$, the map 
%$
%\Psi:\R^M\times\R^{(m-1)s}\longrightarrow W^1(r,m)\times \centinadue
%$
%associating 
%\begin{equation}\label{luna}
%\left((p_\mu)_{\substack{\mu\in \N^m\\1\le |\mu |\le r}},\ta,a_{22},a_{23},...,a_{ms}\right)\longmapsto ((\tp_\mu(p_\mu,\ta),\ta),a_{22},a_{23},...,a_{ms})
%\end{equation}
%is invertible. 

Furthermore, setting 
\begin{equation}\label{trasformata}
\jeta:=\cL_\ta\circ\jet 
\end{equation}
we have $\jeta\in \centinadue$ by construction because in the adapted variables - as we had shown in $\eqref{tfj}$ by setting $\tb=\ta$ - with the exception of the parametrizing component, any truncation starts at order two. Taking the notation in \eqref{pol_trasf} into account, we can also define the chart
\begin{equation}\label{Upsilon}
\Upsilon^1:\Polm\times \centinax \longrightarrow W^1(r,m)\times \R^{(m-1)(s-1)}
\end{equation}
associating 
\begin{equation}\label{sole}
(P,\jet)\longmapsto \left(\tp_{\mu'}\left((p_\mu)_{\substack{\mu\in \N^m\\1\le |\mu |\le r}},\ta\right)_{\substack{\mu'\in \N^m\\1\le |\mu' |\le r}},\ \ta,a_{22},a_{23},...,a_{ms}\right)\ .
\end{equation}
\begin{rmk}\label{assafaddij}
	For further convenience, we also denote by 
	\begin{equation}
	\fU^1: \Polm\times \centinauno \longrightarrow W^1(r,m)
	\end{equation} 
	the restriction of $\Upsilon^1$ to $\Polm\times \centinauno$. It is plain to check by formulas \eqref{pol_trasf}, \eqref{lunetta}, and \eqref{sole} that $\fU^1$ is polynomial, invertible and that its inverse is a polynomial map. \end{rmk}

\begin{rmk}\label{altri i}
	The generalization of the arguments above to the case in which the curve $\gamma \in\arco$ is parametrized by the $i$-th coordinate, with $i\in\{2,\dots,m\}$, is immediate. In particular, one can define functions $\mathscr{F}^i$, $\Upsilon^i$, $\fU^i$, together with sets $W^i(r,m)$.
\end{rmk}

\begin{rmk}
	With slight abuse of notation, in the rest of this work we will often write $(\tP_\ta, \ta, \jeta)$ and $(\tP_\ta,\ta)$ to indicate a point of $\Upsilon^1(P,\jet)$ and $W^1(r,m)$ respectively. 
\end{rmk}

For further convenience, we also observe that
\begin{lemma}\label{haha}
	Any polynomial $P\in\Polm$ satisfies the $s$-vanishing condition  
	\begin{equation}\label{boh}
	\frac{d^\alpha}{dt^\alpha}\left(\left.\frac{\partial P(x)}{\partial x_\ell}\right|_{\jet(t)}\right)_{t=0}=0 \quad \forall \alpha\in\{0,...,s\}\ ,\ \  \forall \ell\in\{1,...,m\}
	\end{equation}
	on the $s$-truncation $\jet\in \centinax$ of some curve $\gamma\in\arcox$, if and only if it satisfies
	$$
	\frac{d^\alpha}{dt^\alpha}\left(\left.\frac{\partial \tP_\ta(\ty)}{\partial \ty_\ell}\right|_{\jeta(t)}\right)_{t=0}=0 \quad \forall \alpha\in\{0,...,s\}\ ,\ \  \forall \ell\in\{1,...,m\}
	$$
	in the adapted coordinates associated to the linear terms of $\jet$.
\end{lemma}
\begin{proof}
	By \eqref{coordinate} one has
	\begin{equation}\label{conviene}
	\cL_\ta^{-1}(\ty):=
	\begin{cases}
	x_1=\ty_1\\
	x_2=a_{21}\,\ty_1+\ty_2\\
	...\\
	x_m=a_{m1}\,\ty_1+\ty_m\\
	\end{cases}\ ,\ \ 
	\cL_\ta^{-1}(\ty)=
	\left(
	\begin{matrix}
	1 & 0 & 0 & ... & 0\\
	a_{21} & 1 & 0 & ... & 0\\
	a_{31} & 0 & 1 & ... & 0\\
	... & ... & ... & ... & ...\\
	a_{m1} & 0 & 0 & ... & 1\\
	\end{matrix}
	\right)\ .
	\end{equation}
	Indicating with $(\cL^{-1}_\ta(\ty))_{k,\ell}$ the $(k,\ell)$-th entry of the Jacobian of the inverse transformation $\cL_\ta^{-1}$ in \eqref{conviene}, by the Leibniz formula one has
	\begin{align}\label{mina}
	\begin{split}
	&\frac{d^\alpha}{dt^\alpha}\left(\left.\frac{\partial \tP_\ta(\ty)}{\partial \ty_\ell}\right|_{\jeta(t)}\right)_{t=0}:=\frac{d^\alpha}{dt^\alpha}\left(\left.\frac{\partial (P\circ\cL^{-1}_\ta(\ty))}{\partial \ty_\ell}\right|_{\jeta(t)}\right)_{t=0}\\
	=&\frac{d^\alpha}{dt^\alpha}\left[\sum_{k=1}^m\left.\frac{\partial P}{\partial x_k}\circ\cL^{-1}_\ta(\ty)\right|_{\cL_\ta\circ\jet(t)}\times\left.  (\cL^{-1}_\ta(\ty))_{k,\ell}\right|_{\jeta(t)}\right]_{t=0}\\
	=&\frac{d^\alpha}{dt^\alpha}\left[\sum_{k=1}^m\left.\frac{\partial P(x)}{\partial x_k}\right|_{\jet(t)}\times\left.  (\cL^{-1}_\ta(\ty))_{k,\ell}\right|_{\jeta(t)}\right]_{t=0}\ .
	\end{split}
	\end{align}
	Since the entries of the matrix $\cL_\ta^{-1}$ in \eqref{conviene} are constant, one has 
	$$
	\left.  (\cL^{-1}_\ta(\ty))_{k,\ell}\right|_{\jeta(t)}=(\cL^{-1}_\ta(\ty))_{k,\ell}\ ,
	$$ 
	so that, finally, \eqref{mina} reads	
	\begin{align}\label{mansomma}
	\begin{split}
	\frac{d^\alpha}{dt^\alpha}\left(\left.\frac{\partial \tP_\ta(\ty)}{\partial \ty_\ell}\right|_{\jeta(t)}\right)_{t=0}=\sum_{k=1}^m\frac{d^\alpha}{dt^\alpha}\left[\left.\frac{\partial P(x)}{\partial x_k}\right|_{\jet(t)}\right]_{t=0}\times  (\cL^{-1}_\ta(\ty))_{k,\ell}\ .
	\end{split}
	\end{align}
	By the expression above, it is immediate to check that if $P$ satisfies the $s$-vanishing condition on $\jet$ then $\tP_\ta$ does the same on $\jeta$. 
	The proof of the converse is immediate by applying the same arguments to $P(x)\equiv \tP_\ta\circ \cL_\ta(x)$. 
\end{proof}
\medskip 

%Therefore, one has the chain of bijections
%$$
%\Polm\times \centina \overset{\Phi}{\longrightarrow} \R^N\times\R^{(m-1)s}\overset{\Psi}{\longrightarrow} \Xi^1\times \centinadue 
%$$
%which associate
%$$
%(P,\boldsymbol{\upgamma})\mapsto (p_\mu, \ta,a_{22},a_{23},...,a_{ms})\mapsto (\tp_\mu(p_\mu,\ta),\ta, a_{22},a_{23},...,a_{ms})\ .
%$$
By the discussion above, one has the choice to write the equations determining the set  $Z^1(r,s,m)$ in \eqref{zrsm} either in the original coordinates, where they assume the form $q^1_{\ell\alpha}\circ\Phi(P,\jet)=0$ for all $\ell\in\{1...,m\}$ and $\alpha\in\{0,...,s\}$, or in the new set of coordinates associated to the change of variables $\cL_\ta$ in $\R^m$, defined in \eqref{coordinate}. In particular, by performing the same computations that led to expression \eqref{qla} in the new variables, and by taking into account the fact that the expansion of $\jeta(t)$ starts at order two in $t$, one can introduce the functions
\begin{equation}\label{tQ}
\tQ_{\ell\alpha}: W^1(r,m)\times \R^{(m-1)(s-1)}\longrightarrow \R\quad , \qquad \ell\in\{1,\dots,m\} \quad \alpha \in \{0,\dots ,s\}
\end{equation}
in the following way:
\begin{align}\label{cuccurucu}
\begin{split}
&	\text{For }\alpha=0, \qquad \tQ_{\ell 0}\circ\Upsilon^1(P,\jet)=\tQ_{\ell 0}(\tP_\ta,\ta,\jeta):=\ \tp_{(0,\dots,0,1,0,\dots,0)} \\
&	\text{For }\alpha\in\{1,...,s\},\ \ \  \tQ_{\ell\alpha}\circ\Upsilon^1(P,\jet)=\tQ_{\ell\alpha}(\tP_\ta,\ta,\jeta):=\frac{d^\alpha}{dt^\alpha}\left(\left.\frac{\partial \tP_\ta(\ty)}{\partial \ty_\ell}\right|_{\jeta(t)}\right)_{t=0}\\
&	=\frac{d^\alpha}{dt^\alpha}\left[\sommar \mu_\ell\  \tp_{\mu}t^{\widetilde{\mu}_1(\ell)}\left(\sum_{k=2}^{s}a_{2k}t^k\right)^{\widetilde{\mu}_2(\ell)}\dots\left(\sum_{j=2}^{s}a_{mj}t^j\right)^{\widetilde{\mu}_m(\ell)}\right]_{t=0}\ .
\end{split}
\end{align}
Expressions \eqref{qla} and \eqref{cuccurucu}, considered together with Lemma \ref{haha}, imply that condition $q_{\ell\alpha}^1\circ\Phi^1(P,\jet)=0 $ holds if and only if
\begin{equation} \tQ_{\ell\alpha}\circ\Upsilon^1 (P,\jet)=\tQ_{\ell\alpha}(\tP_\ta,\ta,\jeta)=0\ ,\ \  \forall  \ell\in\{1,\dots,m\}, \forall \alpha\in\{0,\dots,s\}\ ,
\end{equation}
so that the ideal of the set $Z^1(r,s,m)$ in the new variables is given by the equations $\tQ_{\ell\alpha}\circ\Upsilon^1(P,\jet)=\tQ_{\ell\alpha}(\tP_\ta,\ta,\jeta)=0$ for all $\ell\in\{1...,m\}$ and $\alpha\in\{0,...,s\}$. 
In the sequel, we will work in these new coordinates, since the involved expressions are nicer.

\subsubsection{Computations and estimate on the codimension of $\sigma(r,s,m)$ (case $n\ge 3$, $2\le m\le n-1$)}\label{computations}

As in the previous paragraphs, we set $n\ge 3$, $2\le m\le n-1$.

Once again, we will only consider the case in which $\gamma$ is parametrized by the first coordinate.

Before stating the main results of this paragraph, we still need to introduce a few notations. For any fixed $\alpha\in\{0,\dots,s\}$, for any $\beta\in\{0,\dots,\alpha\}$, and for any $i\in\{1,\dots,m\}$, we set 
\begin{align}\label{mula}
\nu(i,\beta):=
\begin{cases}
(\beta+1,0,\dots,0)\ ,\ &\text{ for } i=1\\
(\beta,0,\dots,0,1,0,\dots,0)\ ,\ &\text{ for } i=2,\dots,m\\
\end{cases}
\end{align}
where the "$1$" fills the  $i$-th slot for $i=2,...,m$. For $\alpha\in\{1,...,s\}$, we also denote the multi-indices $\mu\in\N^m$ of length $2\le|\mu|\le\alpha+1$ not belonging to this family with
\begin{equation}\label{malpha}
\mathcal{M}_m(\alpha):=\{\mu\in\N^m\ ,\ \ 2\le|\mu|\le \alpha+1\}\backslash \bigcup_{\substack{i=1,...,m\\\beta=1,...,\alpha}}\{\nu(i,\beta)\}\ .
\end{equation}

Moreover, for any given $\alpha\in\{1,..,s\}$, $\mu\in\N^m$ and $\ell\in\{1,...,m\}$ we introduce
\begin{align}\label{gma}
\begin{split}
\mathcal{G}_m(\widetilde{\mu}(\ell),\alpha):=\biggl\{&(k_{j2},...,k_{j\alpha})\in\N^{(m-1)\times(\alpha-1)}, j\in\{2,...,m\}:\\
&\sum_{i=2}^{\alpha}k_{ji}=\mtj(\ell)\ ,\ \ \widetilde{\mu}_1(\ell)+\sum_{j=2}^{m}\sum_{i=2}^{\alpha}i\ k_{ji}=\alpha\biggl\} 
\end{split}
\end{align} 
and we set
\begin{equation}\label{basta}
\cE_m(\ell,\alpha):=\{\mu\in \N^m\ |\  \mathcal{G}_m(\widetilde{\mu}(\ell),\alpha)\neq \varnothing\}\ .
\end{equation}
Finally, for any $\ell\in\{1,\dots,m\}$ and for any $\mu \in \N^m$, we remind that (see \eqref{gradiente})
$$
\widetilde\mu_j(\ell):=\mu_j-\delta_{j\ell}\quad j\in\{1,\dots,m\}\ .
$$

With these notations, we can now state the following 
\begin{thm}\label{ideal}
	For any choice of integers $m\ge 1$, $r\ge2$, $1\le s\le r-1$, the set $Z(r,s,m)$ in \eqref{pang} is an algebraic set of $\Polm\times \centina$, whose ideal can be explicitly computed. In particular, with the notations in \eqref{cuccurucu}, \eqref{mula}, \eqref{malpha} and \eqref{gma}, the set $Z^1(r,s,m)$ in \eqref{zrsm} is the image through the inverse of the transformation $\Upsilon^1$ in \eqref{sole} of the algebraic set determined by the following equations:
		\begin{align}\label{vigna}
		\begin{split}
		\tQ_{10}(\tP_\ta,\ta,\jeta)&=\tp_{\nu(1,0)}=0\ ,\\ 
		\tQ_{11}(\tP_\ta,\ta,\jeta)&=2\,\tp_{\nu(1,1)}=0\ ,\\
		\frac{\tQ_{1\alpha}(\tP_\ta,\ta,\jeta)}{\alpha!}&=(\alpha+1)\tp_{\nu(1,\alpha)}+\sommalphaquater \beta\ \tp_{\nu(i,\beta)} a_{i(\alpha-(\beta-1))}\\
		+&\sommalphateruno \mu_1\ \tp_{\mu}
		\sum_{k\in \cG_m(\widetilde \mu(1),\alpha)} 	\left( \prod_{j=2}^m
		\binom{\mtj(1)}{k_{j2}\ ...\  k_{j\alpha}}\ 
		a_{j2}\strut^{k_{j2}}...a_{j\alpha}\strut^{k_{j\alpha}}\right)=0\ ,
		\\
		\end{split}
		\end{align}
		for $\ell=1$, $\alpha=2,...,s$, and	
		\begin{align}\label{pigna}
		\begin{split}
		\tQ_{\ell 0}(\tP_\ta,\ta,\jeta)&=\tp_{\nu(\ell,0)}=0\ ,\\ 
		\tQ_{\ell 1}(\tP_\ta,\ta,\jeta)&=\tp_{\nu(\ell,1)}=0\ ,\\
		\frac{\tQ_{\ell\alpha}(\tP_\ta,\ta,\jeta)}{\alpha!}&= \tp_{\nu(\ell,\alpha)}\\
	+	&  \sommalphater \mu_\ell\ \tp_{\mu}
		\sum_{k\in \cG_m(\widetilde \mu(\ell),\alpha)} 	\left( \prod_{j=2}^m
		\binom{\mtj(\ell)}{k_{j2}\ ...\  k_{j\alpha}}\ 
		a_{j2}\strut^{k_{j2}}...a_{j\alpha}\strut^{k_{j\alpha}}\right)=0\ ,
		\\
		\text{for }\ell=2,...,m\ ,\ \ & \alpha=2,...,s\ .
		\end{split}
		\end{align}
	
\end{thm}

\begin{rmk}\label{nolinear}
	It is plain to check that the coefficients of the vector $\ta\in\centinauno$ containing the linear terms of the truncation $\jet$ do not appear explicitly in expressions \eqref{vigna}-\eqref{pigna}. However, they are "hidden" in the terms $\tp_\mu=\tp_\mu(p_\mu,\ta)$ (see \eqref{pol_trasf}). 
\end{rmk}
As an almost immediate consequence of Theorem \ref{ideal}, we have that $s$-vanishing polynomials are rare in $\Polm$, namely
\begin{cor}\label{codimensione}
	$Z(r,s,m)$ has codimension $m(s+1)$ in $\Polm\times \centina$ and $\sigma(r,s,m)$ is a semi-algebraic set of codimension $s+m$ in $\Polm$. 
\end{cor}

%Moreover, by looking at the form of the equations in \eqref{vigna} and \eqref{pigna}, one can also state the following
%\begin{cor}
%	Consider two polynomials $P,S\in \Polm$, $0<\deg P<\deg S\le r$, having identical coefficients associated to the monomials up to degree $\deg P$ (i.e., $P$ is a truncation of $S$ at order $\deg P$). Then, if for some curve $\gamma\in\arco$ one has $(P,\jet)\not \in Z(r,\deg P-1,m)$, then one also has $(S,\jet)\not \in Z(r,\deg P-1,m)$. 
%\end{cor} 
%\begin{proof}
%	Since $(P,\jet)\not \in Z(r,\deg P-1,m)$, there must exist $\overline \ell\in\{1,\dots,m\}$ and $\overline \alpha\in\{0,\dots,\deg P-1\}$ such that $\tQ_{\overline \ell\overline \alpha}\circ\Upsilon(P,\jet)\neq 0$. By expressions \eqref{vigna} and \eqref{pigna} in Lemma \ref{ideal}, the functions $\tQ_{\ell\alpha}\circ\Upsilon$ at a given order $\alpha\in\{0,\dots,s\}$ involve only the coefficients associated to monomials up to degree $\alpha+1$. Hence, the quantity $\tQ_{\overline \ell\overline \alpha}\circ\Upsilon(S,\jet)$ involves only the monomials of $S$ up to order $\deg P$. But since $P$ is a truncation of $S$ at order $\deg P$, then $\tQ_{\overline \ell\overline \alpha}\circ\Upsilon(S,\jet)=\tQ_{\overline \ell\overline \alpha}\circ\Upsilon(P,\jet)\neq 0$, which concludes the proof. 
%\end{proof}

\begin{proof}{\it (Theorem \ref{ideal})}
	For fixed $s\in\{1,...,r-1\}$, we consider a polynomial $P\in\Polm$ verifying the $s$-vanishing condition on some truncation $\jet(t)\in\centinax$. 
	
	{\it Step 1.} 
	By Lemma \ref{haha}, in the adapted coordinates \eqref{coordinate} one must have 
	\begin{equation}\label{bellaterra}
	\frac{d^\alpha}{dt^\alpha}\left(\left.	\frac{\partial \tP_\ta(\ty)}{\partial \ty_\ell}\right|_{\jeta(t)}\right)_{t=0}=0\qquad  \forall \alpha\in\{0,...,s\} \ ,\ \  \forall \ell\in\{1,...,m\}\ .
	\end{equation} 
	For $\alpha=0$, it is plain to check that the terms of order zero in $t$ in \eqref{bellaterra} are the linear terms of $\tP_\ta$, for which $|\mu|=1$. Expressions \eqref{bellaterra} and \eqref{cuccurucu} yield the thesis for this value of $\alpha$.
	
	Then, as we did in \eqref{gradiente}, we drop the linear terms in $\tP_\ta$ and we write the quantity $\partial \tP_\ta(\ty)/\partial \ty_\ell$ explicitly
	\begin{equation}\label{gradiente2}
	\frac{\partial \tP_\ta(\ty)}{\partial \ty_\ell}:=\sommar \mu_\ell\ \tp_{\mu} \ty^{\widetilde{\mu}(\ell)}\ ,\ \ \widetilde{\mu}_j(\ell):=\mu_j-\delta_{j\ell}\ ,\ \ j=1,...,m\ ,\ \ |\mu|=|\widetilde{\mu}(\ell)|+1\ .
	\end{equation}
	Taking \eqref{tfj} into account (with $\tb:=(b_{21},\dots,b_{m1})=\ta:=(a_{21},\dots,a_{m1})$), we inject in \eqref{gradiente2} the components of the $s$-truncation $\jeta(t)$, namely
	\begin{equation}\label{tfj2}
	\ty_1(t)=t\quad , \qquad \ty_j(t)=\sum_{i=2}^{s}a_{ji}t^i\quad,\qquad  j\in\{2,...,m\}\ ,
	\end{equation} 
	and we obtain
	\begin{align}\label{biscotto}
	\begin{split}
	\left.\frac{\partial \tP_\ta(\ty)}{\partial \ty_\ell}\right|_{\jeta(t)}=\sommar \mu_\ell\  \tp_{\mu}t^{\widetilde{\mu}_1(\ell)}\left(\sum_{i=2}^{s}a_{2i}t^i\right)^{\widetilde{\mu}_2(\ell)}...\left(\sum_{u=2}^{s}a_{mu}t^u\right)^{\widetilde{\mu}_m(\ell)}
	\ .
	\end{split}
	\end{align}	
	
	{\it Step 2.}	
	For $\alpha=1$, we must look for the coefficient of the linear term (in $t$) in expression \eqref{biscotto}. Hence, for fixed $\ell=1,...,m$, since the sums in \eqref{biscotto} start at order two in $t$, only the multi-index for which $\widetilde\mu_j(\ell)=0$ for all $j\in\{2,...,m\}$ and $\widetilde\mu_1(\ell)=1$ must be retained in the sum in expression \eqref{biscotto}. The first condition implies $ \mu_j(\ell)=\delta_{j\ell}$ for all $j\in\{2,...,m\}$, whereas the second yields $\mu_1(\ell)=1+\delta_{1\ell}$. Therefore, by definition \eqref{mula}, for fixed $\ell\in\{1,...,m\}$ only the multi-index $\nu(\ell,1)$ appears in expression \eqref{biscotto} for $\alpha=1$. Again by \eqref{mula}, one has $\nu_1(1,1)=2$ and $\nu_\ell(\ell,1)=1$ for $\ell\in\{2,...,m\}$ so that the thesis in the case $\alpha=1$ follows. 
	
	{\it Step 3.} For any given $\alpha\in\{2,...,s\}$, we are interested in the coefficients of the terms of order $t^\alpha$ in \eqref{biscotto}. Hence, we can truncate the internal sums in \eqref{biscotto} at order $\alpha$. For the same reason, for any $\ell\in\{1,...,m\}$, we can neglect from the leftmost sum in \eqref{biscotto} the monomials $\mu$ satisfying $|\widetilde \mu(\ell)|> \alpha$ (hence $|\mu|> \alpha+1$), as their contribution is of order at least $t^{\alpha+1}$. Thus, for fixed $\alpha\in\{2,...,s\}$ and $\ell\in\{1,...,m\}$,  we have  
	\begin{align}\label{qla2}
	\begin{split}
	&\tQ_{\ell\alpha}(\tP_\ta,\ta,\jeta)\\
	&=\frac{d^\alpha}{dt^\alpha}\left(\sommalpha \mu_\ell \tp_{\mu}t^{\widetilde{\mu}_1(\ell)}\left(\sum_{i=2}^{\alpha}a_{2i}t^i\right)^{\widetilde{\mu}_2(\ell)}...\left(\sum_{u=2}^{\alpha}a_{mu}t^u\right)^{\widetilde{\mu}_m(\ell)}\right)_{t=0}
		\end{split}
	\end{align}
	due to formula \eqref{cuccurucu}.
	Now, for any $j=2,\dots,m$ and $\ell\in\{1,\dots,m\}$, the multinomial expansion yields:
	\begin{align}\label{multinomial}
	\begin{split}
	& \left(\sum_{i=2}^\alpha a_{ji}t^i\right)^{\mtj(\ell)}= \sum_{\substack{k_{j2},...,k_{j\alpha}\in \mathbb{N}\\k_{j2}+...+k_{j\alpha}=\widetilde \mu_j(\ell)}}
	\binom{\mtj(\ell)}{k_{j2}\ ...\  k_{j\alpha}}\ 
	a_{j2}\strut^{k_{j2}}\ ...\ a_{j\alpha}\strut^{k_{j\alpha}}\ \ t^{2k_{j2}+....+\alpha k_{j\alpha}}
	\end{split}\ ,
	\end{align}
	where we have used the notation $\displaystyle \binom{\mtj(\ell)}{k_{j2}\ ...\ k_{j\alpha}}:=\frac{\mtj(\ell)!}{k_{j2}!\ ...\ k_{j\alpha}!}$ . 
	
	Replacing each truncated Taylor development in \eqref{qla2} by its multinomial expansion \eqref{multinomial}, expression \eqref{qla2} reads
	{\small
		\begin{align}\label{qla3}
		\begin{split}
		\frac{d^\alpha}{dt^\alpha}&\left[\rule{0cm}{1.1cm}\right. \sommalpha \mu_\ell\ \tp_{\mu}\ t^{\widetilde{\mu}_1(\ell)}\prod_{j=2}^m\left( \sum_{\substack{k_{j2},...,k_{j\alpha}\in \mathbb{N}\\k_{j2}+...+k_{j\alpha}=\widetilde \mu_j(\ell)}}
		\binom{\mtj(\ell)}{k_{j2}\ ...\  k_{j\alpha}}\ 
		a_{j2}\strut^{k_{j2}}...a_{j\alpha}\strut^{k_{j\alpha}}\ \ t^{2k_{j2}+....+\alpha k_{j\alpha}}\right)\left.\rule{0cm}{1.1cm}\right]_{t=0}=\\
		\frac{d^\alpha}{dt^\alpha}&\left[\rule{0cm}{1.1cm}\right. \sommalpha \mu_\ell\ \tp_{\mu}\ t^{\widetilde{\mu}_1(\ell)} \sum_{\substack{k\in \mathbb{N}^{(m-1)\times(\alpha-1)}\\k=(k_{22},\dots,k_{2\alpha},\dots,k_{m2},\dots,k_{m\alpha})\\\forall i\in\{2,\dots,m\}\\k_{i2}+...+k_{i\alpha}=\widetilde \mu_i(\ell)}}\prod_{j=2}^m
		\binom{\mtj(\ell)}{k_{j2}\ ...  k_{j\alpha}}
		a_{j2}\strut^{k_{j2}}...a_{j\alpha}\strut^{k_{j\alpha}}\ \ t^{2k_{j2}+....+\alpha k_{j\alpha}}\left.\rule{0cm}{1.1cm}\right]_{t=0}.
		\end{split}
		\end{align}} 
	
	Moreover, taking \eqref{gma} into account, the class of multi-indices $\cG_m(\widetilde \mu(\ell),\alpha)$ selects those terms whose contribution inside the brackets of \eqref{qla3} is of order $t^\alpha$. Hence, by the above discussion, by \eqref{cuccurucu} and by \eqref{basta}, for any fixed $\alpha\in\{2,...,s\}$, and $\ell\in\{1,...,m\}$, we can write
	\begin{align}\label{qla4}
	\begin{split}
	\tQ_{\ell\alpha}(\tP_\ta,\ta,\jeta)=&\alpha! \sum_{\substack{\mu\in\cE_m(\ell,\alpha)\\\mu_\ell\neq 0 }} \mu_\ell\ \tp_{\mu}
	\sum_{k\in \cG_m(\widetilde \mu(\ell),\alpha)} 	\left( \prod_{j=2}^m
	\binom{\mtj(\ell)}{k_{j2}\ ...\  k_{j\alpha}}\ 
	a_{j2}\strut^{k_{j2}}...a_{j\alpha}\strut^{k_{j\alpha}}\right)\,.
	\end{split}
	\end{align}
	
	Now, we split the leftmost sum in \eqref{qla4} into the partial sums with respect to the families of indices defined in \eqref{mula} and \eqref{malpha}, namely for any fixed $\alpha\in\{2,...,s\}$ and $\ell\in\{1,...,m\}$, we write
	\begin{align}\label{qla4bis}
	\begin{split}
	&\frac{\tQ_{\ell\alpha}(\tP_\ta,\ta,\jeta)}{\alpha!}\\
	=& \sommalphabis \nu_\ell(i,\beta)\ \tp_{\nu(i,\beta)}
	\sum_{k\in \cG_m(\widetilde \mu(\ell),\alpha)} 	\left(
	\prod_{j=2}^m\binom{\widetilde \nu_j(i,\beta)(\ell)}{k_{j2}\ \dots\  k_{j\alpha}}\ 
	a_{j2}\strut^{k_{j2}}\dots a_{j\alpha}\strut^{k_{j\alpha}}\right)\\
	+&  \sommalphater \mu_\ell\ \tp_{\mu}
	\sum_{k\in \cG_m(\widetilde \mu(\ell),\alpha)} 	\left( \prod_{j=2}^m
	\binom{\mtj(\ell)}{k_{j2}\ \dots\  k_{j\alpha}}\ 
	a_{j2}\strut^{k_{j2}}\dots a_{j\alpha}\strut^{k_{j\alpha}}\right)\ .
	\end{split}
	\end{align}

	{\it Step 4.} We first study the case in which $\ell\neq 1$. For fixed $\ell\in\{2,...,m\}$, for any $i\in\{1,...,m\}$, $i\neq \ell$, and for any $\beta\in\{1,...,\alpha\}$, the monomials corresponding to the indices $\nu(i,\beta)$ do not contribute to the leftmost sum at the r.h.s. of \eqref{qla4bis}. Infact, by \eqref{mula}, the $\ell$-th element $\nu_\ell(i,\beta)$ of multi-index $\nu(i,\beta)$ is equal to zero for $\ell\in\{2,...,m\}$ and $i\in\{1,...,m\}$, $i\neq \ell$. 
	
	Moreover, still for fixed $\ell\in\{2,...,m\}$, the indices $\nu(\ell,\beta)$, with $\beta\in\{1,...,\alpha\}$, satisfy $\widetilde{\nu}_1(\ell,\beta)(\ell)=\beta$ and $\widetilde{\nu}_j(\ell,\beta)(\ell)=0$ for all $j\in\{2,...,m\}$, so that by \eqref{gma} we have
	$$
	\mathcal{G}_m(\widetilde\nu(\ell,\beta)(\ell),\alpha)=
	\begin{cases}
	\varnothing\ ,\ \ &\text{if } \beta=1,...,\alpha-1\\
	\{0\}\ ,\ \ &\text{if } \beta=\alpha\\
	\end{cases}\ .
	$$
	Consequently, the only monomial that contributes to the leftmost sum at the r.h.s. of \eqref{qla4bis} is the one associated to the multi-index $\nu(\ell,\alpha)$, and one has $k_{j2}=0,\dots, k_{j\alpha}=0$ when $\mu=\nu(\ell,\alpha)$. 
	Moreover, by hypothesis we have $\nu_\ell(\ell,\alpha)=1$ for any $\ell\in\{2,...,m\}$.
	Due to these arguments, for any fixed  $\ell\in\{2,...,m\}$, we can rewrite \eqref{qla4bis} in the form
	\begin{align}\label{qla4bis2m}
	\begin{split}
	&\frac{\tQ_{\ell\alpha}(\tP_\ta,\ta,\jeta)}{\alpha!}\\
	& = \tp_{\nu(\ell,\alpha)}
	+ \sommalphater \mu_\ell\ \tp_{\mu}
	\sum_{k\in \cG_m(\widetilde \mu(\ell),\alpha)} 	\left( \prod_{j=2}^m
	\binom{\mtj(\ell)}{k_{j2}\ ...\  k_{j\alpha}}\ 
	a_{j2}\strut^{k_{j2}}...a_{j\alpha}\strut^{k_{j\alpha}}\right)\ .
	\end{split}
	\end{align}
	
	This proves the Lemma for $\ell=2,...,m$, $\alpha=2,...,s$.
	
	{\it Step 5.} We now consider the case $\ell=1$. For all $j\in\{2,...,m\}$, the sub-family of indices $\nu(1,\beta)$, with $\beta\in\{1,...,\alpha\}$, satisfies $\widetilde{\nu}_1(1,\beta)(1)=\beta$ and $\widetilde \nu_j(1,\beta)(1)=0$. Hence, thanks to \eqref{gma}, we find
	\begin{equation}\label{topolino}
	\mathcal{G}_m(\widetilde \nu(1,\beta)(1),\alpha)=
	\begin{cases}
	\varnothing\ ,\ \ &\text{if } \beta=1,...,\alpha-1\\
	\{0\}\ ,\ \ &\text{if } \beta=\alpha\\
	\end{cases}\ .
	\end{equation}
	Moreover, we have $\nu_1(1,\alpha)=\alpha+1$ by construction.  
	
	On the other hand, for $\ell=1$, $\beta\in\{1,...,\alpha\}$ and $i,j\in\{2,...,m\}$, the multi-indices $\nu(i,\beta)$ satisfy $\widetilde{\nu}_1(i,\beta)(1)=\beta-1$ and $\widetilde \nu_j(i,\beta)(1)=\delta_{ji}$. Hence, by \eqref{gma} one can write
	\begin{equation}\label{stanlio}
	\widetilde \nu_j(i,\beta)(1)=\delta_{ji}=\sum_{u=2}^\alpha k_{ju}\ \Longleftrightarrow \ k_{ju}=\delta_{ji}\delta_{uv} \text{ for some $v\in\{2,...,\alpha\}$}
	\end{equation}
	and
	\begin{equation}\label{ollio}
	\widetilde{\nu}_1(i,\beta)(1)+\sum_{j=2}^{m}\sum_{u=2}^{\alpha}u\ k_{ju}=\beta-1+\sum_{j=2}^{m}\sum_{u=2}^{\alpha}u\ k_{ju}=\alpha\ .
	\end{equation}
	By \eqref{stanlio}, we see that condition \eqref{ollio} can be satisfied by some vector of multi-integers $(k_{22},\dots,k_{2\alpha},\dots,k_{m2},\dots,k_{m\alpha})$ if $\beta\in\{1,...,\alpha-1\}$, but cannot be fulfilled for $\beta=\alpha$. 
	Injecting \eqref{stanlio} into \eqref{ollio} one has
	\begin{equation} \sum_{j=2}^{m}\sum_{u=2}^{\alpha}\,u\,\delta_{ji}\delta_{uv} =\alpha-(\beta-1) \quad \text{  for all }\quad  \beta\in\{1,...,\alpha-1\}\ ,\ \ i\in\{2,...,m\}
	\end{equation}
	which implies
	\begin{equation}
	k_{ju}=\delta_{ji}\,\delta_{uv}\,\delta_{v,\alpha-(\beta-1)}\quad \text{  for all }\quad  \beta\in \{1,...,\alpha-1\}\ ,\ \ i\in\{2,...,m\}\ .
	\end{equation}
	Hence, for all $\beta\in\{1,...,\alpha-1\}$, and $i\in\{2,...,m\}$, we can finally write
	\begin{equation}\label{tom}
	\mathcal{G}_m(\widetilde\nu(i,\beta)(1),\alpha)=\{(k_{j1},...,k_{j\alpha}),\ j\in\{2,...,m\},\ k_{ju}=\delta_{ji}\delta_{u,\alpha-(\beta-1)}\}
	\end{equation} 
	and 
	\begin{equation}\label{jerry}
	\mathcal{G}_m(\widetilde\nu(i,\alpha)(1),\alpha)=\varnothing\ .
	\end{equation}
	Moreover, for $i=2,...,m$, by \eqref{mula} we have $\nu_1(i,\beta)=\beta$.
	
	By taking \eqref{topolino}, \eqref{tom}, \eqref{jerry} into account, expression \eqref{qla4bis} with $\ell=1$ yields
	\begin{align}\label{qla41}
	\begin{split}
	&\frac{\tQ_{1\alpha}(\tP_\ta,\ta,\jeta)}{\alpha!} =(\alpha+1)\tp_{\nu(1,\alpha)}+\sommalphaquater \beta\ \tp_{\nu(i,\beta)} a_{i(\alpha-(\beta-1))}\\
	+&\sommalphateruno \mu_1\ \tp_{\mu}
	\sum_{k\in \cG_m(\widetilde \mu(1),\alpha)} 	\left( \prod_{j=2}^m
	\binom{\mtj(1)}{k_{j2}\ ...\  k_{j\alpha}}\ 
	a_{j2}\strut^{k_{j2}}...a_{j\alpha}\strut^{k_{j\alpha}}\right)\ .
	\\
	\end{split}
	\end{align}
	This concludes the proof for the case in which $\jet\in \centinax$. 
	
	The proof of the case in which $\jet\in\centinai$, with $i=2,\dots,m$, is the same: one just has to take into account that the rôle of the special index is played by $i$ instead of $1$. Hence, the ideals of the sets $Z^i(r,s,m)$ can be explicitly computed and, by expression \eqref{pang}, the proof is concluded.

\end{proof}

We are now able to prove that $s$-vanishing polynomials are rare in $\Polm$. 
\begin{proof}{\it (Corollary \ref{codimensione})} 
	We want to show that, for a given pair $(P,\jet)\in\Polm\times\centinax$, the $ms+m$ equations in \eqref{vigna} and \eqref{pigna} are all linearly independent. 
	
	For all $\ell=1,...,m$ and $\alpha=0,...,s$, we collect in table \eqref{Jacobien} the derivatives of the functions $\tQ_{\ell\alpha}/\alpha! $ defined in \eqref{tQ}-\eqref{cuccurucu} - and whose action is made explicit in \eqref{vigna} - \eqref{pigna} - with respect to the coefficients of $\tP_\ta$, $\ta$, and to the Taylor coefficients of $\jeta$. We have indicated 
	\begin{enumerate}
		\item with the symbol $\mathbb{D}$, the $s\times s$ diagonal matrix whose entries are the numbers $\alpha+1$, for $\alpha=1,...,s$;
		\item with the symbol $\mathbb{I}_s$, the $s\times s$ identity matrix;
		\item with the symbol $\mathbb{B}_i$, $i\in\{2,...,m\}$, an $s\times s$ matrix whose entry at position $\alpha,\beta$, with $\alpha\in\{1,...,s\}$ and  $\beta\in\{1,..., s\}$, reads
		\begin{equation}\label{B}
		(\mathbb{B}_i)_{\alpha,\beta}:=
		\begin{cases}
		0\ ,\ \ &\text{if }\alpha=1\ ,\\
		\beta a_{i(\alpha-(\beta-1))}\ ,\ \ &\text{if } 2\le \alpha\le s \ ,\ \ 1\le \beta\le \alpha-1\ ,\\
		0 \ ,\ \ &\text{if }2\le \alpha\le s\ ,\ \ \alpha\le\beta\le s\ .
		\end{cases}
		\end{equation}

	\end{enumerate}
	
	\begin{eqfloat}
		{\footnotesize
			\begin{equation}\label{Jacobien}
			{
				\begin{matrix}
				\ell &  \alpha & \partial_{a_{ji}} & \partial_{\tp_\mu} & \partial_{\tp_{\nu(i,0)}} & \partial_{\tp_{\nu(1,\beta)}}  &  \partial_{\tp_{\nu(2, \beta)}} & \dots & \partial_{\tp_{\nu(m, \beta)}}\\
				& & & \mu\in \mathcal{M}(s) & i=1,\dots,m &  \beta=1,\dots,s & \beta=1,\dots,s& \dots  & \beta=1,\dots,s\\
				\\
				1,\dots,m & 0& 0& 0& \mathbb{I}_m& 0& 0& 0& 0\\
				\\
				1 & 1,\dots,s & \dots & \dots & 0 & \mathbb{D} &\mathbb{B}_2 & \dots & \mathbb{B}_m  \\
				\\
				2 & 1,\dots,s & \dots & \dots & 0  & 0& \mathbb{I}_{s} & 0 & 0  \\
				\\
				\dots & \dots & \dots & \dots & 0 & 0 & 0 & \dots & 0& \\
				\\
				m & 1,\dots,s & \dots & \dots & 0 & 0 & 0 & 0 & \mathbb{I}_s\\
				\end{matrix}
			}
			\end{equation}}
		
		\caption{Jacobian of $\tQ_{\ell\alpha}(\tP_\ta,\ta,\jeta)=0$ with $\ell\in\{1,...,m\}$ and $\alpha=\{0,..,s\}$. The first and the second column contain, respectively, all the possible values for the parameters $\ell$ and $\alpha$. The third column corresponds to the derivatives of $\tQ_{\ell\alpha}(\tP_\ta,\ta,\jeta)$  with respect to the variables of the vector $\ta\in \centinauno$, and to the Taylor coefficients of the $s$-jet $\jeta$. The remaining columns contain the derivatives with respect to the coefficients $\tp_\mu$ of $\tP_\ta$ associated with the families of multi-indices \eqref{mula} and \eqref{malpha}, in suitable order.}
	\end{eqfloat}
	
	It is plain to check that matrix \eqref{Jacobien} contains a submatrix of maximal rank $ms+m$ - corresponding to the derivatives w.r.t. those coefficients associated to the family of multi-indices \eqref{mula} - independently of $\ta$ and of the $s$-truncation $\jeta$ on which the $s$-vanishing condition is realized. Hence, since the transformation $\Upsilon$ in \eqref{Upsilon} is invertible, by Theorem \ref{ideal} the set $Z^1(r,s,m)$ is determined by $ms+m$ linearly independent algebraic equations and has codimension $ms+m$ in $\Polm\times\centinax$.
	
	As it was the case in the proof of Theorem \ref{ideal}, the same strategy of proof applies for $Z^j(r,s,m)$, $j=2,\dots,m$, one just has to switch the rôle of the indices $1$ and $j$. 
	
	Since $\sigma^1(r,s,m):=\Pi_{\Polm}Z^1(r,s,m)$ (see \eqref{pong}) and $Z^1(r,s,m)$ is algebraic, by the Theorem of Tarski and Seidenberg (see Th. \ref{Tarski_Seidenberg}) $\sigma^1(r,s,m)$ is a semi-algebraic set of $\Polm$. Moreover, as Jacobian \eqref{Jacobien} has rank $ms+m$ w.r.t. the $ms+m$ polynomial coefficients associated to the multi-indices of the family \eqref{mula}, for $\alpha\in\{2,...,s\}$ and $i\in\{1,...,m\}$, by Theorem \ref{ideal} and by the implicit function theorem, the conditions $\tQ_{\ell \alpha}(\tP_\ta,\ta,\jeta)=0$ imply
	\begin{align}\label{parametrizzazione}
	\begin{split}
	&\tp_{\nu(i,0)}=\tp_{\nu(i,1)}=0\\
	&\tp_{\nu(i,\alpha)}=g_{i\alpha}(\tp_\mu, \ta,a_{22},...,a_{2s},...,a_{m2},...,a_{ms})\quad , \qquad \mu\in\cM(\alpha)
	\end{split} 
	\end{align}
	for some implicit functions $g_{i\alpha}$. That is, one can express the polynomial coefficients $\tp_{\nu(i,0)},\tp_{\nu(i,1)},\tp_{\nu(i,\alpha)}$ as implicit functions of the remaining coefficients - associated to the multi-indices in the family $\cM(\alpha)$ defined in \eqref{mula} - and of the $(m-1)s$ parameters of $\ta$ and $\jeta$. Moreover, since the functions $\tQ_{\ell\alpha}(\tP_\ta,\ta,\jeta)$ are polynomial for all $\ell\in\{1,...,m\}$ and $\alpha\in\{0,..,s\}$, the implicit functions $g_{i\alpha}$ are all analytic. Therefore, one has an analytic parametrization of $\sigma^1(r,s,m)$ given by the $m(s+1)$ independent equations \eqref{parametrizzazione}, for $i\in\{1,...,m\}$, $\alpha\in\{0,...,s\}$.  This, in turn, yields that
	\begin{align}
	\begin{split}
	\text{dim}\,\sigma^1(r,s,m)=&\dim W^1(r,m)-m(s+1)=\dim \Polm+\dim\centinax-m(s+1)\\
	=& \dim \Polm+(m-1)s-m(s+1)\ ,
	\end{split}
	\end{align}
	which implies that the codimension of $\sigma^1(r,s,m)$ in $\Polm$ is
	\begin{align}
	\begin{split}
	\text{codim}\,\sigma^1(r,s,m)=&m(s+1)-(m-1)s=m+s\ .
	\end{split}
	\end{align}
	Once again, it is plain to check that the same result holds true also in the case in which the parametrizing coordinate of the curve gamma is the $j$-th, with $j=2,\dots,m$. Hence, one finds $\text{codim } \sigma^j(r,s,m)=m+s$ for $j=2,\dots,m$, which, together with expression \eqref{ping}, proves the statement.

\end{proof}

\subsection{Geometric properties}\label{geometric}

For fixed integers $r\ge 2$, $m\ge 2$, $1\le s\le r-1$, and for any $i\in\{1,\dots,m\}$, we indicate respectively by $\Sigma(r,s,m):=\overline{\sigma}(r,s,m)$ and $\Sigma^i(r,s,m):=\overline{\sigma}^i(r,s,m)$ the closures in $\Polm$ of the sets $\sigma(r,s,m)$  and ${\sigma}^i(r,s,m)$ introduced in the previous section. Taking \eqref{ping} into account,  one has
\begin{equation}\label{unione}
\Sigma(r,s,m)=\bigcup_{i=1}^m \Sigma^i(r,s,m)\ .
\end{equation}
For $m\ge 2$, Corollary \ref{codimensione} and Proposition \ref{dim_chiusura} ensure that for any $i\in\{1,\dots,m\}$
$$
\text{codim }\Sigma^i(r,s,m)=\text{  codim }\sigma^i(r,s,m)=s+m>0
$$ in $\Polm$, so that $\Polm\backslash\Sigma^i(r,s,m)$ is an open set of full Lebesgue measure. Therefore,  by \eqref{unione}, the same holds true also for $\Polm\backslash\Sigma(r,s,m)$. 

As we did previously,  when $m\ge 2$ we only consider the case in which the index of the parametrizing coordinate is $i=1$, as the other cases are immediate generalizations.

In case $m=1$, instead, Lemma \ref{m=1} ensures that
\begin{equation}\label{giàfatto}
\sigma(r,s,1)=\overline \sigma(r,s,1)=:\Sigma(r,s,1)\ .
\end{equation}
and
\begin{equation}
\text{codim } \sigma(r,s,1)=\text{codim } \Sigma(r,s,1)=s+1\ .
\end{equation}

Still for $m=1$, in order to make use of uniform notations w.r.t. the case $m=2$ and to simplify the exposition in the sequel, we write
\begin{equation}\label{uniforme}
\sigma^1(r,s,m)\equiv \Sigma^1(r,s,1):=\Sigma(r,s,1)
\end{equation} 
and we extend the notations of subsection \ref{appropriata} by setting
\begin{equation}\label{extend}
\cL_\ta:=\id\quad, \qquad \tP_\ta(\ty):=P(\ty)\quad ,\qquad \text{  for } m=1\ .
\end{equation}

The rest of this section will be devoted to the proof of the following
\begin{lemma}\label{stima_inf}
	Let $m$ be a positive integer. For any open set $\tD\subset\Polm\backslash \Sigma^1(r,s,m)$ verifying $\overline \tD\cap \Sigma^1(r,s,m)=\varnothing$, there exist positive constants $C_1(\tD)$, $C_2(s,m)$ such that for any polynomial $P(x)\in\tD$ and for any arc $\gamma\in\arcox$ 
	one has the following lower estimates 
	\begin{align}\label{stima}
	\begin{split}
	&\max_{\substack{\ell=1,...,m\\\alpha=1,...,s}}\left|\frac{d^\alpha}{dt^\alpha}\left(\left.\frac{\partial \tP_\ta(\ty)}{\partial \ty_\ell}\right|_{\cL_\ta\circ\gamma(t)}\right)_{t=0}\right|> C_1(\tD) \\  &\text{ in case } s=1 \text{ or } m=1,\\
	&\max_{\substack{\ell=1,...,m\\\alpha=1,...,s}}\left|\frac{d^\alpha}{dt^\alpha}\left(\left.\frac{\partial \tP_\ta(\ty)}{\partial \ty_\ell}\right|_{\cL_\ta\circ\gamma(t)}\right)_{t=0}\right|> \frac{C_1(\tD)}{1+C_2(s,m)\times  \max_{\substack{\ell=2,...,m\\ \alpha=2,...,s}}|a_{\ell\alpha}|}\\ &\text{ in case } 2\le s\le r-1 \text{ and } \ m\ge 2 \ ,
	\end{split}
	\end{align}
	where - for $m\ge 2$ - $\tP_\ta$ is the polynomial $P$ written in the adapted system of coordinates for $\gamma$ introduced in paragraph \ref{appropriata} and $\cL_\ta$ is the associated transformation defined in \eqref{conviene}, whereas for $m=1$ the symbols $\tP_\ta$ and $\cL_\ta$ are to be intended as in \eqref{extend}. 
\end{lemma}
%\begin{rmk}
%By taking expression \eqref{qla} into account, estimate \eqref{stima} can be rewritten in the form
%\begin{equation}
%\max_{ \substack{\ell\in\{1,...m\}\\ \alpha\in\{1,...,s\}}} |q_{\ell\alpha}\circ\,\Phi(P,\gamma)|>
%\begin{cases}
%C(\tD) \quad &\text{ if } s=1\\
%C(\tD) \  &\text{ if } m=1\\
% \displaystyle \frac{C(\tD)}{C'(s,m)\times  \max_{\substack{\ell=2,...,m\\ \alpha=2,...,s}}|a_{\ell\alpha}|}\quad &\text{ if } 2\le s\le r-1\ , \ m\ge 2\ .
%\end{cases}
%\end{equation}
%\end{rmk}
\begin{rmk}\label{lower_bound_linear_terms}
	As we shall see in the next section, for any $\lambda>0$, when $\gamma$ is the minimal arc of Theorem \ref{arco_minimale}, one can give a positive upper bound - that only depends on $r,s,m,\lambda$ - to the quantity $\max_{\substack{\ell=1,...,m\\ \alpha=2,...,s}}|a_{\ell\alpha}|$ at the denominator of \eqref{stima}. This is due to the fact that all minimal arcs satisfy a uniform Bernstein-like inequality on their Taylor coefficients (see formula \eqref{Bernie} in Theorem \ref{arco_minimale}). 
\end{rmk}
Before proving Lemma \ref{stima_inf}, we need  an intermediate result and a few additional notations.

%We take the definition of the family of multi-indices \eqref{mula} of paragraph \ref{algebraic} into account, and we set
%\begin{equation}
%	\cN(s):=\{\mu\in\N^m\ ,\ \ 2\le|\mu|\le r\}\backslash \bigcup_{\substack{i=1,...,m\\\beta=1,...,s}}\{\nu(i,\beta)\}\ .
%\end{equation}
%We observe that for $s\equiv r-1$, by \eqref{malpha}, $\cN(r-1)\equiv \cM(r-1)$. 

In case $m\ge 2$, for any given arc $\gamma\in \arcox$ with associated change of coordinates $\cL_\ta$ (see paragraph \ref{appropriata}, in particular formulas \eqref{coordinate}-\eqref{conviene}), we define the direct sum 
$$
\Polm= \cP_\ta^\sharp(r,m)\oplus\cP_\ta^\flat(r,m)
$$
in the following way: for
any polynomial $R(x)\in\Polm$, we consider its expression $\tR_\ta(\ty):=R\circ\cL_\ta^{-1}(\ty)\in\Polm$ in the adapted coordinates \eqref{coordinate} for $\gamma$; $\tR_\ta(\ty)$ can be decomposed uniquely into the partial sums with respect to the families of multi-indices defined in \eqref{mula} and \eqref{malpha}, namely:
\begin{equation}\label{split}
\tR_\ta(\ty)=\sum_{\substack{\mu\in\N^m\\1\le |\mu|\le r}} \tir_\mu \ty^\mu=\sum_{i=1}^m\sum_{\beta=0}^s \tir_{\nu(i,\beta)}\ty^{\nu(i,\beta)}+\sum_{\substack{\mu\in\N^m\\ \mu\in\mathcal{M}(s)}}\tir_\mu \ty^\mu=:\tR_\ta^\sharp(\ty)+\tR_\ta^\flat(\ty)\ ,
\end{equation}
and we set $R^\sharp(x):=\tR_\ta^\sharp\circ \cL_\ta(x)\in \cP_\ta^\sharp (r,m)$ and $R^\flat(x):=\tR_\ta^\flat\circ \cL_\ta(x)\in \cP_\ta^\flat (r,m)$. Clearly, the decomposition $R(x)=R^\sharp(x)+R^\flat(x)$ is unique, as the function associating $R\longmapsto \tR_\ta:=R\circ \cL_\ta^{-1}$, with $R\in \Polm$, is invertible.   

Finally, we set
\begin{equation}\label{Q}
\begin{split}
&\tQ\circ\Upsilon^1:\  \Polm\times \centinax\longrightarrow \R^{m(s+1)}\\ 
&(R,\jet)\longmapsto \tQ_{\ell\alpha}\circ\Upsilon^1(R,\jet)\equiv \tQ_{\ell\alpha}(\tR_\ta,\ta,\jeta)\ ,
\end{split}
\end{equation}
where $\ell=1,...,m$, $\alpha=0,...,s$, $\jet$ is the $s$-truncation of the curve $\gamma$, and the explicit form of $\tQ_{\ell\alpha}\circ\Upsilon^1$ is given in Theorem \ref{ideal}. We also indicate by $\mathscr{N}(\cdot)$ the zero sets of the transformations which will appear henceforth.
 
With this setting, one has the following intermediate result:
\begin{lemma}\label{invertibilita}
	In case $m\ge 2$, for any given $\gamma\in\arcox$ with associated $s$-truncation $\jet \in \centinax$, and for any given polynomial $R(x)\in\Polm\backslash\Sigma^1(r,s,m)$, there exists a unique polynomial $S(x)\in \sigma^1(r,s,m)$ such that 
	$$
	(S,\jet)\in\mathscr N(\tQ\circ\Upsilon^1)\quad ,\qquad S^\flat=R^\flat\ .
	$$
	In particular, $S$ satisfies the $s$-vanishing condition on the truncation $\jet$.

\end{lemma}
\begin{proof}
	Given $\gamma\in\arcox$ with its associated $s$-truncation $\jet\in\centinax$ and a polynomial $R(x)\in \Polm\backslash\Sigma^1(r,s,m)$, we denote by
	\begin{equation}\label{diesis} 
	\tA^\sharp_{R^\flat,\jet}: \Upsilon^1(\cP_\ta^\sharp\oplus\{ R^\flat\} \times \{\jet\}) \rightarrow\R^{m(s+1)}
	\end{equation}
	the restriction of $\tQ$ to the set $\Upsilon^1(\cP_\ta^\sharp\oplus\{ R^\flat\} \times \{\jet\})$ .
	
	As it was shown in the proof of Corollary \ref{codimensione} (see Table \eqref{Jacobien}), for all $\alpha\in\{0,...,s\}$ and $\ell\in\{1,...,m\}$, the Jacobian matrix of $\tA^\sharp_{R^\flat,\jet}$ reads 
	\begin{equation}\label{Jacobien2}
	\mathcal A:=	\left(
	\begin{matrix}
	\mathbb{I}_m & 0 & 0 & 0 & ... & 0 \\
	\\
	0& \mathbb{D} & \mathbb{B}_2 & \mathbb{B}_3 & ... & \mathbb{B}_m \\
	\\
	0 & 0&  \mathbb{I}_{s} & 0 & ... & 0  \\
	\\
	0 & 0& 0 &\mathbb{I}_s & ...&0  \\
	\\
	0&0&0 &...& 0& 0\\
	\\
	0 & 0 & 0 & 0 & 0& \mathbb{I}_s \\
	\end{matrix}
	\right)\ ,
	\end{equation}
	where the blocks $\mathbb{D}$ and $\mathbb{B}_i$, $i=2,...,m$, were defined in \eqref{B}. 
	Also, by Theorem \ref{ideal}, when $R^\flat$ and $\jet$ are fixed, the restriction of the function $\tQ$ to the set $\Upsilon^1(\cP_\ta^\sharp\oplus\{ R^\flat\} \times \{\jet\})$ is affine. Hence, $\tA^\sharp_{R^\flat,\jet}$ is represented by matrix \eqref{Jacobien2} acting on $\Upsilon^1(\cP_\ta^\sharp\oplus\{ R^\flat\} \times \{\jet\})$ plus a constant term depending only on $\tR_\ta^\flat\equiv R^\flat\circ \cL_\ta^{-1}$ and $\jeta\equiv \cL_\ta\circ\jet$; thus, it is globally invertible in $\Upsilon^1(\cP_\ta^\sharp\oplus\{ R^\flat\} \times \{\jet\})$. We set
	$$
	S^\sharp:=\left(\Upsilon^1\right)^{-1}\left(\mathscr N \left(\tQ^\sharp_{R,\jet}\right)\right)\in \cP_\ta^\sharp\ ,
	$$ 
	which is equivalent to saying that
	$$
	(S^\sharp+R^\flat,\jet)\in\mathscr N( \tQ\circ \Upsilon^1 )\ ,
	$$
	i.e., by \eqref{Q} and \eqref{cuccurucu}, $S^\sharp+R^\flat$ satisfies the $s$-vanishing condition on $\jet$. 
\end{proof}
We are now able to prove Lemma \ref{stima_inf}.
\begin{proof}{\itshape (Lemma \ref{stima_inf})}
	
	We consider a polynomial $P\in\tD\subset \Polm\backslash\Sigma^1(r,s,m)$, with $\overline \tD\cap \Sigma^1(r,s,m)=\varnothing$.
	
	{\bf Case $m=1$.} In case $m=1$, by Lemma \ref{m=1} there exists a constant $C_1(\tD)$ such that the truncation at order $s+1$ of $P$ - indicated by $P_{s+1}$ - satisfies 
	$$
	||P_{s+1}||_\infty>C_1(\tD)\ .
	$$
	The thesis follows easily by the expression above and by Definition \ref{s-vanishing}. 
	
	{\bf Case $m\ge 2$.} For any fixed arc $\gamma\in\arcox$, we shift to its associated adapted coordinates, and we consider the $s$-truncation $\jeta\in \centinax$, together with the pull-back $\tP_\ta:=P\circ\cL_\ta^{-1}$ of the polynomial $P$ w.r.t. the change of coordinates $\cL_\ta$ introduced in paragraph \ref{appropriata}. Due to the hypothesis, to Lemma \ref{haha} (see especially formulas \eqref{conviene} and \eqref{mansomma}), and to the fact that the linear terms of the curve $\gamma$ are uniformly bounded (see the Bernstein's estimate \eqref{Bernie}), there exists a constant $C_1(\tD)>0$ such that 
	\begin{equation}\label{distanza}
	||\tP_\ta-\Sigma^1(r,s,m)||_\infty:=\inf_{\widehat \tP_\ta\in\Sigma^1(r,s,m)}||\tP_\ta-\widehat \tP_\ta||_\infty>C_1(\tD)>0 \ .
	\end{equation}
	It suffices to prove the statement for the quantity
	$$
	\max_{\substack{\ell=1,...,m\\\alpha=1,...,s}}\left|\frac{d^\alpha}{dt^\alpha}\left(\left.\frac{\partial \tP_\ta}{\partial \ty_\ell}\right|_{\jeta(t)}\right)_{t=0}\right|
	$$
	instead of 
	\begin{equation}\label{Lione}
	\max_{\substack{\ell=1,...,m\\\alpha=1,...,s}}\left|\frac{d^\alpha}{dt^\alpha}\left(\left.\frac{\partial \tP_\ta}{\partial \ty_\ell}\right|_{\cL_\ta\circ \gamma(t)}\right)_{t=0}\right|
	\end{equation}
	because - as we had already pointed out in paragraph \ref{computations} - the terms of order higher than $s$ in the Taylor developement of $\cL_\ta\circ\gamma$ yield a null contribution to \eqref{Lione}.  With the decomposition in \eqref{split}, by Lemmata \ref{haha} and \ref{invertibilita} there exists a unique polynomial $\tS_\ta=\tS_\ta^\sharp+\tS_\ta^\flat$ fulfilling the $s$-vanishing condition on $\jeta$ and satisfying $\tS_\ta^\flat=\tP_\ta^\flat$.  Hence, \eqref{distanza} yields
	\begin{equation}\label{panforte}
	||\tP_\ta-\tS_\ta||_\infty=	||\tP_\ta^\sharp-\tS_\ta^\sharp||_\infty>C_1(\tD)>0 \ .
	\end{equation}
	By the proof of Lemma \ref{invertibilita}, we also know that - since $\tS_\ta^\flat=\tP_\ta^\flat$ and $\jeta$ are kept fixed - the function $\tA^\sharp_{P^\flat,\jet}$ in \eqref{diesis}   is affine and invertible in $ \Upsilon^1(\cP_\ta^\sharp\oplus\{ P^\flat\} \times \{\jet\})$. In particular, it is represented by matrix $\cA$ in \eqref{Jacobien2} plus a constant term depending only on $\tP_\ta^\flat,\jeta$. Taking into account the fact that $\tA^\sharp_{P^\flat,\jet}$ is the restriction of $\tQ$ to the set $ \Upsilon^1(\cP_\ta^\sharp\oplus\{ P^\flat\} \times \{\jet\})$,  one can write
	\begin{equation}\label{amatriciana}
	||\tP_\ta^\sharp-\tS_\ta^\sharp||_\infty\le ||\cA^{-1}||_\infty ||\tQ(\tP_\ta^\sharp+\tP_\ta^\flat,\jeta)-\tQ(\tS_\ta^\sharp+\tP_\ta^\flat,\jeta)||_\infty\ ,
	\end{equation}
	where $||\cA^{-1}||_\infty$ indicates the matrix norm of the inverse.
	Expressions \eqref{panforte} and \eqref{amatriciana} together yield
	\begin{equation}\label{pastiera}
	||\tQ(\tP_\ta^\sharp+\tP_\ta^\flat,\jeta)-\tQ(\tS_\ta^\sharp+\tP_\ta^\flat,\jeta)||_\infty>\frac{C_1(\tD)}{||\cA^{-1}||_\infty}\ .
	\end{equation}
	Moreover, by construction one has $	\tS^\sharp_\ta\in \mathscr N (\tA^\sharp_{P^\flat,\jet})$, that is $(\tS^\sharp_\ta+\tP_\ta^\flat,\jeta)\in \mathscr{N} (\tQ)$, so that  \eqref{pastiera} implies
	\begin{equation}\label{panettone_Motta}
	||\tQ(\tP_\ta^\sharp+\tP_\ta^\flat,\jeta)||_\infty>\frac{C_1(\tD)}{||\cA^{-1}||_\infty}\ .
	\end{equation}
	From the explicit form \eqref{Jacobien2} of the $(ms+m)\times (ms+m)$ matrix $\cA$, one can easily infer the form of $\cA^{-1}$, namely
	\begin{equation}\label{Jacobien3}
	{
		\cA^{-1}:=	\left(
		\begin{matrix}
		\mathbb{I}_m& 0&  0 & 0 & ... & 0 \\
		\\
		0 &\mathbb{D}^{-1} & -\mathbb{D}^{-1}\mathbb{B}_2 & -\mathbb{D}^{-1}\mathbb{B}_3 & ... & -\mathbb{D}^{-1}\mathbb{B}_m\\
		\\
		0& 0&  \mathbb{I}_{s} & 0 & ... & 0 \\
		\\
		0& 0 & 0 &\mathbb{I}_s & ...&0 \\
		\\
		0& 0&0&0 &...& 0\\
		\\
		0 &0 & 0 & 0 & 0& \mathbb{I}_s\\
		\end{matrix}
		\right)\ .
	}
	\end{equation}
	The induced matrix norm is, by construction,
	$
	||\cA^{-1}||_\infty:=\sup_{i\in\{1,...,n\}}\sum_{j=1}^{ms+m}|\cA^{-1}_{ij}|\ .
	$
	By the definition of $\mathbb{D}$ given above table \eqref{Jacobien}, and by the definition of the blocks $\mathbb{B}_i$, with  $i\in\{2,...,m\}$, in \eqref{B} one has $||\mathbb{D}^{-1}||_\infty=1$ and for $m\in\{2,...,m\}$ one can write $$
	\sup_{\ell\in\{2,...,m\}}||\mathbb{B}_\ell||\le
	\begin{cases}
	0 \quad &\text{if } s=1\\
	s\times (s-1)\max_{\substack{\ell=2,...,m\\ \alpha=2,...,s}}|a_{\ell \alpha}| &\text{if } 2\le s\le r-1\ .
	\end{cases}
	$$ 
	Hence, one finally has
	
		\begin{equation}\label{normQmenouno}
		||\cA^{-1}||_\infty\le 
		\begin{cases}
		1 \quad &\text{if } s=1\\
		1+(m-1)\, s\,(s-1)\times  \max_{\substack{\ell=2,...,m\\ \alpha=2,...,s}}|a_{\ell\alpha}|\quad  &\text{if } 2\le s\le r-1\ \ .
		\end{cases}
		\end{equation}
	
	Estimate \eqref{normQmenouno}, together with formulas \eqref{panettone_Motta}, \eqref{Q}, and \eqref{cuccurucu} implies the thesis, with $C_2(s,m)=(m-1)\,s\,(s-1)$
\end{proof}
\section{Proof of Theorem A}\label{genericity}
In order to prove Theorem A, we need to combine the results of the previous sections with several intermediate Lemmata.

\subsection{Codimension estimates}

For any pair of integers $n\ge 2$ and $1\le m\le n-1$, we indicate by $O(n,m)$ the space of $n\times m$ real matrices whose columns are orthonormal vectors of $\R^n$. Clearly, a matrix $A\in O(n,m)$ induces a map from $\R^m$ to $\R^n$ associating $\R^m\ni x\longmapsto I=Ax\in\R^n$. From a geometric point of view, for any integer $r\ge 2$, the restriction of any polynomial $Q(I)\in\Pol$ to any $m$-dimensional subspace $\Gamma^m\subset \R^n$ endowed with the Euclidean metric can be identified through $Q|_{\Gamma^m}(x):=Q(Ax)=:P(x)\in\Polm$, where the columns of $A\in O(n,m)$ span $\Gamma^m$.

We also indicate by $O(m)$ the $m\times m$ orthogonal group and by $\tG(m,n)$ the $m$-dimensional Grassmannian, which is locally isomorphic to $O(n,m)/O(m)$.

With this setting, for any integer $1\le s\le r-1$, we define
\begin{align}\label{U}
\begin{split}
\cU=\cU(r,s,m,n):=\{&(Q,A,P)\in \Pollo\times O(n,m)\times \Polm| \\
&  P(x)=Q(Ax)\ ,\ \  P(x)\in \Sigma(r,s,m)\}\ ,
\end{split}
\end{align}
and we indicate by
\begin{equation}\label{V}
\cV=\cV(r,s,m,n):=\Pi_{\Pollo}\,\cU(r,s,m,n)
\end{equation}
its projection onto the first component, i.e. the set of those polynomials $Q\in \Pollo$ for which the origin is non-critical, and such that, for some $m$-dimensional subspace $\Gamma^m$ orthogonal to $\nabla Q(0)$, the restriction $Q|_{\Gamma^m}\in\Polm$ belongs to the closure $\Sigma(r,s,m)$ of the set of $s$-vanishing polynomials introduced in section \ref{rs_vanishing_polynomials}.

\begin{rmk}\label{pleonastico}
	We observe that it is implicit in Definition \eqref{U} that $\Gamma^m$ must be orthogonal to $\nabla Q(0)$. Infact, any $P\in \sigma(r,s,m)$ must satisfy $\nabla P(0)=0$ (see equation \eqref{linear_zero}). Hence, the limit $\widehat P$ of any converging sequence $\{P_n\in \sigma(r,s,m)\}_{n\in\N}$ must have the same property. Since $\Sigma(r,s,m)=\overline\sigma(r,s,m)$, one has $\nabla \widehat P(0)=0$ for any $\widehat P\in\Sigma(r,s,m)$. As in our case we are considering $\widehat P(x)=\widehat Q(Ax)$ for some $\widehat Q\in\Pollo$, we have $\nabla \widehat P(0)=A^\dag\nabla \widehat Q(0)=0$, which is equivalent to saying that all the columns of $A$ are orthogonal to $\nabla \widehat Q(0)$. 
\end{rmk}

The first result that we prove in this section is the following
\begin{lemma}\label{closedness steep}
	$\cV(r,s,m,n)$ is a closed set in $\Pollo$ for the topology induced by $\Pol$. 
\end{lemma}
\begin{proof}
	Consider a sequence $\{Q_j\}_{j\in \N}$ in $\cV(r,s,m,n)$, converging to some polynomial $\overline Q\in \Pollo$ for the topology induced by $\Pol$ on $\Pollo$. Then, for any fixed $j\in \N$, by \eqref{U}-\eqref{V} there exists $A_j\in O(n,m)$ such that $P_j(x):=Q_j(A_jx)\in \Sigma(r,s,m)$. Since $O(n,m)$ is compact, there exists $\overline A\in O(n,m)$ and a subsequence $\{A_{j_k}\}_{k\in \N}\longrightarrow \overline A$. Hence, there exists a polynomial $\overline P\in \Polm$ such that the subsequence $\{P_{j_k}(x):=Q_{j_k}(A_{j_k}x)\}_{k\in \N}$ converges to $\overline P(x):=\overline Q(\overline Ax)$. Since $P_{j_k}(x)\in \Sigma(r,s,m)$ for all $k\in \N$ by construction, and $\Sigma(r,s,m)$ is closed, $\overline P(\overline Ax)\in \Sigma(r,s,m)$, whence the thesis. 
\end{proof}
Moreover, for given values of $m,n$, when $r$ and $1\le s\le r-1$ are sufficiently high, the set $\cV(r,s,m,n)$ becomes generic, namely
\begin{lemma}\label{positive_codimension}
	$\cV(r,s,m,n)$ is a semi-algebraic set of $\Pollo$ satisfying 
	\begin{equation}
	\text{  codim } \cV(r,s,m,n)\ge \max\{0,s-m(n-m-1)\}\ . 
	\end{equation}
\end{lemma}
\begin{proof}
	By hypothesis, $\Sigma(r,s,m):=\overline \sigma(r,s,m)$ and $\sigma(r,s,m)$ is a semi-algebraic set (see Corollary \ref{codimensione}). Hence, Proposition \ref{closure-interior-boundary} assures that $\Sigma(r,s,m)$ is also semi-algebraic. Therefore, set $\cU$ in \eqref{U} is clearly semi-algebraic, since it is determined by a finite number of semi-algebraic relations. Finally, the Theorem of Tarski and Seidenberg \ref{Tarski_Seidenberg} implies that $\cV$ is semi-algebraic since it is obtained by projecting $\cU$ onto its first component. 
	
	As for the codimension of $\cV$, we start by estimating the dimension of $\cU$. We remark that, for a fixed choice of $\overline A\in O(n,m)$ and $\overline P\in \Sigma(r,s,m)$, one has 
	\begin{align}\label{Fecamp}
	\begin{split}
	&\dim (\cU\cap \{(Q,A,P)\in\Pollo\times O(n,m)\times \Polm: A=\overline A\ ,\ \ P=\overline P\})\\
	&=\dim \Pollo-\dim \Polm\ .
	\end{split}
	\end{align}
	Infact, since the matrix $A:=(A_1|\dots|A_m)$, $A_1,\dots, A_m\in \R^n$, has been fixed, one can construct an orthonormal basis of $\R^n$ by completing $A_1,\dots,A_m$, with $n-m$ vectors $A_{m+1},\dots,A_n$. Since in the set above the restriction of any polynomial $Q\in \Pol$ to the space generated by $A_1,\dots,A_m$ is fixed, all the monomials of $Q$ corresponding to the coordinates associated to $A_1,\dots,A_m$ are uniquely determined. The number of these monomials is $\dim \Polm$, and the total number of monomials in $Q$ is $\dim \Pol$, whence equality \eqref{Fecamp}. In order to compute $\dim \cU$, one must add to the r.h.s. of \eqref{Fecamp} the dimension of the spaces corresponding to the fixed variables, namely
	\begin{align}\label{bella}
	\begin{split}
	\dim \cU &= \dim (\cU\cap \{(Q,A,P)\in\Pollo\times O(n,m)\times \Polm: A=\overline A\ ,\ \ P=\overline P\})\\
	&\quad +\dim O(n,m)+\dim \Sigma(r,s,m)\\
	&= \dim \Pollo-\dim \Polm+\dim O(n,m)+\dim \Sigma(r,s,m)\ .
	\end{split}
	\end{align}	
	We observe that, by Definition \ref{U}, if $(Q(I),A,Q(Ax))\in \cU$, for any orthogonal matrix $M\in O(m)$ also $(Q(I),AM,Q(AMx))\in \cU$, since one has the freedom to choose the orthonormal basis spanning the $m$-dimensional subspace $\Gamma^m\in \tG(m,n)$. More precisely, if we define the action of $O(m)$ on any element $(Q(I),A,Q(Ax))\in\cU$ as
	\begin{equation}
	(Q(I),A,Q(Ax))\longmapsto (Q(I),AM,Q(AMx))\ , 
	\end{equation}
	we can define an orbit of $O(m)$ starting at a given point $(Q(I),A,Q(Ax))\in\cU$ as
	\begin{equation}\label{orbit}
	\{(Q(I),AM,Q(AMx))\in\cU\ ,\ \ M\in O(m) \}\ .
	\end{equation}
	Since the first component in \eqref{orbit} is invariant, by \eqref{V} we see that the set $\cV$ can be found as the projection of the set of orbits $\cU/O(m)$ onto $\Pollo$, namely 
	\begin{equation}\label{proiezione_bella}
	\cV:=\Pi_{\Pollo}\,\cU=\Pi_{\Pollo}\,(\cU/O(m))\ ,
	\end{equation} 
so that one can write	\begin{equation}\label{proiezione_bellissima}
	\dim\cV=\dim\Pi_{\Pollo}\,(\cU/O(m))\ .
	\end{equation}
	Obviously, the action of $O(m)$ on $\cU$ is free and smooth, hence by \ref{Azione} it is also proper. Therefore, Theorem \ref{Quoziente} ensures that 
	\begin{equation}\label{pluto}
	\dim (\cU/O(m))=\dim \cU-\dim O(m)\ .
	\end{equation}
	By \eqref{proiezione_bellissima}, we have
	\begin{equation}\label{splendida}
	\text{  codim }\cV\ge \text{  codim } (\cU/O(m))-\dim O(n,m)-\dim\Polm 
	\end{equation} 
	and equations \eqref{bella} and \eqref{pluto} imply
	\begin{align}
	\begin{split}\label{meravigliosa}
	&\text{  codim }(\cU/O(m))\\
	&= \dim\Pollo+\dim O(n,m)+\dim \Polm -\dim \cU+\dim O(m)\\
	&\ge 2\dim \Polm -\dim\Sigma(r,s,m)+\dim O(m)
	\end{split}\ .
	\end{align}
	Expressions \eqref{splendida} and \eqref{meravigliosa} together yield
	\begin{align}\label{notte}
	\begin{split}
	\text{  codim } \cV\ge& \dim \Polm -\dim\Sigma(r,s,m)+\dim O(m)-\dim O(n,m)\\
	& = \text{  codim }\Sigma (r,s,m) +\dim O(m)-\dim O(n,m)\ .
	\end{split}
	\end{align}
	
	By Proposition \ref{dim_chiusura}, $\text{codim}\,\Sigma(r,s,m)=\text{codim}\,\sigma(r,s,m)$, so by Corollary \ref{codimensione} one has $\text{codim }\Sigma(r,s,m)=s+m$. Moreover, since $\dim O(m)=m(m-1)/2$ and $\dim O(n,m)=mn-m(m-1)/2-m$, \eqref{notte} reads
	\begin{equation}
	\text{  codim } \cV\ge 
	s-m(n-m-1)\ .
	\end{equation}
	Since the codimension is a nonnegative number, the thesis follows.

\end{proof}

\subsection{Stable lower estimates}
The set $\cV(r,s,m,n)$ introduced in the previous section is important because, for any given polynomial $Q\in \cP^\star(r',n)$ - with $r'\ge r$ - whose truncation at order $r$ lies outside of $\cV(r,s,m,n)$, all polynomials belonging to a small open neighborhood of $Q$ in $\cP^\star(r',n)$ are steep around the origin on the subspaces of dimension $m$, with uniform indices and coefficients. More precisely, one has 
\begin{thm}\label{steep_polinomi}
	Take five integers $r'\ge r\ge 2$, $1\le s\le r-1$, $n\ge 2$, $m\in\{1,...,n-1\}$. Consider a polynomial $Q\in\cP^\star(r',n)$, and suppose that for some $\tau>0$ its truncation at order $r$, indicated by $Q_r$, satisfies 
	\begin{equation}\label{fuori}
	\left|\left|Q_r- \cV\left(r,s,m,n\right)\right|\right|_\infty:=\inf_{R\in  \cV\left(r,s,m,n\right)}||Q_r-R||_\infty>\tau \ .
	\end{equation}
	There exist constants $\eps_0=\eps_0(r,s,m,\tau,n)$, $\tC_m=\tC_m(r',r,s,\tau,n)$, $\lambda_0=\lambda_0(r,s,m,\tau) $, and $\widehat \delta=\widehat \delta(r,s,m,\tau,n)$ such that any polynomial $S\in\Polpo$ contained in a ball of radius $\eps\in[0,\eps_0]$ around $Q$ in $\Polpo$
	is steep on the subspaces of dimension $m$ at any point $I\in B^n(0,\widehat \delta)$, with uniform steepness coefficients $\tC_m$, $\lambda_0$ and with steepness indices bounded by
	\begin{equation}\label{indice}
	\overline \alpha_m(s):=	\begin{cases}
	s\quad , \qquad &\text{ if $m=1$ }\\
	2s-1\quad, \qquad &\text{ if $m\ge 2$ }\ .\\
	\end{cases} 
	\end{equation}
\end{thm}

\begin{rmk}\label{diventa_generico}
	By Lemma \ref{positive_codimension}, since $1\le s\le r-1$, if $r>m(n-m-1)+1$ and $s\ge m(n-m-1)+1$ one has $\text{codim } \cV(r,s,m,n)\ge 1$, so that hypothesis \eqref{fuori} is generic in $\Pollo$. 
\end{rmk}
In order to prove Theorem \ref{steep_polinomi}, we need a Lemma used by Pyartli in the study Diophantine approximation. This result is crucial in KAM Theory, as it was shown by Herman, Rüssmann,  Sevryuk, and others (see \cite{Broer_Huitema_Sevryuk_1998} and the references therein). We give its statement in a version provided by Rüssmann \cite{Russmann_2001}. 

\begin{lemma}\label{Pyartli}
	Let $f : [a, b] \longrightarrow \R$, with $a < b$, be a $q$-times continuously differentiable function
	satisfying
	$$|f^{(q)}(t)| > \beta$$ 
	for all $t \in [a, b]$,
	for some $q \in \N$ and $\beta>0$.
	
	Then one has the estimate
	$$
	\meas\{t \in [a, b]\ |\ |f(t)| \le \rho\} \le 4
	\left( q!\frac{\rho}{2\beta}\right)^{1/q}
	$$
	for all  $\rho > 0$.

\end{lemma} 

We also need the following auxiliary 
\begin{lemma}\label{aiutati}
	With the hypotheses of Theorem \ref{steep_polinomi}, there exist positive constants $\eps^\star=\eps^\star(r,s,m,\tau,n),\chi=\chi(r,s,m,\tau,n)$, and $\zeta=\zeta(r,s,m,\tau)$, such that, for any $\eps\in[0,\eps^\star]$, the truncation $S_r\in \Pollo$ of any polynomial $S\in \Polpo$ contained in a ball of radius $\eps$ around $Q$ in $\Polpo$ verifies
	\begin{equation}\label{prima}
	||S_r- \cV(r,s,m,n)||_\infty>\chi\ ,  
	\end{equation}
	and for any $m$-dimensional subspace $\Gamma^m$ orthogonal to $\nabla S(0)$, one has
	\begin{equation}\label{seconda}
	||S_r|_{\Gamma^m}-\Sigma (r,s,m)||_\infty>\zeta\ .
	\end{equation}
\end{lemma}
\begin{proof}{\it (Lemma \ref{aiutati})}
	%If $r>m(n-m-1)$ and $s\ge m(n-m-1)$, by Lemma \ref{positive_codimension} and Proposition \ref{dim_chiusura} one has $\text{codim } \cV=\text{codim } \overline\cV\ge 1$ and the thesis follows immediately. 
	
	We split the proof into three steps. In the first one, we introduce appropriate sets and notations that are helpful in the proof. In second step, suitable continuous functions are defined by exploiting the existence of local continuous sections for the Grassmannian. In the third and last step, the statement is proved by exploiting the first two steps and the compactness of the Grassmannian. 
	
	We also observe that, due to Remark \ref{diventa_generico}, estimate \eqref{prima} is trivial for $r>m(n-m-1)+1$ and $s\ge m(n-m-1)+1$. 
	
	{\it Step 1.}
	For any given polynomial $V\in\Pollo$, we denote by $\tG_{ V}(m,n)\subset \tG(m,n)$ the compact subset of $m$-dimensional subspaces orthogonal to $\nabla V(0)\neq 0$. We also define the set
	\begin{equation}
	\Lambda^m:=\{(V,\Gamma^m)\,|\, V\in\Pollo,  \Gamma^m\in \tG_V(m,n)\}\ .
	\end{equation}
	%together with its projection
	%\begin{equation}
	%\pi: \Lambda^m\rightarrow \Pollo\ ,\ \ \{(V,\Gamma^m)\,|\, V\in\Pollo,   \Gamma^m\in \tG_V(m,n)\}\longmapsto V\ .
	%\end{equation}
	Now, setting $N:=\dim\Polpo$, for sufficiently small $\eps$ one has $\nabla S(0)\neq 0$ for any $S\in B^N(Q,\eps)$. Moreover, the map $\tf:\Polpo\longrightarrow \R^n$ associating $S\longmapsto \nabla S(0)$ is obviously continuous and surjective, and the same holds true for the function $\mathtt h: \R^n \longrightarrow\tG(n-1,n)$ which to a vector $\omega$ associates $\omega^\perp$. Hence, $\mathtt h\circ \tf$ is also continuous and surjective. %its $n-1$-dimensional orthogonal hyperplane $\Gamma_S^{n-1}\in \tG_S(n-1,n)$ is also continuous. 
	Therefore there exists an open set of $n-1$ dimensional hyperplanes - indicated by $\cW^{n-1}(Q,\eps)\subset \tG(n-1,n)$ - whose inverse image $\tf^{-1}(\mathtt h^{-1}(\cW^{n-1}(Q,\eps)))$ contains $B^N(Q,\eps)$. Hence, for $m\in\{1,\dots,n-1\}$, we can define the open set 
	$$
	\cW^m(Q,\eps):=\{\Gamma^m\in \tG(m,n)\,|\, \Gamma^m\subseteq  \Gamma^{n-1} \text{ for some }\Gamma^{n-1}\in \cW^{n-1}(Q,\eps)\}\ .
	$$ 	
	The above construction implies that, for any $m\in\{1,...,n-1\}$ and for sufficiently small $\eps$, the choice of an open ball $B^N(Q,\eps)$ determines a set 
	\begin{equation}\label{omegaemme}
	\Xi^m(Q,\eps):=\{(S,\Gamma^m)\,|\, S\in B^N(Q,\eps),\Gamma^m\in \cW^m(Q,\eps),\Gamma^m\in \tG_{S}(m,n)\}\subset \Lambda^m\ .
	\end{equation}
	
	\begin{rmk}
		To carry out the construction at this step, one only needs to perturb the linear terms of $Q$. The bound on $\eps$ that must be considered at this step, therefore, does not depend on the degree of the polynomial $Q$.
	\end{rmk}

	{\it Step 2.}	Now, we take into account the fact that it is always possible to define a local continuous section for the Grassmannian $\tG(m,n)$. Namely, for any element $\Gamma\in\tG(m,n)$ there exists an open neighborhood $\cE_\Gamma\subset \tG(m,n)$ of $\Gamma$ and a continuous map $\xi:\cE_\Gamma\longrightarrow O(n,m) $ such that, if $\pi: O(n,m)\longrightarrow \tG(m,n)$ is the canonical projection, then $\pi\circ\xi$ is the identity. Hence, for any $m$-dimensional subspace $\Gamma^m\in \tG_{Q}(m,n)$, we consider its associated open neighborhood $\cE_{\Gamma^m}$, and a compact neighborhood $\overline{ \mathsf{V}}_{\Gamma^m}\subset \cE_{\Gamma^m}$ centered at $\Gamma^m$. Since $\tG_{Q}(m,n)$ is compact, it can be covered by a finite number $L>0$ of compact neighborhoods $\overline{ \mathsf{V}}_{i}$ and open neighborhoods $\cE_i$, with $i=1,...,L$, of this kind. Hence, if $\eps$ is sufficiently small, then one has 
	\begin{equation}\label{unioni_varie}
	\cW^m({Q},\eps)\subset \bigcup_{i=1}^L \overline{ \mathsf{V}}_{i} \subset  \bigcup_{i=1}^L \cE_i\ .
	\end{equation} Moreover, if we indicate by $\xi_i$, $i=1,...,L$ the continuous section associated to each neighborhood $\cE_i$, it makes sense to define the sets
	\begin{equation}\label{omegaemmei}
	\Xi^m_i({Q},\eps):=\{(S,\Gamma^m)\,|\,(S,\Gamma^m)\in\Xi^m({Q},\eps), \Gamma^m\in\cE_i\}\ ,\ \    \Xi^m({Q},\eps)=\bigcup_{i=1}^L \Xi^m_i({Q},\eps)
	\end{equation}
	and the continuous functions
	\begin{align}\label{fi}
	\begin{split}
	F_i: &\ \Xi^m_i(Q,\eps)\longrightarrow \cP(r',m)\\
	&(S(I),\Gamma^m)\longmapsto  T(x):=S(Ax)\ ,\ \  A:=\xi_i(\Gamma^m)\in O(n,m)\ .
	\end{split}
	\end{align}
	
	{\it Step 3.} Fix $i=1,...,L$. By hypothesis $||Q_r-\cV||_\infty>\tau$ and $\nabla Q(0)=\nabla Q_r(0)\neq 0$, so that by the definition of $\cU$ and $\cV$ in \eqref{U}-\eqref{V}, by Remark \ref{pleonastico}, and by the compactness of $\tG_Q(m,n)=\tG_{Q_r}(m,n)$, there exists $\zeta_i=\zeta_i(r,s,m,\tau)>0$ such that - on any subspace $\Gamma^m\in \tG_{Q}(m,n)$, $\Gamma^m\in \cE_i$ - one has \begin{equation}\label{cipresso}
	||P_r-\Sigma(r,s,m)||_\infty>2\,\zeta_i\ ,
	\end{equation} 
	where $P_r(x):=Q_r(Ax)$ - with $A=\xi_i(\Gamma^m)$ - is the restriction to the subspace $\Gamma^m$ of the truncation $Q_r$. 
	
	Now, fix $\Gamma^m\in( \tG_{Q}(m,n)\cap  \overline{ \mathsf{V}}_{i})\subset \cE_i$.  
	By the continuity of $F_i$, there exists $\eps_{i,\Gamma^m}^\star=\eps_{i,\Gamma^m}^\star(r,s,m,\tau,n)>0$ such that, for any $\eps\in]0, \eps^\star_{i,\Gamma^m}]$, the open ball $\upsilon^m_{i,\Gamma^m}(\eps)\subset\Xi^m_i(Q,\eps)$ centered at $(Q,\Gamma^m)$ verifies the following property: for any $(S,\widehat\Gamma^m)\in \upsilon_{i,\Gamma^m}^m(\eps)$, the restricted truncated polynomial $\widehat T_r(x):=S_r(\widehat Ax)$, with $\widehat A=\xi_i(\widehat \Gamma^m)$, is contained in an open ball of radius $\zeta_i$ around $P_r(x):=Q_r(Ax)$, with $A=\xi_i(\Gamma^m)$. Hence, on the one hand by \eqref{cipresso} one infers
	\begin{equation}\label{ginestra}
	||\widehat T_r-P_r||_\infty<\zeta_i\quad \Longrightarrow \quad  ||\widehat T_r-\Sigma(r,s,m)||_\infty> \zeta_i\ .
	\end{equation}
	On the other hand, by construction one has 
	\begin{equation}\label{unione_fa_la_forza}
	\left(Q,\tG_{Q}(m,n)\cap  \overline{ \mathsf{V}}_{i}\right)\,\subset\, 	\bigcup_{\Gamma^m\in\,\tG_{Q}(m,n)\cap \overline{ \mathsf{V}}_i}\upsilon_{i,\Gamma^m}^m(\eps^\star_{i,\Gamma^m})
	\end{equation}
	and - due to the compactness of the fiber $(Q,\,\tG_{Q}(m,n)\cap  \overline{ \mathsf{V}}_{i})$ - it is possible to extract a finite number $J_i$ of subspaces $\Gamma^m\in \tG_{Q}(m,n)\cap\overline{ \mathsf{V}}_i$ from \eqref{unione_fa_la_forza} and write (with slight abuse of notation)
	\begin{equation}\label{trentatre_trentini}
	\left(Q,\tG_{Q}(m,n)\cap  \overline{ \mathsf{V}}_{i}\right)\,\subset\, 	\bigcup_{j=1}^{J_i}\upsilon_{i,j}^m(\eps)\,\subset\, \Xi^m_i(Q,\eps)\quad , \quad  \eps \in \left]0,\eps^\star_i\right]\ ,
	\end{equation}
	where we have set
	\begin{equation}
	\eps^\star_i=\eps^\star_i(r,s,m,\tau,n):=\min_{j\in \{1,\dots,J_i\}}\{\eps^\star_{i,j}\}\ .
	\end{equation}
	Inclusion \eqref{trentatre_trentini}, together with \eqref{omegaemmei}, yields that the finite union  $\upsilon^m(\eps):=\cup_{i=1}^L\cup_{j=1}^{J_i}\upsilon^m_{i,j}(\eps)$ is an open neighborhood of $\Xi^m(Q,\eps)$ containing the fiber $(Q, \tG_{Q}(m,n))$. Therefore, by setting 
	$$
	\zeta=\zeta(r,s,m,\tau):=\min_{i\in\{1,\dots,L\}}\{\zeta_i\}>0\ ,\ \   \eps^\star=\eps^\star(r,s,m,\tau,n):=\min_{i=1,\dots,L}\{\eps^\star_i\}>0\ ,
	$$
	and by taking \eqref{ginestra} into account, one has that for $0< \eps\le  \eps^\star$ and for any $(S,\widehat \Gamma^m)\in \upsilon^m(\eps)\subset \Xi^m(Q,\eps)$, the restricted truncated polynomial $ \widehat T_r(x):=S_r( Ax)$ - with $ A=\xi_j( \widehat \Gamma^m)$ for some $j=1,...,L$ - verifies 
	\begin{equation}\label{ginestra2}
	|| \widehat T_r-\Sigma(r,s,m)||_\infty> \zeta_j\geq \zeta\ .
	\end{equation} 
	Therefore, we have proved that, for any $0\le \eps\le \eps^\star$ there exists $\zeta>0$ such that
	\begin{equation}
	||S_r|_{\Gamma^m}- \Sigma(r,s,m)||_\infty>\zeta\ ,
	\end{equation}
	for any $S\in B^N(Q,\eps)$ and for any $\Gamma^m$ orthogonal to $\nabla S(0)=\nabla S_r(0)\neq 0$. Hence, the Definition of set $\cV(r,s,m,n)$ in \eqref{V} ensures that  for any $0\le \eps\le \eps^\star$ there exists $\chi=\chi(r,s,m,\tau,n)>0$ such that for any $S\in B^N(Q,\eps)$ one has 
	\begin{equation}
	||S_r- \cV(r,s,m,n)||_\infty>\chi\ .
	\end{equation}
	
	This concludes the proof.
	
\end{proof}
We need another intermediate result in order to demonstrate Theorem \ref{steep_polinomi}. Before giving its statement, for any polynomial $S\in \cP^\star(r',n)$, we firstly consider its associated minimal arc $\gamma$ constructed in Theorem \ref{arco_minimale}. Also, for any $\lambda>0$ we indicate by $\tI'_\lambda\subset [-\lambda,\lambda]$ the  interval obtained by cutting the interval $\tI_\lambda$ at point 3) of Theorem \ref{arco_minimale} into three equal pieces and by taking the central one. In particular, we have 
$$
|\tI'_\lambda|=\frac{|\tI_\lambda|}{3}=\frac{\lambda}{3\tK}\ ,
$$
where $\tK=\tK(r',m,n)$ is a suitable constant. 

We also assume the setting of Lemmata \ref{steep_polinomi} and \ref{aiutati}, and we consider a ball $B^N(Q,\eps)$ of radius $\eps<\eps^\star/2$ around $Q\in \Polpo$. Within this framework, one has

\begin{lemma}\label{Palermo}
	There exist two constants $\cC'=\cC'(r',r,s,m,\tau)$ and $\lambda_0=\lambda_0(r,s,m,\tau)$  such that - for any number $0<\lambda\le \lambda_0$, for any polynomial $S\in \Polpo$ contained in $B^N(Q,\eps)$, and for any $m$-dimensional subspace $\Gamma^m$ orthogonal to $\grad S(0)$, the restriction $T:=S|_{\Gamma^m}$ satisfies the following estimates 
	\begin{align}\label{basso}
	\begin{split}
	&\max_{\substack{\alpha=1,...,s}}\left|\frac{d^\alpha}{dt^\alpha}\left(\left.\frac{\partial\, T(x)}{\partial\, x_1}\right|_{ \gamma(t)}\right)_{t=t^\star}\right|> \cC' \quad, \qquad    \text{ for } m=1,\\
	&\max_{\substack{\ell=1,...,m\\\alpha=1,...,s}}\left|\frac{d^\alpha}{dt^\alpha}\left(\left.\frac{\partial\, T(x)}{\partial\, x_\ell}\right|_{ \gamma(t)}\right)_{t=t^\star}\right|> \cC'\lambda^{s-1}\quad, \qquad  \text{ for }  m\ge 2 
	\end{split}
	\end{align}
	at any point $t^\star\in \tI_\lambda'$.
	
\end{lemma}
\begin{proof}{\it (Lemma \ref{Palermo})}
	We proceed by steps.
	
	{\it Step 1.}   We consider the value $\eps^\star$ of Lemma \ref{aiutati} and we fix a polynomial $S$ in the ball  $B^N(Q,\eps)$, where $\eps\in[0,\eps^\star/2]$.  By Lemma \ref{aiutati}, there exists a parameter %$\chi=\chi(r,s,m,\tau,n)$,
	$\zeta=\zeta(r,s,m,\tau)>0$ such that %$||S_r- \cV(r,s,m,n)||_\infty>\chi$ and -
	on any $m$-dimensional subspace $  \Gamma^m$ orthogonal to $\grad S(0)\neq 0$ - the truncation $S_r$ varifies 
	\begin{equation}\label{vicino2}
	||S_r|_{\Gamma^m}-\Sigma(r,s,m)||_\infty>\zeta\ .
	\end{equation} 
	%	\red{Eliminare $\zeta$ e lasciare solo $\chi$.}
	
	Now, for a given subspace $\Gamma^m\in \tG_{S}(m,n)$, one can choose a matrix $A\in O(n,m)$ whose columns span $ \Gamma^m$ and set $T(x):=S(Ax)$. Then, for any $\lambda>0$, by Theorem \ref{arco_minimale} there exists a minimal real-analytic arc
	\begin{equation}\label{archetto}
	\gamma(t):=
	\begin{cases}
	x_1(t)=t\\
	x_j(t)=f_j(t),\ j\in\{2,...,m\}
	\end{cases}
	t\in \tI_\lambda\subset[-\lambda,\lambda]\ ,\ \ |\tI_\lambda|=\lambda/\tK(r',n,m)\ ,
	\end{equation}
	whose image is contained in the thalweg $\cT(S,\Gamma^m)$. We observe that, up to a change in the order of the vectors spanning $\Gamma^m$, in Theorem \ref{arco_minimale}, we can always suppose that the coordinate parametrizing $\gamma$ is the first one. We divide the interval $\tI_\lambda$ into three equal parts of length $\lambda/(3\tK)$ and we denote by $\tI'_\lambda:=[\lambda_{min},\lambda_{max}]$ the central one. Then, for any given $t^\star \in \tI'_\lambda$ associated to the point $\gamma(t^\star)=x^\star$, we consider the affine reparametrization 
	$$
	\gamma^\star(u):=
	\begin{cases}
	x_1(u)=u+t^\star\\
	x_j(u)=f_j(u+t^\star)\ ,\ \  j\in\{2,...,m\}
	\end{cases}
	u\in[\lambda_{min}-t^\star,\lambda_{max}-t^\star]\ .
	$$
	It is clear that $\gamma$ and $\gamma^\star$ share the same image, have the same speed everywhere, and that $\gamma^\star(0)=\gamma(t^\star)=x^\star$, so that for all $\alpha\in\{0,...,s\}$ and $\ell\in\{1,...,m\}$, one has
	\begin{equation}\label{quasi_s_vanishing}
	\frac{d^\alpha}{dt^\alpha}\left(\left.\frac{\partial T(x)}{\partial x_\ell}\right|_{\gamma(t)}\right)_{t=t^\star}=\frac{d^\alpha}{du^\alpha}\left(\left.\frac{\partial T(x)}{\partial x_\ell}\right|_{\gamma^\star(u)}\right)_{u=0}\ .
	\end{equation}
	Now, indicating by $L_\star:\R^m\longrightarrow\R^m, x\longmapsto x-x^\star=:\widetilde{x}$ the translation w.r.t. $x^\star$, the curve $\gamma^\star$ is mapped into $\widetilde \gamma^\star:=L_\star\circ \gamma^\star$, which reads
	\begin{equation}\label{curvetta}
	\widetilde{\gamma}^\star (u):=
	\begin{cases}
	\widetilde{x}_1(u)=x_1(u)-x_1^\star=u+t^\star-t^\star=u\\
	\widetilde{x}_j(u):=x_j(u)-x_j(t^\star)=f_j(u+t^\star)-f_j(t^\star)=:\widetilde{f}_j(u)
	\end{cases}\ ,
	\end{equation}
	with $u\in [\lambda_{min}-t^\star,\lambda_{max}-t^\star]$. The polynomial $T$ written in the new coordinates reads 
	\begin{equation}\label{TU}
	T(x)=T\circ L_\star^{-1}(\widetilde x)=T(\widetilde{x}+x^\star)=:U^\star(\widetilde{x})\ .
	\end{equation}
	Since $x^\star$ is fixed, one has 
	\begin{equation}\label{lallero}
	\frac{\partial T(x)}{\partial x_\ell}=
	\frac{\partial U^\star(\widetilde x)}{\partial \widetilde{x}_\ell}\ ,\ \ \forall \ell\in\{1,...,m\}\ .
	\end{equation}
	Moreover, if one takes into account the fact that $\widetilde \gamma^\star(0)=0$, and that the origin for the coordinates $\widetilde x$ corresponds to the point $x=x^\star$ in the old coordinates, equality \eqref{quasi_s_vanishing}, together with \eqref{lallero}, yields, for all $\alpha\in\{0,...,s\}$ and for all $\ell\in\{1,...,m\}$,
	\begin{equation}\label{conlemani}
	\frac{d^\alpha}{dt^\alpha}\left(\left.\frac{\partial T(x)}{\partial x_\ell}\right|_{\gamma(t)}\right)_{t=t^\star}=\frac{d^\alpha}{du^\alpha}\left(\left.\frac{\partial T(x)}{\partial x_\ell}\right|_{\gamma^\star(u)}\right)_{u=0}=\frac{d^\alpha}{du^\alpha}\left(\left.\frac{\partial U^\star(\widetilde{x})}{\partial \widetilde{x}_\ell}\right|_{\widetilde\gamma^\star(u)}\right)_{u=0}\ .
	\end{equation}
	{\it Step 2. } Bernstein's estimate \eqref{Bernie} applied to the  components of $\gamma$ in \eqref{archetto} reads
	\begin{equation}
	\max_{t\in \tI_\lambda'}|f_j(t)|\le \tM_0\lambda\quad , \qquad j=2,\dots,m
	\end{equation}
	for some uniform constant $\tM_0=\tM(r',n,m,0)$, so that - for any $t^\star\in\tI_\lambda'$ - one has 
	\begin{equation}\label{Termini}
	||x^\star||_\infty:=|| \gamma(t^\star)||_\infty \le\tM_0\lambda\ .
	\end{equation}
	For any given $\eps\in[0,\eps^\star]$, with the help of the arguments in the proof of Lemma \ref{aiutati} - in particular taking \eqref{omegaemme}, \eqref{omegaemmei}, and \eqref{fi} into account - the set $\Xi^m(Q,\eps)$ admits the covering $\Xi^m(Q,\eps)=\cup_{i=1}^L\, \Xi^m_i(Q,\eps)$, and for any index $i\in\{1,...,L\}$ there exists a continuous function
	\begin{equation}\label{si}
	F_i: \ \Xi^m_i(Q,\eps)\longrightarrow \cP(r',m)
	\end{equation}
	that maps
	\begin{equation}
	(R, \Gamma^m)\longmapsto  R( A x)\ ,\ \  A:=\xi_i( \Gamma^m)\in O(n,m)\ ,
	\end{equation}
	where $\xi_i$ is a local continuous section for the Grassmannian $\tG(m,n)$. 
	
	Hence, taking Lemma \ref{aiutati} and formula \eqref{TU} into account, the function 
	\begin{align}
	\begin{split}
	&\tF_i^\star: \Xi^m_i(Q,\eps)\times \R^m\longrightarrow \cP(r',m) \\
	&(R, \Gamma^m,x^\star)\longmapsto R(A(\widetilde x+x^\star))\ ,\ \  A:=\xi_i(\Gamma^m)\in O(n,m)
	\end{split}
	\end{align}
	is continuous. Moreover, \eqref{Termini} and the Theorem of Heine-Cantor ensure that $\tF^\star_i$ is uniformly continuous over the restricted compact domain $\overline \Xi^m_i(Q,\eps^\star/2)\times \overline B^n_\infty(0,\tM_0\lambda)$.
	Choosing the value $\zeta(r,s,m,\tau)>0$ in \eqref{vicino2}, there exists a uniform positive real $0<\lambda_i=\lambda_i(\zeta)\le 1$ such that, for any $0<\lambda\le \lambda_i$ and for any $(S,\Gamma^m)\in \Xi^m_i(Q,\eps^\star/2)$, the image of the set $(S,\Gamma^m,\overline B^n_\infty(0,\tM_0\lambda))$ through $\tF_i^\star$ verifies
	\begin{equation}\label{ammazza}
	\tF_i^\star\left(S,\Gamma^m,\overline B^n_\infty(0,\tM_0\lambda)\right) \subset B^M_\infty\left(\tF_i^\star(S,\Gamma^m,0),\frac{\zeta}{2}\right)\ ,\ \   M:=\dim \cP(r',m)\ .
	\end{equation}
	The above reasonings imply for any $t^\star \in\tI_\lambda'\subset [-\lambda,\lambda]$, with $0<\lambda\le \lambda_i$, one has
	\begin{equation}\label{ambe}
	||(T\circ L_\star^{-1})_r-T_r||_\infty=:||U_r^\star-T_r||_\infty\le ||U^\star-T||_\infty< \frac{\zeta}{2} \ ,
	\end{equation}
	where $U_r^\star, T_r$ are the truncations at order $r$ of polynomial $U^\star$ introduced in \eqref{TU} and of polynomial $T:=S|_{\Gamma^m}$, respectively. 
	
	Repeating the same argument for any index $j\in\{1,...,L\}$, relations \eqref{vicino2}, \eqref{ammazza}, and \eqref{ambe} imply that, if 
	\begin{equation}\label{lambdazero}
	0<\lambda\le \lambda_0:=\min_{i\in\{1,...,L\}}\{\lambda_i=\lambda_i(\xi(r,s,m,\tau))\}\le 1\ ,
	\end{equation} then, for any $(S_r,\Gamma^m)\in \Xi^m(Q_r,\eps/2)$ and for any $t^\star\in \tI_\lambda'$ one has 
	\begin{equation}\label{aho}
	||U_r^\star-\Sigma(r,m,n)||_\infty=||T_r\circ L_\star^{-1}-\Sigma(r,m,n)||_\infty> \frac{\zeta}{2} \ .
	\end{equation}
	
	{\it Step 3.} 
	Without any loss of generality, we suppose that the minimal curve $\gamma\subset \cT(S,\Gamma^m)$ is parametrized by the first coordinate. Hence, for $n\ge 3$ and $2\le m\le n-1$ we indicate by $\ta=(a_{12},...,a_{1m})$ the linear coefficients of the Taylor expansion of the translated curve $\widetilde \gamma^\star$ in \eqref{curvetta}. One can make use of the set of adapted coordinates $\widetilde\ty:=\cL_\ta(\widetilde x)$ for the curve $\widetilde{\gamma}^\star$, as defined in paragraph \ref{appropriata}. We remind that $\cL_\ta:= \id$ in case $m=1$ (see also \eqref{extend}). 
	
	\medskip 
	
	In particular, for all $m\in\{1,\dots,n-1\}$ we write $\tU^\star_{r,\ta}:=U^\star_r\circ \cL^{-1}_{\ta}$ and $ \widetilde \gamma^\star_\ta:=\cL_\ta\circ\widetilde\gamma^\star$. 
	
	\medskip 
	
	By construction, the curve $\widetilde \gamma^\star_\ta\in \Theta_m^1$ (see Definition \eqref{arco} and Remark \ref{1-coord}) is analytic in $[\lambda_{min}-t^\star,\lambda_{max}-t^\star]$ with complex analyticity width $\lambda/\tK$, and $|\tI_\lambda'|=|\lambda_{max}-\lambda_{min}|= \lambda/(3\tK)$, as  $\tI_\lambda'$ was obtained by cutting $\tI_\lambda$ into three equal pieces and by taking the central one. By \eqref{aho} and Lemma \ref{stima_inf}, there exist constants\footnote{In Lemma \ref{stima_inf}, $C_1$ depends on the open set $\tD\in \Polm\backslash \Sigma(r,s,m)$. In our case, by \eqref{aho}, $\tD$ is the open ball of radius $\frac12\zeta$ around $T$, which is at distance at least $\frac12 \zeta$ from $\Sigma(r,s,m)$; hence, with slight abuse of notation, we can write $C_1=C_1(\zeta)$.} $C_1=C_1(\zeta)$ and $C_2=C_2(s,m)$ such that one has the lower estimate
	\begin{align}\label{Firenze}
	\begin{split}
	&\max_{\substack{\alpha=1,...,s}}\left|\frac{d^\alpha}{du^\alpha}\left(\left.\frac{\partial\, \tU^\star_{r,\ta}( \ty)}{\partial\,  \ty_1}\right|_{\widetilde \gamma^\star_\ta(u)}\right)_{u=0}\right|> C_1 \\  &\text{ in case } s=1 \text{ or } m=1,\\
	&\max_{\substack{\ell=1,...,m\\\alpha=1,...,s}}\left|\frac{d^\alpha}{du^\alpha}\left(\left.\frac{\partial\, \tU^\star_{r,\ta}(\ty)}{\partial\, \ty_\ell}\right|_{\widetilde \gamma^\star_\ta(u)}\right)_{u=0}\right|> \displaystyle \frac{C_1}{1+C_2\times  \max_{\substack{\ell=2,...,m\\ \alpha=2,...,s}}|a_{\ell\alpha}|}\\ &\text{ in case } 2\le s\le r-1 \text{ and } \ m\ge 2 \ ,
	\end{split}
	\end{align}
	where the $a_{\ell \alpha}$'s, with $\ell\in\{2,...,m\}$ and $\alpha\in\{2,...,s\}$, are the Taylor coefficients of $\widetilde\gamma^\star_\ta(u)$ at the origin. Definition \eqref{curvetta} assures that the Taylor coefficients of order equal or higher than one of the curve $\widetilde{ \gamma}^\star(u)$ at $u=0$, and those of the curve $\gamma$ calculated at $t=t^\star$ coincide. Moreover, by construction (see paragraph \eqref{appropriata}) $\widetilde\gamma^\star_\ta(u)$ and $\widetilde\gamma^\star(u)$ share the same Taylor coefficients of order greater or equal than two calculated at the origin. Hence, the Bernstein estimate in \eqref{Bernie} applied to the second relation in \eqref{Firenze} and the fact that $0<\lambda\le\lambda_0\le 1$ (see \eqref{lambdazero}) yield that there exists a uniform constant $\tM=\tM(r',n,m,s)\ge 1$ such that estimate
	\begin{align}\label{Siena}
	\begin{split}
	&\max_{\substack{\ell=1,...,m\\\alpha=1,...,s}}\left|\frac{d^\alpha}{du^\alpha}\left(\left.\frac{\partial\, \tU^\star_{r,\ta}(\ty)}{\partial\,\ty_\ell}\right|_{\widetilde \gamma^\star_\ta(u)}\right)_{u=0}\right|> \frac{C_1}{(1+C_2)\,\tM}\,\lambda^{s-1}
	\end{split}
	\end{align}
	holds in case $2\le s\le r-1$ and $m\ge 2$.
	
	Now, expression \eqref{conviene} together with estimate \eqref{Bernie} yields
	$$
	||\cL_\ta^{-1}||_\infty\le 1+\tM\ , 
	$$ 
	for the matrix norm of $\cL_\ta^{-1}$. Therefore, by \eqref{mansomma}, by the first estimate in \eqref{Firenze} and by \eqref{Siena}, we infer
	\begin{align}\label{San_Gimignano}
	\begin{split}
	&\max_{\substack{\alpha=1,...,s}}\left|\frac{d^\alpha}{du^\alpha}\left(\left.\frac{\partial\, U^\star_r(x)}{\partial\, x_1}\right|_{\widetilde \gamma^\star(u)}\right)_{u=0}\right|> \frac{C_1}{1+\tM}\\ \qquad &   \text{ for } s=1 \text{ or } m=1,\\
	\\
	&\max_{\substack{\ell=1,...,m\\\alpha=1,...,s}}\left|\frac{d^\alpha}{du^\alpha}\left(\left.\frac{\partial\, U^\star_r(x)}{\partial\,x_\ell}\right|_{\widetilde \gamma^\star(u)}\right)_{u=0}\right|> \frac{C_1}{(1+C_2)\tM(1+\tM)}\lambda^{s-1}\\
	\qquad & \text{ for } 2\le s\le r-1 , \ m\ge 2 \ .
	\end{split}
	\end{align}
	Since in expression \eqref{San_Gimignano} one is considering only the derivatives up to order $s\in\{1\dots,r-1\}$ at the origin $u=0$ and $\widetilde \gamma^\star(u)$ contains no constant terms, the same estimate holds true for the polynomial $U^\star$.
	%\begin{align}\label{Certaldo}
	%\begin{split}
	%&\max_{\substack{\ell=1,...,m\\\alpha=1,...,s}}\left|\frac{d^\alpha}{du^\alpha}\left(\left.\frac{\partial\, U^\star(x)}{\partial\, x_\ell}\right|_{\widetilde \gamma^\star(u)}\right)_{u=0}\right|> \frac{C_1}{1+\tM} \quad   \text{ for } s=1 \text{ or } m=1,\\
	%&\max_{\substack{\ell=1,...,m\\\alpha=1,...,s}}\left|\frac{d^\alpha}{du^\alpha}\left(\left.\frac{\partial\, U^\star(x)}{\partial\,x_\ell}\right|_{\widetilde \gamma^\star(u)}\right)_{u=0}\right|> \frac{C_1}{C_2(1+\tM)\, \tM}\lambda^{s-1}\quad \text{ for } 2\le s\le r-1 , \ m\ge 2 \ .
	%\end{split}
	%\end{align}
	
	The thesis follows from \eqref{San_Gimignano} and from \eqref{conlemani} by setting $\cC'=\displaystyle\frac{C_1}{(1+C_2)\tM(1+\tM)}$. 
	
\end{proof}
With the help of Lemmata \ref{aiutati} and \ref{Palermo}, we are now able to prove Theorem \ref{steep_polinomi}. 
\begin{proof}{\it (Theorem \ref{steep_polinomi})} 
	
	{\it Introduction.} We assume the setting and the notations of Lemmata \ref{aiutati}-\ref{Palermo}. In particular, for $0<\eps\le\eps^\star/2 $ we consider a polynomial $S\in \Pollo$ in the ball $B^N(Q,\eps)$, and a given $m$-dimensional subspace $ \Gamma^m$ orthogonal to $\grad S(0)\neq 0$. We denote by $\gamma$ the minimal arc of Theorem \ref{arco_minimale} - whose image is contained in the thalweg $\cT(S, \Gamma^m)$ - and for any $0<\lambda\le  \lambda_0=\lambda_0(r,s,m,\tau)$ we indicate by $\tI_\lambda$ its interval of analyticity of length $\lambda/\tK$, where $\tK=\tK(r',n,m)$ is a uniform constant. We also indicate by $\tI'_\lambda$ the interval which is obtained by dividing $\tI_\lambda$ into three equal parts and by taking the central one. 
	
	Finally, we set $T(x):=S|_{\Gamma^m}(x):=S(Ax)$, with $A\in O(n,m)$ a matrix belonging to the image of the continuous section $\xi_i:\cE_i(\Gamma^m)\longrightarrow O(n,m)$, with $i\in\{1,...,L\}$,  and whose columns span $ \Gamma^m$ (see the proof of Lemma \ref{aiutati}). 
	
	We proceed by steps.
	
	{\it Step 1. } For $\ell=1,...,m$ and $\alpha=\{1,...,s\}$, we consider the functions
	\begin{equation}\label{gigi}
	g_\ell^{(\alpha)}(t^\star):=\frac{d^\alpha}{dt^\alpha}\left(\left.\frac{\partial T(x)}{\partial x_\ell}\right|_{\gamma(t)}\right)_{t=t^\star}\ ,\ \ t^\star\in\tI_\lambda'
	\end{equation}
	and the constant functions
	$$
	g_{m+1}(t^\star):=\cC'\ ,\ \ g_{m+2}(t^\star):=-\cC'
	\ ,\ \
	g_{m+3}(t^\star):= \cC'\lambda^{s-1} \ ,\ \
	g_{m+4}(t^\star):=- \cC'\lambda^{s-1}\ .
	$$ 
	
	The degree of $T$ is  bounded by $r'$ and - on the interval $\tI_\lambda$ - $\gamma$ is an analytic-algebraic function whose diagram is bounded by a positive integer $\td=\td(r',n,m)$ (see Point 2 of the thesis in Theorem \ref{arco_minimale}). Hence, for any given choice of $\alpha\in\{1,...,s\}$ and $\ell\in\{1,...,m\}$, the function $g^{(\alpha)}_\ell(t^\star)$ is Nash (i.e. semi-algebraic of class $C^\infty$) due to Propositions \ref{composta} and \ref{derivata}, and its diagram is bounded by a quantity depending only on $r',m,n$. In addition, Proposition \ref{Nash} ensures that $g^{(\alpha)}_\ell(t^\star)$ is  actually analytic-algebraic in $\tI_\lambda'$. The same is obviously true also for $g_{m+j}(t^\star)$ for $j\in\{1,2,3,4\}$.
	Therefore, we set $\widehat d=\widehat d(r',m,n):=\max_{i\in\{1,...,m+4\}}\max_{\alpha\in\{1,...,s\}}\{\diag\,( g_i^{(\alpha)})\}$. 
	
	Now, for any choice of $\alpha\in\{1,...,s\}$ and $i\in\{1,...,m+4\}$, the graph of $g^{(\alpha)}_i$ over $\tI_\lambda'$ belongs to the algebraic curve of some non-constant polynomial $V_i^{(\alpha)}\in \R[x,y]$ of two variables, whose degree depends on $\widehat d(r',m,n)$. If we indicate by 
	\begin{equation}
	V_{i}^{(\alpha)}(x,y)=\Pi_{k=1}^{K(i,\alpha)}\,\left(V_{i,k}^{(\alpha)}(x,y)\right)^{p_k}\quad, \qquad \text{ with }\quad  K(i,\alpha), p_k\in \N^\star 
	\end{equation}
	the decomposition of $V_{i}^{(\alpha)}(x,y)$ into its irreducible factors, by Bézout's Theorem (see Th. \ref{Th_Bezout}) the irreducible components $\left\{(x,y)\in \R^2| V_{i,k}^{(\alpha)}(x,y)=0 \right\}$ of the algebraic curve $\left\{(x,y)\in \R^2| V_{i}^{(\alpha)}(x,y)=0 \right\}$ intersect at most at a finite number of points which is bounded by $\left(\deg V_{i}^{(\alpha)}\right)^2$, which in turn is a quantity depending only on  $\widehat d(r',m,n)$. This fact, together with the regularity of $g_i^{(\alpha)}$ in $\tI_\lambda'$, implies that there exist two positive integers $\overline k(i,\alpha)\in \{1,\dots,K(i,\alpha)\}$, $\mathscr{w}=\mathscr{w}(r',m,n)$, and a subinterval of $\tI^\star_{\lambda,i,\alpha}\subset \tI_\lambda'$ of length $|\tI_\lambda'|/\mathscr{w}$  verifying the two following conditions:
	\begin{align}\label{appartenenza_irriducibile}
	\begin{split}
	&\text{graph} \left(\left.g^{(\alpha)}_i\right|_{\tI_{\lambda,i,\alpha}^\star}\right)\subset \left\{(x,y)\in \R^2|  V_{i,\overline k(i,\alpha)}^{(\alpha)}(x,y)=0\right\}\\
	&\text{graph} \left(\left.g^{(\alpha)}_i\right|_{\tI_{\lambda,i,\alpha}^\star}\right)\cap  \left\{(x,y)\in \R^2|  V_{i, k}^{(\alpha)}(x,y)=0\right\}=\varnothing
	\end{split}
	\end{align}
	for all $k\in \{1,\dots,K(i,\alpha)\}\backslash\{\overline k(i,\alpha)\}$. 
	
	The above reasoning can be repeated for all other pairs of integers belonging to $ \{1,\dots,s\}\times \{1,\dots,m+4\}$ and which are different from $ (\alpha,i)$. Hence, finally there exists an interval $\tI^\star_\lambda\subset \tI_\lambda'$ of length $|\tI'_\lambda|/(\mathscr{w}^{(m+4)\,s})$ on which	for any $(i',\alpha')\in \{1,\dots,s\}\times \{1,\dots,m+4\}$  the relations 
	\begin{align}\label{appartenenza_irriducibile_2}
	\begin{split}
	&\text{graph} \left(\left.g^{(\alpha')}_{i'}\right|_{\tI_{\lambda}^\star}\right)\subset \left\{(x,y)\in \R^2|  V_{i',\overline k(i',\alpha')}^{(\alpha')}(x,y)=0\right\}\\
	&\text{graph} \left(\left.g^{(\alpha')}_{i'}\right|_{\tI_{\lambda}^\star}\right)\cap  \left\{(x,y)\in \R^2|  V_{i', k}^{(\alpha')}(x,y)=0\right\}=\varnothing
	\end{split}
	\end{align}
	are verified for some integer $\overline k(i',\alpha')\in \{1,\dots,K(i',\alpha')\}$ and for any integer $k\in \{1,\dots,K(i',\alpha')\}\backslash\{\overline k(i',\alpha')\}$.

	Then, by \eqref{appartenenza_irriducibile_2} and again by Bézout's Theorem, there exists a positive integer $\mathsf{N}=\mathsf{N}(r',m,n)$ such that for any distinct pairs of integers $(\alpha, i)$ and $(\beta,j)$ belonging to $\{1,\dots,s\}\times \{1,\dots,m+4\}$, the algebraic curves $\{(x,y)\in (\tI_\lambda^\star,\R)| V_{i,\overline k(i,\alpha)}^{(\alpha)}(x,y)=0 \}$ and $\{(x,y)\in (\tI_\lambda^\star,\R)| V_{j,\overline k(j,\beta)}^{(\beta)}(x,y)=0 \}$ either coincide or intersect at most at $\mathsf{N}=\mathsf{N}(r',m,n)$  points. 
	
	By repeating this reasoning for all possible distinct pairs and by taking \eqref{appartenenza_irriducibile_2} into account, one finally has that there exists a positive constant $\mathsf{M}=\mathsf{M}(r',s,m,n)$ and an interval $\tJ^\star_\lambda$ of uniform length $|\tJ^\star_\lambda|=|\tI^\star_\lambda|/\mathsf{M}$ over which the graphs of any pair of functions among $g_1^{(1)},...,g_m^{(1)},...,g_1^{(s)},...,g_m^{(s)},g_{m+1},...,g_{m+4}$ either do not intersect or coincide. 
	
	These reasonings - together with the fact that expression \eqref{basso} in Lemma \ref{Palermo} holds for all $t^\star\in \tJ^\star_\lambda\subset \tI'_\lambda$ - yield that there must exist $\overline \alpha\in\{1,...,s\}$ and $\overline \ell\in\{1,...,m\}$ verifying 
	\begin{align}\label{violoncello}
	\begin{split}
	&\min_{t^\star\in\tJ^\star_\lambda}\left|g^{(\overline \alpha)}_{\overline \ell}(t^\star)\right|> \cC' \quad   \text{ for } m=1\\
	&\min_{t^\star\in\tJ^\star_\lambda}\left|g^{(\overline \alpha)}_{\overline \ell}(t^\star)\right|> \cC'\lambda^{s-1}\quad \text{ for }  \ m\ge 2 \ .
	\end{split}
	\end{align}
	
	{\it Step 2.} We apply Lemma \ref{Pyartli} to $g_{\overline \ell}$, with $q\equiv\overline\alpha$, $[a,b]\equiv \tJ_\lambda^\star$ and with $\beta$ equal to the r.h.s. of \eqref{violoncello}. If we ask for
	\begin{equation}
	4
	\left( \overline \alpha!\frac{\rho}{2\beta}\right)^{1/\overline \alpha}\le  \frac{|\tJ_\lambda^\star|}{2}=\frac{\lambda}{6\,\tK\,\mathsf{M}\,{\mathscr{w}}^{(m+4)\,s}}
	\end{equation}
	and we take into account the fact that $\overline \alpha\in\{1,...,s\}$, we can choose
	\begin{equation}
	\rho=	\frac{2\,\lambda^{s}}{s \,!\times \left[24\,\tK\,\mathsf{M}\,{\mathscr{w}}^{(m+4)\,s}\right]^{s}}
	\times 
	\begin{cases}
	\cC' \quad   \text{ for }  m=1\\
	\cC'\lambda^{s-1}\quad \text{ for } \ m\ge 2\ .
	\end{cases}
	\end{equation}
	Hence, in an open set $\mathtt \tA_\lambda\subset \tJ_\lambda^\star$ of measure $\displaystyle \frac{|\tJ_\lambda^\star|}{2}=\frac{\lambda}{6\,\tK\,\mathsf{M}\,{\mathscr{w}}^{(m+4)\,s}}$, one has 
	\begin{equation}\label{idea}
	|g_{\overline \ell}(t)|>\rho=
	\begin{cases}
	\tC_1 \lambda^s \quad , \qquad   &\text{ for }  m=1\\
	\tC_m\lambda^{2s-1}\quad , \qquad  &\text{ for }  \ m\ge 2
	\end{cases}\qquad \forall t\in\tA_\lambda
	\end{equation}
	for some $\overline \ell\in\{1,...,m\}$, and for a constant 
	\begin{equation}
	\tC_m=\tC_m(r',r,s,\tau,n)=	\frac{2\cC'(r',r,s,m,\tau)}{s \,!\times \left[24\,\tK(r',m,n)\,\mathsf{M}(r',s,m,n)\,{\mathscr{w}(r',m,n)}^{(m+4)\,s}\right]^{s}}\ ,
	\end{equation}
	where $m\in\{1,\dots,n-1\}$.
	
	{\it Step 3.}
	Taking the definition of $\tA_\lambda$ into account, by construction (see \eqref{gigi}) we have
	\begin{equation}\label{campagna}
	\max_{t\in\tA_\lambda}\left|g_{\overline \ell}(t)\right|:=\max_{t\in\tA_\lambda}\left|\left.\frac{\partial T(x)}{\partial x_{\overline \ell}}\right|_{\gamma(t)}\right|:=\max_{t\in\tA_\lambda}\left|\left.\frac{\partial S|_{ \Gamma^m}(x)}{\partial x_{\overline \ell}}\right|_{\gamma(t)}\right|\ .
	\end{equation}
	Due to Theorem \ref{arco_minimale}, the image of $\gamma$ is contained in the thalweg $\cT(S, \Gamma^m)$, that is in the locus of minima of $T:=S|_{ \Gamma^m}$ on the spheres $\cS^m_\eta\subset \Gamma^m$ of radius $\eta>0$ centered at the origin. Moreover, the curve $\gamma$ was constructed by a uniform local inversion theorem applied to the curve $\phi$ of Lemma \ref{Jess} that was parametrized by the radius $\eta>0$ of the spheres $\cS^m_\eta\subset \Gamma^m$ and shared the same image with $\gamma$. So, to any value of $t\in\tA_\lambda$ there corresponds a unique radius $\eta(t)$ associated to a sphere $\cS^m_{\eta(t)}\subset \Gamma^m$. 
	
	Hence, taken any pair of real numbers $\lambda,\xi$ satisfying $0<\lambda\le \xi\le \lambda_0$ - where $\lambda_0$ is the quantity defined in Lemma \ref{Palermo} - by the discussion at step 3 of the proof of Theorem \ref{arco_minimale} (in particular, the inclusions in \eqref{inclusione}), one has that the inverse image of $\tA_\lambda$ is contained in the interval $\cI_\xi\subset [0,\xi]$ defined in Lemma \ref{Jess}. This argument and \eqref{campagna} imply that for some $\overline \ell\in\{1,...,m\}$ one has
	\begin{equation}
	\max_{t\in\tA_\lambda}\left|\left.\frac{\partial S|_{ \Gamma^m}(x)}{\partial x_{\overline \ell}}\right|_{\gamma(t)}\right|\le\max_{\eta\in[0,\xi]}\left|\left.\frac{\partial S|_{ \Gamma^m}(x)}{\partial x_{\overline \ell}}\right|_{\phi(\eta)}\right|= \max_{\eta\in[0,\xi]}\min_{||x||_2=\eta}\left|\frac{\partial S|_{ \Gamma^m}(x)}{\partial x_{\overline \ell}}\right|
	\end{equation} 
	which in turn, as $0<\lambda\le \xi\in(0,\lambda_0]$, by taking \eqref{idea}-\eqref{campagna} and the equivalence of norms into account, implies that 
	\begin{equation}\label{quasi}
	\max_{\eta\in[0,\xi]}\min_{||x||_2=\eta}\left|\left|\grad S|_{ \Gamma^m}(x)\right|\right|_2	>	\begin{cases}
	\tC_1 \lambda^s \ ,\ \    &\text{ for }  m=1\\
	\tC_m\lambda^{2s-1}\ ,\ \  &\text{ for }  \ m\ge 2
	\end{cases}
	\ ,\ \ \forall\ 0<\lambda\le \xi\in(0,\lambda_0]\ .
	\end{equation}
	Since the coordinates $x$ are associated to an orthonormal basis spanning $\Gamma^m$, for any point $I\in\R^n$ contained in the subspace $ \Gamma^m$ one has $\pi_{ \Gamma^m}(\nabla_I S(I))\equiv \nabla_x (S|_{ \Gamma^m})(x)$, and by choosing $\lambda=\xi$ in \eqref{quasi}, we have proved that any polynomial $S\in \Polpo$  in the ball $ B^N(Q,\eps)$, with $\eps\le \eps^\star/2$, is steep at the origin on the subspaces of dimension $m$ with index bounded as in \eqref{indice} and with coefficients $\tC_m,\lambda_0$. It remains to prove that this holds true also in a neighborhood of the origin.
	
	{\it Step 4.} For any polynomial $S\in \Polpo$, we consider the translation
	\begin{equation}
	\tH^\star: \Polpo \times \R^n \longrightarrow \Polpo\ ,\ \ (S(I),I^\star)\longmapsto S(I+I^\star)\ .
	\end{equation}
	$\tH^\star$ is uniformly continuous over the compact $\overline  B^N(Q,\eps^\star/4)\times \overline B^n(0,1)$. Hence, there exists a number $\widehat \delta=\widehat \delta(\eps^\star)>0$ such that for any $S\in \overline B^N(Q,\eps^\star/4)$, one has
	\begin{equation}
	\tH^\star(\{S\}\times \overline B^n(0,\widehat \delta))\subset B^N(S,\eps^\star/4)\ .
	\end{equation}
	Hence, for any given point $I^\star$ satisfying $||I^\star||_2<\widehat \delta $ and for any polynomial $S\in \overline B^N(Q,\eps^\star/4)$, the polynomial $S(I+I^\star)$ belongs to $ B^N(Q,\eps^\star/2)$.
	
	Now, we consider a polynomial $S\in \overline B^N(Q,\eps^\star/4)$. By the above reasonings, for any $I^\star$ verifying $||I^\star||_2<\widehat \delta$, one has that its translation $S(I+I^\star)$ belongs to $B^N(Q,\eps^\star/2)$. We have proved at Step 3 that any polynomial in $\Polpo$ belonging to $ B^N(Q,\eps)$ - with $\eps\in[0,\eps^\star/2]$  - is steep at the origin on the subspaces of dimension $m$, with index as in \eqref{indice}, and with uniform coefficients $\tC_m,\lambda_0$.  Consequently, for any given $I^\star$ satisfying $||I^\star||_2<\widehat \delta$, the polynomial  $S(I+I^\star)$ is steep at the origin on the $m$-dimensional subspaces, with uniform index and coefficients. This is equivalent to stating the same property for polynomial $S$ at any point $I^\star$ in a ball of radius $\widehat \delta$ around the origin. 
	
	The thesis follows by setting $\eps_0=\eps_0(r,s,m,\tau,n):=\eps^\star/4$.

\end{proof}

\subsection{Proof of the genericity of steepness}\label{proof genericity steepness}

With the help of Theorem \ref{steep_polinomi}, we are finally able to prove Theorem A.

\begin{proof}({\it Theorem A})
	
	It is sufficient to study the case in which $I_0=0$, else one considers the translated function $h_0(I):=h(I+I_0)$. We proceed by steps.
	
	{\it Step 1.}	For any choice of integers $r,n\ge 2$, and for any given $\ts=(s_1,\dots,s_{n-1})\in \N^{n-1}$, where $s_m\in\{1,\dots,r-1\}$ for all $m\in\{1,\dots,n-1\}$,   by taking \eqref{V} into account, we define
	\begin{equation}\label{Omega}
	\Omegone:=\bigcup_{m=1}^{n-1} \cV(r,s_m,m,n)\subset \Pollo\ .
	\end{equation}
	The above set is closed due to Lemma \ref{closedness steep}.
	
	For any given pair $\varrho,\tau>0$, we consider a function $h\in\mathscr{D}\subset C_b^{2r-1}(\overline B^n(0,\varrho))$ satisfying 
	\begin{equation}\label{fuoriporta}
	\nabla h(0)\neq 0\quad ,\qquad \left|\left|\tT_0(h,r,n)-\Omegone\right|\right|_\infty>\tau \ .
	\end{equation} 
	
	Now, for $m=1,\dots,n-1$, taking the definition of $\eps_0(r,s_m,m,\tau,n)$ in Theorem \ref{steep_polinomi} into account, we set
	\begin{equation}\label{soglia}
	\overline \epsilon=\overline \epsilon(r,\ts,\tau,n):=\frac12\times \min_{m\in\{1,...,n-1\}}\{\eps_0(r,s_m,m,\tau,n)>0\}\ .
	\end{equation}
	Then, for $\eps\in[0,\overline \epsilon]$, we consider a function $f\in\mathscr{D}\subset  C_b^{2r-1}(\overline B^n(0,\varrho))$ satifying
	\begin{equation}\label{prima_o_poi_finira}
	f\in  \boldsymbol{\fB}^{2r-2}(h,\eps,\overline B^n(0,\varrho))\ .
	\end{equation}
	Due to \eqref{prima_o_poi_finira}, $\tT_0\,(f,2r-2,n)$ is contained in a ball of radius $\eps$ around $\tT_0\,(h,2r-2,n)$ in $\Pol$. Hence, as by construction we have set $\eps\le \overline \epsilon$, where $\overline \epsilon$ was defined in \eqref{soglia}, the definition of set $\Omegone$ in \eqref{Omega}, together with condition \eqref{fuoriporta} yields that we can apply Theorem \ref{steep_polinomi} with $r'=2r-2$. In turn, this ensures the existence of positive constants $\tC_m=\tC_m(r'=2r-2,r,s_m,\tau,n)$, 
	$$
	\widehat  d=\widehat  d(r,\ts,\tau,n):=\min_{m\in\{1,\dots,n-1\}}\widehat \delta(r,s_m,m,\tau,n)\ ,
	$$ 
	and
	$$
	\overline \lambda=\overline \lambda(r,\ts,\tau) :=\min_{m\in\{1,...,n-1\}}\lambda_0(r,s_m,m,\tau)
	$$
	such that $\tT_0(f,2r-2,n)$ is steep in an open ball of radius $\widehat d$ around the origin with coefficients $\bar \lambda,\tC_m$, $m=1,\dots,n-1$, and with indices 
	\begin{equation}\label{pollice}
	\overline \alpha_m(s_m):=	
	\begin{cases}
	s_1\quad , \qquad &\text{ if $m=1$ }\\
	2s_m-1\quad, \qquad &\text{ if  $m\ge 2$ }\ .\\
	\end{cases} 
	\end{equation}
	
	{\it Step 2.}	
	For any $I\in B^n(0,R)$ - with
	\begin{equation}\label{tamburo}
	R=R(r,\ts,\tau,n,\varrho):= \min\left\{\frac{\varrho}{3},\frac{\widehat d(r,\ts,\tau,n)}{2}\right\}\ ,
	\end{equation} for any $m\in\{1,\dots,n-1\}$, and for any $m$-dimensional affine subspace $\Gamma^m=\Gamma^m(I)$ passing through $I$ and orthogonal to $\grad f(I)\neq 0$, we indicate by $f|_{\Gamma^m}$ the restriction of $f$ to $\Gamma^m$. We assume that any given $\Gamma^m(I)$ is endowed with the induced euclidean metric, and we indicate by $x$ a suitable system of coordinates on $\Gamma^m(I)$ whose origin $x=0$ corresponds to point $I$. Moreover, for any $\beta\in \{1,\dots,n\}$, we set
	$\partial_\beta:=\frac{\partial}{\partial x_\beta}$ and
	$$
	\kappa:=\min\left\{\overline \lambda,\frac{\varrho}{3}\right\}\ ,
	$$ 
	
	Now, we fix both $I\in B^n(0,R)$ and $\Gamma^m(I)$. By standard calculus, at any point $x$ verifying $||x||\le \kappa$ (hence, sufficiently close to $I$),  one can write
	\begin{align}\label{caldo}
	\begin{split}
	|\partial_\beta (f|_{\Gamma^m})(x)&- \tT_0(\partial_\beta (f|_{\Gamma^m}),2r-3,m)(x)|\\ &\le K(r,m) \max_{\substack{\upalpha\in \N^n\\|\upalpha|=2r-1}}\max_{I'\in \overline B^n(0,\varrho)}|D^{\alpha}  f(I')|\,\frac{||x||_{2}^{2r-2}}{(2r-2)!}
	\end{split}
	\end{align}
	for some constant $K(r,m)>0$.

	Since $\tT_0\,(\partial_\beta (f|_{\Gamma^m}),2r-3,m)(x)=\partial_\beta[\tT_0(f|_{\Gamma^m},2r-2,m)](x)$, taking \eqref{caldo} into account, we have
	\begin{align}\label{finimola}
	\begin{split}
	&|\partial_\beta (f|_{\Gamma^m})(x)|\\
	&\ge \left|\,|\,\partial_\beta[\,\tT_0\,(f|_{\Gamma^m},2r-2,m)\,]\,(x)\,|-|\partial_\beta (f|_{\Gamma^m})(x)- \tT_0(\partial_\beta (f|_{\Gamma^m}),2r-3,m)\,(x)|\,\right|\\
	&\ge  |\,\partial_\beta[\,\tT_0\,(f|_{\Gamma^m},2r-2,m)\,]\,(x)\,|-c(r,n,m,\varrho,\mathscr{D})\, {||x||_2}^{2r-2}\ ,
	\end{split}
	\end{align}
	where we have indicated 
	$$
	c=c(r,n,m,\varrho,\mathscr{D}):=\frac{K(r,m)}{(2r-2)!}\max_{g\in \overline{\mathscr{D}}}||g||_{C_b^{2r-1}(\overline B^n(0,\varrho))}\ .
	$$
	Estimate \eqref{finimola} implies that for any $x\in \Gamma^m(I)$ verifying $||x||\le \kappa$ we can write
	$$
	||\grad (f|_{\Gamma^m})(x)||_1\ge  \left|\left|\grad  \tT_0\,(f|_{\Gamma^m},2r-2,m)(x)\right|\right|_1-c\,n\, {||x||_2}^{2r-2}\ ,
	$$
	and, by the equivalence of norms,
	\begin{equation}\label{treni_tozeur}
	||\grad (f|_{\Gamma^m})(x)||_2\ge\frac{1}{n}  \left|\left|\grad  \tT_0\,(f|_{\Gamma^m},2r-2,m)(x)\right|\right|_2-c\, {||x||_2}^{2r-2}\ .
	\end{equation}
	
	{\it Step 3.} By the discussion at Step 1, $\tT_0(f,2r-2,n)$ is steep in an open ball of radius $\widehat d$ around the origin $I=0$, with coefficients $\bar \lambda,\tC_m$, $m=1,\dots,n-1$, and with indices as in \eqref{pollice}. This property, together with  expression \eqref{treni_tozeur} and with the fact that 
	\begin{itemize}
		\item the origin $x=0$ on $\Gamma^m(I)$ corresponds to point $I\in B^n(0,R)$ by construction;
		\item $R\le \widehat d/2$ by \eqref{tamburo};
	\end{itemize}
	yields   
	\begin{equation}\label{alice}
	\max_{\eta\in [0,\xi]}\,\min_{\substack{||x||_2\in \Gamma^1(I)\\||x||_2=\eta} }||\grad (f|_{\Gamma^1})(x)||_2> \frac{\tC_1}{n} \xi^{s_1}-c\, \xi^{2r-2}\quad  \forall \xi\in[0,\kappa]\quad (m=1)
	\end{equation}
	\begin{equation}\label{battiato}
	\max_{\eta\in [0,\xi]}\,\min_{\substack{||x||_2\in \Gamma^m(I)\\||x||_2=\eta} }||\grad (f|_{\Gamma^m})(x)||_2> \frac{\tC_m}{n} \xi^{2s_m-1}-c\, \xi^{2r-2}\quad  \forall \xi\in[0,\kappa]\quad (2\le m\le n-1)\ .
	\end{equation}

	If we impose
	\begin{equation}
	\begin{cases}
	c\, \xi^{2r-2}\le \displaystyle \frac{\tC_1}{2n} \xi^{s_1} \quad, \qquad &\text{  if } m=1\\
	\\
	c\, \xi^{2r-2}\le \displaystyle \frac{\tC_m}{2n} \xi^{2s_m-1} \quad, \qquad  &\text{  if } m=2,\dots,n-1
	\end{cases}
	\end{equation}
	by \eqref{alice}-\eqref{battiato}, and by the fact that $s_m\le r-1$ for all $m=1,\dots,n-1$, we have that $f$ is steep in a ball of radius 
	$R$
	around the origin with coefficients (we have set $r'=2r-2$)
	{\small
		\begin{align}
		\begin{split}
		\delta&=\delta(r,\ts,\tau,n, \varrho,\mathscr{D})\\ 
		&:=
		\min\left\{\kappa,\left(\frac{\tC_1(r',r,s_1,\tau,n)}{2\,n\,c(r,n,m,\varrho,\mathscr{D})}\right)^{\frac{1}{2r-2-s_1}},\min_{m\in\{2,\dots,n-1\}}\left\{\displaystyle\left( \frac{\tC_m(r',r,s_m,\tau,n)}{2\,n\,c(r,n,m,\varrho,\mathscr{D})}\right)^{\frac{1}{2(r-s_m)-1}}\right\}\right\} \ ,
		\end{split}
		\end{align}}
	$$
	C_m(r',r,s_m,\tau,n):=\frac{\tC_m(r',r,s_m,m,\tau,n)}{2n}\ ,
	$$ 
	and with indices bounded as in \eqref{pollice}.
	
	It remains to prove the estimate on the codimension of $\Omegone$. By \eqref{Omega}, Lemma \ref{positive_codimension} and Proposition \ref{dim_chiusura} we have
	$$
	\text{  codim } \Omegone \ge \max\left\{0,\min_{m\in\{1,\dots,n-1\}}\{s_m-m(n-m-1)\}\right\}\ .
	$$
	%Since, for fixed $n$, the quantity $m(n-m-1)$ reaches its maximum for $m=n/2$ when $n$ is even, and for $m=(n-1)/2$ when $n$ is odd, we find 
	%\begin{equation}
	%\text{  codim }\OmegoneI\ge \begin{cases}
	%\max\left\{0,s+1-\displaystyle\frac{n(n-2)}{4}\right\},\ \text{if $n$ is even}\\
	%\max\left\{0,s+1-\displaystyle\frac{(n-1)^2}{4}\right\},\ \text{if $n$ is odd}\\
	% \end{cases}\ .
	% \end{equation}
	This concludes the proof.

\end{proof}

\section{Proof of Theorem B and of its Corollaries}\label{Prova Th B}
Hereafter, we assume the notations and the results of the previous sections. 

\subsection{Proof of Theorem B}

It suffices to prove the statement for $I_0=0$, otherwise one considers $h_0(I):=h(I+I_0)$. 

\subsubsection{Case $m=1$.}
Let $\Gamma^1$ be a $1$-dimensional subspace (a line) orthogonal to $\grad h(0)\neq 0$, and let $w\in \S^n$ be its generating vector. By standard results of calculus, the restriction of the Taylor polynomial $\tT_0(h,r,n)$ to $\Gamma^1$, indicated by $\tT_0(h|_{\Gamma^m},r,1)$, reads
\begin{equation}\label{chickenuno}
\tT_0(h|_{\Gamma^1},r,1)(x)=h(0)+\sum_{i=1}^r \frac{1}{i!}h^i[w,\dots,w] x^i\ , 
\end{equation}
where the multi-linear notation in \eqref{forma} has been taken into account, and where $x$ is the coordinate associated to the vector $w$.

By \eqref{chickenuno} and Lemma \ref{m=1}, condition \eqref{s_1+1 jet} amounts to asking that, for any subspace $\Gamma^1$, the polynomial $\tT_0(h|_{\Gamma^1},r,1)$ belongs to the complementary of the set of $s_1$-vanishing polynomials $\sigma(r,s_1,1)$ in $\cP(r,1)$. Moreover, again by Lemma \ref{m=1}, one has $\sigma(r,s_1,1)=\overline \sigma(r,s_1,1)=:\Sigma(r,s_1,1)$. Hence, by definitions \eqref{U}-\eqref{V}, by Theorem A and by \eqref{Omega}, $h$ is steep on the subspaces of dimension one in a neighborhood of the origin, with steepness index bounded by $s_1$.

\subsubsection{Case $n\ge 3$, $2\le m \le n-1$. }

It is sufficient to prove that, for fixed $m\in\{1,\dots,n-1\}$, under the assumptions at point $ii)$ of Theorem B, one has $\tT_0(h,r,n)\in \Pol \backslash \cV(r,s_m,m,n)$, where the set $\cV(r,s_m,m,n)$ was defined in \eqref{V}. The thesis then follows by Theorem A and by expression \eqref{Omega}. 

By absurd, suppose that the claim is false. Then, by \eqref{U}-\eqref{V}, there exists some subspace $\Gamma^m$ orthogonal to $\nabla h(0)\neq 0$ such that $\tT_0(h|_{\Gamma^m},r,m)\in \Sigma(r,s_m,m)$. 

Hence, since $\Sigma(r,s_m,m):=\overline \sigma(r,s_m,m)$ by construction, there are two possibilities:
\begin{enumerate}
	\item either $\tT_0(h|_{\Gamma^m},r,m)\in \sigma(r,s_m,m)$\ ;
	\item or $\tT_0(h|_{\Gamma^m},r,m)\in \Sigma(r,s_m,m)\backslash \sigma(r,s_m,m)$\ .
\end{enumerate}
We consider the two cases separately and we prove that in both cases we end up being in contradiction with the hypotheses.

{\it Case 1.} If $\tT_0(h|_{\Gamma^m},r,m)\in \sigma(r,s_m,m)$, then by construction $\tT_0(h|_{\Gamma^m},r,m)$ verifies the $s_m$ vanishing condition at the origin on some curve $\gamma\in \arco$, whose image is contained in $\Gamma^m$. Since one is free to choose the orthonormal basis $\{u_1,\dots,u_m\}\in \tU(m,n)$ spanning $\Gamma^m$, up to a changement in the order of the vectors we can suppose without any loss of generality that the coordinate which parametrizes the curve $\gamma$ is the first one, that is $\gamma\in\arcox$, and $\tT_0(h|_{\Gamma^m},r,m)\in \sigma^1(r,s_m,m)$. Moreover, following section \ref{appropriata}, we can make use of the adapted coordinates for the curve $\gamma$, which are associated to the basis (see expression \eqref{chgt_base}) 
\begin{equation}\label{base_adatta}
\left\{v:=u_1+\sum_{i=2}^ma_{i1} u_i,\,u_2,\dots,u_m\right\}\in \scV^1(m,n)\ ,
\end{equation}
where, as we did in \ref{appropriata}, we indicate by  $\ta:=(a_{21},\dots,a_{1m})\in \R^{m-1}$ the vector containing the linear coefficients of the Taylor expansion of $\gamma$ at the origin. Hence, following the notations of section \ref{appropriata} (especially, formula \eqref{pol_trasf}), we write
$$
\tT_{0,\ta}(h|_{\Gamma^m},r,m)(\ty):= 	\tT_0(h|_{\Gamma^m},r,m)\circ  \cL^{-1}_\ta (\ty)\ .
$$
Then, by standard results of calculus, taking \eqref{forma} into account, one can write
\begin{equation}\label{orasiride}
\tT_{0,\ta}(h|_{\Gamma^m},r,m)(\ty)=\sum_{\substack{\mu \in \N^m\\1\le |\mu| \le r}}\frac{1}{\mu!} h_0^{|\mu|}\Big[\stackrel{\mu_1}{\overbrace{v}},\stackrel{\mu_2}{\overbrace{u_2}},\dots,\stackrel{\mu_m}{\overbrace{u_m}}\Big]\,\ty_1^{\mu_1}\dots \ty_m^{\mu_m}\ , 
\end{equation}
where $\mu!:=\mu_1!\dots, \mu_m!$. 

Since $\tT_0(h|_{\Gamma^m},r,m)\in \sigma^1(r,s_m,m)$, by \eqref{pong} and Theorem \ref{ideal}, one has 
\begin{equation}\label{si_annulla_tutto}
\tQ_{i\alpha}(\tT_{0,\ta}(h|_{\Gamma^m},r,m),\ta,\jetam)=0\qquad \forall\  i\in\{1,\dots,m\}\ ,\ \ \forall \alpha\in\{0,\dots,s\}\ ,
\end{equation}
where $\jetam:=\cL_\ta\circ \jetm$ denotes the $s_m$-truncation at the origin of the curve $\gamma$ expressed in the adapted variables. 

We now try to simplify the expressions in \eqref{si_annulla_tutto}, by taking formulas \eqref{vigna}-\eqref{pigna} in Theorem \ref{ideal} into account and by exploiting the form of the polynomial $\tT_{0,\ta}(h|_{\Gamma^m},r,m)$ in \eqref{orasiride}. Namely - due to \eqref{orasiride} and \eqref{mulan} - the coefficients $\tp_{\nu(1,\alpha)}$ and $\tp_{\nu(\ell,\alpha)}$, with $\ell\in\{2,\dots,m\}$, $\alpha\in \{0,\dots,s_m\}$, appearing in \eqref{vigna}-\eqref{pigna} read
\begin{align}\label{Rossini}
\begin{split}
\tp_{\nu(1,\alpha)}=\frac{1}{(\alpha+1)!} h_0^{\alpha+1}[v,\dots,v]\quad , \qquad  \tp_{\nu(\ell,\alpha)}=\frac{1}{\alpha!}h_0^{\alpha+1}\Big[\stackrel{\alpha}{\overbrace{v}},u_\ell\Big]\ .
\end{split}
\end{align} 
Moreover, if $s_m\ge 2$, for $\alpha\in\{2,\dots,s_m\}$, exploiting \eqref{orasiride} and the linearity, in our case the second addend at the right hand side at the third line of \eqref{vigna} reads
\begin{align}\label{Verdi}
\begin{split}
\sommalphaquater \beta\ \tp_{\nu(i,\beta)} a_{i(\alpha-(\beta-1))}=&	\sommalphaquater  \frac{\beta}{\beta!}h_0^{\beta+1}\Big[\stackrel{\beta}{\overbrace{v}},u_i\Big] a_{i(\alpha-(\beta-1))}\\
=&	\sum_{\beta=1}^{\alpha-1} \frac{1}{(\beta-1)!}h_0^{\beta+1}\Big[\stackrel{\beta}{\overbrace{v}},\sum_{i=2}^ma_{i(\alpha-(\beta-1))} u_i\Big]\ . \\
\end{split}
\end{align}
Henceforth, in order to simplify our formulas, we set
$$
{\bf V}_i:=
\begin{cases}
v \ , \text{ if } i=1\\
u_i \ , \text{ if } i\in\{2,\dots,m\}
\end{cases}\ .
$$
Considering again the case $s_m\ge 2$, for $\alpha\in\{2,\dots,s_m\}$, and for any $i=1,\dots,m$, by \eqref{orasiride} and by \eqref{giammai}-\eqref{nonvuoto}, the last addend at the right hand side of the third relation in \eqref{vigna}-\eqref{pigna} reads 
	\begin{align}\label{Puccini}
	\begin{split}
	&\sommalphateri \mu_i\ \tp_{\mu}
	\sum_{k\in \cG_m(\widetilde \mu(i),\alpha)} 	\left( \prod_{j=2}^m
	\binom{\mtj(i)}{k_{j2}\ ...\  k_{j\alpha}}\ 
	a_{j2}\strut^{k_{j2}}...a_{j\alpha}\strut^{k_{j\alpha}}\right)\\
	&=\sommalphateri \frac{\mu_i}{\mu!} h_0^{|\mu|}\Big[\stackrel{\widetilde \mu_1(i)}{\overbrace{v}},\stackrel{\widetilde \mu_2(i)}{\overbrace{u_2}},\dots,\stackrel{\widetilde \mu_m(i)}{\overbrace{u_m}},{\bf V}_i\Big]\\
	 &\qquad \qquad \qquad  \times \sum_{k\in \cG_m(\widetilde \mu(i),\alpha)} 	\left( \prod_{j=2}^m
	\binom{\mtj(i)}{k_{j2}... k_{j\alpha}}
	a_{j2}\strut^{k_{j2}}...a_{j\alpha}\strut^{k_{j\alpha}}\right)\\
	&=\sum_{\substack{\mu\in\cE_m(i,\alpha)\\ \mu\in\cM_m(\alpha)\\ \mu_i\neq 0}} \sum_{k\in\mathcal G_m(\widetilde \mu(i),\alpha) }\frac{h_0^{|\mu|}\Big[\stackrel{\widetilde \mu_1(i)}{\overbrace{v}},\stackrel{k_{22}}{\overbrace{a_{22} \ u_2}},\dots,\stackrel{k_{2\alpha}}{\overbrace{a_{2\alpha} \ u_2}},\stackrel{k_{m2}}{\overbrace{a_{m2} \ u_m}},\dots,\stackrel{k_{m\alpha}}{\overbrace{a_{m\alpha} \ u_m}},{\bf V}_i\Big]}{\widetilde \mu_1(i)!\  k!}\ ,
	\end{split}
	\end{align}
where the last passage is a consequence of the multi-linearity and of the fact that, for all $i\in\{1,\dots,m\},j\in\{2,\dots,m\}$, we have $\sum_{u=2}^{\alpha}k_{ju}=\mtj(i)$ by construction (see \eqref{giammai}).

Now, by \eqref{vigna}-\eqref{pigna}, it is trivial to observe that for $\alpha=0$ and for all $i\in\{1,\dots,m\}$, one has 
\begin{equation}\label{0}
\tQ_{i,0}(\tT_{0,\ta}(h|_{\Gamma^m},r,m),\ta,\jetam)=0 \quad \Longleftrightarrow \quad h_0^1[v]=h^1_0[u_2]=\dots h^1_0[u_m]=0
\end{equation}
which simply means that the basis vectors $\{v,u_2,\dots,u_\ell\}$ are orthogonal to $\nabla h(0)\neq 0$.

We observe that, in case $i=1,\dots,m$ and $\alpha=1$, taking \eqref{Rossini} and \eqref{vigna}-\eqref{pigna} into account, condition  \eqref{si_annulla_tutto} amounts to requiring that for all $i\in\{1,\dots,m\}$
\begin{equation}\label{alpha=2}
\tQ_{i,1}(\tT_{0,\ta}(h|_{\Gamma^m},r,m),\ta,\jetam)=0 \ \Longleftrightarrow \	h_0^2[v,v]=h^2_0[v,u_2]=\dots = h^2_0[v,u_m]=0 \ .
\end{equation}

For $i=1$ and $\alpha=2$, instead, we have 
\begin{equation}\label{3}
\tp_{\nu(1,\alpha)}=\frac{1}{3!}h^3[v,v,v]
\end{equation}
and for $\beta=1$ the term in \eqref{Verdi} is null due to \eqref{alpha=2}. 

Moreover, still for $i=1$ and $\alpha=2$ we observe that also the term in \eqref{Puccini} does not yield any contribution to condition \eqref{si_annulla_tutto}.  In order to see this, we start by remarking that the multi-indices to be taken into account in \eqref{Puccini} for $i=2$ and $\alpha=2$ must verify $\mu \in \cE_m(1,2)$, that is, by \eqref{giammai}-\eqref{nonvuoto}
\begin{equation}\label{condizioni}
k_{j2}=\widetilde \mu_j(1)\quad \forall\ j\in\{2,\dots,m\}\quad , \qquad \widetilde \mu_1(1)+\sum_{j=2}^m 2k_{j2}=2 \ .
\end{equation}
Conditions \eqref{condizioni} are only possible if 
\begin{enumerate}
	\item $\widetilde \mu_1(1)=2$ and $\widetilde \mu_j(1)=0$ for all $j\in\{2,\dots,m\}$, which implies $\mu=(3,0,\dots,0)$.  
	\item  $k_{j2}=\widetilde\mu_j(1)=\delta_{jp}$ for some $p\in\{2,\dots,m\}$ and $\widetilde \mu_1(1)=0$. This implies that $\mu_1=1$ and $\mu_j=\delta_{jp}$ for all $j\in\{2,\dots,m\}$, so that finally $|\mu|=2$. 
\end{enumerate}

The first case is incompatible with the condition $\mu\in \cM_m(2)$ required in \eqref{Puccini} (see \eqref{sid}), as $(3,0,\dots,0)\equiv \nu (1,2)$. The second case does not yield any contribution to \eqref{Puccini} because of  \eqref{alpha=2}. 

Hence one has 
\begin{equation}\label{1-2}
\tQ_{1,2}(\tT_{0,\ta}(h|_{\Gamma^m},r,m),\ta,\jetam)=0 \quad \Longleftrightarrow \quad h^3[v,v,v]=0\ .
\end{equation}

Finally, for $i=1$ and $\alpha\ge 3$, and for $i\in \{2,\dots,m\}$ and $\alpha\ge 2$, comparing expressions \eqref{Rossini}-\eqref{Verdi}-\eqref{Puccini} with the quantities \eqref{vigna}-\eqref{pigna} in Theorem \ref{ideal}, and taking the definition of the quantities 
$
\cH_{m,\ell,\alpha}^{h,0}(v,u_2,\dots,u_m,a(m))
$
in \eqref{polpette}-\eqref{polpettone} into account, one has that 
\begin{equation}\label{Inter}
\tQ_{i,\alpha}(\tT_{0,\ta}(h|_{\Gamma^m},r,m),\ta,\jetam)=0 \quad \Longleftrightarrow \quad  \cH_{m,\ell,\alpha}^{h,0}(v,u_2,\dots,u_m,a(m,s_m))=0 \ .
\end{equation}
Putting together \eqref{0}-\eqref{alpha=2}-\eqref{1-2}-\eqref{Inter} with \eqref{base_adatta}, we see that the polynomial $\tT_{0,\ta}(h|_{\Gamma^m},r,m)$ belongs to $\sigma^1(r,s_m,m)$ if and only if the system
\begin{align}\label{Roma}
\begin{split}
\begin{cases}
(u_1,\dots,u_m)\in \tU(m,n)\\[0.5em]
a(m):=(a_{21},\dots,a_{2s_m},\dots,a_{m1},\dots,a_{ms_m})\in \R^{(m-1)\times s_m}\\[0.5em]
v=u_1+\sum_{j=2}^m a_{j1}\,u_j\\[0.5em]
h^1_{0}[v]=h^1_{0}[u_2]=\dots= h^1_{0}[u_m]=0\\[0.5em]
\cH_{m,\ell,\alpha}^{h,0}(v,u_2,\dots,u_m,a(m,s_m))=0 \quad  i=1,\dots,m\ ,\ \   \alpha=1,\dots,s
\end{cases}
\end{split}
\end{align}
admits a solution. However, this is in contradiction with hypothesis \eqref{non_tocca} in the statement, therefore $\tT_0(h|_{\Gamma^m},r,m)\not\in \sigma^1(r,s_m,m)$. 

{\it Case 2.} We now assume that $\tT_0(h|_{\Gamma^m},r,m)\in \Sigma(r,s_m,m)\backslash \sigma(r,s_m,m)$. Up to changing the order of the vectors spanning $\Gamma^m$, by \eqref{unione} without any loss of generality we can suppose $\tT_0(h|_{\Gamma^m},r,m)\in \Sigma^1(r,s_m,m)\backslash \sigma(r,s_m,m)$. Then, there must exist a sequence of polynomials $\{P_k\in \sigma^1(r,s_m,m)\}_{k\in \N}$ approaching $\tT_0(h|_{\Gamma^m},r,m)$. To conclude the proof of Case 2, we need the following

\begin{lemma}\label{Luminy}
	 There exists a sequence $\{S_k\in \Pol\}_{k\in \N}$ converging to $\tT_0(h,r,n)$ in $\Pol$ and verifying $S_k|_{\Gamma^m}=P_k$ for any given $k\in \N$. 
 \end{lemma}
  
 \begin{proof}
 	 We indicate by $A_1,\dots,A_m\in \tU(m,n)$ an orthonormal basis of $\Gamma^m$, and we choose $n-m$ orthonormal supplementary vectors $A_{m+1},\dots,A_n$ to form a orthonormal basis of $\R^n$. In the coordinates $(x_1,\dots,x_n)$ associated to $A_1,\dots,A_n$, the restriction of any polynomial $Q\in \Pol$ to the subspace $\Gamma^m$ is obtained  by simply setting $x_{m+1}=\dots=x_n=0$ in the expression of $Q$. Conversely,  any polynomial $P(x_1,\dots,x_m)\in \Polm$ depending only on the first $m$ variables is the projection on $\Gamma^m$ of any polynomial $Q\in \Pol$ of the form $Q(x_1,\dots,x_n)=P(x_1,\dots,x_m)+q(x_1,\dots,x_n)$, where $q\in \Pol$ verifies $q(x_1,\dots,x_m,0)=0$. 
 
By the above discussion and with slight abuse of notation, one can define the polynomial
\begin{equation}\label{quacca}
	q_h=q_h(x_1,\dots,x_n):= \tT_0(h,r,n)(x_1,\dots,x_n)-\tT_0(h|_{\Gamma^m},r,m)(x_1,\dots,x_m)
\end{equation}
which verifies $q_h(x_1,\dots,x_m,0)=0$ by construction. 

Then, for $k\in \N$, we consider the polynomials
\begin{equation}\label{quacchera}
	S_k=S_k(x_1,\dots,x_n):= P_k(x_1,\dots,x_m)+q_h(x_1,\dots,x_n)\ ,
\end{equation}
where $\{P_k\}_{k\in \N}$ is the sequence approaching $\tT_0(h|_{\Gamma^m},r,m)$ introduced above. The sequence $\{S_k\}_{k\in \N}$ has the properties we seek. Infact, as $q_h(x_1,\dots,x_m,0)=0$, on the one hand $S_k$ verifies 
\begin{equation}
	S_k|_{\Gamma^m}=S_k(x_1,\dots,x_m,0)=P_k\qquad \forall \ k \in \N\ ;
\end{equation}
on the other hand, as $P_k\longrightarrow \tT_0(h|_{\Gamma^m},r,m)$ by hypothesis, by taking \eqref{quacca}-\eqref{quacchera} into account one has 
\begin{equation}
	S_k\longrightarrow \tT_0(h|_{\Gamma^m},r,m)+q_h= \tT_0(h,r,n)\ .
\end{equation}

\end{proof}

   Since by Lemma \ref{Luminy} one has  $S_k|_{\Gamma^m}=P_k$, and since $P_k\in \sigma^1(r,s_m,m)$ by construction, the same arguments developed at Case 1 yield that for any $k\in \N$ the system
\begin{align}\label{Napoli}
\begin{split}
\begin{cases}
(u_1,\dots,u_m)\in \tU(m,n)\\[0.5em]
a(m):=(a_{21},\dots,a_{2s_m},\dots,a_{m1},\dots,a_{ms_m})\in \R^{(m-1)\times s_m}\\[0.5em]
v=u_1+\sum_{j=2}^m a_{j1}\,u_j\\[0.5em]
(S_k)^1_{0}[v]=(S_k)^1_{0}[u_2]=\dots= (S_k)^1_{0}[u_m]=0\\[0.5em]
\cH_{m,\ell,\alpha}^{S_k,0}(v,u_2,\dots,u_m,a(m,s_m))=0 \quad  i=1,\dots,m\ ,\ \   \alpha=1,\dots,s
\end{cases}
\end{split}
\end{align}
must be verified. However, this fact and the fact that, by Lemma \ref{Luminy}, one also has $S_k\longrightarrow \tT_0(h,r,n)$, contradicts the hypotheses of Theorem B. Hence, we must have $\tT_0(h|_{\Gamma^m},r,m)\not\in \Sigma(r,s_m,m)\backslash \sigma(r,s_m,m) $. 

\medskip 

By the discussion at Cases 1-2 above, the assumptions of Theorem B imply that - for any $m$-dimensional subspace $\Gamma^m$, with $m\in\{2,\dots,n-1\}$ - the Taylor polynomial $\tT_0(h|_{\Gamma^m},r,m)$ lies outside of $ \Sigma(r,s_m,m)$. Hence $\tT_0(h,r,n)\in \Pol \backslash \cV(r,s_m,m,n)$. This, together with \eqref{Omega} and Theorem A, concludes the proof. 

\subsection{Proof of the Corollaries}
\subsubsection{Proof of Corollary B1}
We start by studying the one-dimensional affine subspaces orthogonal to $\grad h(I_0)\neq 0$. Hypothesis \eqref{3-jet} is equivalent to hypothesis \eqref{s_1+1 jet} in Theorem B with $r=3$, $s_1=2$, whence the thesis. 

On the other hand, for any fixed $m\in\{2,\dots,n-1\}$, since $\tG(m,n)$ and $\S^n$ are both compact, by hypothesis \eqref{3-jet} there exists $\tau_m>0$ such that for any $m$-dimensional affine subspace $I_0+\Gamma^m$ orthogonal to $\grad h(I_0)$, and for any vector $w\in \S^n\cap \Gamma^m$ we have 
\begin{equation}\label{alto}
h_{I_0}^1[w]=0\quad , \qquad 	|h_{I_0}^2[w,w]|+|h^3_{I_0}[w,w,w]|\ge \tau_m>0\ .
\end{equation}
By continuity, up to dividing the constant $\tau_m$ by a factor two, relation \eqref{alto} holds uniformly also for all polynomials in a small ball around $\tT_0(h,3,n)$. Hence,  
$h$ matches the hypotheses at point $ii)$ of Theorem B for $r=3$ and $s_m=2$, whence the thesis for affine subspaces of dimension higher or equal than two. 

\subsubsection{Proof of Corollary B2}

In the proof of Theorem A we have set $\Omegone:=\bigcup_{m=1}^{n-1} \cV(r,s_m,m,n)$ (see \eqref{Omega}). Furthermore, formula \eqref{V} ensures that for any given function $h$ of class $C^{2r-1}$ around $I_0$
\begin{equation}
\tT_{I_0}(h,r,n)\in \cV(r,s_m,m,n) \quad \stackrel{\text{by definition}}{\Longleftrightarrow} \quad \substack{\exists \Gamma^m \in \tG(m,n),\ \Gamma^m\perp \grad h(I_0) \text{ s.t.}\\ \tT_{I_0}(h|_{\Gamma^m},r,m)\in \Sigma(r,s_m,m):=\overline \sigma(r,s_m,m)}\ .
\end{equation}  
Since one is free to choose the order of the orthonormal vectors spanning $\Gamma^m$, without any loss of generality we can also write 
\begin{equation}
\tT_{I_0}(h,r,n)\in \cV(r,s_m,m,n) \quad \Longleftrightarrow \quad \substack{\exists \Gamma^m \in \tG(m,n),\ \Gamma^m\perp \grad h(I_0) \text{ s.t.}\\ \tT_{I_0}(h|_{\Gamma^m},r,m)\in \Sigma^1(r,s_m,m):=\overline \sigma^1(r,s_m,m)}\ .
\end{equation}  

In the proof of Theorem B, we have also seen that for any $m\in\{1,\dots,n-1\}$, and for any given subspace $\Gamma^m\in \tG(m,n)$ orthogonal to $\grad h(I_0)$, condition $\tT_{I_0}(h|_{\Gamma^m},r,m)\in \sigma^1(r,s_m,m)$ holds if and only if system \eqref{so un cavolo} (if $m=1$) or \eqref{machenneso} (if $m\ge 2$) admits a solution when $Q$ is set to be the Taylor polynomial at order $r$ of function $h$.

Hence we have that, for any fixed $m\in\{1,\dots,n-1\}$, condition
$
\tT_{I_0}(h,r,n)\in \cV(r,s_m,m,n)
$
is equivalent to asking that $ \tT_{I_0}(h,r,n)$ belongs to the closure in $\Pollo$ of the set of polynomials solving system \eqref{so un cavolo} (if $m=1$) or system \eqref{machenneso} (for $m\in\{2,\dots,n-1\}$). 

The thesis follows by the arguments above and by \eqref{Omega}. 

\subsubsection{Proof of Corollary B3}

By the proof of Theorem B, the hypotheses in $i)$ and $ii)$ amount to asking for the existence of a real-analytic curve 
\begin{equation}
\gamma(t):=
\begin{cases}
wt \quad &\text{ for } m=1\\
\left(t,\sum_{i=1}^{+\infty}a_{2i}t^i,\dots,\sum_{i=1}^{+\infty}a_{mi}t^i\right)\quad &\text{ for } m=\{2,\dots,n-1\}
\end{cases}
\end{equation}
whose image is contained in some $m$-dimensional subspace orthogonal to $\grad h(I_0)\neq 0$, and such that $(\pi_{\Gamma^m}\grad h)\circ  \gamma$ has a zero of infinite order at $\gamma(0)=I_0$. By analyticity, then, $(\pi_{\Gamma^m}\grad h)\circ \gamma$ is identically zero. This implies by Definition \ref{def steep} that $h$ is not steep. 

\section{Partition of the set of $s$-vanishing polynomials}\label{Theorem C prova}

The proof of Theorems C1-C2-C3 is quite long and requires intermediate results which will be presented in this section. Before stating them, in the following two paragraphs we introduce some definitions and notations.

\subsection{Initial setting}\label{initial}
Consider three integers $r,m\ge 2$, and $s\in\{2,\dots,r-1\}$.	In sections \ref{rs_vanishing_polynomials}-\ref{genericity}, we have indicated by $\Sigma(r,s,m)\subset \Polm$ the closure of the set $\sigma(r,s,m)$ of $s$-vanishing polynomials. In particular, by \eqref{ping}-\eqref{pong} one has
\begin{align}\label{sigma}
\begin{split}
&	\sigma(r,s,m)=\bigcup_{i=1}^m \sigma^i(r,s,m)\quad , \qquad \sigma^i(r,s,m):=\Pi_{\Polm} Z^i(r,s,m)\ ,\\
&\Sigma(r,s,m)=\bigcup_{i=1}^n\Sigma^i(r,s,m):=\bigcup_{i=1}^n\overline \sigma^i(r,s,m):= \bigcup_{i=1}^n\overline{ \Pi_{\Polm} Z^i(r,s,m)}\ ,
\end{split} 
\end{align}
where the sets $Z^i(r,s,m)\subset \Polm\times \centinai$, with $i\in\{1,\dots,m\}$, are defined in \eqref{zrsm}, and one has decomposition \eqref{pang}, namely
$$
\Polm\times \centina \supset Z(r,s,m)=:\bigcup_{i=1}^m Z^i(r,s,m)\ .
$$
The expression of the sets $Z^i(r,s,m)$,  $i\in\{1,\dots,m\}$, is given explicitly \footnote{Actually, in Theorem \ref{ideal}, only the expression of $Z^1(r,s,m)$ is explicit. However, as it was already pointed out in section \ref{rs_vanishing_polynomials}, the cases $i=2,\dots,m$ are trivial generalizations of the case $i=1$: in order to find the expression of $Z^i(r,s,m)$ for $i\neq 1$, one simply has to follow the same steps needed to find the expression for $Z^1(r,s,m)$, and to exchange the rôle of the first coordinate with the $i$-th one.} in Theorem \ref{ideal}.

%We now develop some arguments of section \ref{genericity}, in particular we delve into some aspects of the proof of Lemma \ref{Palermo}. 

In the previous sections, we have seen that, in order to check if a given polynomial $Q\in \Pollo$ is steep at the origin\footnote{It is clear that the arguments devoped in the sequel hold also if the considered point is not the origin.} on a fixed subspace $\Gamma^m$ orthogonal to $\grad Q(0)\neq 0$, one must check whether the restriction $P:=Q|_{\Gamma^m}\in \Polm$ belongs to the complementary of $\Sigma(r,s,m):=\overline{\sigma}(r,s,m)$. We now claim that it is not strictly necessary to consider the closure of the whole set $\sigma(r,s,m)$. Indeed, in practice, the curves on which the $s$-vanishing condition must be tested are minimal arcs with uniform characteristics, like the one defined in Theorem \ref{arco_minimale}.

Namely, by the arguments in the proof of Lemma \ref{Palermo} - for any given $Q\in \Pollo$ and for any fixed $\Gamma^m$ perpendicular to $\grad Q(0)\neq 0$, it is sufficient to check if there exists a threshold $\lambda_0>0$ such that, for any $0<\lambda\le \lambda_0$, there exists an interval $I_\lambda\subset [-\lambda,\lambda]$ of length $\lambda/\tK$ \footnote{More details about the threshold $\lambda_0$ are given in Lemma \ref{Palermo}, whereas the constant $\tK$ is the one introduced in Theorem \ref{arco_minimale}.}, on which - for any curve $\gamma\in \arco$ verifying the Bernstein's inequality \eqref{Bernie} - the composition $(\grad P)\circ\gamma:=(\grad Q|_{\Gamma^m})\circ \gamma$ has no zeros of order greater or equal than $s$. In particular, we are interested in testing the $s$-vanishing conditions on those analytic curves $\gamma$ over $I_\lambda$ that, for some $i\in\{1,\dots,m\}$, verify
\begin{equation}\label{Sanders}
\gamma(t)=
\left(
\begin{matrix}
x_1(t)=\sum_{k=1}^{+\infty} a_{1k}\, t^k\\
\dots \\
x_{i-1}(t)=\sum_{k=1}^{+\infty} a_{{(i-1)}k}\, t^k\\
x_i(t)=t\\
x_{i+1}(t)=\sum_{k=1}^{+\infty} a_{{(i+1)}k}\, t^k\\
\dots\\
x_m(t)=\sum_{k=1}^{+\infty} a_{mk}\, t^k\\
\end{matrix}
\right)
\quad , \qquad  	\max_{u\in I_\lambda}\,\max_{\substack{j\in\{1,\dots,m\}\\j\neq i}}|a_{jk}(u)|\le \frac{\tM(r,n,m,k)}{\lambda^{k-1}}\ .
\end{equation}

By Theorem \ref{arco_minimale}, the constant $\tM=\tM(r,n,m,k)$ in \eqref{Sanders} can be explicitly computed and it is uniform for all curves $\gamma\in \arco$. 

As we have shown in the proof of Lemma \ref{Palermo}, for any given $P\in\Polm$, the threshold $\lambda_0$ - if it exists - goes to zero with the distance of $P$  to the "bad" set $\sigma(r,s,m)$. Therefore, by formula \eqref{Sanders}, the Taylor coefficients of the curves $\gamma$ on which the $s$-vanishing condition must be tested may take any value, except for those of order one which, independently from the choice of $\lambda>0$, are always uniformly bounded by $\tM(r,n,m,1)$. 

Inspired by the above reasonings, with the notations in \eqref{Sanders}, we give the following

\begin{defn}\label{comunismo}
	
	For $i=1,\dots,m$, we introduce the sets
	\begin{align}\label{cappello}
	\begin{split}
	&	\arcoih :=\left\{\gamma \in\arcoi| \max_{\substack{j\in\{1,\dots,m\}\\j\neq i}}\{|a_{j1}(0)|\}\le \tM(r,n,m,1)\right\}\\
	\end{split}
	\end{align}
	and
	\begin{equation}
	\arcoh:=\bigcup_{i=1}^m \arcoih \ ,
	\end{equation}
	and we denote respectively  by $\centinaih\subset \centinai$ and by $\centinah\subset \centina$ their associated subsets of $s$-truncations. 
	
\end{defn}

\begin{rmk}\label{palla steep}
	By formula \eqref{cappello}, if one introduces the decomposition
	\begin{equation}\label{ordinedueh}
	\centinaih=\centinaunoih\times\centinadueih
	\end{equation}
	as in \eqref{ordinedue}, the space $\centinaunoih$ is compact and $\centinadueih\equiv \centinaduei$.
\end{rmk}

\begin{defn} 
	In section \ref{appropriata} (see formula \eqref{W}), we defined the set $W^1(r,m)$ as
	
	\begin{equation}\label{W1}
	W^1(r,m):=\mathscr{F}^1(\Polm\times \R^{m-1})\ ,
	\end{equation}
	where the function $\mathscr{F}^1$ was introduced in \eqref{effeuno}-\eqref{lunetta}. 
	Similarly, for any $j\in\{2,\dots,m\}$, taking Remark \ref{altri i} into account, we had set 
	\begin{align}\label{Wi}
	\begin{split}
	W^j(r,m):=\mathscr{F}^j(\Polm\times \R^{m-1})
	\ .
	\end{split}
	\end{align}
	
	Now, by \eqref{cappello} and by the above discussion, for any $i\in\{1,\dots,m\}$ it makes sense to define also
	\begin{align}\label{W'}
	\begin{split}
\widehat W^i(r,m):=\mathscr{F}^i(\Polm\times \overline B^{m-1}(0,\tM(r,n,m,1)))\ .
	%\widehat W(r,m):=&\bigcup_{i=1}^m \widehat W^i(r,m)\ .
	\end{split}
	\end{align}

\end{defn}

We remind that, due to Definition \ref{comunismo} and to Remark \ref{assafaddij}, for any given $i\in\{1,\dots,m\}$ there exists a polynomial bijection $\fU^i$ between $\Polm\times \centinaunoih$ and $\widehat W^i(r,m)$: one is free to work either in the standard coordinates $ (p_\mu,\ta)\in\R^M\times \R^{(m-1)s}$, with $M:=\dim \Polm$, or in the adapted coordinates $(\tp_\mu,\ta)\in  \widehat W^i(r,m)$.

By the arguments above, without any loss of generality, for any fixed $i\in\{1,\dots,m\}$ it is sufficient to consider the set of those polynomials $P\in \Polm$ verifying the $s$-vanishing condition on the $s$-jet $\jeth\in\centinaih$ of some curve $\gamma\in \arcoih$. Namely, following \eqref{zrsm} and \eqref{sigma}, for $i=1,\dots,m$ we introduce the semi-algebraic sets 
\begin{align}\label{sigma2}
\begin{split}
\widehat Z^i(r,s,m)
:= &\{  (P,\jeth)\in\Polm\times \centinaih\}|(P,\jeth) \text{ satisfies }\\ &\quad q_{\ell \alpha}^i\circ\Phi^i(P,\jeth)=0
\text{  for all } \ell\in\{1,...,m\}, \alpha\in\{0,...,s\}\}\\
\widehat Z(r,s,m):= & \bigcup_{i=1}^m \widehat Z^i(r,s,m)\\ 
\widehat \sigma^i(r,s,m):= & \Pi_{\Polm} \widehat Z^i(r,s,m)\quad , \qquad \widehat \sigma(r,s,m) =  \bigcup_{i=1}^m \widehat \sigma^i(r,s,m)\\	
\widehat \Sigma(r,s,m):= &\bigcup_{i=1}^n\widehat\Sigma^i(r,s,m):=\bigcup_{i=1}^n \text{closure}\left(\widehat \sigma^i(r,s,m)\right) 
\end{split}
\end{align}
and, of course, we have 
\begin{equation}\label{contenuto}
\widehat Z^i(r,s,m)\subset Z^i(r,s,m)\quad , \qquad 	\widehat \sigma^i(r,s,m)\subset \sigma^i(r,s,m)\quad \forall \  i\in \{1,\dots,m\}. 
\end{equation}

For further convenience, we also state the following simple
\begin{lemma}\label{ventana}
	$\widehat \sigma(r,s,m)$ has codimension $s+m=\text{codim }\sigma(r,s,m)$ in $\Polm$. 
\end{lemma}
\begin{proof}
	By the third line of \eqref{sigma2} and by the second inclusion in \eqref{contenuto}, it suffices to prove the statement for $\widehat \sigma^1(r,s,m)$. By the first line of \eqref{sigma2} and by Definition \ref{comunismo}, one has 
	\begin{align}\label{zcappello}
	\begin{split}
	\widehat Z^1(r,s,m)=Z^1(r,s,m)\bigcap \left\{(P,\jeth)\in \Polm \times \centinaxh \right\}\ .
	\end{split}
	\end{align}

	As it was shown in Corollary \ref{codimensione}, the Jacobian associated to the equalities determining $Z^1(r,s,m)$ has full rank $ms+m$. 
	Namely, by the discussion below expression \eqref{Jacobien}, in the adapted coordinates of section \ref{appropriata},  such a Jacobian has non-zero pivots corresponding to the derivatives w.r.t. the coefficients of the polynomial $\tP_\ta$ associated to multi-indices in the family \eqref{mula}. As we had shown in the proof of Corollary \ref{codimensione}, this fact and the Implicit Function Theorem imply that for any pair $(P,\jet)$ belonging to $Z^1(r,s,m)$, one can express $ms+m$ coefficients of $P$ as implicit functions of the other coefficients of $P$ and of the $(m-1)s$ parameters of $\jet$. This was the argument that led to estimate $\text{codim } \sigma^1(r,s,m)=s+m$ in Corollary \ref{codimensione}. The thesis follows by putting this argument together  with formulas \eqref{sigma2}-\eqref{zcappello} and with the fact that $\dim \centinax=\dim \centinaxh$.  \end{proof}

\subsection{Partition of $\Polm$ and $\widehat W^1(r,m)$}

Let $r,m\ge 2$ be two integers. In this paragraph, we introduce a partition of the spaces $\Polm$ and $\widehat W^1(r,m)$ which will turn out to be useful in the sequel. In order to do this, we first need to introduce a family of multi-indices.

\subsubsection{A family of multi-indices}

For $b,c\in\{2,\dots,m\} $, $b\le c$ \footnote{We have set $b\le c$ in \eqref{pi copto} only for convenience, in order not to have two indices $b,c$ corresponding to the same multi-index $\mu\in \N^m$. Infact, it is clear that if we eliminate this constraint we have $\varpi(b,c)=\varpi(c,b)$ for all $b,c\in\{2,\dots,m\}$.}, we set
\begin{equation}\label{pi copto}
\varpi(b,c ):=\mu=(\mu_1,\dots,\mu_m) \in \N^m \ | \ \mu_1=0\ ,\ \  \mu_j=\delta_{jb}+\delta_{j c} \quad \forall j\in\{2,\dots,m\}\ .
\end{equation}
Comparing \eqref{pi copto} with the sub-family $\nu(\ti,1)$ defined in \eqref{mula}, it is plain to check that one has the disjoint union
\begin{equation}\label{altri_indici}
\left(\bigcup_{\substack{b,c=2\\b\le c}}^m  \{\varpi(b,c)\}\right) \bigsqcup \left(\bigcup_{\ti=1}^m \{\nu(\ti, 1)\}\right)=\{\mu\in \N^m\ |\  |\mu|=2\}\ .
\end{equation}

Moreover, we have the following

\begin{lemma}\label{invarianza_indici}
	For any polynomial $P\in \Polm$, 	the coefficients $p_\mu$  associated to the multi-indices $\mu$ belonging to the family \eqref{pi copto} are invariant under the transformations of paragraph \ref{appropriata}. Namely, using the notations in \eqref{pol_trasf}, for any given $\ta\in\centinaunoh$ one has 
	$$
	\tp_{\varpi(b,c)}=p_{\varpi(b,c)} \qquad \text{ for all } b,c\in\{2,\dots,m\}\quad , \qquad  b\le c\ .
	$$
\end{lemma}

\begin{proof}
	We indicate by $A_1,\dots,A_m$ the orthonormal basis of $\R^m$ associated to the coordinates $x_1,\dots,x_m$ on which polynomials in $\Polm$ depend.
	
	For any $\ta:=(a_{21},\dots,a_{m1})\in \centinaunoh$, we also denote by
	\begin{equation}\label{old}  v_\ta:=A_1+a_{21}A_2+a_{31}A_3+...+a_{m1}A_m\ ,\ \ u_2:=A_2\ ,\ \ ...\ ,\ \ u_m:=A_m\ ,
	\end{equation}
	the basis associated to the adapted variables defined in Section \ref{appropriata} (see \eqref{chgt_base}), namely
	\begin{equation}\label{coordinate2}
	\ty_1:=x_1\quad ,\qquad  \ty_2=\ty_2(\ta):=x_2-a_{21}\,x_1\quad \dots \quad \ty_m=\ty_m(\ta):=x_m-a_{m1}\,x_1\ .
	\end{equation}
	By \eqref{pol_trasf}, \eqref{mula} and \eqref{altri_indici}, the quadratic terms of the transformed polynomial $\tP_\ta$ read
	\begin{align}\label{Carrara}
	\begin{split}
	\begin{cases}
	\tp_{\nu(1,1)}\ty_1^2=\tp_{\nu(1,1)}x_1^2\\
	\tp_{\nu(\ell,1)}\ty_1\ty_\ell=\tp_{\nu(\ell,1)}x_1(x_\ell-a_{\ell 1}x_1)\quad \ell=2,\dots,m\\
	\tp_{\varpi(j,\ell)}\ty_j\ty_\ell=\tp_{\varpi(j,\ell)}(x_j-a_{j 1}x_1)(x_\ell-a_{\ell 1}x_1)\quad j,\ell=2,\dots,m\ ,\ \  j\le \ell\ .	 	\end{cases}
	\end{split}
	\end{align}
	
	By expression \eqref{Carrara}, we infer that - in the original variables $x_1,\dots,x_m$ - for any $j,\ell\in\{2,\dots,m\}$, $j\le \ell$, the coefficient associated to the monomial $x_j x_\ell$ is $\tp_{\varpi(j,\ell)}$, that is
	$
	p_{\varpi(j,\ell)}= \tp_{\varpi(j,\ell)}\ .
	$

\end{proof}

\subsubsection{Partition}

For any $P\in \Polm$, we set 
\begin{equation}\label{bH}
\Hp: =\left(
\begin{matrix}
2p_{\varpi(2,2)} & p_{\varpi(2,3)} & \dots & p_{\varpi(2,m)} \\
p_{\varpi(2,3)} & 2p_{\varpi(3,3)} & \dots & p_{\varpi(3,m)} \\
\dots & \dots & \ddots & \dots \\
p_{\varpi(2,m)} & p_{\varpi(3,m)} & \dots & 2p_{\varpi(m,m)} 
\end{matrix}
\right)
\end{equation}
and
\begin{equation}\label{S1}
\sS_1^1(r,m):=\{P\in \Polm | \   \det\Hp\neq 0\}\ .
\end{equation}

\begin{rmk}
	Matrix $\Hp$ is invariant under the transformations of section \ref{appropriata}. Namely, by Lemma \ref{invarianza_indici}, we have 
	\begin{equation}
	\Hpa:=\left(
	\begin{matrix}
	2\tp_{\varpi(2,2)} & \tp_{\varpi(2,3)} & \dots & \tp_{\varpi(2,m)} \\
	\tp_{\varpi(2,3)} & 2\tp_{\varpi(3,3)} & \dots & \tp_{\varpi(3,m)} \\
	\dots & \dots & \ddots & \dots \\
	\tp_{\varpi(2,m)} & \tp_{\varpi(3,m)} & \dots & 2\tp_{\varpi(m,m)} 
	\end{matrix}
	\right)=\Hp\ .
	\end{equation}
\end{rmk}

We also define 
\begin{equation}\label{S2}
\sS_2^1(r,m):=\Polm\backslash \sS_1^1(r,m)=\{P\in \Polm | \   \det\Hp= 0\}\ ,
\end{equation}
so that we have the disjoint union
\begin{equation}
\Polm= \sS_1^1(r,m)\bigsqcup \sS_2^1(r,m) \ .
\end{equation}

We now consider the images of $\sS_1^1(r,m)$ and $\sS_2^1(r,m)$ through the transformation $\fU^1$ defined in Remark \ref{assafaddij}, namely
\begin{align}\label{Civitavecchia}
\begin{split}
\scS^1_1(r,m):=\fU^1\left(\sS_1^1(r,m)\times \centinaunoh\right)=&\left\{(\tP_\ta,\ta) \in \widehat W^1(r,m) | \det \Hpa\neq 0\right\}\\
\scS^1_2(r,m):=\fU^1\left(\sS_2^1(r,m)\times \centinaunoh\right)=&\left\{(\tP_\ta,\ta) \in \widehat W^1(r,m)\right. | \det \Hpa=0\Big\}\ .
\end{split}
\end{align}

By \eqref{Civitavecchia}, we have the partition
\begin{equation}
\widehat W^1(r,m)=  \scS^1_1(r,m) \bigsqcup \scS^1_2(r,m) \ .
\end{equation}

\begin{rmk}
	It is clear that the above partition can be implemented also in case one considers adapted variables $(\tP_\ta,\ta)\in \widehat W^i(r,m)$, with $i\in\{2,\dots,m\}$. By suitably modifying the family of indices in \eqref{pi copto}, as well as by introducing an adapted matrix $\Hpi$, it is possible to define sets $\sS_1^i,\sS_2^i$ whose disjoint union yields $\Polm$ and sets $\scS_1^i, \scS_2^i$ whose disjoint union yields $\widehat W^i(r,m)$. However, the underlying reasonings are not conceptually different from the ones we did above, therefore we omit them in order not to burden the exposition.   
\end{rmk}

\subsection{Two important results}

Consider three integers $r,m\ge 2$ and $s\in\{1,\dots,r-1\}$. The two results below are the cornerstones of the proof of Theorems C1-C2-C3.
\begin{thm}\label{chiusuraquattro}
	In case $r\ge 2, s=1$, for any $i\in\{1,\dots,m\}$ the semi-algebraic sets $\widehat \sigma^i(r,1,m)$, and $\widehat \sigma(r,1,m)$ are closed in $ \Polm$, that is, by formulas \eqref{sigma2},
	$$
	\widehat \sigma^i(r,1,m)=\widehat \Sigma^i(r,1,m) \quad  \forall \ i\in\{1,\dots,m\}\quad ,\qquad  \widehat \sigma(r,1,m)=\widehat \Sigma(r,1,m)\ .
	$$
	Moreover, for any $i\in\{1,\dots,m\}$, taking the definition of transformation $\Upsilon^i$ into account (see \eqref{Upsilon}),
	%	\begin{enumerate}
	%	\item for $r\ge 2, s=1$, and for any $i\in\{1,\dots,m\}$, the form of $\widehat \sigma^i(r,s,m)$ can be computed by the means of an algorithm which involves only linear operations. 
	%\item for $r\ge 3$, $s=2$, or $r\ge 4$, $s=3$, and for any $i\in\{1,\dots,m\}$ 
	the set
	$
	\Pi_{\widehat W^i(r,m)}\Upsilon^i(\widehat Z^i(r,1,m))
	$
	is closed in $\widehat W^i(r,m)$, and its form can be explicitly computed.
	
	\medskip 
	
	%\item  in the space $\Polm$ there exist two semi-algebraic sets 
	%\begin{equation}
	%	X^i_1(r,s,m)\subset \sS^i_1(r,m)\ ,\ \ X^i_2(r,s,m)\subset \Polm\backslash \sS^i_1(r,m)
	%\end{equation}
	%satisfying
	%\begin{enumerate}
	%	\item $X_1^i(r,s,m), X_2^i(r,s,m)$ are closed for the induced topology in $\sS^i_1(r,m)$ and in $\Polm\backslash \sS^i_1(r,m)$, respectively;
	%	\item $
	%	\widehat \sigma^i(r,s,m)= X_1^i(r,s,m)\bigsqcup X_2^i(r,s,m) \ ;
	%	$
	%	\item the form of $X_1^i(r,s,m)$ can be explicitly computed by the means of an algorithm involving only linear operations; 
	%	\item in $\Polm$ one has
	%	\begin{align}
	%	\begin{split}
	%	&\text{  codim } X_1^i(r,s,m) =  s+m\\
	%	&\text{  codim } X_2^i(r,s,m) =  s+2m-1\ .
	%	\end{split}
	%	\end{align} 
	%\end{enumerate}
	
	%	\end{enumerate}

\end{thm}
\begin{thm}\label{partition}
	For any given values of $r\ge 3$, $s\ge 2$, and $i\in\{1,\dots,m\}$ there exist two semi-algebraic subsets of $\Polm$
	\begin{align}
	\begin{split}
	&X_1^i(r,s,m)\subset \sS_1^i(r,m)\ ,\ \ X_2^i(r,s,m)\subset \sS_2^i(r,m)\ , \end{split}
	\end{align}
	and two semi-algebraic subsets of $\widehat W^i(r,m)$
	\begin{align}
	\begin{split}
	&Y_1^i(r,s,m)\subset \scS_1^i(r,m)\ ,\ \ Y_2^i(r,s,m)\subset \scS_2^i(r,m)\ ,\ 
	\end{split}
	\end{align} 
	verifying the following properties:
	\begin{enumerate}
		\item one has the partition
		$
		\widehat \sigma^i(r,s,m)= X_1^i(r,s,m)\bigsqcup X_2^i(r,s,m)\ ;
		$
		\item for $j\in\{1,2\}$, one has $$
		X_j^i(r,s,m)=\Pi_{\Polm}\left((\fU^i)^{-1}(Y_j^i(r,s,m))\right)\ ,
		$$ where $\fU^i$ was defined in Remarks \ref{assafaddij}-\ref{altri i};
		\item $Y_1^i(r,s,m)$ is closed in $\scS_1^i(r,m)$ for the topology induced by $\Polm$;
		\item $X_1^i(r,s,m)$ is closed in $\sS_1^i(r,m)$ for the topology induced by $\Polm$;	
		\item the form of $Y_1^i(r,s,m)$ can be explicitly computed starting from the expression of $\widehat Z^i(r,s,m)$ by the means of an algorithm involving only linear operations.

	\end{enumerate}
	
\end{thm}

The rest of this section is devoted to the proof of the above results.

We will only prove Theorems \ref{chiusuraquattro}-\ref{partition} in the case $i=1$, as the other cases are simple generalizations. 

\subsubsection{Proof of Theorem \ref{chiusuraquattro} and strategy of proof of Theorem \ref{partition}.}\label{A}	

We have seen in section \ref{rs_vanishing_polynomials} that the equations determining $Z^1(r,s,m)$, can be written in the adapted coordinates introduced in section \ref{appropriata}. Namely, by recalling the functions \begin{equation}\label{funzioni_utili}
\tQ_{\ti\alpha}(\tP_\ta,\ta,\jeta): W^1(r,m)\times \R^{(m-1)(s-1)}\longrightarrow \R\quad , \qquad \ti\in\{1,\dots,m\} \quad \alpha \in \{0,\dots ,s\}\ ,
\end{equation}
presented in \eqref{tQ}-\eqref{cuccurucu}, Theorem \ref{ideal} ensures that $Z^1(r,s,m)$ is the image through the inverse of the transformation $\Upsilon^1$ in \eqref{sole} of the zero set
\begin{equation}\label{utilino}
\bigcap_{\substack{\ti\in\{1,\dots,m\}\\ \alpha\in\{0,\dots,s\}}}\{(\tP_\ta,\ta,\jeta)\in W^1(r,m)\times\R^{(m-1)(s-1)}|\tQ_{\ti\alpha}(\tP_\ta,\ta,\jeta)=0\} \ ,
\end{equation}
and the explicit form of the quantities $\tQ_{i\alpha}$ is given in \eqref{vigna}-\eqref{pigna}.

Then, if we indicate by $\widehat \tQ_{\ti\alpha}$ the restriction of $\tQ_{\ti\alpha}$ to the subset $\widehat W^1(r,m)\times \centinadue\subset  W^1(r,m)\times \centinadue$, and if we denote by $\nucleo \left(\{\widehat \tQ_{\ti\alpha}\}_{\substack{\ti\in\{1,\dots,m\}\\ \alpha\in\{0,\dots,s\}}}\right)$ the zero set of the non-linear maps $\{\widehat \tQ_{\ti\alpha}\}_{\substack{\ti\in\{1,\dots,m\}\\ \alpha\in\{0,\dots,s\}}}$ , it is clear by the discussion at paragraph \ref{initial}, in particular by \eqref{zcappello}, that   
\begin{align}
\begin{split}
\widehat Z^1(r,s,m):=(\Upsilon^1)^{-1}\left(\nucleo \left(\{\widehat \tQ_{\ti\alpha}\}_{\substack{\ti\in\{1,\dots,m\}\\ \alpha\in\{0,\dots,s\}}}\right)\right)  \\
\end{split}
\end{align}
that is 
\begin{align}\label{utile}
\begin{split}
\Upsilon^1(\widehat Z^1(r,s,m))= \nucleo \left(\{\widehat \tQ_{\ti\alpha}\}_{\substack{\ti\in\{1,\dots,m\}\\ \alpha\in\{0,\dots,s\}}}\right)\ .
\end{split}
\end{align}

\begin{proof}{\it (Theorem \ref{chiusuraquattro})}
	When $s=1$, then in \eqref{vigna}-\eqref{pigna} one must consider $\alpha\in\{0,1\}$. We observe that  \begin{equation}\label{01}
	\forall \ \alpha\in\{0,1\}\ ,\ \    \forall \ti=1,\dots,m\qquad 	\widehat \tQ_{\ti\alpha}(\tP_\ta,\ta,\jetah)=0 \quad \Longleftrightarrow \quad \tp_{\nu(\ti,\alpha)}=0\ ,
	\end{equation}
	where the family of multi-indices $\nu(\ti,\alpha)$ was defined in \eqref{mula}. As we see, no parameters belonging to the space $\centinadueh$ appear in equation \eqref{01}, which determines a closed subset of $\widehat W^1(r,m)$. In other words, by \eqref{utile} with $s=1$, we have that the projection $\Pi_{\widehat W^1(r,m)}\left(\Upsilon^1(\widehat Z^1(r,1,m))\right)$
 is closed in $\widehat W^1(r,m)$ and its form is given in \eqref{01}. Moreover, as the function $\fU^1$ introduced in Remark \ref{assafaddij} is continuous by construction, the inverse image $(\fU^1)^{-1}\left(\Pi_{\widehat W^1(r,m)}\left(\Upsilon^1(\widehat Z^1(r,1,m))\right)\right)$ is closed in $\Pol\times \centinaunoh$. Finally, as all the elements of $\centinaunoh$ live in a compact ball (see \eqref{cappello} and Remark \eqref{palla steep}), using Lemma \ref{projection} with $E\equiv\Polm$ and $K\equiv \centinaunoh$ and taking formulas \eqref{sigma2}-\eqref{utile} into account one has that  
 \begin{equation}
 	\widehat \sigma^1(r,1,m)=\Pi_{\Polm}\left[(\fU^1)^{-1}\left(\Pi_{\widehat W^1(r,m)}\left(\Upsilon^1(\widehat Z^1(r,1,m))\right)\right)\right]
 \end{equation}
 is closed. This concludes the proof. 
\end{proof}

%\subsubsection{Case $r\ge 2, s=1$.}\label{r2s1}

Now, taking \eqref{utile}, the third line of \eqref{sigma2} and the proof of Theorem \ref{chiusuraquattro} into account, the key idea behind the proof of Theorem \ref{partition} consists in understanding under which conditions the set $ \Upsilon^1(\widehat Z^1(r,s,m))$ admits a closed projection onto $\widehat W^1(r,m)$.

In particular, we will see that if we only consider polynomials in the set $\sS^1_1(r,m)\subset \Polm$, it is possible to extrapolate linearly from the $ms+m$ equations in \eqref{utile} the $(m-1)s$ parameters of $\centinadue$ as explicit algebraic functions of the parameters of $\scS^1_1(r,m):=\fU^1(\sS^1_1(r,m)\times \centinaunoh)$, as the discussion at the following paragraph shows. 

We start by setting
\begin{align}\label{Ww'}
\begin{split}
\sA(r,s,m):= & \nucleo \left(\{\widehat \tQ_{\ti\alpha}\}_{\substack{\ti\in\{1,\dots,m\}\\ \alpha\in\{0,\dots,s\}}}\right) \bigcap \Upsilon^1(\sS_1^1(r,s,m)\times \centinaxh)\ ,
\end{split}
\end{align} 
and we remind that, by Remark \ref{assafaddij},
\begin{equation}\label{Cartesio}
\Upsilon^1(\sS_1^1(r,s,m)\times \centinaxh)=\fU^{1}\left(\sS_1^1(r,s,m)\times \centinaunoh\right)\times \centinadue=\scS^1_1(r,m)\times \centinadue \ .
\end{equation}

We now claim that $\sA(r,s,m)$ is a graph of the form

\begin{align}\label{Ettore Scola}
\begin{split}
\sA(r,s,m)= Y^1_1(r,s,m)\times \overline \theta(r,s,m)\ ,
\end{split}
\end{align} 
where 
$$
Y^1_1(r,s,m):=\Pi_{\widehat W^1(r,m)}\sA(r,s,m)
$$
is a closed subset of $\scS^1_1(r,m)$ for the induced topology determined by algebraic equations involving the coordinates of elements in $\scS^1_1(r,m)$, and the points of $\overline \theta(r,s,m)\subset  \centinadue$ are parametrized by $Y^1_1(r,s,m)$. In the following subparagraphs, we will prove this claim and show how it can be used in order to demonstrate Theorem \ref{partition}. 

\subsubsection{Explicit form of the equations (case $r\ge 3$, $s\ge 2$)} \label{preliminary}

The goal of this paragraph is to give a more explicit expression of $\nucleo \left(\{\widehat \tQ_{\ti\alpha}\}_{\substack{\ti\in\{1,\dots,m\}\\\alpha\in\{0,1\}}}\right)$, in case $r\ge 3$, $s\ge 2$.

\begin{rmk}\label{archeologia}
	We remind that equations \eqref{vigna}-\eqref{pigna} are recursive w.r.t. the parameters of the curve $\gamma$. Namely, for any given integer $\beta\in\{2,\dots,s\}$, the coefficients of order $\beta$ belonging to the space $\centinadueh $ - that is $a_{22}, \dots, a_{2\beta}, \dots, a_{m2}, \dots,a_{m\beta}$ - appear in equations \eqref{vigna}-\eqref{pigna} only for $\alpha\ge \beta$. 
	
\end{rmk}
%\medskip

%We now show that the parameters of the space $\centinax$ can be iteratively reduced from the equations $\tQ_{\ell\alpha}(\tP_\ta,\jeta)=0$, $\ell\in\{1,\dots,m\}$, $\alpha\in\{2,\dots,s\}$.

For any polynomial $P\in \Polm$ and any curve $\gamma\in \arcoxh$, we indicate by
$
\tP_\ta^{(>2)}
$
the associated polynomial $\tP_\ta$ written in the adapted coordinates for $\gamma$ (see paragraph \ref{appropriata}) deprived of its monomials of degree less or equal than two \footnote{One has $\tP_\ta^{(> 2)}\neq 0$ since we are considering the case of polynomials having degree $r\ge 3$.}. Also, if  $\jeth\in \centinaunoh$ is the $s$-truncation of $\gamma$, for any given $\alpha\in\{2,\dots,s\}$ we denote by
$$
\jetah^{(<\alpha)}(t):= \jetah -
\left(\begin{matrix}
0\\
\sum_{i=\alpha}^s a_{2i}t^i \\
\dots \\
\sum_{j=\alpha}^s a_{mj}t^j \\
\end{matrix}\right)
$$
its truncation at order $\alpha-1$ written in the adapted coordinates for $\gamma$.

\begin{rmk}\label{jet_tagliato}
	We observe that, for $\alpha=2$, $\jetah^{(<2)}(t)$ reduces to the line $(t,0,\dots, 0)$, since with the exception of the parametrizing coordinate, the components of $\jeta$ start at order two in $t$ (see paragraph \ref{appropriata}).
\end{rmk} 

With this setting, we have
\begin{lemma}\label{ciaone}
	For any polynomial $P\in \Polm$, there exists a linear change of coordinates $\fD=\fD(P):\R^m\longrightarrow\R^m$ such that for any given $\alpha\in\{2,\dots,s\}$, and for $\ti=1,\dots,m$, the algebraic equations $\widehat \tQ_{\ti\alpha}(\tP_\ta, \ta,\jeta)=0$ in \eqref{vigna}-\eqref{pigna} take the form
	\begin{equation}\label{struttura}
	\left(
	\begin{matrix}
	0 & 0 & 0 & \dots & 0  \\
	0 & 2\tp'_{\varpi(2,2)} & 0 & \dots & 0 \\
	0 & 0 & 2\tp'_{\varpi(3,3)} & \dots &0 \\
	0 & \dots & \dots & \dots & \dots \\
	0 & 0 & 0 & \dots & 2\tp'_{\varpi(m,m)} \\
	\end{matrix}
	\right) 
	\left(
	\begin{matrix}
	0\\
	a'_{2\alpha}\\
	a'_{3\alpha}\\
	\dots\\
	a'_{m\alpha}
	\end{matrix}
	\right)+\fD
	\left(
	\begin{matrix}
	\widehat \tQ_{1\alpha}\left( \tP_\ta^{(>2)}, \jeta^{(<\alpha)}\right)\\
	\widehat \tQ_{2\alpha}\left( \tP_\ta^{(>2)}, \jeta^{(<\alpha)}\right)\\
	\widehat \tQ_{3\alpha}\left( \tP_\ta^{(>2)}, \jeta^{(<\alpha)}\right)\\
	\dots \\
	\widehat \tQ_{m\alpha}\left( \tP_\ta^{(>2)}, \jeta^{(<\alpha)}\right)
	\end{matrix}
	\right)
	=0\ .
	\end{equation}
\end{lemma}

\begin{proof}
	
	{\it Step 1.} We firstly claim that equations \eqref{vigna}-\eqref{pigna} can be but into the form
	\begin{equation}\label{struttura2}
	\left(
	\begin{matrix}
	0 & 0 & 0 & \dots & 0  \\
	0 & 2\tp_{\varpi(2,2)} & \tp_{\varpi(2,3)} & \dots & \tp_{\varpi(2,m)} \\
	0 & \tp_{\varpi(2,3)} & 2\tp_{\varpi(3,3)} & \dots & \tp_{\varpi(3,m)} \\
	0 & \dots & \dots & \dots & \dots \\
	0 & \tp_{\varpi(2,m)} & \tp_{\varpi(3,m)} & \dots & 2\tp_{\varpi(m,m)} \\
	\end{matrix}
	\right) 
	\left(
	\begin{matrix}
	0\\
	a_{2\alpha}\\
	a_{3\alpha}\\
	\dots\\
	a_{m\alpha}
	\end{matrix}
	\right)+
	\left(
	\begin{matrix}
	\widehat \tQ_{1\alpha}\left( \tP_\ta^{(>2)}, \jeta^{(<\alpha)}\right)\\
	\widehat \tQ_{2\alpha}\left( \tP_\ta^{(>2)}, \jeta^{(<\alpha)}\right)\\
	\widehat \tQ_{3\alpha}\left( \tP_\ta^{(>2)}, \jeta^{(<\alpha)}\right)\\
	\dots \\
	\widehat \tQ_{m\alpha}\left( \tP_\ta^{(>2)}, \jeta^{(<\alpha)}\right)
	\end{matrix}
	\right)
	=0\ .
	\end{equation}
	
	By \eqref{01}, for any $\alpha\in\{2,\dots,s\}$, and for any $\ti,\tj\in\{1,\dots,m\}$, the monomials $\tp_{\nu(\tj,0)},\tp_{\nu(\tj,1)}$ do not yield any contribution to equations $\tQ_{\ti\alpha}(\tP_\ta, \ta,\jeta)=0$. Therefore, taking \eqref{altri_indici} into account, the only monomials of order two which may appear in equation $\tQ_{\ti\alpha}(\tP_\ta, \ta,\jeta)=0$ are those associated to the multi-indices $\varpi(b,c)$, with $b,c\in\{2,\dots,m\}$, $b\le c$.
	
	Moreover, fixing the values of $\ti\in\{1,\dots,m\}$ and $\alpha\in\{2,\dots,s\}$, by \eqref{vigna}-\eqref{pigna}, the multi-indices $\mu\in \N^m$ contributing to equation $\widehat \tQ_{\ti\alpha}(\tP_\ta, \ta,\jeta)=0$ are those for which the set $\cG_m(\widetilde \mu(\ti),\alpha)$ in \eqref{gma} is non-empty. This amounts to requiring that the components of the vector $(k_{22},\dots,k_{2\alpha},\dots, k_{m2},\dots,k_{m\alpha})\in \N^{(m-1)(\alpha-1)}$ appearing in formulas \eqref{vigna}-\eqref{pigna} satisfy
	\begin{equation}\label{conservazione}
	\sum_{i=2}^{\alpha}k_{ji}=\mtj(\ti) \quad  \forall j\in\{2,\dots,m\}\quad , \qquad  \widetilde{\mu}_1(\ti)+\sum_{j=2}^{m}\sum_{i=2}^{\alpha}i\ k_{ji}=\alpha\ .
	\end{equation}
	In particular, for fixed $\ti\in\{1,\dots,m\}$, $\alpha\in\{2,\dots,s\}$ and $\ell\in\{2,\dots,m\}$, if we look at the monomials containing the coefficient $a_{\ell\alpha}$ in equation $\widehat \tQ_{\ti\alpha}(\tP_\ta,\jeta)=0$ - that is, at the form of the terms for which $k_{\ell\alpha}\neq 0$ in \eqref{vigna}-\eqref{pigna} - by \eqref{conservazione} we must have
	\begin{align}\label{tarda notte}
	\begin{split}
	\widetilde \mu_1(\ti)=0\quad , \qquad  k_{ji}=\delta_{i\alpha}\delta_{j\ell} \quad , \qquad \widetilde\mu_j(\ti)=\delta_{j\ell}\quad ,  \qquad  j\in\{2,\dots,m\}\ .
	\end{split}
	\end{align}
	Firstly, expression \eqref{tarda notte} implies $|\mu|=2$. Secondly, as we said above, the only multi-indices of length $|\mu|=2$ which may yield a contribution to $\tQ_{\ti\alpha}(\tP_\ta, \ta,\jeta)=0$ are those belonging to the family $\{\varpi(b,c)\}_{\substack{b,c\in\{2,\dots,m\} b \le c}}$ in \eqref{pi copto}. Therefore, we have two cases.
	
	{\it Case $\ti=1$.} If $\mu\in\{\varpi(b,\ell)\}_{\substack{b,\ell\in\{2,\dots,m\} b \le \ell}}$ then $\mu_1=0$, and all terms in formula \eqref{vigna} which are associated to these indices are multiplied by zero. Hence, the coefficients $a_{\ell \alpha}$ do not appear in  $\widehat \tQ_{1\alpha}(\tP_\ta, \ta,\jeta)=0$, nor do any of the monomials of order two in $\tP_\ta$. This, together with Remark \ref{archeologia} proves the claim for $\ti=1$ (the first line of \eqref{struttura2}).
	
	{\it Case $\ti\in\{2,\dots,m\}$.} Taking \eqref{tarda notte} into account, for any given $\ell\in\{2,\dots,m\}$ one has that
	\begin{enumerate}
		\item if $\ti \le \ell$, the only term to which the coefficient $a_{\ell\alpha}$ is associated in equation $\widehat \tQ_{\ti\alpha}(\tP_\ta,\ta,\jeta)=0$ is the one corresponding to the multi-index $\mu\in \varpi(\ti,\ell)$, that is, by \eqref{pigna}, the monomial $(1+\delta_{\ell\ti})\tp_{\varpi(\ti,\ell)} a_{\ell\alpha}$;
		\item if $\ti>\ell$, the term containing $a_{\ell\alpha}$ is the one associated to the multi-index $\mu\in \varpi(\ell,\ti)$, that is, by \eqref{pigna},  $(1+\delta_{\ell\ti})\tp_{\varpi(\ell,\ti)} a_{\ell\alpha}$.
	\end{enumerate} 
	
	Conversely, if a monomial associated to an index $\varpi(\ti,\ell)$, with $\ti,\ell\in\{2,\dots,m\}$, $\ti \le \ell$, appears in equations $\widehat \tQ_{\ti\alpha}(\tP_\ta, \ta,\jeta)=0$, then, by \eqref{pi copto} and by \eqref{conservazione}, one must necessarily have
	\begin{align}
	\begin{cases}
	\widetilde{\mu}_1(\ti)+\sum_{j=2}^{m}\sum_{i=2}^{\alpha}i\ k_{ji}=\alpha \quad , \qquad 
	\widetilde{\mu}_1(\ti)= 0\\
	\sum_{i=2}^{\alpha}k_{ji}=\mtj(\ti):=\delta_{j\ell}+\delta_{j\ti}-\delta_{j\ti}=\delta_{j\ell} \quad \forall j\in\{2,\dots,m\}\ ,
	\end{cases}
	\end{align}
	which is true if and only if for some $v\in \{2,\dots,\alpha\}$ one has
	\begin{align}
	\begin{cases} k_{ji}=\delta_{j\ell}\delta_{iv}\\ \sum_{j=2}^{m}\sum_{i=2}^{\alpha}i\ \delta_{j\ell}\delta_{iv}=\alpha
	\end{cases}
	\end{align}
	that is if and only if $k_{ji}=\delta_{j\ell}\delta_{i\alpha}$ .
	One can check by formula \eqref{pigna} that this ensures that such a term must be of the form $(1+\delta_{\ell\ti})\tp_{\varpi(\ti,\ell)} a_{\ell\alpha}$. This reasoning, together with Remark \ref{archeologia}, and with the fact that - as we showed at the beginning of the proof - no monomials of order two appear in equation $\widehat \tQ_{\ti\alpha}(\tP_\ta, \ta,\jeta)=0$ other than those associated to the family \eqref{pi copto}, proves the claim for $\ti\in\{2,\dots,m\}$ (rows $2,\dots,m$ of \eqref{struttura2}).  
	
	{\it Step 2.}  For any $P\in \Polm$, taking Lemma \ref{invarianza_indici} into account we indicate by $\Gp$ the matrix
	\begin{equation}\label{gipi}
	\left(
	\begin{matrix}
	0 & 0 &  \dots & 0  \\
	0 & 2p_{\varpi(2,2)} & \dots & p_{\varpi(2,m)} \\
	0 & \dots &  \ddots & \dots \\
	0 & p_{\varpi(2,m)} & \dots & 2p_{\varpi(m,m)} \\
	\end{matrix}
	\right) = \left(
	\begin{matrix}
	0 & 0 &  \dots & 0  \\
	0 & 2\tp_{\varpi(2,2)} & \dots & \tp_{\varpi(2,m)} \\
	0 & \dots &  \ddots & \dots \\
	0 & \tp_{\varpi(2,m)} & \dots & 2\tp_{\varpi(m,m)} \\
	\end{matrix}
	\right)=
	\left(
	\begin{matrix}
	0&0\\
	0&\Hp
	\end{matrix}
	\right) 	
	\end{equation}
	appearing in \eqref{struttura2}. 
	$\Gp$ is symmetric, hence diagonalizable. Hence, for any $P\in \Polm$, there exist a basis of eigenvectors, indicated by
	\begin{equation}\label{new}
	v_\ta,\,u'_2=\tu'_2(P),\dots,\, u'_m=u'_m(P)\ ,
	\end{equation}
	and a real $m\times m$ invertible matrix $\fD=\fD(P)$, such that equation \eqref{struttura} takes the form
	\begin{equation}\label{diagonale}
	\left(
	\begin{matrix}
	0 & 0 & 0 & \dots & 0  \\
	0 & 2\tp'_{\varpi(2,2)} & 0 & \dots & 0 \\
	0 & 0 & 2\tp_{\varpi(3,3)}' & \dots & 0 \\
	0 & \dots & \dots & \dots & \dots \\
	0 & 0 & 0 & \dots & 2\tp'_{\varpi(m,m)} \\
	\end{matrix}
	\right) 
	\left(
	\begin{matrix}
	0\\
	a'_{2\alpha}\\
	a'_{3\alpha}\\
	\dots\\
	a'_{m\alpha}
	\end{matrix}
	\right)+\fD
	\left(
	\begin{matrix}
	\tQ_{1\alpha}\left( \tP_\ta^{(>2)}, \jeta^{(<\alpha)}\right)\\
	\tQ_{2\alpha}\left( \tP_\ta^{(>2)}, \jeta^{(<\alpha)}\right)\\
	\tQ_{3\alpha}\left( \tP_\ta^{(>2)}, \jeta^{(<\alpha)}\right)\\
	\dots \\
	\tQ_{m\alpha}\left( \tP_\ta^{(>2)}, \jeta^{(<\alpha)}\right)
	\end{matrix}
	\right)
	=0\ ,
	\end{equation}
	where the primed quantities indicate that we are working in the new basis \eqref{new}. 
	
	\begin{rmk}\label{immutato}
		Comparing \eqref{old} with \eqref{new}, we observe that the vector $v_\ta$ was left unchanged. This is due to the fact that, by \eqref{gipi}, $v_\ta$ is already an eigenvector of $\Gp$ (associated to a null eigenvalue). %However, whereas $\tu_2', \dots, \tu_m'$ can be taken to be pairwise orthonormal due to the spectral Theorem,  $v_\ta$ is not necessarily normalized. 
	\end{rmk}
	
\end{proof}

\subsubsection{Proof of Theorem \ref{partition}}

%\subsubsection{Definition of the sets $X_1^1(r,s,m), X_2^1(r,s,m)$ (for $r\ge 3, s\ge 2$).} \label{fretta}

For $r\ge 3$, $s\ge 2$, we define
\begin{equation}\label{X1}
X_1^1(r,s,m):=\widehat \sigma^1(r,s,m) \bigcap \sS_1^1(r,m)\ .
\end{equation}

{\it Step 1.} Firstly, we show that, if $P\in \sS_1^1(r,m)$, then the parameters of the space $\centinadueh$ can be reduced iteratively from equation \eqref{struttura} for $\alpha\in\{2,\dots,s\}$.

When $\alpha=2$ the second term at the l.h.s. of formula \eqref{struttura} in Lemma \ref{preliminary} does not depend on the parameters of $\centinadueh$ (see Remark \ref{jet_tagliato}), so that the coefficients $a_{22},\dots,a_{m2}$ can be reduced, as matrix $\Hp$ is invertible by construction when $P\in \sS_1^1(r,m)$.

If, for $\alpha\in\{3,\dots,s\}$, we assume that the parameters $a_{j\beta}$, with $j\in\{2,\dots,m\}$ and $\beta\in\{2,\dots,\alpha-1\}$, have been reduced, then the first equation in \eqref{struttura} does not contain any new parameter, whereas the terms $a_{2\alpha},\dots,a_{m\alpha}$ can be found by inverting $\Hp$ once again. 

The above considerations and \eqref{X1} imply that if $P\in X_1^1(r,s,m)$ then
\begin{enumerate}
	\item the parameters of the space $\centinadueh$ can be reduced from the equations in \eqref{struttura} by the means of a recursive algorithm which only involves linear computations and the inversion of $\Hp$;
	\item there exists a truncation $\jet\in\centinax$ such that $\fU^{1}(P\times \jet)=(\tP_\ta,\ta,\jeta)$ solves \eqref{struttura}. 
\end{enumerate}

Taking \eqref{Ww'} and the above arguments into account, with the notations of section \ref{A} we have that $\sA(r,s,m)$ is determined by a system of $ms+m-(m-1)(s-1)=s+2m-1$ explicit algebraic equations involving only the coordinates of $\scS_1^1(r,m)$, and of $(m-1)(s-1)$ explicit algebraic equations that parametrize the coefficients of $\centinadueh$ as functions of the points in $\scS_1^1(r,m)$. In other words,  $\sA(r,s,m)$ has the form of a graph of the type indicated in \eqref{Ettore Scola}, and $Y^1_1(r,s,m)$ is a closed algebraic subset of $\scS^1_1(r,m)$ whose form can be explicitly computed - starting from the expression of $\Upsilon^1(\widehat Z^1(r,s,m))$ - by the means of an algorithm involving only linear operations. 

Moreover, by Remark \ref{assafaddij} the function $\fU^1$ is polynomial, and $\sS_1^1(r,s,m)$ is obviously semi-algebraic in $\Polm$ by construction. Therefore, by \eqref{Ww'}-\eqref{Cartesio} the set $\sA(r,s,m)$ is semi-algebraic in $\widehat W^1(r,m)\times \centinadue$, and $Y^1_1(r,s,m)$ is semi-algebraic in $\widehat W^1(r,m)$ by the Theorem of Tarski and Seidenberg \ref{Tarski_Seidenberg}. 

{\it Step 2.} Since the invertible transformation $\fU^1$ defined in Remark \ref{assafaddij} is polynomial, due to Step 1 (see \eqref{Ww'}-\eqref{Cartesio}-\eqref{Ettore Scola}) and to continuity we have that the inverse image
\begin{equation}
(\fU^1)^{-1}(Y^1_1(r,s,m))\subset  \sS_1^1(r,m)\times \centinaunoh
\end{equation}
is closed in $\sS_1^1(r,m)\times \centinaunoh$ for the induced topology. Finally by taking \eqref{Ww'}-\eqref{Ettore Scola}-\eqref{X1} into account, and by considering the fact that, as we have already pointed out in Remark \ref{palla steep}, $\centinaunoh$ is compact - Lemma \ref{projection} ensures that the projection
\begin{equation}\label{proiezione}
\Pi_{\Polm}\left((\fU^1)^{-1}(Y^1_1(r,s,m))\right)=X^1_1(r,s,m)=\widehat \sigma^1(r,s,m) \bigcap \sS_1^1(r,m)
\end{equation}  is closed in $\sS_1^1(r,s,m)$ for the topology induced by $\Polm$.

The semi-algebraicness of the projection in \eqref{proiezione} is a consequence of the semi-algebraicness of $Y^1_1(r,s,m)$ demonstrated at Step 1, of the fact that $\fU^1$ is polynomial, and of the Theorem of Tarski and Seidenberg.

The above arguments prove Theorem \ref{partition} once one sets
\begin{equation}\label{Unità Italia}
\begin{split}
&	X_2^1(r,s,m):=\widehat \sigma^1(r,s,m)\backslash X_1^1(r,s,m)\\ &Y_2^1(r,s,m):=\Pi_{\widehat W^1(r,m)}\left(\nucleo \left(\{\widehat \tQ_{\ti\alpha}\}_{\substack{\ti\in\{1,\dots,m\}\\ \alpha\in\{0,\dots,s\}}}\right) \right)\backslash Y_1^1(r,s,m)\ .
\end{split}
\end{equation}

\section{Proof of Theorems C1-C2-C3}\label{Prova C1-C4}

We assume the notations of section \ref{Theorem C prova}. In this section, we consider two positive integers $r\ge 2$, $n\ge 3$, a vector $\ts:=(s_1,\dots,s_{n-1})\in \N^{n-1}$,  with $1\le s_i\le r-1$ for all $i=1,\dots,n-1$, and a function $h$ of class $C_b^{2r-1}$ around the origin, verifying $\grad  h(0) \neq  0$.

Also, for any $m\in\{2,\dots,n-1\}$, we set $\scK(r,n,m):=\tM(r,n,m,1)$, where the constant $\tM(r,n,m,1)$ is the one appearing in the statement of Theorem \ref{arco_minimale} (see formula \eqref{Bernie}). 
\subsection{Proof of Theorem C1}

Fix $m\in\{2,\dots,n-1\}$.  Let $\Gamma^m$ be a given $m$-dimensional subspace belonging to the subset $\sL_0(h,m,n)\subset \tG(m,n)$ introduced in Definition \ref{D}. 

Taking Definition \ref{cappello} into account, we consider a curve $\gamma\in \arcoh$ whose image is contained in $\Gamma^m$. Without any loss of generality, up to changing the order of the vectors spanning $\Gamma^m$, we can suppose that $\gamma$ is parametrized by the first coordinate, hence that  $\gamma\in \arcoxh$. Following \eqref{Sanders}, we indicate by $\ta=(a_{21},\dots,a_{m1})\in \overline B^{m-1}(0,\scK(r,n,m))$ the linear coefficients coefficients of the Taylor expansion of $\gamma$ at the origin, and by $\jet$ its $s$-truncation (with $1\le s \le r-1$).

We also indicate by $u_1,\dots,u_m\in \tU(m,n)$ a orthonormal basis spanning $\Gamma^m$, and by $v:=u_1+\sum_{i=2}^m a_{i1} u_i,u_2 \dots,u_m$ the basis associated to the adapted coordinates for $\gamma$ introduced in section \ref{appropriata}. As we have already shown in \eqref{orasiride}, the Taylor polynomial $\tT_0(h,r,n)$ restricted to $\Gamma^m$ written in the adapted coordinates reads
\begin{equation}\label{ma che riderone}
\tT_{0,\ta}(h|_{\Gamma^m},r,m)(\ty)=\sum_{\substack{\mu \in \N^m\\1\le |\mu| \le r}}\frac{1}{\mu!} h_0^{|\mu|}\Big[\stackrel{\mu_1}{\overbrace{v}},\stackrel{\mu_2}{\overbrace{u_2}},\dots,\stackrel{\mu_m}{\overbrace{u_m}}\Big]\,\ty_1^{\mu_1}\dots \ty_m^{\mu_m}\ ,
\end{equation}
where we have used the notation introduced in formula \eqref{forma}. Moreover, as customary, $\jet$ reads $\jeta$ in the new coordinates.

By the arguments at paragraph \ref{initial}, and by taking \eqref{Upsilon} and \eqref{sigma2} into account, if we manage to prove that condition 
$$
(\tT_{0,\ta}(h|_{\Gamma^m},r,m), \ta,\jetunoa) \in \Upsilon^1(\widehat Z^1(r,1,m))
$$ 
is never satisfied for any choice of the subspace $\Gamma^m\in \sL_0(h,m,n)$ and of the curve $\gamma$, which is equivalent - due to Theorem \ref{chiusuraquattro} for $s=1$ - to condition
%\footnote{$\jetunoa$ is trivially $(t,0,\dots,0)$, as we had already observed in Remark \ref{jet_tagliato}, so that equation $
%	(\tT_{0,\ta}(h|_{\Gamma^m},r,n), \ta,\jetunoa) \in \Upsilon^1(\widehat \Sigma^1(r,1,m))
%	$ involves only the coefficients of $\tT_{0,\ta}(h|_{\Gamma^m},r,n)$ and $\ta$. Hence, the pull-back of the above equation to the original variables reads simply $(\tT_0(h|_{\Gamma^m},r,m),\ta)\not\in \widehat \Sigma^1(r,1,m)$}
\begin{equation}\label{none}
\tT_0(h|_{\Gamma^m},r,m)\not\in \widehat \sigma^1(r,1,m)=\widehat \Sigma^1(r,1,m)\quad , \qquad \forall\  \Gamma^m\in \sL_0(h,m,n)\ ,
\end{equation}
then by the definitions in \eqref{U}-\eqref{V} we have 
\begin{align}\label{sine}
\begin{split}
	\tT_0(h,r,n)\in\qquad &\\
	 \Pollo\backslash \Pi_{\Pollo}\{&(Q,A,P)\in \Pollo\times O(n,m)\times \Polm|\\ &\text{Span }(A)\in \sL_0\ ,\ \  P(x)=Q(Ax)\ ,\ \  P(x)\in \Sigma(r,1,m)\}\ .
	\end{split}
\end{align}
Hence, if claim \eqref{none} is true, taking \eqref{sine}  into account one retraces exactly the same steps used to demonstrate Theorem A, (see section \ref{genericity}), and the thesis follows. As it has already been pointed out in subparagraph \ref{si rompe la steepness}, one just needs to be careful about the fact that the steepness coefficients, in this case, do not admit a uniform lower bound, as they may tend to zero when a subspace not belonging to $\sL_0(h,m,n)$ and on which $h$ is not steep is approached. Moreover, one cannot expect to have steepness on the subspaces of $\sL_0(h,m,n)$ for an open set of functions around $h$. 

For any choice of $\Gamma^m\in \sL_0(h,m,n)$, by Theorem \ref{ideal} and by \eqref{cuccurucu}, condition $$
(\tT_{0,\ta}(h|_{\Gamma^m},r,m), \ta,\jetunoa)\not \in \Upsilon^1(\widehat Z^1(r,1,m))
$$ is equivalent to asking that system 
\begin{equation}\label{Chaikovskij}
\tQ_{i\alpha}(\tT_{0,\ta}(h|_{\Gamma^m},r,m),\ta,\jetunoa)=0\qquad \forall\  i\in\{1,\dots,m\}\ ,\ \ \forall \alpha\in\{0,1\}\ ,
\end{equation}
has no solution.
Then, by expressions \eqref{0}-\eqref{alpha=2}, the system in \eqref{Chaikovskij} is not satisfied if and only if
\begin{equation}\label{Handel}
\begin{cases}
(u_1,\dots,u_m) \in \tU(m,n)\quad , \qquad  
\text{ Span }(u_1,\dots,u_m) =\Gamma^m\\  (a_{21},\dots,a_{m1})\in \overline B^{m-1}(0,\scK(r,n,m))\quad , \qquad v:=u_1+a_{21}u_2+\dots+a_{m1}u_m\\
h_0^1[v]=h_0^1[u_2]=\dots=h^1_0[u_m]\\
h^2_0[v,v]=h^2_0[v,u_2]=\dots=h^2_0[v,u_m]=0
\end{cases}
\end{equation}
has no solution. 

By construction, the Hessian of the restriction $h|_{\Gamma^m}$ has no null eigenvalues, so that system \eqref{Handel} admits no solution, and both claim \eqref{none} and expression \eqref{sine} hold true, as wished. This concludes the proof.

\subsection{Proof of Theorem C2}

\subsubsection{Case of a subspace belonging to $\sL_1(h,m,n)$}\label{Lambda1}
With the usual setting, let $m\in\{2,\dots,n-1\}$ be an integer, and $\Gamma^m\in \sL_1(h,m,n)\subset \tG(m,n)$ be a $m$-dimensional subspace spanned by vectors $u_1,\dots,u_m\in \tU(m,n)$. 

As the Hessian matrix of the restriction $h|_{\Gamma^m}$ has at most one null eigenvalue, without any loss of generality one can choose $u_1$ to be the eigenvector associated to the unique null eigenvalue, that is
\begin{equation}\label{Rubens}
\begin{cases}
h^1_0[u_1]=h^1_0[u_2]=\dots=h_0^1[u_m]=0\\
h_0^2[u_1,u_1]=h^2_0[u_1,u_2]= h^2_0[u_1,u_m]=0\\
\text{Span}(u_1,u_2,\dots,u_m)= \Gamma^m
\end{cases}
\end{equation}
so that at the same time one must have 
\begin{equation}\label{Van Dyck}
\det \left(
\begin{matrix}
h^2_0[u_2,u_2]& h^2_0[u_2,u_3] & \dots & h^2_0[u_2,u_m]\\
h^2_0[u_3,u_2]& h^2_0[u_3,u_3] & \dots & h^2_0[u_3,u_m]\\
\dots\\
h^2_0[u_m,u_2]& h^2_0[u_m,u_3] & \dots & h^2_0[u_m,u_m]\\
\end{matrix}
\right)\neq 0\ .
\end{equation}

The expression of $\tT_{0}(h|_{\Gamma^m},r,m)$ w.r.t. the coordinates $x_1,\dots,x_m$ associated to the vectors $u_1,u_2,\dots,u_m$ reads
\begin{equation}\label{hohoho}
\tT_{0}(h|_{\Gamma^m},r,m)(x)=\sum_{\substack{\mu \in \N^m\\1\le |\mu| \le r}}\frac{1}{\mu!} h_0^{|\mu|}\Big[\stackrel{\mu_1}{\overbrace{u_1}},\stackrel{\mu_2}{\overbrace{u_2}},\dots,\stackrel{\mu_m}{\overbrace{u_m}}\Big]\,x_1^{\mu_1}\dots x_m^{\mu_m}\ .
\end{equation}

We now claim that 

\begin{lemma}\label{Odilon Redon}
	If $h$ is non-steep at the origin on $\Gamma^m$ at some given order $s_m\ge 2$, then
	\begin{equation}
	\tT_0(h|_{\Gamma^m},r,m)\in X_1^1(r,s_m,m):= \widehat \sigma^1(r,s_m,m) \bigcap \sS_1^1(r,m)
	\end{equation}
	and $\tT_0(h|_{\Gamma^m},r,m)$ satisfies the $s_m$-vanishing condition on $\Gamma^m$ on some curve $\gamma\in \arcoxh$ whose Taylor expansion at the origin has null linear terms. 
\end{lemma}
\begin{proof}
	Looking at \eqref{hohoho}, it is easy to check that the coefficients of $\tT_{0}(h|_{\Gamma^m},r,m)$ associated to the family of indices
	$
	\{\varpi(b,c)\}_{\substack{b,c\in\{2,\dots,m\}, b\le c}}
	$ 
	introduced in \eqref{pi copto} read
	\begin{equation}\label{Arnolfo}
	p_{\varpi(b,c)}= \left(\tT_{0}(h|_{\Gamma^m},r,m)\right)_{\varpi(b,c)}=\frac{1}{1+\delta_{bc}} h^2_0[u_b,u_c]\ ,
	\end{equation}
	where $\delta_{bc}$ is the Kronecker delta. 
	
	By putting together expressions \eqref{Van Dyck} - \eqref{Arnolfo} with the definition of set $\sS_1^1(r,m)$ in \eqref{S1}, one has that
	\begin{equation}\label{Durer}
	\tT_{0}(h|_{\Gamma^m},r,m)\in \sS_1^1(r,m)\ .
	\end{equation} 
	
	Since we have assumed $h$ is non-steep at the origin on $\Gamma^m$ at some order $s_m\ge 2$ then, by the discussions at section \ref{genericity} and at paragraph \ref{initial}, one has 
	\begin{equation}
	\tT_0(h|_{\Gamma^m},r,m)\in \widehat \Sigma(r,s_m,m)
	\end{equation}
	which, together with \eqref{Durer}, yields
	\begin{equation}\label{Murillo}
	\tT_0(h|_{\Gamma^m},r,m)\in \widehat \Sigma(r,s_m,m)\bigcap \sS_1^1(r,m)\ .
	\end{equation}
	
	We now claim that
	
	\medskip 
	
	\begin{lemma}\label{Leonardo}
		
		\begin{equation}
		\tT_0(h|_{\Gamma^m},r,m)\not \in \widehat \Sigma^i(r,s_m,m) \qquad \forall\  i\in\{2,\dots,m\}\ .
		\end{equation}
		
	\end{lemma}
	\begin{proof}
		Suppose, by absurd, that $\tT_0(h|_{\Gamma^m},r,m) \in \widehat \Sigma^i(r,s_m,m)$ for some $i\in\{2,\dots,m\}$. It is clear from Theorem \ref{ideal} and \eqref{sigma2} that, as $s_m\ge 2$, one has $\widehat\sigma^i(r,s_m,m)\subset \widehat\sigma^i(r,1,m)$, and therefore $\widehat\Sigma^i(r,s_m,m)\subset \widehat\Sigma^i(r,1,m)$. This fact and the initial hypothesis imply $\tT_0(h|_{\Gamma^m},r,m) \in \widehat \Sigma^i(r,1,m)$, so that by Theorem \ref{chiusuraquattro}, one must have 
		\begin{equation}\label{Raffaello}
		\tT_0(h|_{\Gamma^m},r,m) \in \widehat \sigma^i(r,1,m)\ .
		\end{equation}
		Relation \eqref{Raffaello} implies that there must exist a curve $\gamma\in\arcoih$ with values in $\Gamma^m$ such that $\tT_0(h|_{\Gamma^m},r,m)$ satisfies the $1$-vanishing condition on $\gamma$. As it was shown in the proof of Theorem C1 (see the discussion around formula \eqref{Chaikovskij}) this is equivalent to asking that system
		\begin{equation}\label{Giorgione}
		\begin{cases}
		h^1_0[u_1]=h^1_0[u_2]=\dots=h^1_0[u_i]=\dots=h^1_0[u_m]=0\\
		h^2_0[u_i,u_1]=h^1_0[u_i,u_2]=\dots=h^1_0[u_i,u_i]=\dots=h^1_0[u_i,u_m]=0
		\end{cases}\qquad i\neq 1
		\end{equation}
		admits a solution, which contradicts \eqref{Van Dyck}. 
	\end{proof}

\medskip 
	
	Due to \eqref{Murillo} and to Lemma \ref{Leonardo}, we then have that 
	\begin{equation}\label{Goya}
	\tT_0(h|_{\Gamma^m},r,m)\in \widehat \Sigma^1(r,s_m,m)\bigcap \sS_1^1(r,m)\ .
	\end{equation}
	
	Moreover, by Theorem \ref{partition}, the set $	X_1^1(r,s,m):=\widehat \sigma^1(r,s,m) \bigcap \sS_1^1(r,m)$ defined in \eqref{X1} is closed in $\sS_1^1(r,m)$ for the topology induced by $\Polm$, whence one deduces that actually
	\begin{equation}\label{Salvator Rosa}
	X_1^1(r,s_m,m)=\text{closure}\left(\widehat \sigma^1(r,s_m,m)\right) \bigcap \sS_1^1(r,m)=\widehat \Sigma^1(r,s_m,m) \bigcap \sS_1^1(r,m)\ .
	\end{equation}
	Relations \eqref{Goya} and \eqref{Salvator Rosa} together imply 
	
	\begin{equation}\label{De Ribera}
	\tT_0(h|_{\Gamma^m},r,m)\in X_1^1(r,s_m,m)\ .
	\end{equation}
	
	Therefore, by \eqref{De Ribera} and by the definition of $X^1_1(r,s_m,m)$ in \eqref{X1}, there exists a curve $\gamma\in \arcoxh$, with image in $\Gamma^m$, on which the Taylor polynomial of the restriction $\tT_0(h|_{\Gamma^m},r,m)$ satisfies the $s_m$-vanishing condition. Namely, if $\ta=(a_{21},\dots,a_{m1})\in \overline B^{m-1}(0,\scK(r,n,m))$ indicates the linear coefficients of $\gamma$ and $\jetamh$ its $s_m$-truncation written in the adapted coordinates, by \eqref{sigma2} one must have 
	\begin{equation}\label{Brueghel}
	(\tT_{0,\ta}(h|_{\Gamma^m},r,m),\ta,\jetamh)\in \Upsilon^1(\widehat Z^1(r,s_m,m))\ ,
	\end{equation}
	that is, by \eqref{cuccurucu},
	\begin{equation}\label{Rembrandt}
	\widehat \tQ_{\ell \alpha}(\tT_{0,\ta}(h|_{\Gamma^m},r,m),\ta,\jetamh)=0\ ,\ \  \ell=1,\dots,m\ ,\ \  \alpha=0,\dots,s_m\ .
	\end{equation}
	
	In particular, as $s_m\ge 2 $, due to Theorem \ref{ideal} and to \eqref{Rossini}, the equations in \eqref{Rembrandt} for $\alpha=0,1$ read
	\begin{equation}\label{Van Gogh}
	\begin{cases}
	h^1_0[v]=h^1_0[u_1]=\dots=h_0^1[u_m]=0\\
	h_0^2[v,v]=h^2_0[v,u_2]= h^2_0[v,u_m]=0\\
	\text{Span}(v,u_2,\dots,u_m)= \Gamma^m
	\end{cases}
	\end{equation}
	where 
	\begin{equation}\label{Rosso Fiorentino}
	v=u_1+\sum_{i=2}^m a_{i1}u_i
	\end{equation}
	is the anisotropic vector associated to the curve $\gamma$. Comparing \eqref{Rubens} to \eqref{Van Gogh}, as the Hessian of $h|_{\Gamma^m}$ has only one null eigenvalue we see that the vectors $u_1$ and $v$ must be parallel, but then \eqref{Rosso Fiorentino} yields
	\begin{equation}\label{Louvre}
	v=u_1\quad , \qquad 	a_{21}=\dots=a_{m1}=0\ ,
	\end{equation}
	so that by the arguments of subsection \eqref{appropriata} the coordinates adapted to the curve $\gamma$ coincide with the original coordinates.

\end{proof}

We now recall that, due to Theorem \ref{partition}, the form of the set
\begin{equation}
Y^1_1(r,s_m,m):=\Pi_{\widehat W^1(r,m)}\left(\nucleo \left(\{\widehat \tQ_{\ti\alpha}\}_{\substack{\ti\in\{1,\dots,m\}\\ \alpha\in\{0,\dots,s_m\}}}\right) \bigcap \Upsilon^1(\sS_1^1(r,m)\times \centinaxh)\right)
\end{equation} satisfying\footnote{The transformation $\fU^1$ was introduced in Remark \ref{assafaddij}.}
\begin{equation}\label{El Greco}
X^1_1(r,s_m,m)=\Pi_{\Polm} \left((\fU^1)^{-1}Y^1_1(r,s_m,m)\right)
\end{equation}
can be explicitly computed - starting from the relations determining $\Upsilon^1(\widehat Z^1(r,s_m,m))$ - by the means of an algorithm involving only linear operations.
By this fact, the form of the set
\begin{equation}\label{Perozzi}
\mathscr{Y}^1_1(r,s_m,m):=Y^1_1(r,s_m,m)\bigcap \{(\tp_\mu,\ta)\in \widehat W^1(r,m)| \ta=0\}
\end{equation}
can also be deduced explicitly starting from $\Upsilon^1(\widehat Z^1(r,s_m,m))$. Then, due to Lemma \ref{Odilon Redon} and to \eqref{El Greco}, one has that 
\begin{equation}\label{Fidia}
\myoverset{\text{$h$ non-steep at $0$}}{\text{at order $s_m$ on $\Gamma^m$}} \ \Longrightarrow \	\fU^1(\tT_0(h|_{\Gamma^m},r,m),0)=(\tT_0(h|_{\Gamma^m},r,m),0)\in \mathscr{Y}^1_1(r,s_m,m)\ .
\end{equation} 

Moreover, we observe the following facts:
\begin{enumerate}
	
	\item in section \ref{Prova Th B}, the explicit expression of set $\cZ^{r,s_m,m}_n\subset \Pollo\times \R^{(m-1)s_m}\times \scV^1(m,n)$ introduced in Corollary B2 is obtained by injecting into the explicit expression for  
	\begin{equation}\label{citta}
	\Upsilon^1(\widehat Z^1(r,s_m,m)):=\nucleo \left\{\tQ_{i,\alpha}(\tT_{0,\ta}(\tP_\ta,r,m),\ta,\jetam)\right\}_{\substack{i=1,\dots,m\\\alpha=0,\dots,s_m}}
	\end{equation}   
	given in Theorem \ref{ideal} the explicit form of the coefficients of $\tT_{0,\ta}(Q|_{\Gamma^m},r,m)$ in \eqref{chickenuno}-\eqref{orasiride} (which we have written again in\eqref{ma che riderone}), with $Q$ any polynomial belonging to $\Pollo$, and by considering the vectors $v,u_2,\dots,u_m$ in \eqref{ma che riderone} as free parameters of $\scV^1(m,n)$. 
	
	\item In the same way, for any $Q\in\Pollo$, one can inject in the expressions determining $\mathscr{Y}^1_1(r,s_m,m)$ the explicit form of the coefficients of $\tT_{0,\ta}(Q|_{\Gamma^m},r,m) $, given in \eqref{ma che riderone}, with the vectors $v,u_2,\dots,u_m$ considered as free parameters of $\scV^1(m,n)$. The resulting subset is indicated by 
	$$
	\scA_1(r,s_m,n,m)\subset \Pollo\times\R^{m-1}\times  \scV^1(m,n)\ .
	$$
	
\end{enumerate}

By the arguments above, and by the fact that the form of $\mathscr{Y}^1_1(r,s_m,m)$ is obtained starting from the expression of $\Upsilon^1(\widehat Z^1(r,s_m,m))$ through linear algorithms, we have that the explicit expression of
$
\scA_1(r,s_m,n,m)
$  
can be found starting from the expressions determining $\cZ^{r,s_m,m}_n$ by only performing linear operations. 

Furthermore, the above arguments together with \eqref{Fidia} and with formula \eqref{Louvre} yield that if system
\begin{align}\label{Melandri}
\begin{split}
&\begin{cases}
%(a_{21},\dots,a_{m1})\in \overline B^{m-1}(\scK)\\
(u_1,\dots,u_m)\in \tU(m,n)\quad , \qquad  % \qquad v:=u_1+\sum_{i=2}^m a_{i1} u_i
\text{ Span }(u_1,u_2,\dots,u_m)=\Gamma^m\in \sL_1(h,m,n)\\
(\tT_{0}(h,r,n),0, u_1,u_2,\dots,u_m)\in \scA_1(r,s_m,n,m)
\end{cases}
\end{split}
\end{align}
has no solution, then $h$ is steep at the origin with index $\alpha_m\le 2s_m-1$  on any subspace $\Gamma^m\in \sL_1(h,m,n)$. 

\subsubsection{Case of a subspace belonging to $\sL_2(h,m,n)$}\label{Lambda2}

As we did in paragraph \ref{Lambda1}, we choose a basis $u_1,\dots,u_m\in \tU(m,n)$ spanning $\Gamma^m$ such that $u_1$ coincides with a normalized eigenvector of the Hessian associated to one of the null eigenvalues. Then, we have 
\begin{lemma}\label{Toulouse Lautrec}
	If $h$ is non-steep at the origin on $\Gamma^m\in \sL_2(h,m,n)$ at a given order $s_m\ge 2$, up to suitably changing the order of the vectors $u_1,\dots,u_m$ one has
	\begin{equation}
	\tT_0(h|_{\Gamma^m},r,m)\in \widehat \Sigma^1(r,s_m,m) \bigcap \sS_2^1(r,m)\ .
	\end{equation}
\end{lemma}
\begin{proof}
	By \eqref{Arnolfo}, and by the fact that the hessian of $h|_{\Gamma^m}$ has two or more null eigenvalues, we have that 
	\begin{equation}\label{Rema Namakula}
	\det \left(
	\begin{matrix}
	h^2_0[u_2,u_2]& h^2_0[u_2,u_3] & \dots & h^2_0[u_2,u_m]\\
	h^2_0[u_3,u_2]& h^2_0[u_3,u_3] & \dots & h^2_0[u_3,u_m]\\
	\dots\\
	h^2_0[u_m,u_2]& h^2_0[u_m,u_3] & \dots & h^2_0[u_m,u_m]\\
	\end{matrix}
	\right)= 0\ ,
	\end{equation}
	hence $	\tT_0(h|_{\Gamma^m},r,m)\in  \sS_2^1(r,m)$ following definition \eqref{S2}. 
	
	Moreover, $h$ is non-steep at the origin on $\Gamma^m$ at a given order $s_m\ge 2$ so that, by the discussions at section \ref{genericity} and at paragraph \ref{initial}, one has 
	\begin{equation}
	\tT_0(h|_{\Gamma^m},r,m)\in \widehat \Sigma(r,s_m,m)
	\end{equation}
	so that by the previous considerations we have
	\begin{equation}\label{Velazquez}
	\tT_0(h|_{\Gamma^m},r,m)\in \widehat \Sigma(r,s_m,m)\bigcap \sS_2^1(r,m)\ .
	\end{equation}
	By the above expression and by \eqref{sigma2}, we have that there must exist $i\in\{1,\dots,m\}$ such that $\tT_0(h|_{\Gamma^m},r,m)\in \widehat \Sigma^i(r,s_m,m)\bigcap \sS_2^1(r,m)$. As it was already discussed in the proof of Theorem C1, if $\tT_0(h|_{\Gamma^m},r,m)\in \widehat \Sigma^i(r,s_m,m)\bigcap \sS_2^1(r,m)$ then the vector $u_i$ must satisfy \eqref{Giorgione}. If $i=1$, there is nothing else to prove. If $i\neq 1$, it suffices to interchange the vector $u_1$ with the vector $u_i$. 
\end{proof}

Now, if for any $Q\in \Pollo$ we inject in the expressions\footnote{Contrary to the case studied in the previous paragraph, here these expressions cannot be found easily, in general.} determining the semi-algebraic subset $\fU^1\left(\left(\widehat \Sigma^1(r,s_m,m) \bigcap \sS_2^1(r,m)\right)\times \centinaunoh\right)\subset \widehat W^1(r,m)$ the explicit expression of the coefficients of the polynomial $\tT_{0,\ta}(Q|_{\Gamma^m},r,m)$, and we let the vectors $v,u_2,\dots,u_m$ appearing in \eqref{ma che riderone} vary in $\scV^1(m,n)$, we obtain a set 
$$
\scA_2(r,s_m,n,m)\subset \Pol\times\R^{m-1}\times \scV^1(m,n)\ .
$$ 

Moreover, by Lemma \ref{Toulouse Lautrec}, we have that if system
\begin{align}\label{Conte Mascetti}
\begin{split} 
&\begin{cases}
(a_{21},\dots,a_{m1})\in \overline B^{m-1}(\scK)\\
(u_1,\dots,u_m)\in \tU(m,n)\quad , \qquad v:=u_1+\sum_{i=2}^m a_{i1} u_i\\
\text{ Span }(v,u_2,\dots,u_m)=\Gamma^m\in \sL_2(h,m,n)\\
(\tT_{0}(h,r,n),a_{21},\dots,a_{m1},v,u_2,\dots,u_m)\in \scA_2(r,s_m,n,m)
\end{cases}
\end{split}
\end{align}
has no solution, then $h$ is steep at the origin with index $\alpha_m\le 2s_m-1$  on any subspace $\Gamma^m\in \sL_2(h,m,n)$. 

\subsection{Proof of Theorem C3}

\subsubsection{Construction of an atlas of analytic maps for the Grassmannian} 

It is well known that for any pair of positive integers $k,n$, with $k< n$, the Grassmannian $\tG(k,n)$ has the structure of a projective algebraic variety (see e.g. \cite{Bochnak_Coste_Roy_1998}, \cite{Perrin_2008}). In this subparagraph, for any positive integer $n\ge 3$ and for any $m\in\{2,\dots,n-1\}$ we will construct a suitable atlas of analytic maps for $\tG(m,n)$ by exploiting classic arguments of real-algebraic geometry. 

Namely, we fix two integers $n\ge 3$, and $m\in\{2,\dots,n-1\}$ and we consider a subset $I:=(i_1,\dots,i_m)\subset \{1,\dots,n\}$ of cardinality $m$, as well as its complementary $J:=(j_1,\dots,j_{n-m})$ of cardinality $n-m$ in $\{1,\dots,n\}$. 

We work in the euclidean space $\R^n$, and we fix a bilinear, symmetric, non-degenerate form $\mathsf{B}:\R^n\times \R^n\longrightarrow \R$. The Spectral Theorem ensures the existence of  an orthonormal basis - indicated by $e_1,\dots,e_n$ - that diagonalizes $\mathsf{B}$. Namely, in the basis $e_1,\dots,e_n$ the form $\mathsf{B}$ is represented by matrix 
\begin{equation}\label{Bi}
\mathsf{B}_{(e_1,\dots,e_n)}:=
\left(
\begin{matrix}
\alpha_1 & 0 & 0&\dots& 0 \\
0 & \alpha_2 & 0 & \dots & 0\\
\dots & \dots & \dots &\dots & \dots \\
0& 0 &0 &\dots & \alpha_n 
\end{matrix}
\right)\ ,
\end{equation}
where $\alpha_1\times \dots \times \alpha_n\neq 0$. 

We indicate by $E_I$ (resp. $E_J$) the $m$-dimensional subspace spanned by the vectors $(e_{i_1},\dots, e_{i_m})$ (resp. the $n-m$-dimensional subspace spanned by $e_{j_1},\dots,e_{j_{n-m}}$). One clearly has $\R^n=E_I\oplus E_J$. We also denote by $U_J$ the subset of $\tG(m,n)$ containing the $m$-dimensional subspaces which are supplementary for $E_J$. 

With this setting, we consider the cartesian product $E_J^m:=\stackrel{\text{$m$ times}}{\overbrace{E_J\times \dots \times E_J}}$ and we have that 
\begin{lemma}\label{Grassmannian_parametrization}
	The map
	\begin{equation}\label{effej}
	\mathscr{F}_J: E_J^m\longrightarrow U_J\qquad (w_1,\dots,w_m)\longmapsto \text{Span}(e_{i_1}+w_1,\dots,e_{i_m}+w_m)
	\end{equation}
	is bijective. 
\end{lemma}

\begin{proof} We proceed by steps. In the first two steps, we check that definition \eqref{effej} is well-posed. In Steps 3 and 4 we prove injectivity and surjectivity. 
	
	{\it Step 1.} We claim that for any choice of $(w_1,\dots,w_m)\in E_J^m$, the vectors $(e_{i_1}+w_1,\dots,e_{i_m}+w_m)$ are linearly independent. Infact, if by absurd there exist $m$ vectors $(w_1,\dots,w_m)\in E_J^m$ such that $(e_{i_1}+w_1,\dots,e_{i_m}+w_m)$ are not linearly independent, then there must be a vector $\lambda=(\lambda_1,\dots,\lambda_m)\in \R^m\backslash\{0\}$ satisfying $\sum_{\ell=1}^m\lambda_\ell(e_{i_\ell}+w_\ell)=0$, that is $\sum_{\ell=1}^m\lambda_\ell e_{i_\ell}=-\sum_{\ell=1}^m\lambda_\ell w_{i_\ell}$. 
	
	As $\sum_{\ell=1}^m\lambda_\ell e_{i_\ell}\in E_I$, and $-\sum_{\ell=1}^m\lambda_\ell w_{i_\ell}\in E_J$, and as $\R^n=E_I\oplus E_J$ by construction, by the assumptions one must have $-\sum_{\ell=1}^m\lambda_\ell w_{i_\ell}=\sum_{\ell=1}^m\lambda_\ell e_{i_\ell}=0$. The previous relation - together with the fact that the vectors $e_{i_1},\dots,e_{i_m}$ are linearly independent by hypothesis, yields $\lambda=0$, in contradiction with the initial assumption $\lambda\neq 0$.
	
	{\it Step 2. } We claim that for any choice of $(w_1,\dots,w_m)\in E^m_J$ one has $\text{ Span }(e_{i_1}+w_1,\dots,e_{i_m}+w_m)\in U_J$. 
	
	By absurd, we suppose that for some $(w_1,\dots,w_m)\in E^m_J$ there exist two non-zero vectors $u\in E_J$, $\sigma=(\sigma_1,\dots,\sigma_m)\in \R^m$, verifying $u=\sum_{\ell=1}^m\sigma_\ell(e_{i_\ell}+w_\ell)$, that is $\sum_{\ell=1}^m\sigma_\ell e_{i_\ell}=u-\sum_{\ell=1}^m\sigma_\ell w_\ell$. By construction, one has $\sum_{\ell=1}^m\sigma_\ell e_{i_\ell}\in E_I$, and $u-\sum_{\ell=1}^m\sigma_\ell w_\ell\in E_J$. Hence, as $\R^n=E_I\oplus E_J$, the previous formula yields $u-\sum_{\ell=1}^m\sigma_\ell w_\ell=\sum_{\ell=1}^m\sigma_\ell e_{i_\ell}=0$, which in turn implies $\sigma=0$, as the vectors $(e_{i_1},\dots,e_{i_m})$ are linearly independent. 
	
	Hence, one has $E_J\cap \text{ Span }(e_{i_1}+w_1,\dots,e_{i_m}+w_m)=\{0\}$. Therefore, since $\dim E_J=n-m$ and $\dim\left( \text{ Span }(e_{i_1}+w_1,\dots,e_{i_m}+w_m)\right)=m$ (the vectors $e_{i_1}+w_1,\dots,e_{i_m}+w_m$ are linearly independent by Step 1), the subspace given by  $\text{ Span }(e_{i_1}+w_1,\dots,e_{i_m}+w_m)$ is supplementary to $E_J$ and thus belongs to $U_J$. 
	
	{\it Step 3.} We prove that $\mathscr{F}_J$ is injective. By absurd, we suppose that there exists a subspace in $U_J$ which has two different pre-images. Namely, we suppose that there exist vectors $(u_1,\dots,u_m)\in E_J^m$, and $(w_1,\dots,w_m)\in E_J^m$, with $(u_1,\dots,u_m)\neq (w_1,\dots,w_m)$, satisfying $\mathscr{F}_J(u_1,\dots,u_m)=\mathscr{F}_J(w_1,\dots,w_m)$. Hence, as $e_{i_1}+u_1,\dots,e_{i_m}+u_m$ and $e_{i_1}+w_1,\dots,e_{i_m}+w_m$ generate the same subspace, for any $\ell\in\{1,\dots,m\}$ there must exist $\beta^\ell=(\beta_1^\ell, \dots,\beta_m^\ell)\in \R^m\backslash\{0\}$ such that
	
	\begin{equation}\label{dipendenza_lineare}
	e_{i_\ell}+w_\ell=\sum_{k=1}^m \beta^\ell_k (e_{i_k}+u_k)\ ,
	\end{equation}  that is
	\begin{equation}\label{dipendenza_lineare_2}
	e_{i_\ell}-\sum_{k=1}^m\beta^\ell_k e_{i_k}=\sum_{i=1}^m \beta^\ell_i u_i - w_\ell\ .
	\end{equation} 
	By construction one has $e_{i_\ell}-\sum_{k=1}^m\beta^\ell_k e_{i_k}\in E_I$ and $\sum_{i=1}^m \beta^\ell_i u_i - w_\ell\in E_J$, so due to \eqref{dipendenza_lineare_2} and to the direct sum $\R^n=E_I\oplus E_J$ one infers $\sum_{i=1}^m \beta^\ell_i u_i - w_\ell=e_{i_\ell}-\sum_{k=1}^m\beta^\ell_k e_{i_k}=0$. Since the vectors $e_{i_1},\dots,e_{i_m}$ are linearly independent, we finally obtain 
	\begin{equation}\label{Kronecker}
	\beta^\ell_k=\delta_{\ell k}\ .
	\end{equation} where $\delta_{\ell k }$ is the Kronecker symbol. Formulas \eqref{dipendenza_lineare} and \eqref{Kronecker} together imply that 
	\begin{equation}
	e_{i_\ell}+w_\ell=e_{i\ell}+ u_\ell \quad \Longleftrightarrow\quad w_\ell=u_\ell \qquad \forall \ell\in\{1,\dots,m\}\ ,
	\end{equation} 
	in contradiction with the hypothesis $(u_1,\dots,u_m)\neq (w_1,\dots,w_m)$. 
	
	{\it Step 4. } We prove that $\mathscr{F}_J$ is surjective. Consider a subspace $V\in U_J$. Since $V$ is supplementary of $E_J$, one has the direct sum $\R^n=V\oplus E_J$, and for any $\ell\in \{1,\dots,m\}$ there exist unique vectors $(v_\ell,w_\ell)\in V\times E_J$ such that $e_{i_\ell}=v_\ell-w_\ell$. Hence, for any $\ell\in\{1,\dots,m\}$ there exists a unique choice of vectors $(w_1,\dots,w_m)$ satisfying $v_\ell=e_{i_\ell}+w_\ell$. The vectors $v_1,\dots,v_m$ belong to $V$ by construction, and are linearly independent by Step 1. Therefore, since $\dim V=m$ by hypothesis ($V\in U_J$), one has $\text{Span}(v_1,\dots,v_m)=V$.   
\end{proof}

We indicate by $\mathscr{J}^{n-m}$ the subsets of cardinality $n-m$ in $\{1,\dots,n\}$. One has the following covering of the $m$-dimensional Grassmannian:
\begin{equation}\label{Magda}
\tG(m,n)=\bigcup_{\substack{J\in \mathscr{J}^{n-m}}}U_J\ .
\end{equation}
By construction, any vector $w_\ell\in E_J$ can be expressed uniquely as
\begin{equation}\label{vudoppio}
w_\ell= \sum_{k=1}^{n-m} M_{\ell k}e_{j_k}\ ,
\end{equation}
where $(M_{\ell k})_{\substack{\ell=1,\dots,m\\k=1,\dots,n-m}}$ is a real $m\times (n-m)$ matrix. By \eqref{Magda}, and by Lemma \ref{Grassmannian_parametrization}, there exists an atlas sending $\tG(m,n)$ to the open union
\begin{equation}\label{qudoppio terrazza}
\bigcup_{\substack{J\in \mathscr{J}^{n-m}}}\mathscr{F}_J^{-1}(U_J)=\bigcup_{\substack{J\in \mathscr{J}^{n-m}}} E_J^m \subset \R^{m\times(n-m)}\ .
\end{equation}

%It is a classic result of real-algebraic geometry (see \cite{Bochnak_Coste_Roy_1998}) that $\tG(m,n)$ has the structure of a projective algebraic variety. 

\subsubsection{Proof of Theorem C3}

Taking \eqref{Magda} into account, we fix $J\in \mathscr{J}^{n-m}$ together with its associated sets $E_J,U_J$. Let $V$ be a $m$-dimensional subspace belonging to $U_J$. By Lemma \ref{Grassmannian_parametrization}, one has 
\begin{equation}\label{Vuspan}
V=\text{ Span } \{e_{i_1}+w_1,\dots,e_{i_m}+w_m\}
\end{equation}
for a unique $(w_1,\dots,w_m) \in E_J^m$. 

Now, as $\tG_{\ge 1}(m,n)$ contains those subsets of $\tG(m,n)$ on which the restriction of the bilinear form $\mathsf{B}$ has at least one null eigenvalue, $V\in \tG_{\ge 1}(m,n)$ if and only if $\mathsf{B}$ is degenerated on $V$. Namely, taking \eqref{Vuspan} into account, $V\in \tG_{\ge 1}(m,n)$ iff there exists a vector $v=\sum_{\ell=1}^m v_\ell (e_{i_\ell}+w_\ell)\in V$ such that for all $\ell'\in\{1,\dots,m\} $ one has
\begin{align}\label{degenere}
\begin{split}
\mathsf{B}(v,e_{i_{\ell'}}+w_{\ell'})&=\sum_{\ell=1}^m v_\ell\, \mathsf{B}(e_{i_\ell}+w_\ell,e_{i_{\ell'}}+w_{\ell'})\\
&= \sum_{\ell=1}^m v_\ell\,\left( \mathsf{B}(e_{i_\ell},e_{i_{\ell'}})+\mathsf{B}(w_{\ell},w_{\ell'})\right)\\
& = v_{\ell'}\,\alpha_{i_{\ell'}}+\sum_{\ell=1}^m v_\ell\,\mathsf{B}(w_{\ell},w_{\ell'})\ .
\end{split}
\end{align}
To pass from the first to the second line in the above expression, we have taken into account the fact that $(w_1,\dots,w_m)\in E_J^m$, that $E_J=\text{Span}(e_{j_1},\dots,e_{j_{n-m}})$, and that the form $\mathsf{B}$ is diagonal in the basis $e_{1},\dots,e_{m}$; in the last passage, we have considered \eqref{Bi}. 
Setting
\begin{equation}
\mathscr{M}_\mathsf{B}=\mathscr{M}_\mathsf{B}(w_1,\dots,w_m):=
\left(
\begin{matrix}
\displaystyle\frac{1}{\alpha_{i_1}}\mathsf{B}(w_1,w_1) & \dots & \displaystyle\frac{1}{\alpha_{i_1}}\mathsf{B}(w_1,w_m)\\
\dots & \dots & \dots\\
\displaystyle\frac{1}{\alpha_{i_m}}\mathsf{B}(w_m,w_1) & \dots & \displaystyle\frac{1}{\alpha_{i_m}}\mathsf{B}(w_m,w_m)\\
\end{matrix}
\right)
\end{equation}
it is plain to check that \eqref{degenere} can be rewritten in the form
\begin{equation}\label{non so piu che scrivere}
\mathscr{M}_\mathsf{B}\, v=-v
\end{equation}
that is, we are asking for $-1$ to be an eigenvalue of $\mathscr{M}_\mathsf{B}$, hence \eqref{non so piu che scrivere} is equivalent to
\begin{equation}\label{ma lasciatemi in pace}
\det (\mathscr{M}_\mathsf{B}+\mathbb I_m)=0\ ,
\end{equation}  
where $\mathbb I_m$ is the $m\times m$ identity matrix. 

Since $\mathscr{M}_\mathsf{B}$ depends on $(w_1,\dots,w_m)\in E_J^m$ and since $(w_1,\dots,w_m)$ are in bijection with $\R^{m(n-m)}$ by \eqref{vudoppio}, the quantity $\det (\mathscr{M}_\mathsf{B}+\mathbb I_m)$ determines a polynomial map $\R^{m(n-m)}\longrightarrow \R$.  

If, by absurd, $\det (\mathscr{M}_\mathsf{B}+\mathbb I_m)$ is the null polynomial, then relation \eqref{ma lasciatemi in pace} holds on the whole inverse image $\mathscr{F}_J^{-1}(U_J)$. In particular, we observe that $(w_1=0,\dots,w_m=0)\in \mathscr{F}_J^{-1}(U_J)$ because $\mathscr{F}_J(0,\dots,0)=\text{Span}\{e_{i_1},\dots,e_{i_m}\}=E_I$ and $E_I$ is supplementary of $E_J$ by construction. Therefore, by the above reasonings one must have 
$$
\det (\mathscr{M}_\mathsf{B}(0,\dots,0)+\mathbb I_m)=\det \mathbb{I}_m=0
$$
which is clearly false. Consequently, the polynomial function $
\det (\mathscr{M}_\mathsf{B}(w_1,\dots,w_m)+\mathbb I_m)
$ is not identically null over $\R^{m(n-m)}$ and, due to Lemma \ref{Luca sei un mito}, its zero set is contained in a submanifold of codimension one in $\R^{m(n-m)}$. Hence, also the subset of degenerated subspaces of $U_J$ is contained in a submanifold of codimension one in $\tG(m,n)$. The reasoning can be repeated for all $J\in \mathscr{J}^{n-m}$. As, by its definition and by \eqref{Magda}, $\tG_1(m,n)$ is the finite union over $J\in \mathscr{J}^{n-m}$ of the degenerated subspaces of $U_J$, we have that $\tG_1(m,n)$ is contained in a submanifold of codimension one in $\tG(m,n)$. This proves point 1 of the statement. 

\medskip 

With the setting above, for any fixed $J\in \mathscr{J}^{n-m}$ we observe that a $m$-dimensional subspace $V'\in U_J$ belongs to $\tG_{\ge 2}(m,n)$ if and only if there exist at least two linearly independent vectors $v,u\in V'$ satisfying \eqref{non so piu che scrivere}. In particular, the subset $\tG_{\ge 2}(m,n)\cap U_J$ of "doubly-degenerated" subspaces of $U_J$ is contained in the intersection of $\tG_{\ge 1}(m,n)\cap U_J$ with the set 
\begin{equation}
T_J:=\{W\in U_J| \Delta(P_{\mathscr{M}_\mathsf{B}})=0\}
\end{equation}
where $P_{\mathscr{M}_\mathsf{B}}$ is the characteristic polynomial of matrix $\mathscr{M}_\mathsf{B}$, and $\Delta(P_{\mathscr{M}_\mathsf{B}})$ is its discriminant. By the above arguments and by \eqref{Magda}, point 2 of the statement follows if we manage to prove that $T_J$ is contained in a submanifold of codimension one in $\tG(m,n)$. The rest of the proof will be devoted to demonstrating this property. 

Clearly, by the same arguments used in the proof of point 1 of the statement, $\Delta(P_{\mathscr{M}_\mathsf{B}})$ is a polynomial function over $\R^{m(n-m)}$. If, by absurd, $\Delta(P_{\mathscr{M}_\mathsf{B}})$ is identically zero in $\R^{m(n-m)}$, then in particular it must be zero over the open set $\mathscr{F}^{-1}_J(U_J)$. 

Now, choose $m$ numbers $j_1,\dots,j_m\in J$, and consider the vectors
\begin{equation}
w'_1:=\sqrt{\left|\frac{\alpha_{i_1}}{\alpha_{j_1}}\right|}e_{j_1},\quad  w'_2:=\sqrt{2\left|\frac{\alpha_{i_2}}{\alpha_{j_2}}\right|}e_{j_2},\quad \dots,\quad w'_m:=\sqrt{m\left|\frac{\alpha_{i_m}}{\alpha_{j_m}}\right|}e_{j_m}\ ,
\end{equation}
which are well defined  by the fact that $\mathsf{B}$ is non-degenerate (see \eqref{Bi}). It is clear that $(w'_1,\dots,w'_m)\in E_J^m$, so that by Lemma \eqref{Grassmannian_parametrization} one has $\mathscr{F}_J(w'_1,\dots,w'_m)\in U_J$. As $\mathsf{B}$ is diagonal for the basis $e_1,\dots,e_n$ by hypothesis, matrix $\mathscr{M}_\mathsf{B}(w'_1,\dots,w'_m)$ in \eqref{Bi} reads

\begin{equation}
\left(
\begin{matrix}
\sgn(\alpha_{i_1})\sgn(\alpha_{j_1})& 0  & 0 &\dots & 0\\
0 & 2\sgn(\alpha_{i_2})\sgn(\alpha_{j_2}) & 0 &\dots  &0\\
\dots & \dots & \dots & \dots & \dots \\
0 & 0 &0 & \dots   &m\sgn(\alpha_{i_m})\sgn(\alpha_{j_m})\\
\end{matrix}
\right)\ ,
\end{equation}
and it is clear that the discriminant of the characteristic polynomial of this matrix cannot be zero, in contradiction with the fact that $\Delta(P_{\mathscr{M}_\mathsf{B}})=0$ on the whole set $\mathscr{F}_J^{-1}(U_J)$. Therefore, the polynomial function $\Delta(P_{\mathscr{M}_\mathsf{B}})$ is not identically null over $\R^{m(n-m)}$ and its zero set is contained in a submanifold of codimension one in $\R^{m(n-m)}$, by Lemma \ref{Luca sei un mito}. This proves that $T_J$ is contained in a submanifold of codimension one in $\tG(m,n)$, which concludes the proof.

\appendix

\section*{Tools of real-algebraic geometry}
The goal of this appendix is to provide the reader with an overview of some standard results of real-algebraic geometry that are used throughout the present work. The interested reader can find a complete exposition in \cite{Bochnak_Coste_Roy_1998} and \cite{Basu_Pollack_Roy_2006}. 

\section{Semi-algebraic sets and semi-algebraic functions}

\begin{defn}\label{union}
	A set $A\subset \R^n$ is said to be semi-algebraic if it can be written in the form
	\begin{equation}\label{moderna}
	\bigcup_{i=1}^s\{R_i(x)=0,Q_{i1}(x)<0,\dots,Q_{i\,r_i}(x)<0\}\ ,
	\end{equation}
	where $R_i,Q_{i1},...,Q_{i\,r_i}\in\R[x]$.
\end{defn}
%\begin{proof}
%	In definition \ref{semi-algebraic-set} fix $i=1,...,s$ and consider those polynomials $P_{ij}$, with $j\in\{1,...,r_s\}$, such that $P_{ij}(x)=0$ for any $x\in A$. The result follows immediately by setting $R_i(x):=\sum_{j=1}^{r_s} P_{ij}^2(x)$.
%\end{proof}

\begin{rmk}\label{Zariski}
	If only equalities are present in \eqref{moderna}, $A$ is said to be algebraic.  
\end{rmk}

It is clear that the polynomials generating a given semi-algebraic set $A$ are not uniquely determined, nor is their number. However, one can introduce a unique quantity associated to a semi-algebraic set, namely
\begin{defn}[Diagram]\label{diagram}
	For any semi-algebraic set $A\subset \R^n$, we denote by
	\begin{itemize}
		\item  $k_1(A)$ the minimal number of polynomials $R_{i}(X),Q_{ij}(X)\in \R[X]$ that are necessary in order to determine $A$ as in \eqref{moderna}.
		\item $k_2(A)$ the minimal value that the sum $\sum_{i=1}^s\deg(\widehat R_i)+\sum_{i=1}^s\sum_{j=1}^{r_i}\deg(\widehat Q_{ij}) $ can attain, with $\widehat R_{i}(X),\widehat Q_{ij}(X)\in \R[X]$ determining $A$ as in \eqref{moderna}. 
	\end{itemize}
	Following \cite{Bourgain_2005} (Def. 9.1), we call diagram of $A$ the quantity
	\begin{equation}\label{diagramma}
	\diag(A):=k_1(A)+k_2(A)\ .
	\end{equation}  
\end{defn}

Semi-algebraic sets are stable under projection, namely
\begin{thm}\label{Tarski_Seidenberg} (Tarski and Seidenberg, quantitative version)
	
	Take $n,m>0$ and let $A\subset \R^n\times \R^m$ be a semi-algebraic set. We indicate by $\Pi_n:\R^n\times \R^m\longrightarrow \R^m,(x,y)\longmapsto x$ the projector onto the first $n$ coordinates. Then, the set $\Pi_n(A)$ is semi-algebraic and its diagram depends only on $\diag(A)$, $m$ and $n$. 
\end{thm}
The classic versions of the Theorem of Tarski and Seidenberg do not usually make any reference to the diagram of the projected set. The statement given here can be found in \cite{Bourgain_2005} (Proposition 9.2) and its proof is contained in \cite{Basu_Pollack_Roy_1996}. 

The Theorem of Tarski and Seidenberg is fundamental in order to demonstrate the following results (see \cite{Bochnak_Coste_Roy_1998} for proofs)

\begin{prop}\label{complementary}
	The complementary of a semi-algebraic set $A\subset \R^n$ is semi-algebraic, and its diagram depends only on the diagram of $A$.
\end{prop}
\begin{prop}\label{closure-interior-boundary}
	The closure, the interior and the boundary of a semi-algebraic set $A\subset \R^n$ are semi-algebraic and their diagrams depend only on the diagram of $A$.
\end{prop}
\begin{prop}\label{dim_chiusura}
	Let $A$ be a semi-algebraic set of $\R^n$. Then, indicating with $\overline A$ the closure of $A$, one has that $\diag(\overline A)$ depends only on $\diag(A)$, and
	$$
	\dim A=\dim(\overline A)\ .
	$$
\end{prop}
The notion of semi-algebraicness can be easily extended to functions by making reference to their graphs, namely
\begin{defn}\label{funzione_semi_algebraica}
	Let $A\subset\R^n$ and $B\in\R^m$ be semi-algebraic sets. A map $\varphi:A\longrightarrow B$ is said to be semi-algebraic if $\text{graph}(\varphi)$ is a semi-algebraic set of $\R^n\times\R^m$.
\end{defn}
Semi-algebraic functions are piecewise algebraic, namely one has 
\begin{prop}\label{pollide}
	Let $A$ be a semialgebraic
	subset $A\subset \R^n$ and $\varphi : A\longrightarrow \R$ be a semi-algebraic function of diagram $d>0$. There exist a positive integer $M(d)$, a partition of $A$ into a finite number of
	semi-algebraic sets $A_i, i = 1,\dots , m$, with $m\le M(d)$, and for every value of $i$ there exists a polynomial $S_i (X,Y)$ in
	$n + 1$ variables such that, for every $x$ in $A_i$, $S_i(x,Y)$ is not identically zero
	and solves $S_i(x, \varphi(x)) = 0$.
	
\end{prop}
\begin{proof}
	Except for the existence of the bound $M=M(d)$, the proof can be found in ref. \cite{Bochnak_Coste_Roy_1998} (Lemma 2.6.3). Assume, by absurd, that the bound $M(d)$ does not exist; then, one can find a sequence $\{\varphi_j\}_{j\in \N}$ of semi-algebraic functions of $A$ with diagram $d$ such that, for each fixed $j\in \N$, the Proposition holds with the minimal number of pieces in the partition of $A$ being equal to $m_j$, and  $m_j\longrightarrow +\infty$. In particular, one can write $A=\sqcup_{i=1}^{m_j}A_{i,j}$ for any $j\in \N$, and there exist polynomials $S_{i,j}(X_1,\dots,X_n,Y)$ with the required properties. On the one hand, for any given $j\in \N$ and $i\in\{1,\dots,m_j\}$, by decomposition \eqref{moderna}, one has that $\text{  graph}(\varphi_j|_{A_{i,j}})$ is the finite union of sets of the kind
	\begin{align}\label{Rieti}
	\begin{split}
	&	\left\{(x,y)\in A_{i,j}\times \R\ | \ S_{i,j}(x,y)=0, Q_{i1,j}(x,y)<0,\dots, Q_{ir_i,j}(x,y)<0\right\}
	\end{split}
	\end{align} 
	for some polynomials $Q_{i1,j},\dots,Q_{ir_i,j}\in \R[X_1,\dots,X_n,Y]$. On the other hand, one has the disjoint union 
	\begin{equation}\label{Poggio_Bustone}
	\text{graph}(\varphi_j)=\bigsqcup_{i=1}^{m_j}\text{  graph}(\varphi_j|_{A_i,j})\ ,
	\end{equation}
	which is a consequence of the fact that $A=\sqcup_{i=1}^{m_j}A_{i,j} $ is a partition. 
	
	By Def. \ref{diagram}, formulas \eqref{Rieti}-\eqref{Poggio_Bustone}, and the fact that $m_j$ is the minimal number of pieces in the partition of $A$ for $\varphi_j$, one has $\diag (\varphi_j)\ge m_j$ and, for sufficiently high $j$, one has also $m_j>d$ since $m_j\longrightarrow +\infty$, in contradiction with the hypothesis $\diag(\varphi_j)= d$ for any $j\in \N$. This concludes the proof.

\end{proof}
An immediate consequence of Proposition \ref{pollide} is the following
\begin{cor}\label{pollo}
	Let $A\subset \R$ be an interval (finite or infinite) and $\varphi : A\longrightarrow \R$ be a semi-algebraic function of diagram $d>0$. There exist a positive integer $M(d)$ and an interval $\cI$ of length $|A|/M(d)$ over which the function $\varphi$ is algebraic, namely there exists a polynomial $S (X,Y)$ in
	$n + 1$ variables such that, for every $x$ in $\cI$, $S(x,Y)$ is not identically zero
	and solves $S(x, \varphi(x)) = 0$. 
\end{cor}
Among semi-algebraic functions, an important class is that of Nash functions:
\begin{defn}
	Let $A$ be an open semi-algebraic subset of $\R^n$. A semi-algebraic function	$\varphi:A\longrightarrow \R$  belonging to the $C^\infty$ class is said to be a Nash function. The set of Nash functions on $A$ is a ring under the usual operations of sum and function multiplication.
\end{defn}

Moreover, if we define analytic-algebraic functions as those real-analytic functions $f$ defined on an open semi-algebraic set $A\subset\R^n$ and satisfying $P(x,f(x))=0$ for some polynomial $P$ of $n+1$ variables and for all $x\in A$, it turns out that 
\begin{prop}[ref. \cite{Bochnak_Coste_Roy_1998}, Prop. 8.1.8]\label{Nash}
	A function $\varphi:A\longrightarrow \R$ is Nash on $A$ if and only if it is analytic-algebraic.
\end{prop}
Another important property of more general complex analytic-algebraic functions is stated in the following
\begin{prop}\label{p-valency}
	Let $D\subset \C$ be an open, bounded domain. An analytic-algebraic function $f:D\longrightarrow \C$, whose graph solves a polynomial $S\in \mathbb{C}[z,w]$ of degree $k\in \N$, is $k$-valent: that is, if $f$ is not constant then each value of ${\rm Im}(f)$ is the image of at most $k$ points in $D$.
\end{prop}
\begin{proof} 
	Assume, by absurd, that $f$ is non-constant and that there exists $w_0\in{\rm Im}(f)$ which is the image of at least $p>k$ points in $D$.  The polynomial $S^{w_0} (z):=S(z,w_0)$ would admit $p>k$ roots while ${\deg}(S^{w_0})\leq k$ by hypothesis. The Fundamental Theorem of Algebra ensures that $S^{w_0}$ must be identically zero and one has the factorization $S(z,w)=(w-w_0)^\alpha{\widehat S}(z,w)$, 
	where $\alpha\in\{1,...,k\}$, while $\widehat S$ cannot be divided by $(w-w_0)$ in $\C[z,w]$. Since $f$ is analytic and not constant, then the preimage $f^{-1}(\{ w_0\})$ is a finite set and the graph of $f$ must fulfill ${\widehat S}(z,f(z))=0$ out of $f^{-1}(\{ w_0\})$. By continuity, one has $\widehat S(z,f(z))=0$ on the whole domain of definition of $f$ since $f^{-1}(\{ w_0\})$ is finite. But $\deg \widehat S^{w_0}\le k$, with $\widehat S^{w_0}(z):=\widehat S(z,w_0)$, and $\widehat S^{w_0}$ admits more than $k$ roots, hence the previous argument ensures that $\widehat S$ can be divided by $(w-w_0)$, in contradiction with the construction.
\end{proof}

Semi-algebraicness is preserved by composition and inversion of semi-algebraic maps. In the following propositions, $A\subset \R^n$ and $B\subset \R^m$ are supposed to be semi-algebraic sets. 
\begin{prop}\label{immagine}
	Let $\varphi:A\longrightarrow B$ a semi-algebraic map. If $S\subset A$ and $T\subset \varphi(A)$ are semi-algebraic, so are $\varphi(S)$ and the inverse image $\varphi^{-1}(T)$. Moreover, their diagrams depend only on the diagram of $\text{graph}(\varphi)$.  
\end{prop}
\begin{prop}\label{composta}
	Let $f: A\longrightarrow B$ and $g:B\longrightarrow C$ be two semi-algebraic functions. Then $f\circ g$ is semi-algebraic and the diagram of its graph depends only on the diagram of $\text{graph}(f)$ and on the diagram of $\text{graph}(g)$. 
\end{prop}
\begin{prop}\label{inversa}
	Let $f: A\longrightarrow B$ be an injective semi-algebraic function. Then, its inverse $f^{-1}: f(A)\longrightarrow A$ is semi-algebraic and the diagram of its graph depends only on the diagram of $\text{graph}(f)$. 
\end{prop}
\begin{prop}\label{derivata}
	Let $I\subset \R$ be an open interval and  $f:I \longrightarrow \R$ be a semi-algebraic function differentiable
	in $I$. Then its derivative $f'$ is a semi-algebraic function and its diagram only depends on the diagram of $\text{graph}(f)$.
\end{prop}

We refer to \cite{Bochnak_Coste_Roy_1998} for the proofs of these statements. The dependence of the diagrams on the diagram of the initial function is, once again, a consequence, of the quantitative version \ref{Tarski_Seidenberg} of the Theorem  of Tarski and Seidenberg.

Finally, we give the following statement, which will prove to be helpful in our work
\begin{prop}{(see e.g. ref. \cite{Ha_Pham_2017}, pag. 23-24)}\label{vietnamiti}
	Let $f:A\longrightarrow\R$ and $g:A\longrightarrow\R^m$ be semi-algebraic functions and suppose that $f$ is bounded from below. Then
	$$
	\varphi: g(A)\longrightarrow\R \qquad y\longmapsto \inf_{x\in g^{-1}(y)}f(x)
	$$
	is semi-algebraic and the diagram of its graph depends only on the diagrams of $\text{graph}(f)$ and $\text{graph}(g)$. 
	
\end{prop}

\section{Analytic reparametrization of semi-algebraic sets}\label{riparametrizzazione}

Generally speaking, the reparametrization of a semi-algebraic set $A$ is a subdivision of $A$ into semi-algebraic pieces $A_j$ each of which is the image of a semi-algebraic function of the unit cube. On the one hand, it is possible to cover the whole of $A$ if one asks for the covering functions to be of finite regularity, with a uniform control on their derivatives (see \cite{Gromov_1987}). On the other hand, if one requires analyticity of the covering functions together with a uniform control on their derivatives, it is only possible to cover $A$ up to a "small" subset. 

%\begin{thm}{(Gromov \cite{Gromov_1987})}
%	Let $Y\subset[0,1]^n\subset \R^n $ be the zero set of a system of polynomials $P_1,...,P_k$ and denote $\ell:=\dim Y$. For each $r\in\N$ there exists an integer $N_0$ which only depends on $n,r$ and on $\sum_{i=1}^k \deg P_i$ and a family of algebraic $C^r$-maps $h_\nu:[0,1]^\ell\longrightarrow Y$, with $\nu=1,...,N_0$ whose images cover all of $Y$ and $\max_{1,...,\nu}||h_\nu||_{C^r([0,1]^\ell)}\le 1$. 

%\end{thm}

Hereafter, we state this result only in the case we need, that is for reparametrizations of graphs of algebraic functions, referring to \cite{Yomdin_2008} for the general theory. 

It is known that algebraic functions can only have two type of complex singularities: ramification points and poles (where the function may also ramify). If we denote by $d$ the diagram of an algebraic function, the number of its complex singularities is bounded by a quantity depending only on $d$ (see e.g. \cite{Barbieri_Niederman_2022}). It is exactly the neighborhoods of these singularities that cannot be analytically covered. 
\begin{defn}\label{reparametrization}
	Let $\delta>0$ and $g:I:=[-1,1]\longrightarrow \R$ be an algebraic function. An analytic $\delta$-reparametrization of $g$ consists of 
	\begin{itemize}
		\item A finite number $N$ of open subintervals $U_i$ of $I$, with $\text{length}(U_i)\le 2\delta$ for any $i=1,...,N$. 
		\item A partition of $I\backslash \cup_{i=1}^N U_i$ into a finite number of subsegments $\Delta_j$, $j=1,...,M$, together with a collection of real-analytic maps $\psi_j:I \longrightarrow \Delta_j$ such that for any $j\in\{1,...,N\}$
		\begin{enumerate}
			\item $\psi_j$ is an affine reparametrization of the segment $\Delta_j$. 
			\item $\psi_j$ and $g\circ\psi_j$ are both holomorphic in
			$
			\cD_3:=\{z\in\C: |z|\le 3\}\ .
			$
			\item $\psi_j$ and $g\circ\psi_j$ both satisfy a Bernstein inequality, that is 
			\begin{equation}\label{Berna}
			\max_{z\in \cD_3}|\psi_j(z)-\psi_j(0)|\le 1
			\qquad 
			\max_{z\in \cD_3}|g\circ\psi_j(z)-g\circ\psi_j(0)|\le 1\ .
			\end{equation}
		\end{enumerate}
	\end{itemize}
\end{defn}
\begin{thm}{(Yomdin, \cite{Yomdin_2008} Th. 3.2)}\label{yomdin}
	Let $d$ be a positive integer. There exist constants $\tY_1=\tY_1(d)$ and $\tY_2=\tY_2(d)$ such that for each algebraic function $g(x)$ of diagram $d$ defined on $I$ satisfying $0\le g(x)\le 1$ and for each $\delta>0$ there is an analytic $\delta$-reparametrization of $g$ with the number $N$ of the removed intervals bounded by $\tY_1$ and the number $M$ of the covering maps bounded by $\tY_2\log_2(1/\delta)$. Each of the removed intervals is centered at the real part of a complex singularity of $g(x)$. 
\end{thm}
Moreover, one has the following auxiliary result concerning the distance of each interval $\Delta_j$ to the singularities of the complex extension of the function $g$
\begin{prop}{(Yomdin, \cite{Yomdin_2008}, Lemma 3.6)}\label{larghezza_analiticita}
	The center of any of the segments $\Delta_j$ in the partition given by Theorem \ref{yomdin} is at distance no less than $3\times|\Delta_j|$ from any complex singularity of $g$. 
\end{prop}

\section*{Other auxiliary Results}

\section{Quantitative local inversion theorem}
We start by stating a Lipschitz inverse function Theorem. Its proof can be found, for example, in \cite{Garling_2014} (Th. 14.6.6). 

\begin{thm}\label{lip-inv}
	Let
	$U$ be an open subset of a Banach space $E$ and that $k : U \longrightarrow E$ is a
	Lipschitz mapping with constant $K < 1$. Set $h(x) = x + k(x)$. If the closed
	ball $\overline \tB_\eps(x)$ of radius $\eps$ around $x$ is contained in $U$ then
	$\overline \tB_{(1-K)\eps}(h(x)) \subset h(\overline \tB_\eps(x)) \subset \tB_{(1+K)\eps}(h(x))$.
	The mapping $h$ is a homeomorphism of $U$ onto $h(U)$, $h^{-1}$ is a Lipschitz
	mapping with constant $1/(1-K)$, and $h(U)$ is an open subset of $E$.
\end{thm}
This result is crucial in order to prove an analytic inverse function theorem, namely

\begin{thm}\label{loc-inv}
	Take a function $f\in C^\omega(\overline  \cD_R{(0)})$ and a point $z^*\in \cD_{R/2}{(0)}$ satisfying $f'(z^*)\neq 0$. 
	Then, $f$ is invertible in the closed disk $\overline \cD_{R'/16}(z^*)$ and its inverse $f^{-1}$ is analytic in $\overline\cD_{|f'(z^*)|R'/8}(f(z^*))$, where
	$$
	R':=\frac12\times  \min\left\{R,\displaystyle \frac{|f'(z^*)|}{ \max_{\overline \cD_R(0)}|f''|},\right\}\ .
	$$
	
\end{thm}
\begin{proof}
	We define  $h(z):=\displaystyle\frac{f(z)-f(z^*)}{f'(z^*)}+z^*$, $k(z):=h(z)-z$; both these functions are obviously holomorphic in $\overline\cD_{R/2}(z^*)$. Since $k'(z)=h'(z)-1=\displaystyle\frac{f'(z)}{f'(z^*)}-1=\frac{f'(z)-f'(z^*)}{f'(z^*)}$, one has $|k'(z)|=\displaystyle \frac{|f'(z)-f'(z^*)|}{|f'(z^*)|}\le\max_{\overline\cD_R(0)}|f''|\ \frac{|z-z^*|}{|f'(z^*)|}$. If we choose to consider only the $z\in \cD_{R'}(z^*)$, we obtain that $k$ is $\displaystyle\frac12-$Lipschitz on this set. 
	
	At this point, we exploit Theorem \ref{lip-inv} and we have that the function $h(z)=k(z)+z$ is a homeomorphism of $\cD_{R'}(z^*)$ onto its image. Moreover, one has $\overline{\cD}_{R'/8}(h(z^*))\subset h(\overline{\cD}_{R'/4}(z^*))\subset \overline{\cD}_{3R'/8}(h(z^*))$ and, since $h(z^*)=z^*$, this yields 
	\begin{equation}\label{imbriquee}
	\overline{\cD}_{R'/8}(z^*)\subset h(\overline{\cD}_{R'/4}(z^*))\subset \overline{\cD}_{3R'/8}(z^*)\ .
	\end{equation}
	
	We can define $f^{-1}$ by exploiting $h$ and its inverse, namely
	\begin{equation}\label{pallosa}
	z=h^{-1}(h(z))=h^{-1}\left(z^*+\displaystyle\frac{f(z)-f(z^*)}{f'(z^*)}\right)=:f^{-1}(f(z))\ 
	\end{equation}
	Indeed, by expressions \eqref{imbriquee} and \eqref{pallosa}, we see that, if we choose
	$$
	\left|z^*+\displaystyle\frac{f(z)-f(z^*)}{f'(z^*)}-z^*\right|\le \displaystyle \frac{R'}{8}\ ,
	$$ 
	that is $\left|f(z)-f(z^*)\right|\le |f'(z^*)|\displaystyle\frac{R'}{8}$, we have defined the inverse over the closed disc $\overline \cD_{|f'(z^*)|R'/8}(f(z^*))$. Finally, we prove that $$
	f(\overline\cD_{R'/16}(z^*))\subset \overline \cD_{|f'(z^*)|R'/8}(f(z^*))\ .
	$$ In order to see this, for $z\in \overline\cD_{R'/16}(z^*)$ we consider the identity
	$$
	f(z)-f(z^*)=f'(z^*)(z-z^*)+\int_0^1f''(tz+(1-t)z^*)\,(z-z^*)^2\,(1-t)dt
	$$
	that yields the estimate
	\begin{align}
	\begin{split}
	|f(z)-f(z^*)|\le &|f'(z^*)||z-z^*|+\max_{\overline \cD_R(0)}|f''|\,|z-z^*|^2\\
	\le & |f'(z^*)|\frac{R'}{16}+\max_{\overline \cD_R(0)}|f''|\frac{R'^2}{256}
	\le |f'(z^*)|\frac{R'}{8}\ ,
	\end{split}
	\end{align}
	where the last estimate is a consequence of the definition of $R'$.

	The fact that $f^{-1}$ inherits the same regularity of $f$ is a standard consequence of the classic local inversion theorem. 	
\end{proof}

\section{Two elementary properties of Lie Groups}
We refer to \cite{Lee_2013} (Cor. 21.6, Th. 21.10) for proofs. 
\begin{prop}\label{Azione}
	Every continuous action by a compact Lie group on a manifold is
	proper.
\end{prop}
\begin{thm}[Quotient manifold]\label{Quoziente}
	Suppose $G$ is a Lie group acting
	smoothly, freely, and properly on a smooth manifold $\cM$. Then the orbit space $\cM/G$
	is a topological manifold of dimension equal to $\dim\cM - \dim G$ and has a unique
	smooth structure with the property that the quotient map $\pi:\cM\longrightarrow \cM/G$ is a smooth
	submersion.
\end{thm}

\section{Three auxiliary Lemmas}

The following Lemma is an elementary criterion to establish when the projection of a closed set is still closed. 
\begin{lemma}\label{projection}
	Let $E$ be a metric space, $K$ a compact subset of some metric space and $\Delta$ a closed subset of $E\times K$. Then, the projection of $\Delta$ on $E$, indicated by $\Pi_E(\Delta)$, is closed. 
\end{lemma}
\begin{proof}
	Let $\{p_n\}_{n\in\mathbb{N}}$ be a sequence in $\Pi_E(\Delta)$ converging to a point $\bar{p}$ and $\{k_n\}_{n\in\mathbb{N}}$ a sequence in $K$ satisying $(p_n,k_n)\in\Delta$. Since $K$ is a compact subset of some metric space, one can extract a subsequence $\{k_{n_l}\}_{l\in\mathbb{N}}$ converging to a point $\bar{k}\in K$. Hence, the sequence $\{(p_{n_l},k_{n_l})\}_{l\in\mathbb{N}}$ in $\Delta$ converges to $(\bar{p},\bar{k})\in\Delta$, since $\Delta$ is closed. This implies that $\bar{p}$ belongs to $\Pi_E(\Delta)$, which is therefore closed.
\end{proof}

The following statement is a known Theorem due to Bézout (see \cite{Kendig_2012}, Th. 3.4a). 

\begin{lemma}\label{Th_Bezout}
	For any couple of positive integers $k_1,k_2$ consider two non-zero irreducible, non-proportional polynomials $Q_1\in \mathcal P(k_1)$ and $Q_2\in \mathcal P(k_2)$. Then the system $Q_1(z,w)=Q_2(z,w)=0$ has at most $k_1\times k_2$ solutions.
\end{lemma}

For the sake of completeness, we also state the following simple result on the codimension of the zero set of a non-null polynomial. 

\begin{lemma}\label{Luca sei un mito}
	Let $n$ be a positive integer. Consider a non-null real polynomial $P\in \R[x]$ of the variables $x=(x_1,\dots,x_n)\in \R^n$. The zero set $Z_P:=\{x\in \R^n| P(x)=0\}$ is contained in a submanifold of codimension one in $\R^n$.

\end{lemma}

\begin{proof}
	
	If $P$ is a non-zero constant, then there is nothing to prove.
	
	If $P$ is non-constant, the proof is by induction on the degree of $P$. 
	
	If $\deg P=1$, then $Z_P$ is a hyperplane, which is obviously a submanifold of codimension one.
	
	Suppose, now,  that the statement is true for polynomials of degree $k-1\ge 1$. Consider a polynomial of degree $k$, together with its associated open set of non-critical points $S_P:=\{x\in \R^n \,|\, \grad P(x)\neq 0\}$. On the one hand, locally around any point of $Z_P\cap S_P$ one can apply the implicit function theorem, so that $Z_P\cap S_P$ is indeed a submanifold of codimension one in $\R^n$. On the other hand, $\R^n\backslash S_P$ is the common zero set of the $n$ polynomials $\partial P/\partial x_1,\dots, \partial P/\partial x_n$; moreover, since $\deg P=k\ge 2$, at least one among $\partial P/\partial x_1,\dots, \partial P/\partial x_n$ has degree $k-1\ge 1$. Hence, by hypothesis, $\R^n\backslash S_P$ is  contained in a submanifold of codimension one in $\R^n$. 
	
	This proves that the set $Z_P\cap (\R^n\backslash S_P)\subset (\R^n\backslash S_P)$ is contained in a submanifold of codimension one in $\R^n$. Obviously, the thesis follows by the fact that $$
	Z_P=(Z_P\cap (\R^n\backslash S_P))\cup (Z_P\cap S_P)\ .
	$$

\end{proof}

\section*{Acknowledgements} I am extremely grateful to my PhD advisors, Laurent Niederman and Luca Biasco, for putting me on this subject, for their useful suggestions and comments, and for reading and correcting the present manuscript. I am also grateful to J.P. Marco for very useful discussions. In the months preceeding the end of the redaction of this work, I have been funded by the ERC project 757802 {\it Haminstab}; therefore, I wish to acknowledge both the ERC and the PI of the project (prof. M. Guardia) for their support. 

\bibliographystyle{plain}
\bibliography{Steep_2}

\end{document}